\documentclass{article}%
\usepackage{hyperref}
\usepackage{amsmath}
\usepackage{bm}
\usepackage{graphicx}
\usepackage{amsfonts}
\usepackage{amssymb}
\usepackage{geometry}
\usepackage{cite}
\usepackage{color}
\usepackage{amscd}%
\providecommand{\U}[1]{\protect\rule{.1in}{.1in}}
\newtheorem{theorem}{Theorem}

\newtheorem{corollary}[theorem]{Corollary}

\newtheorem{definition}[theorem]{Definition}
\newtheorem{example}[theorem]{Example}

\newtheorem{lemma}[theorem]{Lemma}

\newtheorem{notation}[theorem]{Notation}

\newtheorem{proposition}[theorem]{Proposition}
\newtheorem{remark}[theorem]{Remark}

\newenvironment{proof}[1][Proof]{\noindent\textbf{#1.} }{\ \rule{0.5em}{0.5em}}
\newcommand{\GraphicsDirectory}{./}
\newcommand{\executeiffilenewer}[3]{%
\ifnum\pdfstrcmp{\pdffilemoddate{#1}}%
{\pdffilemoddate{#2}}>0%
{\immediate\write18{#3}}\fi%
}

\graphicspath{{\GraphicsDirectory}}

\def\svgwidth{2in}

\newcommand{\psize}[1]{\def\svgwidth{#1}}
\makeatletter
\def\blfootnote{\gdef\@thefnmark{}\@footnotetext}
\makeatother
\begin{document} 
\author{Bruce K. Driver\thanks{The research of Bruce K. Driver was 
supported in part by NSF Grant DMS-0739164. The author also gratefully 
acknowledges the generosity and hospitality of the Imperial College 
Mathematics Department, where part of this research was carried out, and was supported by EPSRC grant
EP/M00516X/1 and a Nelder Visiting Fellowship during the Fall of 2014.}
\and Jeremy S. Semko\thanks{The research of Jeremy S. Semko was
supported in part by NSF Grant DMS-0739164.}}
\title{Controlled Rough Paths on Manifolds I}
\maketitle

\begin{abstract}In this paper, we build the foundation for a theory of controlled rough paths
on manifolds. A number of natural candidates for the definition of manifold
valued controlled rough paths are developed and shown to be equivalent. The
theory of controlled rough one-forms along such a controlled path and their
resulting integrals are then defined. This general integration theory does
require the introduction of an additional geometric structure on the manifold
which we refer to as a \textquotedblleft parallelism.\textquotedblright\ The
transformation properties of the theory under change of parallelisms is
explored. Using these transformation properties, it is shown that the
integration of a smooth one-form along a manifold valued controlled rough path
is in fact well defined independent of any additional geometric structures. We
present a theory of push-forwards and show how it is compatible with our
integration theory. Lastly, we give a number of characterizations for solving
a rough differential equation when the solution is interpreted as a controlled
rough path on a manifold and then show such solutions exist and are unique.

\end{abstract}
\tableofcontents


\numberwithin{theorem}{section} \numberwithin{equation}{section}

\section{Introduction\label{sec.1}}

In a series of papers \cite{lyons-94,lyons-95,lyons-98}, Terry Lyons
introduced and developed the far reaching theory of rough path analysis. This
theory allows one to solve (deterministically) differential equations driven
by rough signals at the expense of \textquotedblleft
enhancing\textquotedblright\ the rough signal with some additional
information. Lyons' theory has found numerous applications to stochastic
calculus and stochastic differential equations, for example see \cite{CF},
\cite{CFV}, \cite{CHLT}, \cite{CLL}, and the references therein. For some more
recent applications, see \cite{Blanchet}, \cite{Kelly}, \cite{Inahama},
\cite{Delaure} , and \cite{Boe}.

The rough path theory mentioned above has been almost exclusively developed in
the context of state spaces being either finite or infinite dimensional Banach
spaces with the two exceptions of \cite{Cass2012a} and \cite{CDL13}. In
\cite{Cass2012a}, a version of manifold valued rough paths is developed in the
context of \textquotedblleft currents,\textquotedblright\ while in
\cite{CDL13} the authors develop a more concrete theory by working with
embedded submanifolds.

The purpose of this paper is to define and develop a third interpretation of
rough paths on manifolds based on Gubinelli's \cite{gubinelli} notions of
\textquotedblleft controlled\textquotedblright\ rough paths. As Gubinelli's
perspective has proved extremely useful in the flat case (most notably see
Hairer \cite{Hairer2014}), it is expected such a theory of controlled rough
paths on manifolds can give new insights as well as applications to the
existing literature. We now will present a brief summary of the results
contained in this paper.

\subsection{Summary of Results\label{sub.1.1}}

Let $M^{d}$ be a $d$ -- dimensional manifold, $\mathbf{X}_{s,t}:=1+x_{s,t}%
+\mathbb{X}_{s,t}$ be a weak-geometric rough path in $W:=\mathbb{R}^{k}$ with
$1\leq p<3.$ A \emph{rough path controlled by} $\mathbf{X}$ on $M$ (see
Definition \ref{def.2.35}) is a pair of continuous functions $y:\left[
0,T\right]  \rightarrow M,$ and $y^{\dag}:\left[  0,T\right]  \rightarrow
L\left(  W,TM\right)  $ such that (somewhat imprecisely speaking) for all
$0\leq s\leq t\leq T;$
\[
\text{1)}~y_{s}^{\dag}:W\rightarrow T_{y_{s}}M,\quad\text{2) }\psi\left(
y_{s},y_{t}\right)  =y_{s}^{\dag}x_{s,t}+O\left(  \left\vert x_{s,t}%
\right\vert ^{2}\right)  ,\text{ and 3) }U\left(  y_{s},y_{t}\right)
y_{t}^{\dag}-y_{s}^{\dag}=O\left(  \left\vert x_{s,t}\right\vert \right)  ,
\]
where $\psi$ is a \textquotedblleft logarithm\textquotedblright\ on $M$ (see
Definition \ref{def.2.15}) and $U$ is a \textquotedblleft
parallelism\textquotedblright\ on $M$ (see Definition \ref{def.2.16}). [When
$M=\mathbb{R}^{d},$ one identifies all tangent spaces in which case one
typically takes $U\left(  m,n\right)  =I$ and $\psi\left(  m,n\right)  =n-m.]$
The pair $\mathcal{G}:=\left(  \psi,U\right)  $ is called a gauge.
Alternatively one can define controlled rough paths locally via a chart $\phi$
by requiring (see Definition \ref{def.2.40})
\[
\phi\left(  y_{t}\right)  -\phi\left(  y_{s}\right)  -d\phi\circ y_{s}^{\dag
}x_{s,t}=O\left(  \left\vert x_{s,t}\right\vert ^{2}\right)  \text{ and }%
d\phi\circ y_{t}^{\dag}-d\phi\circ y_{s}^{\dag}=O\left(  \left\vert
x_{s,t}\right\vert \right)  .
\]
It is shown in Theorem \ref{the.2.44} that these two notions of controlled
rough paths agree. Moreover, these manifold-valued rough paths may also be
characterized as pairs $\mathbf{y}=\left(  y,y^{\dag}\right)  $ whose
\textquotedblleft push-forwards\textquotedblright\ under smooth real-valued
functions are controlled rough paths on $\mathbb{\mathbb{R}}$ (See Theorem
\ref{the.2.57}). 

Two natural examples of manifold valued controlled rough paths are as follows.
1) If $M^{d}$ is an embedded submanifold (see subsection \ref{sub.2.6}) and
the path $x_{s}\in W$ happens to lie in $M$ (i.e. $x_{s}\in M$ for all $s$ in
$\left[  0,T\right]  ),$ then $\left(  x_{s},P(x_{s})\right)  $ is an $M$ --
valued rough path controlled by $\mathbf{X}$ where $P\left(  m\right)  $ is
orthogonal projection onto $T_{m}M$ (see Example \ref{exa.2.55}). 2) If
$f:W\rightarrow M^{d}\subseteq\mathbb{R}^{\tilde{k}}$ is smooth, then $\left(
f\left(  x_{s}\right)  ,f^{\prime}\left(  x_{s}\right)  \right)  $ is a rough
path controlled by $\mathbf{X}$ (see Example \ref{exa.2.56}).

Now let $\mathcal{G}=\left(  \psi,U\right)  $ be a gauge, $V$ be a Banach
space, and $\mathbf{y}=\left(  y,y^{\dag}\right)  $ be an $M$ -- valued
controlled rough path as above. A pair of continuous functions $\alpha:\left[
0,T\right]  \rightarrow L\left(  TM,V\right)  $ and $\alpha^{\dag}:\left[
0,T\right]  \rightarrow L\left(  W\otimes TM,V\right)  $ is a $U-$%
\emph{controlled (rough) one-form along }$y$ with values in a Banach space $V$
provided (see Definition \ref{def.3.1} for details);

\begin{enumerate}
\item $\alpha_{s}:T_{y_{s}}M\rightarrow V$ for all $s,$

\item $\alpha_{s}^{\dag}:W\otimes T_{y_{s}}M\rightarrow V$ for all $s,$

\item $\alpha_{t}\circ U\left(  y_{t},y_{s}\right)  -\alpha_{s}-\alpha
_{s}^{\dag}\left(  x_{s,t}\otimes\left(  \cdot\right)  \right)  =O\left(
\left\vert x_{s,t}\right\vert ^{2}\right)  $, and

\item $\alpha_{t}^{\dag}\circ\left(  I\otimes U\left(  y_{t},y_{s}\right)
\right)  -\alpha_{s}^{\dag}=O\left(  \left\vert x_{s,t}\right\vert \right)  .$
\end{enumerate}

To abbreviate notation we write $\boldsymbol{\alpha}_{s}=\left(  \alpha
_{s},\alpha_{s}^{\dag}\right)  .$ As an example, if $\alpha\in\Omega
^{1}\left(  M,V\right)  $ is a smooth one-form on $M$ and $U$ is a parallelism
it is shown in Proposition \ref{pro.4.2} how to construct $\alpha_{s}^{\dag
U}$ so that $\boldsymbol{\alpha}_{s}^{U}=\left(  \alpha_{s}:=\alpha
|_{T_{y_{s}}M},\alpha_{s}^{\dag U}\right)  $ is a $U-$controlled (rough)
one-form along $y.$

Theorem \ref{the.3.21} below constructs the integral, $\int\left\langle
\boldsymbol{\alpha}\mathbf{,}d\mathbf{y}^{\mathcal{G}}\right\rangle ,$ of
$\boldsymbol{\alpha}$ along $\mathbf{y}=\left(  y,y^{\dag}\right)  .$ This
integral is a standard flat $V-$valued controlled rough path along
$\mathbf{X}$ which, as the notation suggests, a priori depends on a choice of
gauge, $\mathcal{G}=\left(  \psi,U\right)  .$ However, it is shown in
Corollary \ref{cor.3.31} that the integral actually only depends on the
parallelism, $U.$ In Theorem \ref{the.3.24} (also see Proposition
\ref{pro.3.6}), we show that the integral, $\int\left\langle
\boldsymbol{\alpha}\mathbf{,}d\mathbf{y}^{\mathcal{G}}\right\rangle ,$
satisfies a basic but useful associativity property.

In Theorem \ref{the.3.32}, it is shown that there are \textquotedblleft
natural\textquotedblright\ transformations relating all of the above
structures under change of parallelism, $U\rightarrow\tilde{U},$ in such a way
that the integral, $\int\left\langle \boldsymbol{\alpha}\mathbf{,}%
d\mathbf{y}^{\mathcal{G}}\right\rangle ,$ is preserved. Consequently, as shown
in Theorem \ref{the.4.3}, if $\alpha\in\Omega^{1}\left(  M,V\right)  $ is a
smooth one-form on $M$ with values in $V$ and $\boldsymbol{\alpha}_{s}%
^{U}=\left(  \alpha_{s}:=\alpha|_{T_{y_{s}}M},\alpha_{s}^{\dag U}\right)  $ is
the associated $U-$controlled (rough) one-form along $y,$ then the resulting
integral $\int\left\langle \boldsymbol{\alpha}^{U}\mathbf{,}d\mathbf{y}%
^{\mathcal{G}}\right\rangle $ is in fact independent of both the parallelism,
$U,$ and the logarithm, $\psi$, used in the construction. Therefore we may
simply denote the resulting integral by $\int\alpha\left(  d\mathbf{y}\right)
.$ A gauge independent formula for this integral using charts is given in
Corollary \ref{cor.4.7}. It is shown in Theorem \ref{the.4.15} that if
$\alpha\in\Omega^{1}\left(  M,V\right)  $ is a smooth one-form on $M$ and
$f:M\rightarrow\tilde{M}$ is a smooth map between two manifolds, then%
\[
\int\left(  f^{\ast}\alpha\right)  \left(  d\mathbf{y}\right)  =\int%
\alpha\left(  d\left(  f_{\ast}\mathbf{y}\right)  \right)  ,
\]
where $f_{\ast}\mathbf{y:}=\left(  f\circ y,f_{\ast}\circ y_{s}^{\dag}\right)
$ is the \textquotedblleft push-forward\textquotedblright\ of $\mathbf{y}$ by
$f$ (see Definition \ref{def.4.11}) and $f^{\ast}\alpha\in\Omega^{1}\left(
\tilde{M},V\right)  $ is the pull-back of $\alpha.$

In Section \ref{sub.5.1}, we discuss the notion of a controlled rough path
$\mathbf{y}=\left(  y,y^{\dag}\right)  $ solving the rough differential
equation (RDE)%
\[
d\mathbf{y}_{t}=F_{d\mathbf{X}}\left(  y_{t}\right)  \text{\quad with\quad
}y_{0}=\bar{y}_{0}%
\]
when $F:W\rightarrow\Gamma\left(  TM\right)  $. Essentially $\mathbf{y}$ will
solve such an equation if the path $y$, when pushed forward by any smooth
function $f$, has the correct \textquotedblleft Taylor
expansion\textquotedblright\ and $y^{\dag}$ is the correct derivative, i.e.
$y_{s}^{\dag}=F_{\left(  \cdot\right)  }\left(  y_{s}\right)  $ (see
Definition \ref{def.5.2}). Theorem \ref{the.5.3} then compiles a list of
alternative characterizations for solving an RDE both by approximating
solutions and by relating them to familiar flat space rough integrals. Next,
the existence and uniqueness of solutions are proved in Theorem \ref{the.5.4}
and Theorem \ref{the.5.5}. Lastly, in Theorem \ref{the.5.9}, we record what it
means to solve an RDE when one takes the gauge perspective.

In a sequel to this paper, we will develop notions of parallel translation
along a controlled rough path along with rough version of Cartan's rolling and
unrolling maps in order to characterize all controlled rough paths on $M.$

\section{Definitions of Controlled Rough Paths with Examples\label{sec.2}}

\subsection{Review of Euclidean Space Rough Paths\label{sub.2.1}}

The presentation here will be brief. For a more thorough development, the
reader can refer to many sources, for example \cite{FH2014} or \cite{FV}.

Throughout this paper, we denote $W=\mathbb{R}^{k}$. Let $1\leq p<3$ and let
\begin{equation}
\Delta_{\left[  S,T\right]  }=\left\{  \left(  s,t\right)  :S\leq s\leq t\leq
T\right\}  . \label{equ.2.1}%
\end{equation}

\begin{definition}
\label{def.2.1}A \textbf{control} $\omega$ is a continuous function
$\omega:\Delta_{\left[  0,T\right]  }\rightarrow\mathbb{R}_{+}$ which is
superadditive\footnote{To say $\omega$ is superadditive means $\omega\left(
s,t\right)  +\omega\left(  t,u\right)  \leq\omega\left(  s,u\right)  $ for all
$0\leq s\leq t\leq u\leq T.$} and such that $\omega\left(  s,s\right)  =0$ for
all $s\in\left[  0,T\right]  $.
\end{definition}

\begin{definition}
\label{def.2.2}Let $\mathbf{X}=\left(  x,\mathbb{X}\right)  $ where
\[
x:\left[  0,T\right]  \rightarrow W\quad\quad\text{and}\quad\quad
\mathbb{X}:\Delta_{\left[  0,T\right]  }\rightarrow W\otimes W
\]
and are continuous. Then $\mathbf{X}$ is a\textbf{ }$p$\textbf{-rough path
with control }$\omega$ if

\begin{enumerate}
\item The Chen identity holds:%
\begin{equation}
\mathbb{X}_{s,u}=\mathbb{X}_{s,t}+\mathbb{X}_{t,u}+x_{s,t}\otimes x_{t,u}
\label{equ.2.2}%
\end{equation}
for all $0\leq s\leq t\leq u\leq T$ where $x_{s,t}:=x_{t}-x_{s}$.

\item For all $0\leq s\leq t\leq T$,
\begin{equation}
\left\vert x_{s,t}\right\vert \leq\omega\left(  s,t\right)  ^{1/p}\quad
\quad\text{and}\quad\quad\left\vert \mathbb{X}_{s,t}\right\vert \leq
\omega\left(  s,t\right)  ^{2/p}. \label{equ.2.3}%
\end{equation}

\end{enumerate}

Further, we say that $\mathbf{X}$ is \textbf{weak-geometric} if the symmetric
part of $\mathbb{X}_{s,t}$ ($\operatorname{sym}\left(  \mathbb{X}%
_{s,t}\right)  $) satisfies the relation
\[
\operatorname{sym}\left(  \mathbb{X}_{s,t}\right)  =\frac{1}{2}x_{s,t}\otimes
x_{s,t}.
\]

\end{definition}

\begin{notation}
\label{not.2.3}Let $F_{s,t}$ and $G_{s,t}$ be a pair of functions into a
normed space. When it is not important to keep careful track of constants we
will often write $F_{s,t}\underset{^{i}}{\approx}G_{s,t}$ (for any
$i\in\mathbb{N)}$ to indicate that there exists $C<\infty$ and $\delta>0$ such
that%
\[
\left\vert F_{s,t}-G_{s,t}\right\vert \leq C\omega\left(  s,t\right)
^{i/p}\text{ for all }0\leq s\leq t\leq T\text{ with }\left\vert
t-s\right\vert \leq\delta.
\]

\end{notation}

In this paper, $V$,$\tilde{V}$, and $\hat{V}$ will denote Banach spaces, and
$L\left(  V,\tilde{V}\right)  $ will denote the bounded linear transformations
from $V$ to $\tilde{V}$.

\begin{example}
\label{exa.2.4}If $x\left(  t\right)  \in C^{\infty}\left(  \left[
0,T\right]  ,V\right)  $ is a smooth curve in $V$ and
\begin{equation}
\mathbb{X}_{s,t}=\int_{s\leq u\leq v\leq t}dx_{u}\otimes dx_{v}=\int_{s}%
^{t}x_{s,v}\otimes dx_{v}, \label{equ.2.4}%
\end{equation}
then $\mathbf{X}=\left(  x,\mathbb{X}\right)  $ is a weak-geometric rough path
controlled by $\omega\left(  s,t\right)  =\left\vert t-s\right\vert .$ In this
example we could take even take $p=1.$
\end{example}

\begin{definition}
\label{def.2.5}Let $\mathbf{X}$ be a $p$-rough path on $W\oplus W^{\otimes2}$
with control $\omega$. The continuous pair $\mathbf{y}:=\left(  y,y^{\dag
}\right)  \in C\left(  \left[  a,b\right]  ,V\right)  \times C\left(  \left[
a,b\right]  ,L\left(  W,V\right)  \right)  $ is a $V$ -- valued\textbf{ rough
path controlled by }$\mathbf{X}$ (denoted $\mathbf{y\in}CRP_{\mathbf{X}%
}\left(  \left[  a,b\right]  ,V\right)  $) if there exists a $C$ such that

\begin{enumerate}
\item $\left\vert y_{t}-y_{s}-y_{s}^{\dag}x_{s,t}\right\vert \leq
C\omega\left(  s,t\right)  ^{2/p}$, and

\item $\left\vert y_{t}^{\dag}-y_{s}^{\dag}\right\vert \leq C\omega\left(
s,t\right)  ^{1/p}$ for all $s\leq t$ in $\left[  0,T\right]  $.
\end{enumerate}

We denote $CRP_{\mathbf{X}}\left(  V\right)  :=CRP_{\mathbf{X}}\left(  \left[
0,T\right]  ,V\right)  $ for some fixed $T<\infty$.
\end{definition}

The approximations in Definition \ref{def.2.5} are statements which only need
to hold locally because of the following (easy) sewing lemma.

\begin{lemma}
[Sewing Lemma]\label{lem.2.6}Let $\mathbf{y}:=\left(  y,y^{\dag}\right)  \in
C\left(  \left[  0,T\right]  ,V\right)  \times C\left(  \left[  0,T\right]
,L\left(  W,V\right)  \right)  $ and let $0=t_{0}<t_{1}<\ldots<t_{l}=T$ be a
partition of $\left[  0,T\right]  $ such that $\mathbf{y}|_{\left[
t_{i},t_{i+i}\right]  }$ is a rough path controlled by $\mathbf{X}%
|_{[t_{i},t_{i+1}]}:=\left(  x|_{\left[  t_{i},t_{i+1}\right]  }%
,\mathbb{X}|_{\Delta_{\left[  t_{i},t_{i+1}\right]  }}\right)  $ for all
$0\leq i\leq l-1.$ Then $\mathbf{y}$ is a rough path controlled by
$\mathbf{X}$.
\end{lemma}

\begin{proof}
Let $C_{i}$ with $0\leq i\leq l-1$ be such that%
\[
\left\vert y_{t}-y_{s}-y_{s}^{\dag}x_{s,t}\right\vert \leq C_{i}\omega\left(
s,t\right)  ^{2/p}\quad\text{and}\quad\left\vert y_{t}^{\dag}-y_{s}^{\dag
}\right\vert \leq C_{i}\omega\left(  s,t\right)  ^{1/p}%
\]
whenever $\left(  s,t\right)  \in\Delta_{\left[  t_{i},t_{i+1}\right]  }$. Let
$\tilde{C}:=\sum_{i=0}^{l-1}C_{i}$. Then by a telescoping series argument and
the fact that $\omega$ is increasing (because it is superadditive), it is
clear that
\[
\left\vert y_{t}^{\dag}-y_{s}^{\dag}\right\vert \leq\tilde{C}\omega\left(
s,t\right)  ^{1/p}~\forall~\left(  s,t\right)  \in\Delta_{\left[  0,T\right]
}.
\]

Now let $C=\left(  2l-1\right)  \tilde{C}.$ If $\left(  s,t\right)  \in
\Delta_{\left[  0,T\right]  }$ then there exists $j$ and $j^{\ast}$ such that
$s\in\lbrack t_{j},t_{j+1}]$ and $t\in\left[  t_{j^{\ast}},t_{j^{\ast}%
+1}\right]  $ with $j\leq j^{\ast}$. If $j=j^{\ast}$ then
\[
\left\vert y_{t}-y_{s}-y_{s}^{\dag}x_{s,t}\right\vert \leq C\omega\left(
s,t\right)  ^{2/p}%
\]
trivially. Otherwise, we have%
\begin{align*}
y_{t}-y_{s}-y_{s}^{\dag}x_{s,t}  &  =\left(  y_{t}-y_{t_{j^{\ast}}}\right)
+\left(  y_{t_{j+1}}-y_{s}\right)  +\sum_{i=j+1}^{j^{\ast}-1}\left(
y_{t_{i+1}}-y_{t_{i}}\right) \\
&  \quad-y_{s}^{\dag}x_{s,t_{j+1}}-y_{s}^{\dag}x_{t_{j^{\ast}},t}-\sum
_{i=j+1}^{j^{\ast}-1}y_{s}^{\dag}x_{t_{i},t_{i+1}}\\
&  =\left(  y_{t}-y_{t_{j^{\ast}}}-y_{t_{j^{\ast}}}^{\dag}x_{t_{j^{\ast}}%
,t}\right)  +\left(  y_{t_{j+1}}-y_{s}-y_{s}^{\dag}x_{s,t_{j+1}}\right)
+\left[  y_{t_{j^{\ast}}}^{\dag}-y_{s}^{\dag}\right]  x_{t_{j^{\ast}},t}\\
&  \quad+\sum_{i=j+1}^{j^{\ast}-1}\left(  y_{t_{i+1}}-y_{t_{i}}-y_{t_{i}%
}^{\dag}x_{t_{i},t_{i+1}}\right)  +\sum_{i=j+1}^{j^{\ast}-1}\left[
y_{s}^{\dag}-y_{t_{i}}^{\dag}\right]  x_{t_{i},t_{i+1}}.
\end{align*}

Taking absolute values and using the fact that $\omega$ is superadditive, we
have that the absolute value of each term on the right (including those within
the summations) is bounded by $\tilde{C}\omega\left(  s,t\right)  ^{2/p}$.
Thus
\begin{align*}
\left\vert y_{t}-y_{s}-y_{s}^{\dag}x_{s,t}\right\vert  &  \leq\left(
2l-1\right)  \tilde{C}\omega\left(  s,t\right)  ^{2/p}\\
&  =C\omega\left(  s,t\right)  ^{2/p}%
\end{align*}

\end{proof}

In \cite[Theorem 1]{gubinelli}, the following generalization of Theorem 3.3.1
of \cite{lyons-98} is proved.

\begin{theorem}
\label{the.2.7}Let $\mathbf{X}$ be a $p$-rough path on $W\oplus W^{\otimes2}$
with control $\omega$ and let $\left(  y,y^{\dag}\right)  $ be an $L\left(
W,V\right)  $ -- valued rough path controlled by $\mathbf{X}$. Then there
exists a $z\in C\left(  \left[  0,T\right]  ,V\right)  $ with $z_{0}=0$ and a
$C\geq0$ such that%
\begin{equation}
\left\vert z_{t}-z_{s}-y_{s}x_{s,t}-y_{s}^{\dag}\mathbb{X}_{s,t}\right\vert
\leq C\omega\left(  s,t\right)  ^{3/p} \label{equ.2.5}%
\end{equation}
for all $s\leq t$ in $\left[  0,T\right]  $.
\end{theorem}

We will more commonly refer to the path $z_{t}$ as $\int_{0}^{t}\left\langle
\mathbf{y}_{\tau},d\mathbf{X}_{\tau}\right\rangle $ and its increment,
$z_{s,t}:=z_{t}-z_{s},$ as $\int_{s}^{t}\left\langle \mathbf{y}_{\tau
},d\mathbf{X}_{\tau}\right\rangle .$ Theorem \ref{the.2.9} below is a
generalization of Theorem \ref{the.2.7}, but before we state it, we will make
a remark about certain identifications of spaces.

\begin{remark}
\label{rem.2.8}If $V,\tilde{V},$ and $\hat{V}$ are vector spaces, we can make
the identification%
\[
L\left(  V,L\left(  \tilde{V},\hat{V}\right)  \right)  \cong L\left(
V\otimes\tilde{V},\hat{V}\right)
\]
via the map $\Xi:L\left(  V,L\left(  \tilde{V},\hat{V}\right)  \right)
\longrightarrow L\left(  V\otimes\tilde{V},\hat{V}\right)  $ given by
\[
\Xi\left(  \alpha\right)  \left[  v\otimes\tilde{v}\right]  =\alpha
\left\langle v\right\rangle \left\langle \tilde{v}\right\rangle .
\]
if $\alpha\in L\left(  V,L\left(  \tilde{V},\hat{V}\right)  \right)  .$
\end{remark}

The proof of the following theorem (modulo a reparameterization) may be found
in \cite{gubinelli} or \cite[Remark 4.11]{FH2014}.

\begin{theorem}
\label{the.2.9}Let $\mathbf{X}$ be a $p$-rough path on $W\oplus W^{\otimes2}$
with control $\omega$, let $\left(  y,y^{\dag}\right)  $ be an $V$ -- valued
rough path controlled by $\mathbf{X}$ and let $\boldsymbol{\alpha}=\left(
\alpha,\alpha^{\dag}\right)  $ be an $L\left(  V,\tilde{V}\right)  $ -- valued
rough path controlled by $\mathbf{X}$ where $\alpha_{s}^{\dag}\in L\left(
W,L\left(  V,\tilde{V}\right)  \right)  \cong L\left(  W\otimes V,\tilde
{V}\right)  $. Then there exists a $z\in C\left(  \left[  0,T\right]
,V\right)  $ with $z_{0}=0$ and a $C>0$ such that%
\begin{equation}
\left\vert z_{t}-z_{s}-\alpha_{s}\left(  y_{t}-y_{s}\right)  -\alpha_{s}%
^{\dag}\left(  I\otimes y_{s}^{\dag}\right)  \mathbb{X}_{s,t}\right\vert \leq
C\omega\left(  s,t\right)  ^{3/p} \label{equ.2.6}%
\end{equation}
for all $s\leq t$ in $\left[  0,T\right]  .$ Moreover if we let $z_{s}^{\dag
}:=\alpha_{s}\circ y_{s}^{\dag},$ then $\mathbf{z}_{s}:=\left(  z_{s}%
,z_{s}^{\dag}\right)  $ is a $\tilde{V}$ -- valued controlled rough path.
\end{theorem}

The path $z_{t}$ in this case will be denoted $\int_{0}^{t}\left\langle
\boldsymbol{\alpha}_{\tau},d\mathbf{y}_{\tau}\right\rangle $ and we will
typically summarize Inequality (\ref{equ.2.6}) by writing%
\begin{equation}
\int_{s}^{t}\left\langle \boldsymbol{\alpha}_{\tau},d\mathbf{y}_{\tau
}\right\rangle \underset{^{3}}{\approx}\left\langle \boldsymbol{\alpha}%
_{s},\mathbf{y}_{s,t}^{\mathbb{X}}\right\rangle :=\alpha_{s}y_{s,t}+\alpha
_{s}^{\dag}\left(  I\otimes y_{s}^{\dag}\right)  \mathbb{X}_{s,t}
\label{equ.2.7}%
\end{equation}
wherein we let $\mathbf{y}_{s,t}^{\mathbb{X}}$ be the increment process
defined by,
\begin{equation}
\mathbf{y}_{s,t}^{\mathbb{X}}:=\left(  y_{s,t},\left(  I\otimes y_{s}^{\dag
}\right)  \mathbb{X}_{s,t}\right)  . \label{equ.2.8}%
\end{equation}
Notice that Theorem \ref{the.2.7} does indeed follow from Theorem
\ref{the.2.9} upon replacing $\left(  \alpha,\alpha^{\dag}\right)  $ by
$\left(  y,y^{\dag}\right)  $ and $\left(  y,y^{\dag}\right)  $ by $\left(
x,I_{W}\right)  $ in Inequality (\ref{equ.2.6}).

\begin{remark}
[Motivations]\label{rem.2.10}In order to develop some intuition for the
expression appearing on the right side of Eq. (\ref{equ.2.7}), suppose for the
moment that all functions $\mathbf{X},$ $\left(  y,y^{\dag}\right)  ,$ and
$\left(  \alpha,\alpha^{\dag}\right)  $ are smooth so that $\mathbb{X}$ is
given by Eq. (\ref{equ.2.4}). In this case we want $z_{s,t}\ $to be the usual
integral $\int_{s}^{t}\alpha_{\tau}\dot{y}_{\tau}d\tau$ and to arrive at the
expression in Inequality (\ref{equ.2.6}) we look for an appropriate second
order approximation to the integral. Since $p=1$ now we may conclude
\[
\alpha_{s,\tau}=\alpha_{s}^{\dag}x_{s,\tau}+O\left(  \left(  \tau-s\right)
^{2}\right)
\]
and
\[
y_{t}-y_{\tau}=y_{\tau}^{\dag}\left(  x_{t}-x_{\tau}\right)  +O\left(  \left(
t-\tau\right)  ^{2}\right)  \implies\dot{y}_{\tau}=y_{\tau}^{\dag}\dot
{x}_{\tau}.
\]

We have the identity;%
\begin{equation}
\int_{s}^{t}\alpha_{\tau}dy_{\tau}=\int_{s}^{t}\left[  \alpha_{s}%
+\alpha_{s,\tau}\right]  \dot{y}_{\tau}d\tau=\alpha_{s}y_{s,t}+\int_{s}%
^{t}\alpha_{s,\tau}\dot{y}_{\tau}d\tau. \label{equ.2.9}%
\end{equation}
The last term on the right hand side is approximated up to an error of size
$O\left(  \left(  t-s\right)  ^{3}\right)  $ as follows,%
\begin{align}
\int_{s}^{t}\alpha_{s,\tau}\dot{y}_{\tau}d\tau &  =\int_{s}^{t}\alpha_{s,\tau
}y_{\tau}^{\dag}\dot{x}_{\tau}d\tau\label{equ.2.10}\\
&  =\int_{s}^{t}\alpha_{s}^{\dag}x_{s,\tau}y_{\tau}^{\dag}\dot{x}_{\tau}%
d\tau+O\left(  \left(  t-s\right)  ^{3}\right) \nonumber\\
&  =\int_{s}^{t}\alpha_{s}^{\dag}x_{s,\tau}y_{s}^{\dag}\dot{x}_{\tau}%
d\tau+O\left(  \left(  t-s\right)  ^{3}\right) \nonumber\\
&  =\alpha_{s}^{\dag}\left(  I\otimes y_{s}^{\dag}\right)  \int_{s}%
^{t}x_{s,\tau}\otimes\dot{x}_{\tau}d\tau+O\left(  \left(  t-s\right)
^{3}\right) \nonumber\\
&  =\alpha_{s}^{\dag}\left(  I\otimes y_{s}^{\dag}\right)  \mathbb{X}%
_{s,t}+O\left(  \left(  t-s\right)  ^{3}\right)  .\nonumber
\end{align}
Combining Eq. (\ref{equ.2.9}) and Eq. (\ref{equ.2.10}) gives the approximate
equality,%
\[
\int_{s}^{t}\alpha_{\tau}dy_{\tau}=\alpha_{s}y_{s,t}+\alpha_{s}^{\dag}\left(
I\otimes y_{s}^{\dag}\right)  \mathbb{X}_{s,t}+O\left(  \left(  t-s\right)
^{3}\right)  .
\]

\end{remark}

Controlled rough paths are also useful in interpreting solutions to rough
differential equations. Let $F:V\rightarrow L\left(  W,V\right)  $ be smooth
where we will write $F\left(  a\right)  w$ as $F_{w}\left(  a\right)  $. We
can then make sense of the rough differential equation
\begin{equation}
d\mathbf{y}_{t}=F_{d\mathbf{X}_{t}}\left(  y_{t}\right)  \label{equ.2.11}%
\end{equation}
with initial condition $y_{0}=\bar{y}_{0}.$ We will need a bit of notation
regarding tensor products before we say what it means to solve such an equation.

\begin{notation}
\label{not.2.11}If $\Xi:W\times W\rightarrow V$ is a bilinear form into a
vector space $V$, by the universal property of tensor products, $\Xi$ factors
through a unique linear function $\Xi^{\otimes}$ on $W\otimes W$ such that
$\Xi^{\otimes}\left(  w\otimes\tilde{w}\right)  =\Xi\left(  w,\tilde
{w}\right)  $ for a simple tensor $w\otimes\tilde{w}$. If $\mathbb{W}\in
W\otimes W$ we will abuse notation and write%
\[
\Xi\left(  w,\tilde{w}\right)  |_{w\otimes\tilde{w}=\mathbb{W}}=\Xi\left(
w\otimes\tilde{w}\right)  |_{w\otimes\tilde{w}=\mathbb{W}}=\Xi^{\otimes
}\left(  \mathbb{W}\right)  ,
\]
where, to be precise, if $\mathbb{W}=\sum w_{i}\otimes\tilde{w}_{i}$ then
\[
\Xi^{\otimes}\left(  \mathbb{W}\right)  =\sum\Xi\left(  w_{i},\tilde{w}%
_{i}\right)  .
\]

\end{notation}

We say the controlled rough path $\mathbf{y}=\left(  y,y^{\dag}\right)  $
defined on\footnote{Here we allow that $\mathbf{y}\in CRP_{\mathbf{X}}%
(I_{0},V)$ if is in an element of $CRP_{\mathbf{X}}(K,V)$ for every compact
interval $K\in I_{0}.$} $I_{0}=[0,T)$ or $I_{0}=\left[  0,T\right]  $ solves
Eq. (\ref{equ.2.11}) if for every $\left[  0,b\right]  \subseteq I_{0}$, we
have
\begin{align*}
y_{s,t}  &  \underset{^{3}}{\approx}F_{x_{s,t}}\left(  y_{s}\right)  +\left(
\partial_{F_{w}\left(  y_{s}\right)  }F_{w}\right)  \left(  y_{s}\right)
|_{w\otimes\tilde{w}=\mathbb{X}_{s,t}}\\
y_{s}^{\dag}  &  =F_{\cdot}\left(  y_{s}\right)
\end{align*}
for all $s,t\in\left[  0,b\right]  .$ If in addition $y_{0}=\bar{y}_{0}$, we
say $\mathbf{y}$ solves Eq. (\ref{equ.2.11}) with initial condition
$y_{0}=\bar{y}_{0}$.

The existence and uniqueness of solutions (at least of the path $y_{s}$) to
these differential equations (provided $F$ is sufficiently regular) is due to
Lyons \cite{lyons-98}.\ Clearly if $y_{s}$ is given, then $y_{s}^{\dag}$
exists and is uniquely determined by $y_{s}^{\dag}=F_{\cdot}\left(
y_{s}\right)  $. One may refer to Subsection \ref{sub.6.5} in the Appendix for
more results regarding rough differential equations on Euclidean space.

\subsection{Manifold-Valued Controlled Rough Paths\label{sub.2.2}}

Let $M=M^{d}$ be a $d$-dimensional manifold, $TM$ be its tangent space, and
$\pi:=\pi_{TM}:TM\rightarrow M$ be the natural projection map. Throughout, let
$\mathbf{X=}\left(  x,\mathbb{X}\right)  $ be a weak-geometric $p$-rough path
on $\left[  0,T\right]  $ with with values in $W\oplus W^{\otimes2}$ and
control $\omega$.

The letters $x$ and $y$ will appear in this paper generally as paths, but
occasionally they will refer to arbitrary points in Euclidean space. The
context will allow the reader to identify their proper usage.

\begin{notation}
\label{not.2.12}When $M=\mathbb{R}^{d}$ we will identify $T\mathbb{R}^{d}$
with $\mathbb{R}^{d}\times\mathbb{R}^{d}$ via
\[
\mathbb{R}^{d}\times\mathbb{R}^{d}\ni\left(  m,v\right)  \rightarrow
v_{m}:=\frac{d}{dt}|_{0}\left(  m+tv\right)  \in T_{m}\mathbb{R}^{d}%
\]
and, by abuse of notation, we let $\left\vert v_{m}\right\vert =\left\vert
v\right\vert $ when $\left\vert \cdot\right\vert $ is the standard Euclidean norm.
\end{notation}

\begin{notation}
\label{not.2.13}Whenever $\phi$ is a map, let $D\left(  \phi\right)  $ and
$R\left(  \phi\right)  $ denote the domain and range of $\phi$ respectively.
If $\phi\in C^{\infty}\left(  M,\mathbb{R}^{d^{\prime}}\right)  $ has open
domain, let $d\phi:TD\left(  \phi\right)  \longrightarrow\mathbb{R}%
^{d^{\prime}}$ be defined by%
\begin{equation}
d\phi\left(  v_{m}\right)  :=\frac{d}{dt}|_{0}\phi\left(  \sigma\left(
t\right)  \right)  \in\mathbb{R}^{d^{\prime}} \label{equ.2.12}%
\end{equation}
where $\sigma$ is such that $\sigma\left(  0\right)  =m\in D\left(
\phi\right)  $ and $\dot{\sigma}\left(  0\right)  =v_{m}\in T_{m}M$. Denote
$d\phi_{m}:=d\phi|_{T_{m}M}$. If $f\in C^{\infty}\left(  M,\tilde{M}\right)  $
where $\tilde{M}$ is another manifold, we let $f_{\ast}$ be the push-forward
of $f$ so that $f_{\ast}:TD\left(  f\right)  \longrightarrow T\tilde{M}$ is
defined by
\[
f_{\ast}\left(  v_{m}\right)  :=\frac{d}{dt}|_{0}f\left(  \sigma\left(
t\right)  \right)  \in T_{f\left(  m\right)  }\tilde{M}%
\]
where again $\dot{\sigma}\left(  0\right)  =v_{m}$. Analogously we let
$f_{\ast m}=f_{\ast}|_{T_{m}M}.$ Note that $\phi_{\ast}\left(  v_{m}\right)
=\left(  \phi\left(  m\right)  ,d\phi\left(  v_{m}\right)  \right)  =\left[
d\phi\left(  v_{m}\right)  \right]  _{\phi\left(  m\right)  }$.
\end{notation}

\subsection{Gauges\label{sub.2.3}}

\begin{definition}
\label{def.2.14}Let $\mathcal{U}$ be an open set on $M$. An open set
$\mathcal{D}^{\mathcal{U}}\subseteq M\times M$ is a\textbf{ }$\mathcal{U}$ --
\textbf{diagonal domain}$\mathbb{\ }$if it contains the diagonal of
$\mathcal{U}$, that is $\Delta^{\mathcal{U}}:=\bigcup_{m\in\mathcal{U}}\left(
m,m\right)  \subseteq\mathcal{D}^{\mathcal{U}}$. A \textbf{local diagonal
domain} is a $\mathcal{V}$ -- diagonal domain for some nonempty open
$\mathcal{V\subseteq}M$.

If $\mathcal{U}=M$ we write $\mathcal{D}:=\mathcal{D}^{M}$ and refer to
$\mathcal{D}$ simply as a diagonal domain.
\end{definition}

Throughout the paper, $\mathcal{D}$ will always denote a diagonal domain.

\begin{definition}
\label{def.2.15}A smooth function $\psi:\mathcal{D}\rightarrow TM$ is called a
\textbf{logarithm }if:

\begin{enumerate}
\item $\psi\left(  m,n\right)  \in T_{m}M$

\item $\psi\left(  m,m\right)  =0_{m}$

\item $\psi\left(  m,\cdot\right)  _{\ast}|_{T_{m}M}=I_{m}$
\end{enumerate}

We also write $\psi_{m}$ for $\psi\left(  m,\cdot\right)  $.

If the above holds for $\psi$ defined on a local diagonal domain, we may refer
to $\psi$ as a \textbf{local logarithm}.
\end{definition}

If $E$ is a any vector bundle, we will denote the smooth sections of $E$ by
$\Gamma\left(  E\right)  $. We define $L\left(  TM,TM\right)  $ as the vector
bundle $\tilde{E}$ over the manifold $M\times M$ such that $\tilde{E}_{\left(
n,m\right)  }=L\left(  T_{m}M,T_{n}M\right)  $ and
\[
\tilde{E}=\bigcup\left\{  \tilde{E}_{\left(  n,m\right)  }:n,m\in M\right\}
.
\]

\begin{definition}
\label{def.2.16}A smooth section $U$ $\in\Gamma\left(  L\left(  TM,TM\right)
\right)  $ with domain $\mathcal{D}$ (i.e. $U\left(  n,m\right)  \in L\left(
T_{m}M,T_{n}M\right)  $ for all $\left(  n,m\right)  \in\mathcal{D}$) is
called a \textbf{parallelism }if $U\left(  m,m\right)  =I_{m}$. If $U$ is only
defined on a local diagonal domain, we refer to $U$ as a \textbf{local
parallelism}.
\end{definition}

\begin{definition}
\label{def.2.17}We call the pair $\mathcal{G}:\mathcal{=}\left(
\psi,U\right)  $ (where $\psi$ and $U$ have common domain $\mathcal{D}$) a
\textbf{gauge} on the manifold $M.$ If $\mathcal{D}$ is replaced by a local
diagonal domain, we call $\mathcal{G}$ a \textbf{local gauge}.
\end{definition}

\begin{example}
\label{exa.2.18}If $M=\mathbb{R}^{d},$ the maps $\psi\left(  x,y\right)
=\left[  y-x\right]  _{x}$ and $U_{x,y}v_{y}=v_{x}$ form the \textbf{standard
gauge }on $\mathbb{R}^{d}.$
\end{example}

\begin{example}
\label{exa.2.19}One natural example of a gauge comes from any covariant
derivative $\nabla$ on $TM.$ The construction is as follows. Choose an
arbitrary Riemannian metric $g$ on $M.$ If $m,n\in M$ are \textquotedblleft
close enough\textquotedblright, there is a unique vector $v_{m}$ with minimum
length such that $n=\exp_{m}^{\nabla}\left(  v_{m}\right)  $. We denote this
vector by $\psi^{\nabla}\left(  m,n\right)  :=\left(  \exp_{m}^{\nabla
}\right)  ^{-1}\left(  n\right)  \ $or by $\exp_{m}^{-1}\left(  n\right)  $ if
$\nabla$ is clear from the context. We further let%
\[
U^{\nabla}\left(  n,m\right)  :=//_{1}\left(  t\rightarrow\exp_{m}\left(
t\exp_{m}^{-1}\left(  n\right)  \right)  \right)  ,
\]
where, for any smooth curve $\sigma:\left[  0,1\right]  \rightarrow M,$ we let
$//_{s}\left(  \sigma\right)  =//_{s}^{\nabla}\left(  \sigma\right)
:T_{\sigma\left(  0\right)  }M\rightarrow T_{\sigma\left(  s\right)  }M$
denote parallel translation along $\sigma$ up to time $s\in\left[  0,1\right]
.$ It is shown in Corollary \ref{cor.2.33} that there is a diagonal domain
$\mathcal{D}\subseteq M\times M$ such that $\left(  \psi^{\nabla},U^{\nabla
}\right)  $ so defined is a gauge on $\mathcal{D}.$
\end{example}

\begin{remark}
\label{rem.2.20}We can also get a covariant derivative from a parallelism. If
$U$ is a parallelism, then we can define covariant derivative $\nabla^{U}$ on
$TM$ by
\[
\nabla_{v_{m}}^{U}\left(  Y\right)  :=\frac{d}{dt}|_{0}U\left(  m,\sigma
_{t}\right)  Y\left(  \sigma_{t}\right)  ,
\]
where $\dot{\sigma}\left(  0\right)  =v_{m}$ and $Y$ is a vector field on $M.$
\end{remark}

\begin{remark}
\label{rem.2.21}Although the definition of a gauge includes stipulating a $U$,
if we have just $\psi$, we can define $U^{\psi}\left(  n,m\right)
:=\psi\left(  n,\cdot\right)  _{\ast m}$ and set $\mathcal{G}^{\psi}:=\left(
\psi,U^{\psi}\right)  .$
\end{remark}

\begin{remark}
\label{rem.2.22}We may make a local gauge out of a chart $\phi$. Indeed, we
pull back the flat gauge in Example \ref{exa.2.18} to $M$ to define
\begin{align*}
\psi^{\phi}\left(  m,n\right)   &  :=\left(  d\phi_{m}\right)  ^{-1}\left[
\phi\left(  n\right)  -\phi\left(  m\right)  \right] \\
U^{\phi}\left(  n,m\right)   &  :=\left(  d\phi_{n}\right)  ^{-1}d\phi_{m}.
\end{align*}
This is a gauge which is also consistent with Remark \ref{rem.2.21} and
$D\left(  \psi^{\phi}\right)  =D\left(  U^{\phi}\right)  =D\left(
\phi\right)  \times D\left(  \phi\right)  $.
\end{remark}

Before moving on to controlled rough paths on manifolds, let us record the
structure of the general gauge on $\mathbb{R}^{d}.$

\begin{notation}
\label{not.2.23}If $\left(  \psi,U\right)  $ is a local gauge on
$\mathbb{R}^{d}$, then we write $\left(  \bar{\psi},\bar{U}\right)  $ to mean
the functions determined by the relations%
\[
\psi\left(  x,y\right)  =\left[  \bar{\psi}\left(  x,y\right)  \right]
_{x}\text{\quad and\quad}U\left(  x,y\right)  \left(  v_{y}\right)  =\left[
\bar{U}\left(  x,y\right)  v\right]  _{x}%
\]
so that $\bar{\psi}\left(  x,y\right)  \in\mathbb{R}^{d}$ and $\bar{U}\left(
x,y\right)  \in\operatorname*{End}\left(  \mathbb{R}^{d}\right)  .$
\end{notation}

\begin{theorem}
\label{the.2.24}If $\mathcal{G}=\left(  \psi,U\right)  $ is a local gauge on
$\mathbb{R}^{d}$, for every open convex subset $\mathcal{V}\subseteq
\mathbb{R}^{d}$ such that $\mathcal{V}\times\mathcal{V}\subseteq D\left(
\mathcal{G}\right)  $, there exists smoothly varying functions $A\left(
x,y\right)  \in L\left(  \left(  \mathbb{R}^{d}\right)  ^{\otimes2}%
,\mathbb{R}^{d}\right)  $ and $B\left(  x,y\right)  \in L\left(
\mathbb{R}^{d},\operatorname*{End}\left(  \mathbb{R}^{d}\right)  \right)  $
defined for $\left(  x,y\right)  \in$ $\mathcal{V}\times\mathcal{V}$ such that%
\begin{align}
\bar{U}\left(  x,y\right)   &  =I+B\left(  x,y\right)  \left(  y-x\right)
,\label{equ.2.13}\\
\bar{\psi}\left(  x,y\right)   &  =y-x+A\left(  x,y\right)  \left(
y-x\right)  ^{\otimes2},\label{equ.2.14}\\
B\left(  x,x\right)   &  =D_{2}\bar{U}\left(  x,x\right)  ,\text{ and
}A\left(  x,x\right)  =\frac{1}{2}\left(  D_{2}^{2}\bar{\psi}\right)  \left(
x,x\right)  . \label{equ.2.15}%
\end{align}
The converse holds as well.

Furthermore, we can find a smoothly varying function $C\left(  x,y\right)  \in
L\left(  \left(  \mathbb{R}^{d}\right)  ^{\otimes3},\mathbb{R}^{d}\right)  $
defined on $\mathcal{V}\times\mathcal{V}$ such that%
\begin{align}
C\left(  x,x\right)   &  =\frac{1}{6}\left(  D_{2}^{3}\bar{\psi}\right)
\left(  x,x\right)  ,\text{ and}\label{equ.2.16}\\
\bar{\psi}\left(  x,y\right)   &  =y-x+\frac{1}{2}\left(  D_{2}^{2}\bar{\psi
}\right)  \left(  x,x\right)  \left(  y-x\right)  ^{\otimes2}+C\left(
x,y\right)  \left(  y-x\right)  ^{\otimes3}. \label{equ.2.17}%
\end{align}

\end{theorem}

\begin{proof}
Let $x,y$ be points in $\mathcal{V}$. Taylor's theorem with integral remainder
applied to the second variable with $x$ fixed gives,
\[
\bar{U}\left(  x,y\right)  =I+\int_{0}^{1}\left(  D_{2}\bar{U}\right)  \left(
x,x+t\left(  y-x\right)  \right)  \left(  y-x\right)  dt
\]
and
\[
\bar{\psi}\left(  x,y\right)  =0+\left(  D_{2}\bar{\psi}\right)  \left(
x,x\right)  \left(  y-x\right)  +\int_{0}^{1}\left(  D_{2}^{2}\bar{\psi
}\right)  \left(  x,x+t\left(  y-x\right)  \right)  \left(  y-x\right)
^{\otimes2}\left(  1-t\right)  dt
\]
from which Eqs. (\ref{equ.2.13}) -- (\ref{equ.2.15}) follow with%
\begin{align*}
B\left(  x,y\right)   &  =\int_{0}^{1}\left(  D_{2}\bar{U}\right)  \left(
x,x+t\left(  y-x\right)  \right)  dt\text{ and }\\
A\left(  x,y\right)   &  =\int_{0}^{1}\left(  D_{2}^{2}\bar{\psi}\right)
\left(  x,x+t\left(  y-x\right)  \right)  \left(  1-t\right)  dt.
\end{align*}
The converse statement is easy to verify. The proof of Eqs. (\ref{equ.2.16})
and (\ref{equ.2.17}) also follow by Taylor's theorem (now to third order) in
which case,%
\[
C\left(  x,y\right)  =\frac{1}{2}\int_{0}^{1}\left(  D_{2}^{3}\bar{\psi
}\right)  \left(  x,x+t\left(  y-x\right)  \right)  \left(  y-x\right)
^{\otimes2}\left(  1-t\right)  ^{2}dt.
\]

\end{proof}

Let $B_{r}\left(  x\right)  \subseteq\mathbb{R}^{d}$ be the open ball of
radius $r$ centered at $x.$

\begin{remark}
\label{rem.2.25}If $\psi$ and $\tilde{\psi}$ are local logarithms on
$\mathbb{R}^{d}$, it is easy to check using Theorem \ref{the.2.24} that for
all $\tilde{x}\in\mathbb{R}^{d}$, there exists an $r>0$ and $C>0$ such that
$\left\vert \psi\left(  x,y\right)  \right\vert \leq C\left\vert \tilde{\psi
}\left(  x,y\right)  \right\vert $ for all $x,y\in B_{r}\left(  \tilde
{x}\right)  $.
\end{remark}

We now wish to transfer these local results to the manifold setting. In order
to do this we need to develop some notation for stating that two objects on a
manifold are \textquotedblleft close\textquotedblright\ up to some order. Let
$g$ be any smooth Riemannian metric on $M$.

\begin{notation}
We write $d_{g}$ for the metric associated to $g$ and define $\left\vert
v_{m}\right\vert _{g}:=\sqrt{g_{m}\left(  v_{m},v_{m}\right)  }$
$\forall~v_{m}\in TM$. Further, we let $\left\vert \cdot\right\vert _{g,op}$
be the operator \textquotedblleft norm\textquotedblright\ induced by
$\left\vert \cdot\right\vert _{g}$ on $L\left(  TM,V\right)  ,$ i.e. if
$f_{m}\in L\left(  T_{m}M,V\right)  $, then
\[
\left\vert f_{m}\right\vert _{g,op}:=\sup\left\{  \left\vert f_{m}\left\langle
v_{m}\right\rangle \right\vert :\left\vert v_{m}\right\vert _{g}=1\right\}  .
\]

\end{notation}

\begin{definition}
\label{def.2.27}Let $F,G$ be smooth $TM$ [respectively $L\left(  TM,TM\right)
$] valued functions with $\mathcal{W}$ -- diagonal domains. The expression
\begin{equation}
F\left(  m,n\right)  =_{k}G\left(  m,n\right)  \text{ on }\mathcal{W}
\label{equ.2.18}%
\end{equation}
indicates that for every point in $w\in\mathcal{W}$, there exists an open
$\mathcal{O}_{w}\subseteq M$ containing $w$ such that $\mathcal{O}_{w}%
\times\mathcal{O}_{w}\subseteq D\left(  F\right)  \cap D\left(  G\right)  $
and a $C>0$ such that%
\begin{equation}
\left\vert F\left(  m,n\right)  -G\left(  m,n\right)  \right\vert
_{g,\text{[}g,op\text{]}}\leq C\left(  d_{g}\left(  m,n\right)  \right)  ^{k}
\label{equ.2.19}%
\end{equation}
for all $m,n\in\mathcal{O}_{w}$.

Occasionally we will omit the reference to $\mathcal{W}$ in which case it we
mean the condition (\ref{equ.2.19}) holds where it makes sense to hold.
\end{definition}

Note that in (\ref{equ.2.18}), the reference to $g$ is not explicit. In fact,
the definition does not depend on the choice of $g$ as all Riemannian metrics
are locally equivalent. [See Corollary \ref{cor.6.4} in the Appendix for
precise statement and proof of this standard fact.]

We may also use the $=_{k}$ notation to make statements in regards to other
measures of distance:

\begin{corollary}
\label{cor.2.28}Let $\mathcal{W}$ be an open subset of $M$ and $g$ and
$\tilde{g}$ be any two Riemannian metrics on $M.$ If $F\left(  m,n\right)
=_{k}G\left(  m,n\right)  $ on $\mathcal{W}$ (so that $F$ and $G$ have
$\mathcal{W}$-diagonal domains), then for every local logarithm $\psi$ and
$w\in\mathcal{W}$ such that $\left(  w,w\right)  \in D\left(  \psi\right)  $,
there exists an open $\mathcal{O}_{w}\subseteq\mathcal{W}$ containing $w$ and
$C>0$ such that
\[
\left\vert F\left(  m,n\right)  -G\left(  m,n\right)  \right\vert
_{g,\text{[}g,op\text{]}}\leq C\left\vert \psi\left(  m,n\right)  \right\vert
_{\tilde{g}}^{k}~\forall~m,n\in\mathcal{O}_{w}.
\]
In particular, using the local logarithm $\psi\left(  m,n\right)  =\left(
d\phi_{m}\right)  ^{-1}\left[  \phi\left(  n\right)  -\phi\left(  m\right)
\right]  $, we have that if $w\in D\left(  \phi\right)  \cap\mathcal{W}$, then
there exists an $\mathcal{O}_{w}\subseteq D\left(  \phi\right)  \cap
\mathcal{W}$ and a $C>0$ such that
\[
\left\vert F\left(  m,n\right)  -G\left(  m,n\right)  \right\vert
_{g,\text{[}g,op\text{]}}\leq C\left\vert \phi\left(  n\right)  -\phi\left(
m\right)  \right\vert ^{k}~\forall~m,n\in\mathcal{O}_{w}.
\]

\end{corollary}

\begin{proof}
The proof of the Corollary will use Remark \ref{rem.2.25} and the local
equivalence of any two Riemannian metrics, Corollary \ref{cor.6.4} in the
Appendix. First we simplify matters by assuming that we are working in
Euclidean space which may be accomplished by pushing the metric and functions
forward using charts. Assuming this, we now derive a local inequality that
holds for any two logarithms $\psi$ and $\tilde{\psi}$ when $\left(
w,w\right)  \in D\left(  \psi\right)  \cap D\left(  \tilde{\psi}\right)  $.
Namely, there exist an open neighborhood, $\mathcal{O}_{w},$ of $w$ such that%
\[
\left\vert \tilde{\psi}\left(  m,n\right)  \right\vert _{g}\leq C_{1}%
\left\vert \tilde{\psi}\left(  m,n\right)  \right\vert \leq C_{2}%
C_{1}\left\vert \psi\left(  m,n\right)  \right\vert \leq C_{3}C_{2}%
C_{1}\left\vert \psi\left(  m,n\right)  \right\vert _{\tilde{g}}%
~\forall\left(  m,n\right)  \in\mathcal{O}_{w}\times\mathcal{O}_{w},
\]
where the first and third inequality follow from Corollary \ref{cor.6.4} with
one metric being the standard Euclidean metric and the other metric being $g$
or $\tilde{g}$ respectively, and the second inequality is true by Remark
\ref{rem.2.25}. Thus, there exists a $\tilde{C}$ such that%
\[
\left\vert \tilde{\psi}\left(  m,n\right)  \right\vert _{g}\leq\tilde
{C}\left\vert \psi\left(  m,n\right)  \right\vert _{\tilde{g}}%
\]

Now let $\nabla^{g}$ be the Levi-Civita covariant derivative associated to
$g$. By setting $\tilde{\psi}\left(  m,n\right)  =\left(  \exp_{m}^{\nabla
^{g}}\right)  ^{-1}\left(  n\right)  $ and shrinking $\mathcal{O}_{w}$ if
necessary to ensure that $\left(  \exp_{\left(  \cdot\right)  }^{\nabla^{g}%
}\right)  ^{-1}\left(  \cdot\right)  $ is defined and injective on
$\mathcal{O}_{w}\times\mathcal{O}_{w}$, we have that%
\[
\left\vert \left(  \exp_{m}^{\nabla^{g}}\right)  ^{-1}\left(  n\right)
\right\vert _{g}\leq\tilde{C}\left\vert \psi\left(  m,n\right)  \right\vert
_{\tilde{g}}.
\]
In this setting, $d_{g}\left(  m,n\right)  =\left\vert \left(  \exp
_{m}^{\nabla^{g}}\right)  ^{-1}\left(  n\right)  \right\vert _{g}$, and since
$F\left(  m,n\right)  =_{k}G\left(  m,n\right)  $ on $\mathcal{W}$ (by
shrinking $\mathcal{O}_{w}$ if necessary), we have
\[
\left\vert F\left(  m,n\right)  -G\left(  m,n\right)  \right\vert
_{g,\text{[}g,op\text{]}}\leq\hat{C}\left(  d_{g}\left(  m,n\right)  \right)
^{k}\text{ ~}\forall~m,n\in\mathcal{O}_{w}%
\]
for some $\hat{C}$. Thus, we have%
\[
\left\vert F\left(  m,n\right)  -G\left(  m,n\right)  \right\vert
_{g,\text{[}g,op\text{]}}\leq\hat{C}\left(  \tilde{C}\right)  ^{k}\left\vert
\psi\left(  m,n\right)  \right\vert _{\tilde{g}}^{k}.
\]
which is the statement of the Corollary with $C:=\hat{C}\left(  \tilde
{C}\right)  ^{k}$.
\end{proof}

In the sequel, Corollary \ref{cor.2.28} will typically be used without further
reference in order reduce the proof of showing $F\left(  m,n\right)
=_{k}G\left(  m,n\right)  $ in the manifold setting to a local statement about
functions on convex neighborhoods in $\mathbb{R}^{d}$ equipped with the
standard Euclidean flat metric structures. The first example of this strategy
will already occur in the proof of Corollary \ref{cor.2.29} below. For a
general parallelism it is not true that $U\left(  n,m\right)  ^{-1}=U\left(
m,n\right)  ,$ yet $U\left(  m,n\right)  $ is always a very good approximation
to $U\left(  n,m\right)  ^{-1}$.

\begin{corollary}
\label{cor.2.29}If $U$ is a parallelism on a manifold, $M,$ then
\[
U\left(  n,m\right)  ^{-1}=_{2}U\left(  m,n\right)  .
\]

\end{corollary}

\begin{proof}
This is a local statement so we may use Corollary \ref{cor.2.28} to reduce to
the case that $M$ is a convex open subset of $\mathbb{R}^{d}$. We then may use
Theorem \ref{the.2.24} to learn%
\[
\bar{U}\left(  n,m\right)  ^{-1}=\left(  I+\left[  B\left(  n,m\right)
\left(  m-n\right)  \right]  \right)  ^{-1}=I+\left[  B\left(  n,m\right)
\left(  n-m\right)  \right]  +O\left(  \left\vert n-m\right\vert ^{2}\right)
\]
while
\[
\bar{U}\left(  m,n\right)  =\left(  I+\left[  B\left(  m,n\right)  \left(
n-m\right)  \right]  \right)  .
\]
Subtracting these two equations shows,%
\begin{align*}
\bar{U}\left(  n,m\right)  ^{-1}-\bar{U}\left(  m,n\right)   &  =\left[
B\left(  n,m\right)  -B\left(  m,n\right)  \right]  \left(  n-m\right)
+O\left(  \left\vert n-m\right\vert ^{2}\right) \\
&  =O\left(  \left\vert n-m\right\vert ^{2}\right)
\end{align*}
wherein we have used $B\left(  n,m\right)  -B\left(  m,n\right)  $ vanishes
for $m=n$ and therefore is of order $\left\vert m-n\right\vert .$
\end{proof}

\subsubsection{A Covariant Derivative Gives Rise to a Gauge\label{sub.2.3.1}}

Let $\nabla$ be a covariant derivative on $TM,$ and $g$ be any fixed
Riemannian metric on $M.$ Let $G:TM\rightarrow M\times M$ be the function on
$TM$ defined by%
\begin{equation}
G\left(  v_{m}\right)  :=\left(  m,\exp_{m}^{\nabla}\left(  v_{m}\right)
\right)  \text{ for all }v_{m}\in D\left(  G\right)  , \label{equ.2.20}%
\end{equation}
where $D\left(  G\right)  $ is the domain of $G$ defined by
\[
D\left(  G\right)  :=\left\{  v_{m}\in TM:t\rightarrow\exp_{m}^{\nabla}\left(
tv_{m}\right)  \text{ exists for }0\leq t\leq1\right\}  .
\]

We will now develop a subset of $D\left(  G\right)  $ for which $G$ is
injective. For each $m\in M,$ let $\Lambda_{m}$ denote the set of $r>0$ so
that $B_{r}\left(  0_{m}\right)  \subseteq D\left(  G\right)  ,$ $\exp
_{m}^{\nabla}\left(  B_{r}\left(  0_{m}\right)  \right)  $ is an open
neighborhood of $m$ in $M,$ and $\exp_{m}^{\nabla}:B_{r}\left(  0_{m}\right)
\rightarrow\exp_{m}^{\nabla}\left(  B_{r}\left(  0_{m}\right)  \right)  $ is a
diffeomorphism (here $B_{r}\left(  0_{m}\right)  $ is the open ball in
$T_{m}M$ centered at $0_{m}$ with radius $r$). The fact that $\Lambda_{m}$ is
not empty is a consequence of the inverse function theorem and the fact that
$\left(  \exp_{m}^{\nabla}\right)  _{\ast0_{m}}=I_{T_{m}M}$ is invertible. We
now define $r_{m}:=\sup\Lambda_{m}$ where $r_{m}=\infty$ is possible and
allowed. A little thought shows that $\exp_{m}^{\nabla}\left(  B_{r_{m}%
}\left(  0_{m}\right)  \right)  $ is open and $\exp_{m}^{\nabla}:B_{r_{m}%
}\left(  0_{m}\right)  \rightarrow\exp_{m}^{\nabla}\left(  B_{r_{m}}\left(
0_{m}\right)  \right)  $ is a diffeomorphism, i.e. either $r_{m}=\infty$ or
$r_{m}\in\Lambda_{m}.$

Let us now set $\mathcal{C}^{\ast}:=\cup_{m\in M}B_{r_{m}}\left(
0_{m}\right)  \subseteq TM$ and let $G^{\ast}:\mathcal{C}^{\ast}\rightarrow
M\times M$ be the map defined by%
\[
G^{\ast}\left(  v_{m}\right)  :=\left(  m,\exp_{m}^{\nabla}\left(
v_{m}\right)  \right)  \text{ for all }v_{m}\in\mathcal{C}^{\ast}.
\]
It is easy to verify that $G^{\ast}$ is injective.

We will now build our domain $\mathcal{C}$ for which $G|_{\mathcal{C}}$ is
diffeomorphic onto its range. First we need a simple local invertibility proposition.

\begin{proposition}
\label{pro.2.30}Let $G$ be the function defined in Eq. (\ref{equ.2.20}). Then
for each $m\in M,$ there exists open subsets $\mathcal{V}_{m}\subseteq TM$ and
$\mathcal{W}_{m}\subseteq M$ such that $0_{m}\in\mathcal{V}_{m},~m\in
\mathcal{W}_{m},$ and $G|_{\mathcal{V}_{m}}:\mathcal{V}_{m}\rightarrow
\mathcal{W}_{m}\times\mathcal{W}_{m}$ is a diffeomorphism.
\end{proposition}

\begin{proof}
As this a local result we may assume that $M=\mathbb{R}^{d}$ and identify $TM$
with $M\times M.$ The function $G:TM\rightarrow M\times M$ then takes on the
form $G\left(  x,v\right)  =\left(  x,\bar{G}\left(  x,v\right)  \right)  $
where $\bar{G}\left(  x,0\right)  =x$ and $\left(  D_{2}\bar{G}\right)
\left(  x,0\right)  =I_{M}$ for all $x\in M.$ A simple computation then shows
\[
G^{\prime}\left(  x,0\right)  =\left[
\begin{array}
[c]{cc}%
I & 0\\
I & I
\end{array}
\right]  \text{ for all }x\in M.
\]
The result now follows by an application of the inverse function theorem.
\end{proof}

\begin{notation}
\label{not.2.31}If $\mathcal{W}$ is an open subset of $M$ and $\epsilon>0,$
let $\mathcal{U}\left(  \mathcal{W},\epsilon\right)  $ be the open subset of
$TM$ defined by
\[
\mathcal{U}\left(  \mathcal{W},\epsilon\right)  :=\left\{  v\in\pi^{-1}\left(
\mathcal{W}\right)  \subseteq TM:\left\vert v\right\vert _{g}<\epsilon
\right\}  .
\]

\end{notation}

\begin{theorem}
\label{the.2.32}Let $\mathcal{C}:=\bigcup\mathcal{U}\left(  \mathcal{W}%
,\epsilon\right)  $ where the union is taken over all open subsets
$\mathcal{W}\subseteq M$ and $\epsilon>0$ such that $\mathcal{U}\left(
\mathcal{W},\epsilon\right)  \subseteq D\left(  G\right)  $ and
$G|_{\mathcal{U}\left(  \mathcal{W},\epsilon\right)  }:\mathcal{U}\left(
\mathcal{W},\epsilon\right)  \rightarrow G\left(  \mathcal{U}\left(
\mathcal{W},\epsilon\right)  \right)  $ is a diffeomorphism. Then
$\mathcal{C}$ is an open subset of $TM$ such that $\mathcal{D}:=G\left(
\mathcal{C}\right)  $ is open in $M\times M,$ $G:\mathcal{C}\rightarrow
\mathcal{D}$ is a diffeomorphism,%
\[
\left\{  0_{m}:m\in M\right\}  \subseteq\mathcal{C}\subseteq\mathcal{C}^{\ast
}\text{\quad and\quad}\Delta^{M}=\left\{  \left(  m,m\right)  :m\in M\right\}
\subseteq\mathcal{D}.
\]

\end{theorem}

\begin{proof}
According to Proposition \ref{pro.2.30}, for each $m\in M$ there exists an
open neighborhood $\mathcal{W}$ of $m\in M$ and $\epsilon>0$ so that
$\mathcal{U}\left(  \mathcal{W},\epsilon\right)  \subseteq D\left(  G\right)
$ and $G:\mathcal{U}\left(  \mathcal{W},\epsilon\right)  \rightarrow G\left(
\mathcal{U}\left(  \mathcal{W},\epsilon\right)  \right)  $ is a
diffeomorphism. From this it follows that $\left\{  0_{m}:m\in\mathcal{W}%
\right\}  \subseteq\mathcal{C}$ and $\mathcal{U}\left(  \mathcal{W}%
,\epsilon\right)  \subseteq\mathcal{C}^{\ast}.$ As $m\in M$ was arbitrary we
may conclude $\left\{  0_{m}:m\in M\right\}  \subseteq\mathcal{C}%
\subseteq\mathcal{C}^{\ast}.$ It is now easily verified that $G\left(
\mathcal{C}\right)  =\cup G\left(  \mathcal{U}\left(  \mathcal{W}%
,\epsilon\right)  \right)  $ is open, $G:\mathcal{C}\rightarrow G\left(
\mathcal{C}\right)  $ is a surjective local diffeomorphism and hence is a
diffeomorphism as $G|_{\mathcal{C}}$ is injective (since $G|_{\mathcal{C}%
^{\ast}}$ is injective).
\end{proof}

\begin{corollary}
\label{cor.2.33}Continuing the notation used in Theorem \ref{the.2.32}, we
have $\mathcal{D}$ is a diagonal domain and $\psi:=G|_{\mathcal{C}}%
^{-1}:\mathcal{D\rightarrow C}\subseteq TM$ is a logarithm. Moreover, if we
define
\[
U\left(  m,n\right)  :=//_{1}\left(  t\longrightarrow\exp^{\nabla}\left(
t\psi\left(  m,n\right)  \right)  \right)  ^{-1}:T_{n}M\rightarrow T_{m}M
\]
for all $\left(  m,n\right)  \in\mathcal{D},$ then $U$ is a parallelism on
$M.$
\end{corollary}

\begin{proof}
The only thing that remains to be proven is that $U\left(  m,n\right)  $ is
smoothly varying. This is a consequence of the fact that solutions to ordinary
differential equations depend smoothly on their starting points and parameter
in the vector fields. To be more explicit in this case, for $a\in
\mathbb{R}^{d}$ let $B_{a}^{\nabla}\left(  \mu\right)  =\dot{u}\left(
0\right)  $ where $u\left(  t\right)  =//_{t}\left(  \exp^{\nabla}\left(
\left(  \cdot\right)  \mu a\right)  \right)  \mu$ for $\mu$ in the frame
bundle $GL\left(  M\right)  $ over $M$, so that $B_{a}^{\nabla}$ are the
$\nabla$ -- horizontal vector fields. Now suppose that $w\in M$ is given and
$O\left(  m\right)  :\mathbb{R}^{d}\rightarrow T_{m}M$ is a local frame
defined for $m$ in an open neighborhood $\mathcal{W}$ of $w.$ For $v\in
\pi^{-1}\left(  \mathcal{W}\right)  \cap\mathcal{C}$ let $\gamma\left(
t\right)  =\exp^{\nabla}\left(  tv\right)  $ and $u\left(  t\right)
:=//_{t}\left(  \gamma\right)  O\left(  \pi\left(  v\right)  \right)  .$ We
then have
\begin{align*}
\dot{\gamma}\left(  t\right)   &  =//_{t}\left(  \gamma\right)  v=u\left(
t\right)  O\left(  \pi\left(  v\right)  \right)  ^{-1}v\text{ and}\\
\frac{\nabla u}{dt}  &  =0\text{ with }u\left(  0\right)  =O\left(  \pi\left(
v\right)  \right)  .
\end{align*}
These equations are equivalent to solving
\begin{equation}
\dot{u}\left(  t\right)  =B_{O\left(  \pi\left(  v\right)  \right)  ^{-1}%
v}^{\nabla}\left(  u\left(  t\right)  \right)  \text{ with }u\left(  0\right)
=O\left(  \pi\left(  v\right)  \right)  \label{equ.2.21}%
\end{equation}
in which case $\gamma\left(  t\right)  =\pi_{O\left(  M\right)  }\left(
u\left(  t\right)  \right)  $ where $\pi_{O\left(  M\right)  }$ is the
projection map from $O\left(  M\right)  $ to $M.$ We now define $F\left(
v\right)  :=u\left(  1\right)  $ provided $v\in\pi^{-1}\left(  \mathcal{W}%
\right)  \cap\mathcal{C}.$ It then follows that $F:\pi^{-1}\left(
\mathcal{W}\right)  \cap\mathcal{C}\rightarrow GL\left(  M\right)  $ is smooth
as the solutions to Eq. (\ref{equ.2.21}) depend smoothly on its starting point
and parameter. From this we learn for $\left(  m,n\right)  \in G\left(
\pi^{-1}\left(  \mathcal{W}\right)  \cap\mathcal{C}\right)  $ that
\[
U\left(  n,m\right)  =F\left(  \psi\left(  m,n\right)  \right)  O\left(
m\right)  ^{-1}%
\]
is a smooth function of $\left(  m,n\right)  .$
\end{proof}

\subsection{Controlled Rough Paths\label{sub.2.4}}

\begin{notation}
\label{not.2.34}Throughout the remainder of this paper, $\mathbf{y:=}\left(
y,y^{\dag}\right)  $ denotes a pair of continuous functions, $y\in C\left(
\left[  0,T\right]  ,M\right)  $ and $y^{\dag}\in C\left(  \left[  0,T\right]
,L\left(  W,TM\right)  \right)  ,$ such that $y_{s}^{\dag}\in L\left(
W,T_{y_{s}}M\right)  $ for all $s.$
\end{notation}

\begin{definition}
\label{def.2.35}Let $\left(  \psi,U\right)  $ be a gauge. The pair $\left(
y_{s},y_{s}^{\dag}\right)  $ is $\left(  \psi,U\right)  -$\textbf{rough path
controlled by} $\mathbf{X}$ if there exists a $C>0$ and $\delta>0$ such that

\begin{enumerate}
\item \label{Ite.12}%
\begin{equation}
\left\vert \psi\left(  y_{s},y_{t}\right)  -y_{s}^{\dag}x_{s,t}\right\vert
_{g}\leq C\omega\left(  s,t\right)  ^{2/p} \label{equ.2.22}%
\end{equation}
and

\item \label{Ite.13}%
\begin{equation}
\left\vert U\left(  y_{s},y_{t}\right)  y_{t}^{\dag}-y_{s}^{\dag}\right\vert
_{g}\leq C\omega\left(  s,t\right)  ^{1/p} \label{equ.2.23}%
\end{equation}
hold whenever $0\leq s\leq t\leq T$ and $\left\vert t-s\right\vert \leq\delta
$. Occasionally we will refer to $y_{s}$ as the path and $y_{s}^{\dag}$ as the
derivative process (or Gubinelli derivative).
\end{enumerate}
\end{definition}

\begin{remark}
In Definition \ref{def.2.35} and in the definitions that follow, we use the
convention that the $\delta$ is small enough to ensure that all of the
expressions are well defined (in particular here it is small enough to ensure
$\left(  y_{s},y_{t}\right)  \in\mathcal{D}$).
\end{remark}

\begin{remark}
\label{rem.2.37}Any path $z_{s}$ in Euclidean space naturally gives rise to a
two-parameter \textquotedblleft increment process,\textquotedblright\ namely
$z_{s,t}=z_{t}-z_{s}.$ If $\varphi$ is any function such that $\varphi\left(
z,\tilde{z}\right)  \approx\tilde{z}-z$, then it makes sense to define
$z_{s,t}^{\varphi}:=\varphi\left(  z_{s},z_{t}\right)  .$ This serves as
motivation for the following notation.
\end{remark}

\begin{notation}
\label{not.2.38}Given a gauge, $\mathcal{G}=\left(  \psi,U\right)  $, let
$y_{s,t}^{\psi}:=\psi\left(  y_{s,}y_{t}\right)  $ and $\left(  y^{\dag
}\right)  _{s,t}^{U}:=U\left(  y_{s},y_{t}\right)  y_{t}^{\dag}-y_{s}^{\dag}$.
These will be referred to as the $\mathcal{G-}$\textbf{local increment
processes of }$\left(  y,y^{\dag}\right)  .$
\end{notation}

\begin{remark}
\label{rem.2.39}With Notation \ref{not.2.38}, (\ref{equ.2.22}) becomes
$\left\vert y_{s,t}^{\psi}-y_{s}^{\dag}x_{s,t}\right\vert \leq C\omega\left(
s,t\right)  ^{2/p}$ and (\ref{equ.2.23}) becomes $\left\vert \left(  y^{\dag
}\right)  _{s,t}^{U}\right\vert \leq C\omega\left(  s,t\right)  ^{1/p}$.
\end{remark}

Definition \ref{def.2.35} gives one possible notion of a controlled rough path
on a manifold. We can also define such an object without having to provide a
metric or gauge by using charts on the manifold.

\begin{definition}
\label{def.2.40}The pair $\mathbf{y}_{s}=\left(  y_{s},y_{s}^{\dag}\right)  $
is a \textbf{chart-rough path controlled by} $\mathbf{X}$ if for every chart
$\phi$ on $M$ and every $\left[  a,b\right]  $ such that $y\left(  \left[
a,b\right]  \right)  \subseteq D\left(  \phi\right)  $ we have the existence
of a $C_{\phi,a,b}\geq0$ such that, for all $a\leq s\leq t\leq b$,

\begin{enumerate}
\item
\begin{equation}
\left\vert \phi\left(  y_{t}\right)  -\phi\left(  y_{s}\right)  -d\phi\circ
y_{s}^{\dag}x_{s,t}\right\vert \leq C_{\phi,a,b}\omega\left(  s,t\right)
^{2/p} \label{equ.2.24}%
\end{equation}
and

\item
\begin{equation}
\left\vert d\phi\circ y_{t}^{\dag}-d\phi\circ y_{s}^{\dag}\right\vert \leq
C_{\phi,a,b}\omega\left(  s,t\right)  ^{1/p} \label{equ.2.25}%
\end{equation}

\end{enumerate}

We will denote $C_{\phi,a,b}$ by $C_{\phi}$ when no confusion is likely to arise.
\end{definition}

\begin{notation}
\label{not.2.41}If $\left(  y_{s},y_{s}^{\dag}\right)  $ is a chart rough path
and $\phi$ is a chart as in Definition \ref{def.2.40}, we will write
$\phi_{\ast}\mathbf{y}_{s}$ to mean
\[
\phi_{\ast}\mathbf{y}_{s}\mathbf{:}=\phi_{\ast}\left(  y_{s},y_{s}^{\dag
}\right)  :=\left(  \phi\circ y_{s},d\phi\circ y_{s}^{\dag}\right)  .
\]

\end{notation}

Note that as long as $y$ remains away from the boundary of $D\left(
\phi\right)  $, then $\phi_{\ast}\mathbf{y}_{s}$ is a controlled rough path on
$\mathbb{R}^{d}$. Another way to think of this is that a chart controlled
rough path is one which pushes forward to a controlled rough path in
$\mathbb{R}^{d}$.

Before moving on, we'll make a few remarks.

\begin{remark}
\label{rem.2.42}If $y^{\dag}$ is any function satisfying the conditions in
either of Definitions \ref{def.2.35} or \ref{def.2.40}, then $s\rightarrow
y_{s}^{\dag}$ is automatically continuous. For example, if $\left(
y_{s},y_{s}^{\dag}\right)  $ satisfies the conditions of a $\left(
\psi,U\right)  -$rough path in Definition \ref{def.2.35}, then the function
$t\rightarrow U\left(  y_{s},y_{t}\right)  y_{t}^{\dag}$ is a continuous at
$s$ and therefore $t\rightarrow y_{t}^{\dag}=U\left(  y_{s},y_{t}\right)
^{-1}U\left(  y_{s},y_{t}\right)  y_{t}^{\dag}$ is continuous at $s.$
\end{remark}

\begin{remark}
\label{rem.2.43}If $M=\mathbb{R}^{d}$ and $\phi=I$ then the chart Definition
\ref{def.2.40} reduces to the usual Definition \ref{def.2.5} of controlled
rough paths. In this case, we identify all the tangent spaces with
$\mathbb{R}^{d}$ and forget the base point in the derivative process.
\end{remark}

\subsection{Chart and Gauge CRP\ Definitions are Equivalent\label{sub.2.5}}

\begin{theorem}
\label{the.2.44}Let $\mathbf{y:=}\left(  y,y^{\dag}\right)  $ be a pair of
continuous functions as in Notation \ref{not.2.34}, $M$ be a manifold, and
$\mathcal{G=}\left(  \psi,U\right)  $ be any gauge on $M$ . Then $\mathbf{y}$
is a chart controlled rough path (Definition \ref{def.2.40}) if and only if it
is a $\left(  \psi,U\right)  $-controlled rough path (Definition
\ref{def.2.35}).
\end{theorem}

\begin{corollary}
\label{cor.2.45}We have the equality of sets%
\[
\left\{  \left(  \psi,U\right)  -\text{rough paths}\right\}  =\left\{  \left(
\tilde{\psi},\tilde{U}\right)  -\text{rough paths}\right\}
\]
for any gauges $\left(  \psi,U\right)  $ and $\left(  \tilde{\psi},\tilde
{U}\right)  $ on $M$.
\end{corollary}

\begin{notation}
\label{not.2.46}Let $CRP_{\mathbf{X}}\left(  M\right)  $ be the collection of
\textbf{controlled rough paths in }$M,$ i.e. pairs of functions $\mathbf{y}%
=\left(  y,y^{\dag}\right)  $ as in Notation \ref{not.2.34} which satisfy
either (and hence both) of Definitions \ref{def.2.35} or \ref{def.2.40}.
\end{notation}

We will prove Theorem \ref{the.2.44} after assembling a number of preliminary
results that will be needed in the proof and in the rest of the paper.

\subsubsection{Results Used in Proof of Theorem \ref{the.2.44}%
\label{sub.2.5.1}}

Our first result is a local version of Theorem \ref{the.2.44}.

\begin{theorem}
\label{the.2.47}Let $\mathcal{G=}\left(  \psi,U\right)  $ be a gauge on
$\mathbb{R}^{d}$, $\mathbf{z}=\left(  z,z^{\dag}\right)  \in C\left(  \left[
a,b\right]  ,\mathbb{R}^{d}\right)  \times C\left(  \left[  a,b\right]
,L\left(  W,\mathbb{R}^{d}\right)  \right)  $, and $\mathcal{W}$ be an open
convex set such that $z\left(  \left[  a,b\right]  \right)  \subseteq
\mathcal{W}$ and $\mathcal{W\times W}\subseteq D\left(  \mathcal{G}\right)  $.
Then $\mathbf{z}\in CRP_{\mathbf{X}}\left(  \mathbb{R}^{d}\right)  $ iff
$\mathbf{z}$ is a $\left(  \psi,U\right)  -$rough path controlled by
$\mathbf{X}$ with the choice $\delta:=$ $b-a$.
\end{theorem}

\begin{proof}
Suppose $\mathbf{z\in}CRP_{\mathbf{X}}\left(  \mathbb{R}^{d}\right)  .$ By
Theorem \ref{the.2.24},%
\[
\bar{\psi}\left(  x,y\right)  =y-x+A\left(  x,y\right)  \left(  y-x\right)
^{\otimes2}\text{~}\forall~x,y\in\mathcal{W}.
\]
Clearly $A$ is bounded if it is restricted to $x,y$ in the convex hull of
$z\left(  \left[  a,b\right]  \right)  $ (which is compact and contained in
$\mathcal{W}$). Thus, for all such points, we have there exists a $C_{1}$ such
that%
\begin{equation}
\left\vert \bar{\psi}\left(  x,y\right)  -\left(  y-x\right)  \right\vert \leq
C_{1}\left\vert y-x\right\vert ^{2}.\label{equ.2.26}%
\end{equation}
Taking $y=z_{t}$ and $x=z_{s}$ in this inequality shows%
\begin{equation}
\left\vert \bar{\psi}\left(  z_{s},z_{t}\right)  -z_{s,t}\right\vert \leq
C_{1}\left\vert z_{t}-z_{s}\right\vert ^{2}.\label{equ.2.27}%
\end{equation}
Since $\mathbf{z\in}CRP_{\mathbf{X}}\left(  \mathbb{R}^{d}\right)  ,$ there
exists a $C_{2}$ such that
\begin{align}
\left\vert z_{s,t}-z_{s}^{\dag}x_{s,t}\right\vert  &  \leq C_{2}\omega\left(
s,t\right)  ^{2/p}\label{equ.2.28}\\
\left\vert z_{s,t}^{\dag}\right\vert  &  \leq C_{2}\omega\left(  s,t\right)
^{1/p}.\label{equ.2.29}%
\end{align}
By enlarging $C_{2}$ if necessary we may further conclude,%
\begin{equation}
\left\vert z_{s,t}\right\vert \leq C_{2}\omega\left(  s,t\right)
^{1/p}.\label{equ.2.30}%
\end{equation}
Using Eqs. (\ref{equ.2.28}) and (\ref{equ.2.30}) in Eq. (\ref{equ.2.27}) gives
the existence of a $C_{3}<\infty$ such that%
\[
\left\vert \bar{\psi}\left(  z_{s},z_{t}\right)  -z_{s}^{\dag}x_{s,t}%
\right\vert \leq C_{3}\omega\left(  s,t\right)  ^{2/p}.
\]
By Theorem \ref{the.2.24} once more, we have%
\begin{equation}
\bar{U}\left(  x,y\right)  =I+B\left(  x,y\right)  \left(  y-x\right)
.\label{equ.2.31}%
\end{equation}
As was the case for $A$, $B$ is bounded on the convex hull of $z\left(
\left[  a,b\right]  \right)  $ so that there exists a $C_{4}$ such that%
\begin{align*}
\left\vert \bar{U}\left(  z_{s},z_{t}\right)  z_{t}^{\dag}-z_{s}^{\dag
}\right\vert  &  \leq\left\vert z_{s,t}^{\dag}\right\vert +C_{4}\left\vert
z_{s,t}\right\vert \\
&  \leq\left(  C_{2}+C_{4}C_{2}\right)  \omega\left(  s,t\right)  ^{1/p}.
\end{align*}
Thus $\mathbf{z}$ is a $\left(  \psi,U\right)  -$rough path controlled by
$\mathbf{X}$ with the choice $\delta:=$ $b-a$ where our $C:=\max\left\{
C_{1},C_{2}\left(  1+C_{4}\right)  \right\}  .$

For the converse direction, suppose $\mathbf{z}$ is a $\left(  \psi,U\right)
-$rough path controlled by $\mathbf{X}$ with the choice $\delta:=$ $b-a$ as in
Definition \ref{def.2.35}. From Eq. (\ref{equ.2.26}) and the triangle
inequality we have%
\[
\left\vert y-x\right\vert \leq C_{1}\left\vert y-x\right\vert ^{2}+\left\vert
\bar{\psi}\left(  x,y\right)  \right\vert .
\]
Taking $x=z_{s}$ and $y=z_{t}$ in this inequality and using Definition
\ref{def.2.35} we may find $C_{2}<\infty$ such that%
\begin{align*}
\left\vert z_{s,t}\right\vert  &  \leq C_{1}\left\vert z_{s,t}\right\vert
^{2}+\left\vert \psi\left(  z_{s},z_{t}\right)  \right\vert \\
&  \leq C_{1}\left\vert z_{s,t}\right\vert ^{2}+C_{2}\omega\left(  s,t\right)
^{1/p}%
\end{align*}
for all $s\leq t$ in $\left[  a,b\right]  $. By the uniform continuity of $z$
on $\left[  a,b\right]  $, there exists $\epsilon>0$ such that $C_{1}%
\left\vert z_{s,t}\right\vert \leq\frac{1}{2}$ when $\left\vert t-s\right\vert
\leq\epsilon$ which combined with the previous inequality implies
\[
\left\vert z_{s,t}\right\vert \leq2C_{2}\omega\left(  s,t\right)  ^{1/p}\text{
when }\left\vert t-s\right\vert \leq\epsilon.
\]
For general $a\leq s\leq t\leq b$ we may write $z_{s,t}$ as a sum of at most
$n\leq\left(  b-a\right)  /\epsilon$ increments whose norms are bounded by
$2C_{2}\omega\left(  s,t\right)  ^{1/p}$ wherein we have repeatedly used the
estimate above along with the monotonicity of $\omega$ resulting from
superactivity. Thus we conclude, with $C_{3}:=2C_{2}\left(  b-a\right)
/\epsilon<\infty,$ that%
\[
\left\vert z_{s,t}\right\vert \leq C_{3}\omega\left(  s,t\right)
^{1/p}~\forall~s,t\in\left[  a,b\right]  .
\]
This estimate along with the inequality in Eq. (\ref{equ.2.26}) gives,%
\[
\left\vert \bar{\psi}\left(  z_{s,}z_{t}\right)  -z_{s,t}\right\vert \leq
C_{1}\left\vert z_{s,t}\right\vert ^{2}\leq C_{1}C_{3}^{2}\omega\left(
s,t\right)  ^{2/p}~\forall~s,t\in\left[  a,b\right]  .
\]
The previous inequality along with the assumption that $\mathbf{z}$ is a
$\left(  \psi,U\right)  -$rough path shows there exists $C_{4}<\infty$ such
that%
\[
\left\vert z_{s,t}-z_{s}^{\dag}x_{s,t}\right\vert \leq\left\vert z_{s,t}%
-\bar{\psi}\left(  z_{s,}z_{t}\right)  \right\vert +\left\vert \bar{\psi
}\left(  z_{s,}z_{t}\right)  -z_{s}^{\dag}x_{s,t}\right\vert \leq C_{4}%
\omega\left(  s,t\right)  ^{2/p}.
\]
From Eq. (\ref{equ.2.31}), there exists a $C_{5}$ such that
\[
\left\vert z_{s,t}^{\dag}\right\vert \leq\left\vert U\left(  z_{s}%
,z_{t}\right)  z_{t}^{\dag}-z_{s}^{\dag}\right\vert +C_{5}\left\vert
z_{s,t}\right\vert .
\]
This inequality along with the assumption that $\mathbf{z}$ is a $\left(
\psi,U\right)  -$rough path shows there exists $C_{6}<\infty$ such that
$\left\vert z_{s,t}^{\dag}\right\vert \leq C_{6}\omega\left(  s,t\right)
^{1/p}$ for all $a\leq s\leq t\leq b.$ Thus we have shown $\mathbf{z}\in
CRP_{\mathbf{X}}\left(  \mathbb{R}^{d}\right)  $.
\end{proof}

The rest of this section is now devoted to a number of \textquotedblleft
stitching\textquotedblright\ arguments which will be used to piece together a
number of local versions of Theorem \ref{the.2.44} over subintervals as
described in Theorem \ref{the.2.47} into the full global version as stated in
Theorem \ref{the.2.44}. For the rest of this section let $\mathcal{X}$ be a
topological space and $0\leq S<T<\infty.$

\begin{lemma}
\label{lem.2.48}If $y:\left[  S,T\right]  \rightarrow\mathcal{X}$ is
continuous and $y\left(  \left[  S,T\right]  \right)  \subseteq\bigcup
_{\alpha\in A}\mathcal{O}_{\alpha}$ where $\left\{  \mathcal{O}_{\alpha
}\right\}  _{\alpha\in A}$ is a collection of open subsets of $\mathcal{X}$,
then there exists a partition of $\left[  S,T\right]  $, $S=t_{0}<t_{1}%
<\ldots<t_{l}=T,$ and $\alpha_{i}\in A$ such that for all $i$ less than $l$,
we have
\[
y\left(  \left[  t_{i},t_{i+1}\right]  \right)  \subseteq\mathcal{O}%
_{\alpha_{i}}%
\]

\end{lemma}

\begin{proof}
Define $T^{\ast}:=\sup\left\{  t:S\leq t\leq T,~\text{the conclusion of the
theorem holds for }\left[  S,t\right]  \right\}  $. Note that trivially
$T^{\ast}>S$. For sake of contradiction, suppose $T^{\ast}<T.$ Then there
exists an $\epsilon>0$ such that $T^{\ast}+\epsilon<T,$ $T^{\ast}-\epsilon>S$
and $y\left(  T^{\ast}-\epsilon,T^{\ast}+\epsilon\right)  \subset
\mathcal{O}_{\alpha^{\ast}}$ for some $\alpha^{\ast}.$ But the condition of
the theorem holds for $T^{\ast}-\epsilon$ for some partition $P$. By appending
$P$ with $T^{\ast}+\lambda\epsilon$ with $\lambda\in(-1,1]$ we have that
$T^{\ast}\geq T^{\ast}+\epsilon$ which is absurd. Thus, we must have that
$T^{\ast}=T.$
\end{proof}

\begin{definition}
\label{def.2.49}The set $\left\{  a_{i},b_{i}\right\}  _{i=0}^{l}%
\subset\left[  S,T\right]  $ is an \textbf{interlaced cover of }$\left[
S,T\right]  $ if $S=a_{0}<a_{1}<b_{0}<a_{2}<b_{1}<a_{3}<b_{2}<\ldots
<a_{l}<b_{l-1}<b_{l}\,=T$. Let $y:\left[  S,T\right]  \rightarrow\mathcal{X}$.
The set $\left\{  a_{i},b_{i}\right\}  _{i=0}^{l}$ is an \textbf{interlaced
cover for }$y$ if $\left\{  a_{i},b_{i}\right\}  _{i=0}^{l}\ $is an interlaced
cover of $\left[  S,T\right]  $ and $y\left(  a_{i+1}\right)  \neq y\left(
b_{i}\right)  $ for all $i$ less than $l.$

\end{definition}

\begin{figure}[ptbh]
\centering
\par
\psize{5.5in} %
\executeiffilenewer{\GraphicsDirectoryInterlacedPar.svg}{\GraphicsDirectoryInterlacedPar.pdf}%
{inkscape -z -D --file=\GraphicsDirectoryInterlacedPar.svg --export-pdf=\GraphicsDirectoryInterlacedPar.pdf --export-latex}%
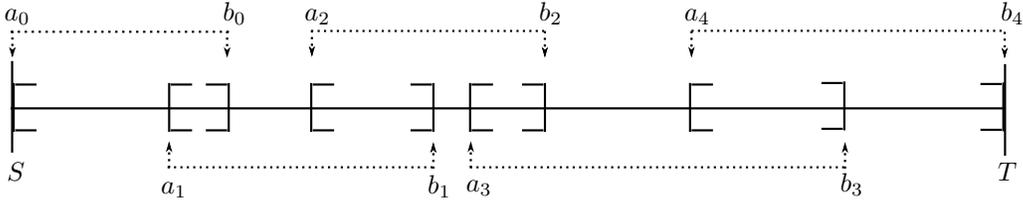%
 \caption{An interlaced cover of
$[S,T]$}%
\label{fig.1}%
\end{figure}

\begin{corollary}
\label{cor.2.50}Suppose $y:\left[  S,T\right]  \rightarrow\mathcal{X}$ is
continuous and $y\left(  \left[  S,T\right]  \right)  \subseteq\bigcup
_{\alpha\in A}\mathcal{O}_{\alpha}$ where $\left\{  \mathcal{O}_{\alpha
}\right\}  _{\alpha\in A}$ is a collection of open sets $\mathcal{O}_{\alpha}%
$. There exists an interlaced cover for $y$, $\left\{  a_{i},b_{i}\right\}
_{i=0}^{l}$ such that $y\left(  \left[  a_{i},b_{i}\right]  \right)
\subseteq\mathcal{O}_{\alpha_{i}}$. Note that for such a setup, this implies
$y\left(  \left[  a_{i+1,}b_{i}\right]  \right)  \subseteq\mathcal{O}%
_{\alpha_{i}}\cap\mathcal{O}_{\alpha_{i+1}}$
\end{corollary}

\begin{proof}
The first step will be a technical one to get rid of unnecessary endpoints.
Let $t_{i}^{\prime}$ and $\alpha_{i}^{\prime}$ be as given in Lemma
\ref{lem.2.48}. Then clearly $y\left(  t_{i}^{\prime}\right)  \in
\mathcal{O}_{\alpha_{i-1}^{\prime}}\cap\mathcal{O}_{\alpha_{i}^{\prime}}$ for
all $1\leq i<l^{\prime}$. Starting with $t_{1}^{\prime}$, we check if
$y\left(  \left[  t_{0}^{\prime},t_{1}^{\prime}\right]  \right)
\subseteq\mathcal{O}_{\alpha_{1}}$. In the case it is, we may renumber our
partition after removing $t_{1}^{\prime}$ and $\mathcal{O}_{\alpha_{0}%
^{\prime}}$ to get a new set of $t_{j}^{\prime}$ and $\alpha_{j}^{\prime}$
which still satisfy the result of the lemma. Continuing this process
inductively, we may assume that we have such a set $\left\{  t_{i},\alpha
_{i}\right\}  _{i=0}^{l}$ such that $y\left(  \left[  t_{i},t_{i+1}\right]
\right)  $ is not contained in $\mathcal{O}_{\alpha_{i+1}}$.

To construct the desired interlaced cover, we define $b_{i}:=t_{i+1}$ for all
$i\leq l:=l^{\prime}-1$ and $a_{0}:=S.$ Note for now that this means $y\left(
\left[  b_{i-1},b_{i}\right]  \right)  \subseteq\mathcal{O}_{\alpha_{i}}$.
Then we define the \textquotedblleft lower end\textquotedblright\ stopping
time $T_{i}$ for all $i>0$ by the formula%
\[
T_{i}:=\inf\left\{  t<b_{i}:y\left(  \left[  t,b_{i}\right]  \right)
\subseteq\mathcal{O}_{\alpha_{i+1}}\right\}  .
\]
By construction and because we refined our partition, $b_{i-1}\leq T_{i}%
<b_{i}.$ It is clear that $y\left(  T_{i}\right)  \neq y\left(  b_{i}\right)
$ by the continuity of $y$. Thus, there exists a time $T_{i}^{\ast}$ such that
$T_{i}<T_{i}^{\ast}$ and $y\left(  T_{i}^{\ast}\right)  \neq y\left(
b_{i}\right)  $. Define
\[
a_{i+i}:=T_{i}^{\ast}%
\]
for all $0<i<n.$ Since $y\left(  \left[  b_{i-1},b_{i}\right]  \right)
\subseteq\mathcal{O}_{\alpha_{i}}$ and $a_{i}>b_{i-1}$, we have that $y\left(
\left[  a_{i},b_{i}\right]  \right)  \subseteq\mathcal{O}_{\alpha_{i}}.$
\end{proof}

Since the following patching trick will be used multiple times in later
proofs, we will prove it here in more generality to avoid too much indexing
notation later.

\begin{lemma}
\label{lem.2.51}Let $\omega$ be a control and $\left\{  a_{i},b_{i}\right\}
_{i=0}^{l}$ be an interlaced cover of $\left[  S,T\right]  $ such that
$\omega\left(  a_{i+1},b_{i}\right)  >0$ for all $i<n.$ Let $\theta>0$ and
$F:D\rightarrow\lbrack0,\infty)$ be a bounded function such that
$D\subseteq\Delta_{\left[  S,T\right]  }$ and for each $1\leq i\leq l$ there
exists $C_{i}<\infty$ such that
\[
F\left(  s,t\right)  \leq C_{i}\omega\left(  s,t\right)  ^{\theta}\text{ for
all }\left(  s,t\right)  \in\Delta_{\left[  a_{i},b_{i}\right]  }\cap D.
\]
Then there exists a $\tilde{C}<\infty$ such that
\begin{equation}
F\left(  s,t\right)  \leq\tilde{C}\omega\left(  s,t\right)  ^{\theta}%
~\forall~\left(  s,t\right)  \in D. \label{equ.2.32}%
\end{equation}

\end{lemma}

\begin{proof}
Let
\begin{align*}
m  &  :=\min\left\{  \omega\left(  a_{i+1},b_{i}\right)  ^{\theta}:0\leq
i<n\right\}  ,\\
C  &  :=\max\left\{  C_{i}:0\leq i\leq n\right\}  ,\text{ and}\\
M  &  :=\sup\left\{  F\left(  s,t\right)  :\left(  s,t\right)  \in D\right\}
<\infty
\end{align*}
and then define $\tilde{C}:=\max\left\{  \frac{M}{m},C\right\}  $. We claim
that Inequality (\ref{equ.2.32}) holds.

If there exists an $i$ such that $s,t\in\left[  a_{i},b_{i}\right]  \cap D$,
then (\ref{equ.2.32}) holds trivially. Otherwise, let $i^{\ast}$ be the
largest $i$ such that $s\in\left[  a_{i},b_{i}\right]  $. Then $s<a_{i^{\ast
}+1}$ and $t>b_{i^{\ast}}$. However this says that $\left[  s,t\right]
\supset\left[  a_{i^{\ast}+1},b_{i^{\ast}}\right]  $ so that
\[
F\left(  s,t\right)  \leq M=\frac{M}{m}m\leq\tilde{C}\omega\left(  a_{i^{\ast
}+1},b_{i^{\ast}}\right)  ^{\theta}\leq\tilde{C}\omega\left(  s,t\right)
^{\theta}.
\]

\end{proof}

\subsubsection{Proof of Theorem \ref{the.2.44}}

The recurring strategy here will be localize appropriately to work in the
$\mathbb{R}^{d}$ case so that we may apply Theorem \ref{the.2.47}.\ We must
choose these localizations carefully so that we may patch the estimates
together (with two different strategies) using the lemmas above. One method of
patching is a bit more involved than the other; therefore we will present it
more formally:

\begin{remark}
[Proof Strategy]\label{rem.2.52}Let $y:\left[  a,b\right]  \rightarrow M$ be
the first component of $\left(  y,y^{\dag}\right)  $ where $\left(  y,y^{\dag
}\right)  $ is either a $\left(  \psi,U\right)  -$ controlled rough path or
chart controlled rough path. Also suppose for each $m\in y\left(  \left[
a,b\right]  \right)  $, we are given an open neighborhood, $\mathcal{W}%
_{m}\subseteq M,$ of $m.$ By Corollary \ref{cor.2.50}, there exists an
interlaced cover for $y$, $\left\{  a_{i},b_{i}\right\}  _{i=1}^{l}$ and
$\left\{  m_{i}\right\}  _{i=1}^{l}$ such that $y\left(  \left[  a_{i}%
,b_{i}\right]  \right)  \subseteq\mathcal{W}_{m_{i}}$ and $\omega\left(
a_{i+1},b_{i}\right)  >0$. Thus, if $F:D\rightarrow\lbrack0,\infty)$ is a
bounded function such that $D\subseteq\Delta_{\lbrack a,b]},$ then in order to
prove that
\begin{equation}
F\left(  s,t\right)  \leq C\omega\left(  s,t\right)  ^{\theta}\text{ }%
\forall~\left(  s,t\right)  \in D, \label{equ.2.33}%
\end{equation}
it suffices to prove; for each $1\leq i\leq l$ there exists $C_{i}<\infty$
such that
\[
F\left(  s,t\right)  \leq C_{i}\omega\left(  s,t\right)  ^{\theta}\text{ for
all }\left(  s,t\right)  \in\Delta_{\left[  a_{i},b_{i}\right]  }\cap D.
\]
Therefore in attempting to prove an assertion in the form of Inequality
(\ref{equ.2.33}), we may assume, without loss of generality, that $y\left(
\left[  a,b\right]  \right)  \subseteq\mathcal{W}$ where the $\mathcal{W}$
will have nice properties dependent on our setting.
\end{remark}

The proof of Theorem \ref{the.2.44} will consist of two steps:

\begin{enumerate}
\item If gauge conditions of (\ref{equ.2.22}) and (\ref{equ.2.23}) hold for
some $C>0$ and $\delta>0$, then the chart conditions of (\ref{equ.2.24}) and
(\ref{equ.2.25}) hold. We will reduce this to the $\mathbb{R}^{d}$ case
immediately, then use Lemma \ref{lem.2.6} to patch the estimates together.

\item If the chart condition of (\ref{equ.2.24}) and (\ref{equ.2.25}) hold,
then gauge condition of (\ref{equ.2.22}) and (\ref{equ.2.23}) hold for an
appropriately chosen $\delta$. Here we will first show which local estimates
we need to satisfy to use Remark \ref{rem.2.52} and then reduce to the
$\mathbb{R}^{d}$ case.
\end{enumerate}

In simple terms, step 1 is \textquotedblleft localize then
patch\textquotedblright\ and step 2 is \textquotedblleft cut nicely, localize,
then patch\textquotedblright.

\begin{proof}
[Proof of Theorem \ref{the.2.44}]\textbf{Step 1: Definition \ref{def.2.35}
}$\implies$\textbf{ Definition \ref{def.2.40}.}

We'll first assume that the gauge definition holds, i.e. that there exists a
$\delta>0$ and a $C_{1}>0$ such that%
\begin{equation}
\left\vert \psi\left(  y_{s},y_{t}\right)  -y_{s}^{\dag}x_{s,t}\right\vert
_{g}\leq C_{1}\omega\left(  s,t\right)  ^{2/p} \label{equ.2.34}%
\end{equation}
and%
\[
\left\vert U\left(  y_{s},y_{t}\right)  y_{t}^{\dag}-y_{s}^{\dag}\right\vert
_{g}\leq C_{1}\omega\left(  s,t\right)  ^{1/p}%
\]
hold for all $0\leq s\leq t\leq T$ such that $\left\vert t-s\right\vert
\leq\delta$. Let $\phi$ be a chart on $M$ and let $\left[  a,b\right]  $ be
such that $y\left(  \left[  a,b\right]  \right)  \subseteq D\left(
\phi\right)  .$ If we define
\begin{align*}
\psi^{\phi}\left(  x,y\right)   &  :=\phi_{\ast}\psi\left(  \phi^{-1}\left(
x\right)  ,\phi^{-1}\left(  y\right)  \right) \\
U^{\phi}\left(  x,y\right)   &  :=\phi_{\ast}U\left(  \phi^{-1}\left(
x\right)  ,\phi^{-1}\left(  y\right)  \right)  \circ\left(  \phi_{\ast}%
^{-1}\right)  _{\phi\left(  y\right)  }\\
\left(  z_{s},z_{s}^{\dag}\right)   &  :=\phi_{\ast}\left(  \mathbf{y}%
_{s}\right)  =\left(  \phi\left(  y_{s}\right)  ,d\phi\circ y_{s}^{\dag
}\right)
\end{align*}
then it is clear that there exists a $C_{2}=C_{2}\left(  \phi_{\ast}\right)  $
such that
\begin{align}
\left\vert \bar{\psi}^{\phi}\left(  z_{s},z_{t}\right)  -z_{s}^{\dag}%
x_{s,t}\right\vert  &  \leq C_{2}\omega\left(  s,t\right)  ^{2/p}%
\label{equ.2.35}\\
\left\vert \bar{U}^{\phi}\left(  z_{s},z_{t}\right)  z_{t}^{\dag}-z_{s}^{\dag
}\right\vert  &  \leq C_{2}\omega\left(  s,t\right)  ^{1/p} \label{equ.2.36}%
\end{align}
for all $a\leq s\leq t\leq b$ such that $t-s\leq\delta$ where $\left(
\psi^{\phi},U^{\phi}\right)  $ is a local gauge on $\mathbb{R}^{d}$ and
$\left(  \bar{\psi}^{\phi},\bar{U}^{\phi}\right)  $ is consistent with
Notation \ref{not.2.23}. Thus $\left(  z,z^{\dag}\right)  $ is a $\left(
\psi^{\phi},U^{\phi}\right)  -$rough path controlled by $\mathbf{X}$. Finally
we need to use this information to show there exists a $C_{\phi,a,b}$ such
that%
\begin{equation}
\left\vert z_{t}-z_{s}-z_{s}^{\dag}x_{s,t}\right\vert \leq C_{\phi,a,b}%
\omega\left(  s,t\right)  ^{2/p}. \label{equ.2.37}%
\end{equation}
and%
\begin{equation}
\left\vert z_{t}^{\dag}-z_{s}^{\dag}\right\vert \leq C_{\phi,a,b}\omega\left(
s,t\right)  ^{1/p} \label{equ.2.38}%
\end{equation}
for all $s,t$ such that $a\leq s\leq t\leq b$.

In light of the Sewing Lemma \ref{lem.2.6} and Lemma \ref{lem.2.48}, we only
need to show that for each $u\in\left[  a,b\right]  $, the inequalities
(\ref{equ.2.37}) and (\ref{equ.2.38}) hold with $C_{\phi,a,b}$ replaced with
$C_{u}$ for all $s,t\in\left(  u-\delta_{u},u+\delta_{u}\right)  \cap\left[
a,b\right]  $ such that $s\leq t$ for some $\delta_{u}>0$.

For any $u\in\left[  a,b\right]  ,$ let $\mathcal{W}_{u}$ be an open convex
set of $z_{u}$ such that $\mathcal{W}_{u}\times\mathcal{W}_{u}\subseteq
D\left(  \psi^{\phi}\right)  $. We then choose $\delta_{u}>0$ to be such that
$z\left(  [u-\delta_{u},u+\delta_{u}]\cap\left[  a,b\right]  \right)
\subseteq\mathcal{W}_{u}$ and $2\delta_{u}\leq\delta.$ However, now we are in
the setting of Theorem \ref{the.2.47} and are therefore finished with this step.

\textbf{Step 2: Definition \ref{def.2.40} }$\implies$ \textbf{Definition
\ref{def.2.35} }

Suppose that the chart item (\ref{equ.2.24}) holds. We must prove that there
exists a $\delta,C>0$ such that%
\begin{align*}
\left\vert \psi\left(  y_{s},y_{t}\right)  -y_{s}^{\dag}x_{s,t}\right\vert
_{g}  &  \leq C\omega\left(  s,t\right)  ^{2/p}\\
\left\vert U\left(  y_{s},y_{t}\right)  y_{t}^{\dag}-y_{s}^{\dag}\right\vert
_{g}  &  \leq C\omega\left(  s,t\right)  ^{1/p}%
\end{align*}
for all $s\leq t$ such that $\left\vert t-s\right\vert \leq\delta$.

We choose $\delta$ such that $\left\vert t-s\right\vert \leq\delta$ for $0\leq
s\leq t\leq T$ implies that both $\left\vert \psi\left(  y_{s},y_{t}\right)
\right\vert _{g}$ and $\left\vert U\left(  y_{s},y_{t}\right)  \right\vert
_{g}$ make sense and are bounded. Around every point $m$ of $y\left(  \left[
0,T\right]  \right)  $, there exists an open $\mathcal{O}_{m}$ containing $m$
and such that $\mathcal{O}_{m}\times\mathcal{O}_{m}\subseteq\mathcal{D}$.
Additionally there exists a chart $\phi^{m}$ such that $m\in D\left(  \phi
^{m}\right)  $. By considering an open ball around $\phi^{m}\left(  m\right)
$ in $R\left(  \phi^{m}\right)  $ and shrinking the radius, we may assume that
$\mathcal{V}_{m}:=D\left(  \phi^{m}\right)  \subseteq\mathcal{O}_{m}$ and the
range, $\mathcal{W}_{m}:=\phi\left(  \mathcal{V}_{m}\right)  ,$ of $\phi^{m}$
is convex. Since $\left\{  \mathcal{V}_{m}\right\}  _{m\in y\left(  \left[
0,T\right]  \right)  }$ is an open cover of $y\left(  \left[  0,T\right]
\right)  $, we may use this cover along with $D=\left\{  \left(  s,t\right)
:0\leq s\leq t\leq T\text{ and }\left\vert t-s\right\vert \leq\delta\right\}
$ to employ the proof strategy in Remark \ref{rem.2.52}. We will do this
twice, with $F\left(  s,t\right)  =\left\vert \psi\left(  y_{s},y_{t}\right)
-y_{s}^{\dag}x_{s,t}\right\vert _{g}$ in the first iteration and $F\left(
s,t\right)  =\left\vert U\left(  y_{s},y_{t}\right)  y_{t}^{\dag}-y_{s}^{\dag
}\right\vert _{g}$ in the second; this will reduce us to considering the case
where there exists a single chart $\phi$ such that $y\left(  \left[
0,T\right]  \right)  \subseteq D\left(  \phi\right)  $, $D\left(  \phi\right)
\times D\left(  \phi\right)  \mathcal{\subseteq D}$ and $\mathcal{W}=R\left(
\phi\right)  $ is convex.

Now that we have reduced to a single chart $\phi$, we may define $\left(
\psi^{\phi},U^{\phi}\right)  $ and the path $\left(  z,z^{\dag}\right)  $ as
in Step 1. Then $z\left(  \left[  0,T\right]  \right)  \subseteq\mathcal{W}$
and $\mathcal{W\times W\subseteq}D\left(  \psi^{\phi}\right)  =D\left(
U^{\phi}\right)  .$ However, by Theorem \ref{the.2.47} we have that the proper
estimates hold because $\mathbf{z}$ is a $\left(  \psi^{\phi},U^{\phi}\right)
-$rough path controlled by $\mathbf{X}.$ Therefore, we are finished by
patching using Remark \ref{rem.2.52}.
\end{proof}

\begin{remark}
\label{rem.2.53}In the proof of Theorem \ref{the.2.44}, we would have been
able to show (and did so somewhat indirectly) that Inequality (\ref{equ.2.24})
implies Inequality (\ref{equ.2.22}) for some $\delta>0$. However, it is not
true in general that, for a fixed $\delta$, Inequality (\ref{equ.2.22})
implies Inequality (\ref{equ.2.24}). See Example \ref{exa.6.7} in the Appendix
for a counterexample.
\end{remark}

In situations in which we are given a covariant derivative $\nabla$ on a
manifold, by Example \ref{exa.2.19}, we have an equivalent definition:

\begin{example}
\label{exa.2.54}The pair $\left(  y_{s},y_{s}^{\dag}\right)  $ is an element
of $CRP_{\mathbf{X}}\left(  M\right)  $ if and only if there exists a $C$ such that

\begin{enumerate}
\item \label{Ite.18}%
\begin{equation}
\left\vert \left(  \exp_{y_{s}}^{\nabla}\right)  ^{-1}\left(  y_{t}\right)
-y_{s}^{\dag}x_{s,t}\right\vert _{g}\leq C\omega\left(  s,t\right)  ^{2/p}
\label{equ.2.39}%
\end{equation}

\item \label{Ite.19}%
\begin{equation}
\left\vert U_{y_{s},y_{t}}^{\nabla}y_{t}^{\dag}-y_{s}^{\dag}\right\vert
_{g}\leq C\omega\left(  s,t\right)  ^{1/p} \label{equ.2.40}%
\end{equation}
where $\left(  \exp_{m}^{\nabla}\right)  ^{-1}$ and $U_{n,m}^{\nabla}$ are
defined as in Example \ref{exa.2.19} and the inequalities hold when
$(y_{s},y_{t})$ are in the domain $\mathcal{D}$ as given in Theorem
\ref{the.2.32}. In particular, on a Riemannian manifold we can use this
definition with the Levi-Civita covariant derivative.
\end{enumerate}
\end{example}

Before providing yet another equivalent definition of controlled rough paths
on manifolds, we will present some examples.

\subsection{Examples of Controlled Rough Paths\label{sub.2.6}}

Recall $\mathbf{X}=\left(  x,\mathbb{X}\right)  $ is a weak-geometric rough
path with values in $W\oplus W^{\otimes2}$ where $W=\mathbb{R}^{k}$. The
results here will rely on basic approximations found in the Appendix, Section
\ref{sec.6}.

\begin{example}
\label{exa.2.55}Let $M^{d}\subseteq W$ be an embedded submanifold and for
every $m\in M^{d}$, let $P\left(  m\right)  $ be the orthogonal projection
onto the tangent space $T_{m}M$. Suppose $x_{s}\in M^{d}$ for all $s$ in
$\left[  0,T\right]  $. Then $\left(  x_{s},P(x_{s})\right)  $ $\in
CRP_{\mathbf{X}}\left(  M\right)  .$
\end{example}

\begin{proof}
We will use the gauge as given in Example \ref{exa.2.54} where the $\nabla$ is
the Levi-Civita covariant derivative from the induced metric from Euclidean
space. Verifying that $P\left(  x_{s}\right)  $ lives in the correct space is trivial.

Next, to show Inequality \ref{equ.2.39} is satisfied, we use item
\ref{Ite.29}. of Lemma \ref{lem.6.2} which says
\[
\exp_{m}^{-1}\left(  \tilde{m}\right)  =P\left(  m\right)  \left(  \tilde
{m}-m\right)  +O\left(  \left\vert \tilde{m}-m\right\vert ^{3}\right)  \text{
for all }m\in M^{d}.
\]
Letting $m=x_{s}$ and $\tilde{m}=x_{t}$, we are done.

Inequality (\ref{equ.2.40}) is also satisfied as a result of Lemma
\ref{lem.6.2} which says that $U_{\tilde{m},m}^{\nabla}=P\left(  m\right)
+O\left(  \left\vert \tilde{m}-m\right\vert \right)  .$ Thus%
\begin{align*}
P\left(  x_{t}\right)  -U_{x_{t},x_{s}}^{\nabla}P\left(  x_{s}\right)   &
\underset{^{1}}{\approx}P\left(  x_{t}\right)  -P\left(  x_{s}\right)
P\left(  x_{s}\right) \\
&  =P\left(  x_{t}\right)  -P\left(  x_{s}\right) \\
&  \underset{^{1}}{\approx}0
\end{align*}

\end{proof}

The next example will be proved in more generality in Section \ref{sub.4.2}.
However, we find it instructive to prove it without charts and in the embedded
context where the reader may be more comfortable.

\begin{example}
\label{exa.2.56}Let $f$ be a smooth function from $W$ to an embedded manifold
$\tilde{M}^{d}\subseteq\mathbb{R}^{\tilde{k}}.$ Then $\left(  f\left(
x_{s}\right)  ,f^{\prime}\left(  x_{s}\right)  \right)  $ $\in CRP_{\mathbf{X}%
}\left(  \tilde{M}\right)  $.
\end{example}

\begin{proof}
Again we will use the Levi-Civita covariant derivative $\tilde{\nabla}$ from
the embedded metric. First we note that $f^{\prime}\left(  x_{s}\right)  $
lives in the correct space as $R\left(  f\right)  \subseteq\tilde{M}^{d}$.

To show Inequality (\ref{equ.2.39}) holds one can use the fact that $\left(
f(\left(  x_{s}\right)  ,f^{\prime}\left(  x_{s}\right)  \right)  $ is a
controlled rough path in the embedded space or Taylor's Theorem to see that%
\[
f\left(  x_{t}\right)  -f\left(  x_{s}\right)  -f^{\prime}\left(
x_{s}\right)  \left(  x_{t}-x_{s}\right)  \underset{^{2}}{\approx}0
\]
which easily implies%
\[
P\left(  f\left(  x_{s}\right)  \right)  \left[  f\left(  x_{t}\right)
-f\left(  x_{s}\right)  -f^{\prime}\left(  x_{s}\right)  \left(  x_{t}%
-x_{s}\right)  \right]  \underset{^{2}}{\approx}0.
\]
But again by Lemma \ref{lem.6.2}%
\begin{align*}
&  P\left(  f\left(  x_{s}\right)  \right)  \left[  f\left(  x_{t}\right)
-f\left(  x_{s}\right)  -f^{\prime}\left(  x_{s}\right)  \left(  x_{t}%
-x_{s}\right)  \right] \\
&  \quad=P\left(  f\left(  x_{s}\right)  \right)  \left[  f\left(
x_{t}\right)  -f\left(  x_{s}\right)  \right]  -f^{\prime}\left(
x_{s}\right)  \left(  x_{t}-x_{s}\right) \\
&  \quad\underset{^{2}}{\approx}\left(  \exp_{f\left(  x_{s}\right)  }%
^{\tilde{\nabla}}\right)  ^{-1}\left(  f\left(  x_{t}\right)  \right)
-f^{\prime}\left(  x_{s}\right)  \left(  x_{t}-x_{s}\right)  .
\end{align*}
Thus%
\[
\left(  \exp_{f\left(  x_{s}\right)  }^{\tilde{\nabla}}\right)  ^{-1}\left(
f\left(  x_{t}\right)  \right)  -f^{\prime}\left(  x_{s}\right)  \left(
x_{t}-x_{s}\right)  \underset{^{2}}{\approx}0.
\]

Lastly to show Inequality (\ref{equ.2.40}), we have%
\[
f^{\prime}\left(  x_{t}\right)  -f^{\prime}\left(  x_{s}\right)
\underset{^{1}}{\approx}0
\]
and therefore
\begin{align*}
0  &  \underset{^{1}}{\approx}P\left(  f\left(  x_{t}\right)  \right)  \left[
f^{\prime}\left(  x_{t}\right)  -f^{\prime}\left(  x_{s}\right)  \right] \\
&  =f^{\prime}\left(  x_{t}\right)  -P\left(  f\left(  x_{t}\right)  \right)
f^{\prime}\left(  x_{s}\right) \\
&  \underset{^{1}}{\approx}f^{\prime}\left(  x_{t}\right)  -U_{f\left(
x_{t}\right)  ,f\left(  x_{s}\right)  }^{\tilde{\nabla}}f^{\prime}\left(
x_{s}\right)  ,
\end{align*}
wherein we have used $P\left(  f\left(  x_{t}\right)  \right)  f^{\prime
}\left(  x_{t}\right)  =f^{\prime}\left(  x_{t}\right)  $ in the second line
and Lemma \ref{lem.6.2} in the last. Thus $\left(  f\left(  x_{s}\right)
,f^{\prime}\left(  x_{s}\right)  \right)  $ $\in CRP_{\mathbf{X}}\left(
\tilde{M}\right)  $
\end{proof}

\subsection{Smooth Function Definition of CRP\label{sub.2.7}}

In the spirit of semi-martingales on manifolds [see for example \cite[Chapter
III]{Emery1989} or \cite{Hsu2002,Elworthy1982,Ikeda1989}], we can define
controlled rough paths on manifolds as elements which, when composed with any
smooth function, give rise to a one-dimensional controlled rough path on flat
space. More precisely we have the following theorem.

\begin{theorem}
\label{the.2.57}$\mathbf{y}=\left(  y,y^{\dag}\right)  \in CRP_{\mathbf{X}%
}\left(  M\right)  $ if and only if for every $f\in C^{\infty}\left(
M\right)  $,
\[
f_{\ast}\mathbf{y}=\left(  f\left(  y\right)  ,df\circ y^{\dag}\right)  \in
CRP_{\mathbf{X}}\left(  \mathbb{R}\right)  .
\]

\end{theorem}

\begin{proof}
The proof that $\mathbf{y}\in CRP_{\mathbf{X}}\left(  M\right)  $ implies that
$f_{\ast}\mathbf{y}\in CRP_{\mathbf{X}}\left(  \mathbb{R}\right)  $ for every
$f\in C^{\infty}\left(  M\right)  $ will be deferred to the more general case
proved in Proposition \ref{pro.4.10} (in which case we consider the codomain
of $f$ to be a manifold $\tilde{M}$).

To prove the converse, let $\phi$ be a chart and $0\leq a<b\leq T$ be such
that $y\left(  \left[  a,b\right]  \right)  \subseteq D\left(  \phi\right)  $
and let $\mathcal{O}\subset M$ be an open set such that $\mathcal{\bar{O}}$ is
compact and
\[
y\left(  \left[  a,b\right]  \right)  \subseteq\mathcal{O\subseteq\bar{O}%
}\subseteq D\left(  \phi\right)  .
\]
Then by using a cutoff function we can manufacture global functions $f^{i}\in
C^{\infty}\left(  M\right)  $ which agree with the coordinates $\phi^{i}$ on
$\mathcal{O}$. The assumption that $f_{\ast}^{i}\mathbf{y\in}CRP_{\mathbf{X}%
}\left(  \left[  a,b\right]  ,\mathbb{R}\right)  $ is a controlled rough path
for $1\leq i\leq d$ then shows the inequalities in Eqs. (\ref{equ.2.24}) and
(\ref{equ.2.25}) of Definition \ref{def.2.40} hold.
\end{proof}

\section{Integration of Controlled One-Forms\label{sec.3}}

In the flat case, a controlled rough path with values in an appropriate
Euclidean spaces can be integrated against another controlled rough path (see
Theorem \ref{the.2.9}) provided their controlling rough path $\mathbf{X}$ is
the same. The integral in this case is another rough path controlled by
$\mathbf{X}$. We can do something similar on manifolds, though it will be
necessary to add some extra structure. As usual let $\mathbf{y}_{s}=\left(
y_{s},y_{s}^{\dag}\right)  $ be a controlled rough path on $M$ controlled by
$\mathbf{X}=\left(  x,\mathbb{X}\right)  \in W\mathcal{\oplus}W^{\otimes2}.$
Let $V$ be a Banach space.

\subsection{Controlled One-Forms Along a Rough Path\label{sub.3.1}}

Let $U$ be a parallelism on $M$.

\begin{definition}
\label{def.3.1}The pair $\left(  \alpha_{s},\alpha_{s}^{\dag}\right)  $ is a
$V$ -- valued $U-$controlled (rough) one-form along $y_{s}$ if

\begin{enumerate}
\item \label{Ite.20}$\alpha_{s}\in L\left(  T_{y_{s}}M,V\right)  $

\item \label{Ite.21}$\alpha_{s}^{\dag}\in L\left(  W\otimes T_{y_{s}%
}M,V\right)  $

\item \label{Ite.22}$\alpha_{t}\circ U\left(  y_{t},y_{s}\right)  -\alpha
_{s}-\alpha_{s}^{\dag}\left(  x_{s,t}\otimes\left(  \cdot\right)  \right)
\underset{^{2}}{\approx}0$

\item \label{Ite.23}$\alpha_{t}^{\dag}\circ\left(  I\otimes U\left(
y_{t},y_{s}\right)  \right)  -\alpha_{s}^{\dag}\underset{^{1}}{\approx}0$
\end{enumerate}

By items \ref{Ite.22} and \ref{Ite.23}, we mean these hold if $\left\vert
t-s\right\vert <\delta$ for some $\delta>0$ to ensure the expressions make sense.
\end{definition}

\begin{remark}
\label{rem.3.2}For the sake of clarity, by item \ref{Ite.22} of Definition
\ref{def.3.1}, we mean that if $s,t$ are close, then there exists a $C$ such
that%
\[
\left\vert \alpha_{t}\circ U\left(  y_{t},y_{s}\right)  -\alpha_{s}-\alpha
_{s}^{\dag}\left(  x_{s,t}\otimes\left(  \cdot\right)  \right)  \right\vert
_{g,op}\leq C\omega\left(  s,t\right)  ^{2/p}.
\]
For item \ref{Ite.23}, we mean for $s,t$ close, there exists a $C$ such that%
\[
\left\vert \alpha_{t}^{\dag}\circ\left(  w\otimes U\left(  y_{t},y_{s}\right)
\right)  -\alpha_{s}^{\dag}\left(  w\otimes\left(  \cdot\right)  \right)
\right\vert _{g,op}\leq C\left\vert w\right\vert \omega\left(  s,t\right)
^{1/p}%
\]
for all $w\in W$. By Corollary \ref{cor.6.4}, it does not matter which
Riemannian metric $g$ we choose here.
\end{remark}

\begin{notation}
\label{not.3.3}Let $CRP_{y}^{U}\left(  M,V\right)  $ denote those
$\boldsymbol{\alpha}_{s}:=\left(  \alpha_{s},\alpha_{s}^{\dag}\right)  $
satisfying Definition \ref{def.3.1}. We refer to $CRP_{y}^{U}\left(
M,V\right)  $ as a \textbf{space of }$U\mathbf{-}$\textbf{controlled one-forms
along} $\mathbf{y}.$
\end{notation}

\begin{remark}
\label{rem.3.4}If $M=\mathbb{R}^{d}$ and $U=I$ and we identify $T_{y_{s}}M$
with $\mathbb{R}^{d}$ then Definition \ref{def.3.1} reduces to the flat case
definition of a $L\left(  \mathbb{R}^{d},V\right)  $ -- valued rough path
controlled by $\mathbf{X}$.
\end{remark}

\begin{remark}
\label{rem.3.5}Note that \ref{Ite.22} and \ref{Ite.23}, of Definition
\ref{def.3.1} force continuity of both $\alpha_{s}$ and $\alpha_{s}^{\dag}$.
\end{remark}

We can take linear combinations of elements of $CRP_{y}^{U}\left(  M,V\right)
$ to form other elements in $CRP_{y}^{U}\left(  M,V\right)  .$ The following
proposition, whose simple proof is left to the reader, shows how to construct
more non-trivial examples of elements in $CRP_{y}^{U}\left(  M,V\right)  .$

\begin{proposition}
\label{pro.3.6}If $V$ and $\tilde{V}$ are Banach spaces, $\boldsymbol{\alpha
}\in CRP_{y}^{U}\left(  M,V\right)  $ and
\[
\mathbf{f}=\left(  f,f^{\dag}\right)  \in CRP_{\mathbf{X}}\left(
\operatorname{Hom}\left(  V,\tilde{V}\right)  \right)  ,
\]
then
\[
\left(  \mathbf{f}\boldsymbol{\alpha}\right)  _{s}:=\left(  f_{s}\alpha
_{s},f_{s}^{\dag}\alpha_{s}+f_{s}\alpha_{s}^{\dag}\right)  \in CRP_{y}%
^{U}\left(  M,\tilde{V}\right)  .
\]
where by $f_{s}^{\dag}\alpha_{s}$ we mean $f_{s}^{\dag}\left(  \left(
\cdot\right)  \otimes\alpha_{s}\left(  \cdot\right)  \right)  .$
\end{proposition}

Our next goal is to define an integral of$\ \boldsymbol{\alpha}_{s}$ along
$\mathbf{y}_{s}$. However, this integral will depend on a choice of
parallelism and for this reason we need to introduce the \textquotedblleft
compatibility tensor\textquotedblright\ which measures the difference between
two parallelisms.

\subsection{The Compatibility Tensors\label{sub.3.2}}

\begin{definition}
\label{def.3.7}The \textbf{compatibility tensor, }$S^{\tilde{U},U}\in
\Gamma\left(  T^{\ast}M\otimes T^{\ast}M\otimes TM\right)  ,$ of two
parallelisms $\tilde{U}$ and $U$ on $M$ is the defined by%
\[
S_{m}^{\tilde{U},U}:=d\left[  U\left(  \cdot,m\right)  ^{-1}\tilde{U}\left(
\cdot,m\right)  \right]  _{m}.
\]
In more detail if $v_{m},w_{m}\in T_{m}M,$ then
\[
S_{m}^{\tilde{U},U}\left[  v_{m}\otimes w_{m}\right]  =v_{m}\left[
x\rightarrow U\left(  x,m\right)  ^{-1}\tilde{U}\left(  x,m\right)
w_{m}\right]  .
\]

\end{definition}

\begin{remark}
\label{rem.3.8}There are actually multiple ways to define $S_{m}^{\tilde{U}%
,U}.$ For example, we have on simple tensors%
\begin{align}
S_{m}^{\tilde{U},U}\left(  v_{m}\otimes w_{m}\right)   &  =d\left[  U\left(
m,\cdot\right)  \tilde{U}\left(  m,\cdot\right)  ^{-1}w_{m}\right]  _{m}%
v_{m}\nonumber\\
&  =\left(  \nabla_{v_{m}}\left[  \tilde{U}\left(  \cdot,m\right)  -U\left(
\cdot,m\right)  \right]  \right)  w_{m}\nonumber\\
&  =\left(  \nabla_{v_{m}}\left[  U\left(  m,\cdot\right)  -\tilde{U}\left(
m,\cdot\right)  \right]  \right)  w_{m} \label{equ.3.1}%
\end{align}
where $\nabla$ is any covariant derivative on $M$. Similar to the proofs of
Corollary \ref{cor.2.29} above and Theorem \ref{the.3.15} below, the
identities in Eq. (\ref{equ.3.1}) are straightforward to prove by employing
charts to reduce them to Euclidean space identities.
\end{remark}

\begin{example}
\label{exa.3.9}If $\nabla$ and $\tilde{\nabla}$ are two covariant derivatives
on $TM,$ $U=U^{\nabla},$ $\tilde{U}=U^{\tilde{\nabla}},$ and $A\in\Omega
^{1}\left(  \operatorname*{End}\left(  TM\right)  \right)  $ such that
$\nabla=\tilde{\nabla}+A,$ then
\[
S_{m}^{\tilde{U},U}\left(  v_{m}\otimes w_{m}\right)  =A\left(  v_{m}\right)
w_{m}\in T_{m}M.
\]
Indeed,
\begin{align*}
v_{m}\left[  U\left(  \cdot,m\right)  ^{-1}\tilde{U}\left(  \cdot,m\right)
w_{m}\right]   &  =\nabla_{v_{m}}\left[  \tilde{U}\left(  \cdot,m\right)
w_{m}\right] \\
&  =\tilde{\nabla}_{v_{m}}\left[  \tilde{U}\left(  \cdot,m\right)
w_{m}\right]  +A\left(  v_{m}\right)  \tilde{U}\left(  m,m\right)  w_{m}\\
&  =0+A\left(  v_{m}\right)  w_{m}=A\left(  v_{m}\right)  w_{m}.
\end{align*}

\end{example}

\begin{example}
[Converse of Example \ref{exa.3.9}]\label{exa.3.10}If $U$ and $\tilde{U}$ are
two parallelisms on $M$ and $\nabla=\nabla^{U}$ and $\tilde{\nabla}%
=\nabla^{\tilde{U}}$ are the corresponding covariant derivatives on $TM$ (as
in Remark \ref{rem.2.20}), then
\[
\nabla_{v_{m}}=\tilde{\nabla}_{v_{m}}+S_{m}^{\tilde{U},U}\left(  v_{m}%
\otimes\left(  \cdot\right)  \right)  ~\forall~v_{m}\in T_{m}M.
\]
The verification is as follows. If $Y$ is a vector-field on $M$ and
$\sigma_{t}$ is such that $\dot{o}_{0}=v_{m}$, we have%
\begin{align*}
\nabla_{v_{m}}Y-\tilde{\nabla}_{v_{m}}Y  &  :=\frac{d}{dt}|_{0}\left[
U\left(  m,\sigma_{t}\right)  -\tilde{U}\left(  m,\sigma_{t}\right)  \right]
Y\left(  \sigma_{t}\right) \\
&  =\left(  \nabla_{v_{m}}\left[  U\left(  m,\cdot\right)  -\tilde{U}\left(
m,\cdot\right)  \right]  \right)  Y\left(  m\right)  +0\cdot\nabla_{v_{m}}Y\\
&  =S_{m}^{\tilde{U},U}\left(  v_{m}\otimes Y\left(  m\right)  \right)
\end{align*}
wherein we have used Eq. (\ref{equ.3.1}) for the last equality.
\end{example}

\begin{lemma}
\label{lem.3.11}If $U,\tilde{U},$ and $\hat{U}$ are three parallelisms, then%
\[
S^{\hat{U},U}=S^{\hat{U},\tilde{U}}+S^{\tilde{U},U}\text{ and }S^{\tilde{U}%
,U}=-S^{U,\tilde{U}}.
\]

\end{lemma}

\begin{proof}
For $v_{m},w_{m}\in T_{m}M,$ an application of the product rules shows%
\begin{align*}
S_{m}^{\hat{U},U}\left(  v_{m}\otimes w_{m}\right)   &  =v_{m}\left[  U\left(
\cdot,m\right)  ^{-1}\hat{U}\left(  \cdot,m\right)  w_{m}\right] \\
&  =v_{m}\left[  \left[  U\left(  \cdot,m\right)  ^{-1}\tilde{U}\left(
\cdot,m\right)  \right]  \left[  \tilde{U}\left(  \cdot,m\right)  ^{-1}\hat
{U}\left(  \cdot,m\right)  \right]  w_{m}\right] \\
&  =S_{m}^{\hat{U},\tilde{U}}\left(  v_{m}\otimes w_{m}\right)  +S_{m}%
^{\tilde{U},U}\left(  v_{m}\otimes w_{m}\right)  .
\end{align*}
Similarly,%
\begin{align*}
S^{U,\tilde{U}}\left[  v_{m}\otimes\left(  \cdot\right)  \right]   &
=v_{m}\left[  \tilde{U}\left(  \cdot,m\right)  ^{-1}U\left(  \cdot,m\right)
\right] \\
&  =v_{m}\left[  U\left(  \cdot,m\right)  ^{-1}\tilde{U}\left(  \cdot
,m\right)  \right]  ^{-1}\\
&  =-v_{m}\left[  U\left(  \cdot,m\right)  ^{-1}\tilde{U}\left(
\cdot,m\right)  \right] \\
&  =-S^{\tilde{U},U}\left[  v_{m}\otimes\left(  \cdot\right)  \right]  .
\end{align*}

\end{proof}

\begin{notation}
\label{not.3.12}If $\mathcal{G}:=\left(  \psi,U\right)  $ is a gauge, we let
$S^{\mathcal{G}}:=S^{\psi_{\ast},U}$ be the compatibility tensor between
$U^{\psi}$ and $U,$ where $U^{\psi}\left(  m,n\right)  :=\psi\left(
m,\cdot\right)  _{\ast n}$ as in Remark \ref{rem.2.21}.
\end{notation}

If we have a covariant derivative $\nabla$ on $M$, then as in Example
\ref{exa.2.19} we have the choice of gauge $\mathcal{G=}\left(  \psi,U\right)
=\left(  \left(  \exp^{\nabla}\right)  ^{-1},U^{\nabla}\right)  $. In this
case, the tensor $S_{m}^{\mathcal{G}}$ is a more familiar object.

\begin{lemma}
\label{lem.3.13}If $\psi=\left(  \exp^{\nabla}\right)  ^{-1}$ and
$U=U^{\nabla}$, then
\[
S_{m}^{\mathcal{G}}=\frac{1}{2}T_{m}^{\nabla}%
\]
where $T^{\nabla}$ is the Torsion tensor of $\nabla.$
\end{lemma}

\begin{proof}
By transferring the covariant derivative and functions using charts, we may
assume we are working on Euclidean space. In this case, by Eq. (\ref{equ.6.10}%
) and Corollary \ref{cor.6.6}, we have%
\begin{align*}
S_{m}^{\mathcal{G}}\left(  \left(  m,v\right)  \otimes\left(  m,w\right)
\right)   &  =\left(  \nabla_{\left(  m,v\right)  }\left[  U_{m,\cdot}%
^{\nabla}-\left(  \exp_{m}^{\nabla}\right)  _{\ast\cdot}^{-1}\right]  \right)
w\\
&  =\left[  \partial_{\left(  m,v\right)  }+A_{m}\left\langle v\right\rangle
\right]  \left[  U_{m,\cdot}^{\nabla}-\left(  \exp_{m}^{\nabla}\right)
_{\ast\cdot}^{-1}\right]  w\\
&  =\left(  U_{m,\cdot}^{\nabla}\right)  ^{\prime}\left(  m\right)  \left[
v\otimes w\right]  -\left(  \left(  \exp_{m}^{\nabla}\right)  ^{-1}\right)
^{^{\prime\prime}}\left(  m\right)  \left[  v\otimes w\right]  +A_{m}%
\left\langle v\right\rangle \left\langle w\right\rangle -A_{m}\left\langle
v\right\rangle \left\langle w\right\rangle \\
&  =A_{m}\left\langle v\right\rangle \left\langle w\right\rangle -\frac{1}%
{2}A_{m}\left\langle v\right\rangle \left\langle w\right\rangle -\frac{1}%
{2}A_{m}\left\langle w\right\rangle \left\langle v\right\rangle \\
&  =\frac{1}{2}\left[  A_{m}\left\langle v\right\rangle \left\langle
w\right\rangle -A_{m}\left\langle w\right\rangle \left\langle v\right\rangle
\right] \\
&  =\frac{1}{2}T_{m}^{\nabla}\left(  \left(  m,v\right)  \otimes\left(
m,w\right)  \right)  .
\end{align*}

\end{proof}

Here is one last example of a gauge and its compatibility tensor.

\begin{proposition}
\label{pro.3.14}Let $G$ be a Lie group and $\nabla$ be the left covariant
derivative on $TG$ uniquely determined by requiring the left invariant vector
fields to be covariantly constant, i.e. $\nabla\tilde{A}=0$ for all
$A\in\mathfrak{g}.$ Then for $g$ near $k,$%
\begin{equation}
U^{\nabla}\left(  g,k\right)  =//\left(  k\rightarrow g\right)  =L_{gk^{-1}%
\ast}, \label{equ.3.2}%
\end{equation}
and
\begin{equation}
\psi^{\nabla}\left(  k,g\right)  =\left(  \exp_{k}^{\nabla}\right)
^{-1}\left(  g\right)  =k\cdot\log\left(  k^{-1}g\right)  \label{equ.3.3}%
\end{equation}
where $L_{g}:G\rightarrow G$ is left multiplication by $g\in G$ and $\log$ is
the local inverse of the map $A\rightarrow e^{A}.$ Moreover the compatibility
tensor for this gauge is given by
\begin{equation}
S\left(  \xi_{g},\eta_{g}\right)  =-\frac{1}{2}L_{g\ast}\left[  \theta\left(
\xi_{g}\right)  ,\theta\left(  \eta_{g}\right)  \right]  ~\text{ for all }%
\xi_{g},\eta_{g}\in T_{g}G \label{equ.3.4}%
\end{equation}
where $\theta$ is the Maurer-Cartan form on $G$ defined by $\theta\left(
\xi\right)  :=L_{g^{-1}\ast}\xi\in\mathfrak{g}:=T_{e}G$ for all $\xi\in
T_{g}G.$
\end{proposition}

\begin{proof}
The torsion of $\nabla$ is given by%
\[
T\left(  \tilde{A},\tilde{B}\right)  =\nabla_{\tilde{A}}\tilde{B}%
-\nabla_{\tilde{B}}\tilde{A}-\left[  \tilde{A},\tilde{B}\right]
=-\widetilde{\left[  A,B\right]  }%
\]
or equivalently as%
\[
T\left(  \xi_{g},\eta_{g}\right)  =-L_{g\ast}\left[  \theta\left(  \xi
_{g}\right)  ,\theta\left(  \eta_{g}\right)  \right]  \text{ for all }\xi
_{g},\eta_{g}\in T_{g}G.
\]
Eq. (\ref{equ.3.4}) follows from the above formula along with the result in
Lemma \ref{lem.3.13}.

If $\xi\left(  t\right)  $ is a path $TG$ above $\sigma\left(  t\right)  \in
G$ it may be written as $\xi\left(  t\right)  =L_{\sigma\left(  t\right)
\ast}\theta\left(  \xi\left(  t\right)  \right)  .$ Since $L_{\sigma\left(
t\right)  \ast}$ is parallel translation, it follows that%
\[
\frac{\nabla\xi\left(  t\right)  }{dt}=L_{\sigma\left(  t\right)  \ast}%
\frac{d}{dt}\theta\left(  \xi\left(  t\right)  \right)  .
\]
Thus $\xi\left(  t\right)  \in TG$ is parallel iff $\theta\left(  \xi\left(
t\right)  \right)  $ is constant for all $t.$ If $\sigma$ is a general curve
in $G$, we may conclude
\[
//\left(  \sigma|_{\left[  s,t\right]  }\right)  =L_{\sigma\left(  t\right)
\ast}L_{\sigma\left(  s\right)  ^{-1}\ast}=L_{\sigma\left(  t\right)
\sigma\left(  s\right)  ^{-1}\ast}%
\]
and therefore $U^{\nabla}$ is given as in Eq. (\ref{equ.3.2}).

A curve $\sigma\left(  t\right)  \in G$ is a geodesic iff $\dot{\sigma}\left(
t\right)  $ is parallel iff $\theta\left(  \dot{\sigma}\left(  t\right)
\right)  =A$ for some $A\in\mathfrak{g.}$ That is $\dot{\sigma}\left(
t\right)  =\tilde{A}\left(  \sigma\left(  t\right)  \right)  $ with
$\sigma\left(  0\right)  =k\in G.$ The solution to this equation is
$\sigma\left(  t\right)  =ke^{tA}$ and hence we have shown that $\exp
_{k}^{\nabla}\left(  k\cdot A\right)  =ke^{A}.$ So setting $g=ke^{A}$ and
solving for $A$ gives $A=\log\left(  k^{-1}g\right)  $ and the formula for
$\psi^{\nabla}$ in Eq. (\ref{equ.3.3}) now follows.
\end{proof}

The last three results of this subsection show how the compatibility tensor
allows us to compare two different parallelisms and two different logarithms
on $M.$

\begin{theorem}
\label{the.3.15}Suppose that $U$ and $\tilde{U}$ are two parallelisms on $M$
and $\psi$ is a logarithm on $M,$ then
\begin{equation}
U\left(  m,n\right)  \tilde{U}\left(  m,n\right)  ^{-1}=_{2}I+S_{m}^{\tilde
{U},U}\left(  \psi\left(  m,n\right)  \otimes\left(  \cdot\right)  \right)  .
\label{equ.3.5}%
\end{equation}

\end{theorem}

\begin{proof}
By using charts it suffices to prove the theorem when $M=\mathbb{R}^{d}.$ By
Taylor's theorem (see Theorem \ref{the.2.24}),
\begin{align*}
U\left(  m,n\right)   &  =_{2}I+\left[  \left(  D_{2}U\right)  \left(
m,m\right)  \left(  n-m\right)  \right]  \text{ and}\\
\tilde{U}\left(  m,n\right)   &  =_{2}I+\left[  \left(  D_{2}\tilde{U}\right)
\left(  m,m\right)  \left(  n-m\right)  \right]
\end{align*}
and therefore
\begin{align}
U\left(  m,n\right)  \tilde{U}\left(  m,n\right)  ^{-1}  &  =_{2}\left(
I+\left[  \left(  D_{2}U\right)  \left(  m,m\right)  \left(  n-m\right)
\right]  \right)  \left(  I-\left[  \left(  D_{2}\tilde{U}\right)  \left(
m,m\right)  \left(  n-m\right)  \right]  \right) \nonumber\\
&  =_{2}I+\left[  \left(  \left(  D_{2}U\right)  \left(  m,m\right)  -\left(
D_{2}\tilde{U}\right)  \left(  m,m\right)  \right)  \left(  n-m\right)
\right]  . \label{equ.3.6}%
\end{align}
However, by Eq. (\ref{equ.3.1}) we have%
\begin{equation}
S_{m}^{\tilde{U},U}=\left(  D_{2}U\right)  \left(  m,m\right)  -\left(
D_{2}\tilde{U}\right)  \left(  m,m\right)  . \label{equ.3.7}%
\end{equation}
Using this identity back in Eq. (\ref{equ.3.6}) shows%
\[
U\left(  m,n\right)  \tilde{U}\left(  m,n\right)  ^{-1}=_{2}I+S_{m}^{\tilde
{U},U}\left(  \left[  n-m\right]  _{m}\otimes\left(  \cdot\right)  \right)
\]
from which Eq. (\ref{equ.3.5}) follows because $\psi\left(  m,n\right)
=_{2}\left[  n-m\right]  _{m}.$
\end{proof}

\begin{corollary}
\label{cor.3.16}If $\mathcal{G=}\left(  \psi,U\right)  $ is a gauge on $M,$
then
\begin{equation}
\psi\left(  n,\cdot\right)  _{\ast m}=_{2}U\left(  n,m\right)  \left[
I+S_{m}^{\mathcal{G}}\left(  \psi\left(  m,n\right)  \otimes\left(
\cdot\right)  \right)  \right]  . \label{equ.3.8}%
\end{equation}

In particular%
\begin{equation}
\psi\left(  y_{t},\cdot\right)  _{\ast y_{s}}\underset{^{2}}{\approx}U\left(
y_{t},y_{s}\right)  \left[  I+S_{y_{s}}^{\mathcal{G}}\left(  \psi\left(
y_{s},y_{t}\right)  \otimes\left(  \cdot\right)  \right)  \right]  .
\label{equ.3.9}%
\end{equation}

\end{corollary}

\begin{proof}
Theorem \ref{the.3.15} implies
\[
U\left(  m,n\right)  \psi\left(  m,\cdot\right)  _{\ast n}^{-1}=_{2}%
I+S_{m}^{\mathcal{G}}\left(  \psi\left(  m,n\right)  \otimes\left(
\cdot\right)  \right)
\]
while Corollary \ref{cor.2.29} shows,%
\[
U\left(  m,n\right)  ^{-1}=_{2}U\left(  n,m\right)  \text{ and }\psi\left(
m,\cdot\right)  _{\ast n}^{-1}=_{2}\psi\left(  n,\cdot\right)  _{\ast m}.
\]
Eq. (\ref{equ.3.8}) now easily follows from the last two displayed equations.
The second statement follows by patching.
\end{proof}

Lastly we may use the compatibility tensor to compare two logarithms.

\begin{proposition}
\label{pro.3.17}Suppose that $\psi$ and $\tilde{\psi}$ are two logarithms on a
manifold $M.$ Then the compatibility tensor, $S^{\psi_{\ast},\tilde{\psi
}_{\ast}}$ is symmetric and
\begin{equation}
\psi\left(  m,n\right)  -\tilde{\psi}\left(  m,n\right)  =_{3}\frac{1}{2}%
S_{m}^{\tilde{\psi}_{\ast},\psi_{\ast},}\left(  \psi\left(  m,n\right)
\otimes\psi\left(  m,n\right)  \right)  . \label{equ.3.10}%
\end{equation}

\end{proposition}

\begin{proof}
As usual it suffices to prove this result when $M=\mathbb{R}^{d}$ in which
case we omit the base points of tangent vectors. From Eq. (\ref{equ.3.7}) with
$U\left(  x,y\right)  =\psi_{x}^{\prime}\left(  y\right)  $ and $\tilde
{U}\left(  x,y\right)  =\tilde{\psi}_{x}^{\prime}\left(  y\right)  ,$ we see
that
\begin{equation}
S_{x}^{\tilde{\psi}_{\ast},\psi_{\ast}}=\psi_{x}^{\prime\prime}\left(
x\right)  -\tilde{\psi}_{x}^{\prime\prime}\left(  x\right)  \label{equ.3.11}%
\end{equation}
which is symmetric since mixed partial derivatives commute. Then by Taylor's
theorem and Eq. (\ref{equ.3.11}),%
\begin{align*}
\psi\left(  x,y\right)  -\tilde{\psi}\left(  x,y\right)   &  =\frac{1}%
{2}\left[  \psi_{x}^{\prime\prime}\left(  x\right)  -\tilde{\psi}_{x}%
^{\prime\prime}\left(  x\right)  \right]  \left(  y-x\right)  ^{\otimes
2}+O\left(  \left\vert y-x\right\vert ^{3}\right) \\
&  =\frac{1}{2}S_{x}^{\tilde{\psi}_{\ast},\psi_{\ast}}\left(  \psi\left(
x,y\right)  ^{\otimes2}\right)  +O\left(  \left\vert y-x\right\vert
^{3}\right)  ,
\end{align*}
wherein we have also used $\left(  y-x\right)  ^{\otimes2}=_{3}\psi\left(
x,y\right)  ^{\otimes2}.$
\end{proof}

\begin{remark}
\label{rem.3.18}If $\nabla$ is any covariant derivative on $TM,$ then
\[
S_{m}^{\tilde{\psi}_{\ast},\psi_{\ast}}=\left[  \nabla d\left(  \psi\left(
m,\cdot\right)  -\tilde{\psi}\left(  m,\cdot\right)  \right)  \right]
_{m}=\mathrm{Hess}_{m}^{\nabla}\left(  \psi_{m}-\tilde{\psi}_{m}\right)
\]
where $\mathrm{Hess}_{m}^{\nabla}f:=\left[  \nabla df\right]  _{m}$. By
choosing $\nabla$ to be Torsion free we again see that $S_{m}^{\tilde{\psi
}_{\ast},\psi_{\ast}}$ is a symmetric tensor.
\end{remark}

\subsection{$U$ -- Controlled Rough Integration\label{sub.3.3}}

Our next goal is to construct \textquotedblleft the\textquotedblright%
\ integral, $\int\left\langle \boldsymbol{\alpha},d\mathbf{y}\right\rangle ,$
where $\mathbf{y}\in CRP_{\mathbf{X}}\left(  M\right)  $ and
$\boldsymbol{\alpha}\in CRP_{y}^{U}\left(  M,V\right)  .$ We begin with the
following proposition in the smooth category which is meant to motivate the
definitions to come.

\begin{proposition}
\label{pro.3.19}Assume (in this proposition only) that all functions,
$\mathbf{y}_{s},$ $\boldsymbol{\alpha}_{s},$ and $x_{s}$ are smooth, $p=1,$
and $\omega\left(  s,t\right)  =\left\vert t-s\right\vert .$ Further assume
$\mathbf{y}$ (respectively $\boldsymbol{\alpha}$) still satisfy the estimates
of being controlled rough path (along $\mathbf{y).}$ Then%
\begin{equation}
\int_{s}^{t}\alpha_{\tau}\dot{y}_{\tau}d\tau=\alpha_{s}\left[  \psi\left(
y_{s},y_{t}\right)  +S_{y_{s}}^{\mathcal{G}}\left(  y_{s}^{\dag}\otimes
y_{s}^{\dag}\mathbb{X}_{s,t}\right)  \right]  +\alpha_{s}^{\dag}\left(
I\otimes y_{s}^{\dag}\right)  \mathbb{X}_{s,t}+O\left(  \left(  t-s\right)
^{3}\right)  . \label{equ.3.12}%
\end{equation}

\end{proposition}

\begin{proof}
Our assumptions give,%
\begin{align*}
&  \psi\left(  y_{s},y_{t}\right)  =y_{s}^{\dag}x_{s,t}+O\left(  \left(
t-s\right)  ^{2}\right)  \implies\dot{y}_{s}=y_{s}^{\dag}\dot{x}_{s},\\
&  \alpha_{t}U\left(  y_{t},y_{s}\right)  =\alpha_{s}+\alpha_{s}^{\dag}%
x_{s,t}+O\left(  \left(  t-s\right)  ^{2}\right)  ,\\
&  U\left(  y_{s},y_{t}\right)  y_{t}^{\dag}=y_{s}+O\left(  t-s\right)
,\text{ and }\\
&  \alpha_{t}^{\dag}\left(  I\otimes U\left(  y_{t},y_{s}\right)  \right)
=\alpha_{s}^{\dag}+O\left(  t-s\right)  .
\end{align*}
We start with the identity,
\begin{align}
\int_{s}^{t}\alpha_{\tau}\dot{y}_{\tau}d\tau &  =\int_{s}^{t}\alpha_{\tau
}U\left(  y_{\tau},y_{s}\right)  U\left(  y_{\tau},y_{s}\right)  ^{-1}\dot
{y}_{\tau}d\tau\nonumber\\
&  =\int_{s}^{t}\left[  \alpha_{s}+\alpha_{s}^{\dag}x_{s,\tau}+O\left(
\left(  \tau-s\right)  ^{2}\right)  \right]  U\left(  y_{\tau},y_{s}\right)
^{-1}\dot{y}_{\tau}d\tau\nonumber\\
&  =\int_{s}^{t}\alpha_{s}U\left(  y_{\tau},y_{s}\right)  ^{-1}\dot{y}_{\tau
}d\tau+\int_{s}^{t}\alpha_{s}^{\dag}x_{s,\tau}U\left(  y_{\tau},y_{s}\right)
^{-1}\dot{y}_{\tau}d\tau+O\left(  \left(  t-s\right)  ^{3}\right) \nonumber\\
&  =\int_{s}^{t}\alpha_{s}U\left(  y_{s},y_{\tau}\right)  \dot{y}_{\tau}%
d\tau+\int_{s}^{t}\alpha_{s}^{\dag}x_{s,\tau}U\left(  y_{s},y_{\tau}\right)
\dot{y}_{\tau}d\tau+O\left(  \left(  t-s\right)  ^{3}\right)
.\label{equ.3.13}\\
&  =:A+B+O\left(  \left(  t-s\right)  ^{3}\right)  \label{equ.3.14}%
\end{align}
wherein we have used Corollary \ref{cor.2.29} in order to show it is
permissible to replace $U\left(  y_{\tau},y_{s}\right)  ^{-1}$ by $U\left(
y_{s},y_{\tau}\right)  $ above. The $B$ term is then easily estimated as%
\begin{align*}
B  &  =\int_{s}^{t}\alpha_{s}^{\dag}x_{s,\tau}U\left(  y_{s},y_{\tau}\right)
\dot{y}_{\tau}d\tau=\int_{s}^{t}\alpha_{s}^{\dag}x_{s,\tau}U\left(
y_{s},y_{\tau}\right)  y_{\tau}^{\dag}\dot{x}_{\tau}d\tau\\
&  =\int_{s}^{t}\alpha_{s}^{\dag}x_{s,\tau}y_{s}^{\dag}\dot{x}_{\tau}%
d\tau+O\left(  \left(  t-s\right)  ^{3}\right)  =\alpha_{s}^{\dag}\left(
I\otimes y_{s}^{\dag}\right)  \mathbb{X}_{s,t}+O\left(  \left(  t-s\right)
^{3}\right)  .
\end{align*}
The estimate of the $A$ term to order $O\left(  \left(  t-s\right)
^{3}\right)  $ requires more care. For this term we use
\[
\frac{d}{dt}\psi\left(  y_{s},y_{t}\right)  =\psi\left(  y_{s},\cdot\right)
_{\ast y_{t}}\dot{y}_{t}\implies\dot{y}_{t}=\psi\left(  y_{s},\cdot\right)
_{\ast y_{t}}^{-1}\frac{d}{dt}\psi\left(  y_{s},y_{t}\right)
\]
and (from Theorem \ref{the.3.15}) that
\[
U\left(  y_{s},y_{\tau}\right)  \psi\left(  y_{s},\cdot\right)  _{\ast
y_{\tau}}^{-1}=_{2}I+S_{y_{s}}^{\mathcal{G}}\left(  \psi\left(  y_{s},y_{\tau
}\right)  \otimes\left(  \cdot\right)  \right)
\]
in order to conclude,%
\begin{align*}
A  &  :=\int_{s}^{t}\alpha_{s}U\left(  y_{s},y_{\tau}\right)  \dot{y}_{\tau
}d\tau=\int_{s}^{t}\alpha_{s}U\left(  y_{s},y_{\tau}\right)  \psi\left(
y_{s},\cdot\right)  _{\ast y_{\tau}}^{-1}\frac{d}{d\tau}\psi\left(
y_{s},y_{\tau}\right)  d\tau\\
&  =\int_{s}^{t}\alpha_{s}\left[  I+S_{y_{s}}^{\mathcal{G}}\left(  \psi\left(
y_{s},y_{\tau}\right)  \otimes\left(  \cdot\right)  \right)  \right]  \frac
{d}{d\tau}\psi\left(  y_{s},y_{\tau}\right)  d\tau+O\left(  \left\vert
t-s\right\vert ^{3}\right) \\
&  =\alpha_{s}\left(  \psi\left(  y_{s},y_{t}\right)  \right)  +\alpha_{s}%
\int_{s}^{t}S_{y_{s}}^{\mathcal{G}}\left(  \psi\left(  y_{s},y_{\tau}\right)
\otimes\frac{d}{d\tau}\psi\left(  y_{s},y_{\tau}\right)  \right)
d\tau+O\left(  \left\vert t-s\right\vert ^{3}\right) \\
&  =\alpha_{s}\left(  \psi\left(  y_{s},y_{t}\right)  \right)  +\alpha_{s}%
\int_{s}^{t}S_{y_{s}}^{\mathcal{G}}\left(  y_{s}^{\dag}x_{s,\tau}\otimes
y_{s}^{\dag}\dot{x}_{\tau}\right)  d\tau+O\left(  \left\vert t-s\right\vert
^{3}\right) \\
&  =\alpha_{s}\left(  \psi\left(  y_{s},y_{t}\right)  \right)  +\alpha
_{s}S_{y_{s}}^{\mathcal{G}}\left(  y_{s}^{\dag}\otimes y_{s}^{\dag}%
\mathbb{X}_{s,t}\right)  +O\left(  \left\vert t-s\right\vert ^{3}\right)  .
\end{align*}
Putting this all together proves Eq. (\ref{equ.3.12}).
\end{proof}

The following definition is motivated by the right hand side of Eq.
(\ref{equ.3.12}).

\begin{definition}
[$\left(  \mathcal{G},\mathbf{y}\right)  $-- integrator]\label{def.3.20}Given
a gauge $\mathcal{G}:=\left(  \psi,U\right)  $ and $\mathbf{y}\in
CRP_{\mathbf{X}}\left(  M\right)  ,$ the $\left(  \mathcal{G},\mathbf{y}%
\right)  $\textbf{-- integrator} is the increment process;%
\[
\mathbf{y}_{s,t}^{\mathcal{G}}:=\left(  \psi\left(  y_{s},y_{t}\right)
+S_{y_{s}}^{\mathcal{G}}\left(  y_{s}^{\dag\otimes2}\mathbb{X}_{s,t}\right)
,\left(  I\otimes y_{s}^{\dag}\right)  \mathbb{X}_{s,t}\right)  \in T_{y_{s}%
}M\times\left[  W\otimes T_{y_{s}}M\right]  .
\]
Moreover, for $\boldsymbol{\alpha}\in CRP_{y}^{U}\left(  M,V\right)  $ (see
Notation \ref{not.3.3}) let%
\begin{equation}
\tilde{z}_{s,t}:=\left\langle \boldsymbol{\alpha}_{s},\mathbf{y}%
_{s,t}^{\mathcal{G}}\right\rangle =\alpha_{s}\left(  \psi\left(  y_{s}%
,y_{t}\right)  +S_{y_{s}}^{\mathcal{G}}\left(  y_{s}^{\dag\otimes2}%
\mathbb{X}_{s,t}\right)  \right)  +\alpha_{s}^{\dag}\left(  I\otimes
y_{s}^{\dag}\right)  \mathbb{X}_{s,t} \label{equ.3.15}%
\end{equation}
which is defined for $\left(  s,t\right)  \in\Delta_{\left[  0,T\right]  }$
with $\left\vert t-s\right\vert <\delta$ for some sufficiently small
$\delta>0.$
\end{definition}

Recall that a two-parameter function $F:\Delta_{\left[  0,T\right]
}\longrightarrow V$ is an almost additive functional if there exists a
$\theta>1$, a control $\tilde{\omega}\left(  s,t\right)  $ and a $C>0$ such
that
\begin{equation}
\left\vert F_{s,u}-F_{s,t}-F_{t,u}\right\vert \leq C\tilde{\omega}\left(
s,t\right)  ^{\theta} \label{equ.3.16}%
\end{equation}
for all $0\leq s\leq t\leq u\leq T$.

\begin{theorem}
\label{the.3.21}Let $\mathcal{G}:=\left(  \psi,U\right)  $ be a gauge,
$\boldsymbol{\alpha}\in CRP_{y}^{U}\left(  M,V\right)  ,$ and $\tilde{z}%
_{s,t}$ be as in Definition \ref{def.3.20}. Then there exists a unique
$\mathbf{z=}\left(  z,z^{\dag}\right)  \mathbf{\in}CRP_{\mathbf{X}}\left(
V\right)  $ such that $z_{0}=0,$ $z_{s,t}\underset{^{3}}{\approx}\tilde
{z}_{s,t},$ and $z_{s}^{\dag}=\alpha_{s}\circ y_{s}^{\dag}.$ We denote this
unique controlled rough path by $\int\left\langle \boldsymbol{\alpha
}\mathbf{,}d\mathbf{y}^{\mathcal{G}}\right\rangle ,$ i.e.
\[
\int_{s}^{t}\left\langle \boldsymbol{\alpha}\mathbf{,}d\mathbf{y}%
^{\mathcal{G}}\right\rangle :=\left[  \int\left\langle \boldsymbol{\alpha
}\mathbf{,}d\mathbf{y}^{\mathcal{G}}\right\rangle \right]  _{s,t}%
^{1}\underset{^{3}}{\approx}\left\langle \boldsymbol{\alpha}_{s}%
,\mathbf{y}_{s,t}^{\mathcal{G}}\right\rangle \text{ and }\left[
\int\left\langle \boldsymbol{\alpha}\mathbf{,}d\mathbf{y}^{\mathcal{G}%
}\right\rangle \right]  _{s}^{\dag}=\alpha_{s}\circ y_{s}^{\dag}.
\]

\end{theorem}

\begin{proof}
By Theorem \ref{the.3.26} below, $\tilde{z}_{s,t}:=\left\langle
\boldsymbol{\alpha}_{s},\mathbf{y}_{s,t}^{\mathcal{G}}\right\rangle $ is an
almost additive functional and therefore by Lyons \cite[Theorem 3.3.1]%
{lyons-98} there exists a unique additive functional $z_{s,t}$ such that
$z_{s,t}\underset{^{3}}{\approx}\tilde{z}_{s,t}.$ Moreover,
\[
z_{s,t}\underset{^{3}}{\approx}\tilde{z}_{s,t}\underset{^{2}}{\approx}%
\alpha_{s}\left(  \psi\left(  y_{s},y_{t}\right)  \right)  \underset{^{2}%
}{\approx}\alpha_{s}\left(  y_{s}^{\dag}x_{s,t}\right)
\]
which shows that $\mathbf{z}_{s}:=\left(  z_{s},\alpha_{s}\circ y_{s}^{\dag
}\right)  $ is indeed a controlled rough path with values in $V.$
\end{proof}

\begin{example}
\label{exa.3.22}In the case that $U=U^{\psi}$ so that
\[
\alpha_{t}\circ\left(  \psi_{y_{t}}\right)  _{\ast y_{s}}-\alpha_{s}%
-\alpha_{s}^{\dag}\left(  x_{s,t}\otimes\left(  \cdot\right)  \right)
\underset{^{2}}{\approx}0
\]
we have that $\mathbf{y}_{s,t}^{\mathcal{G}}:=\left(  \psi\left(  y_{s}%
,y_{t}\right)  ,\left(  I\otimes y_{s}^{\dag}\right)  \mathbb{X}_{s,t}\right)
$ and so%
\[
\int_{s}^{t}\left\langle \boldsymbol{\alpha},d\mathbf{y}^{\mathcal{G}^{\psi}%
}\right\rangle \underset{^{3}}{\approx}\alpha_{s}\left(  \psi\left(
y_{s},y_{t}\right)  \right)  +\alpha_{s}^{\dag}\left(  I\otimes y_{s}^{\dag
}\right)  \mathbb{X}_{s,t}.
\]

\end{example}

\begin{example}
If $\mathcal{G}^{\nabla}=\left(  \left(  \exp^{\nabla}\right)  ^{-1}%
,U^{\nabla}\right)  $, then by Lemma \ref{lem.3.13}, we have that%
\[
\int_{s}^{t}\left\langle \boldsymbol{\alpha},d\mathbf{y}^{\mathcal{G}^{\nabla
}}\right\rangle \underset{^{3}}{\approx}\alpha_{s}\left(  \exp_{y_{s}}%
^{-1}\left(  y_{t}\right)  \right)  +\alpha_{s}^{\dag}\left(  I\otimes
y_{s}^{\dag}\right)  \mathbb{X}_{s,t}+\alpha_{s}\left(  \frac{1}{2}T_{y_{s}%
}^{\nabla}\circ y_{s}^{\dag\otimes2}\mathbb{X}_{s,t}\right)  .
\]

\end{example}

If $\mathbf{f},\boldsymbol{\alpha}$, and $\mathbf{f}\boldsymbol{\alpha}$ $\in
CRP_{y}^{U}\left(  M,\tilde{V}\right)  $ are as in Proposition \ref{pro.3.6},
then the following expected associativity property holds.

\begin{theorem}
[Associativity Theorem I]\label{the.3.24}Let us continue the notation in
Theorem \ref{the.3.21}. If $\mathbf{f}$ and $\mathbf{f}\boldsymbol{\alpha
}:=\left(  f_{s}\alpha_{s},f_{s}^{\dag}\left(  I\otimes\alpha_{s}\right)
+f_{s}\alpha_{s}^{\dag}\right)  $ are as in Proposition \ref{pro.3.6} and
$\mathbf{z}=\left(  z,z^{\dag}\right)  =\int\left\langle \boldsymbol{\alpha
},d\mathbf{y}^{\mathcal{G}}\right\rangle ,$ then%
\[
\int\left\langle \mathbf{f},d\mathbf{z}\right\rangle =\int\left\langle
\mathbf{f}\boldsymbol{\alpha},d\mathbf{y}^{\mathcal{G}}\right\rangle ,
\]
or in other words,
\[
\int\left\langle \mathbf{f},d\int\left\langle \boldsymbol{\alpha}%
,d\mathbf{y}^{\mathcal{G}}\right\rangle \right\rangle =\int\left\langle
\mathbf{f}\boldsymbol{\alpha},d\mathbf{y}^{\mathcal{G}}\right\rangle .
\]

\end{theorem}

\begin{proof}
We have the approximations
\begin{align*}
\left[  \int\left\langle \mathbf{f}\boldsymbol{\alpha},d\mathbf{y}%
^{\mathcal{G}}\right\rangle \right]  _{s,t}^{1}  &  \underset{^{3}}{\approx
}f_{s}\alpha_{s}\left(  \psi\left(  y_{s,}y_{t}\right)  +\mathcal{S}_{y_{s}%
}^{\mathcal{G}}\left(  y_{s}^{\dag\otimes2}\mathbb{X}_{s,t}\right)  \right)
+\left[  \left(  f_{s}^{\dag}\left(  I\otimes\alpha_{s}\right)  +f_{s}%
\alpha_{s}^{\dag}\right)  \right]  \left(  I\otimes y_{s}^{\dag}\right)
\mathbb{X}_{s,t}\\
&  =f_{s}\left(  \alpha_{s}\left(  \psi\left(  y_{s,}y_{t}\right)
+\mathcal{S}_{y_{s}}^{\mathcal{G}}\left(  y_{s}^{\dag\otimes2}\mathbb{X}%
_{s,t}\right)  \right)  +\alpha_{s}^{\dag}\left(  I\otimes y_{s}^{\dag
}\right)  \mathbb{X}_{s,t}\right)  +f_{s}^{\dag}\left(  I\otimes\alpha
_{s}y_{s}^{\dag}\right)  \mathbb{X}_{s,t}\\
&  \underset{^{3}}{\approx}f_{s}\left(  z_{s,t}\right)  +f_{s}^{\dag}\left(
I\otimes z_{s}^{\dag}\right)  \mathbb{X}_{s,t}\\
&  \underset{^{3}}{\approx}\left[  \int\left\langle \mathbf{f},d\mathbf{z}%
\right\rangle \right]  _{s,t}^{1}.
\end{align*}
As the first and last terms of this equation are additive functionals, they
must be equal.

Secondly%
\[
\left[  \int\left\langle \mathbf{f}\boldsymbol{\alpha},d\mathbf{y}%
^{\mathcal{G}}\right\rangle \right]  _{s}^{\dag}=f_{s}\alpha_{s}\left(
y_{s}^{\dag}\right)  =f_{s}z_{s}^{\dag}=\left[  \int\left\langle
\mathbf{f},d\mathbf{z}\right\rangle \right]  _{s}^{\dag}.
\]
Thus, the two controlled rough paths are equal.
\end{proof}

\begin{remark}
\label{rem.3.25}The $\left(  \mathcal{G},\mathbf{y}\right)  $-- integrator
$\mathbf{y}_{s,t}^{\mathcal{G}}$ is helpful in easing notation so that the
integral is simply written $\int_{s}^{t}\left\langle \boldsymbol{\alpha
},d\mathbf{y}^{\mathcal{G}}\right\rangle .$ A more honest notation for this
integral would be
\[
\int_{s}^{t}\left\langle \left(  \alpha,\alpha^{\dag}\right)  ,d\left(
y^{\psi},\mathbb{X}\right)  \right\rangle _{\mathcal{S}_{y^{\dag}%
}^{\mathcal{G}}}%
\]
where $\mathcal{S}_{y^{\dag}}^{\mathcal{G}}\left(  s\right)  $ is the block
matrix defined by
\[
\mathcal{S}_{y^{\dag}}^{\mathcal{G}}\left(  s\right)  :=%
\begin{pmatrix}
I & S_{y_{s}}^{\mathcal{G}}\circ\left(  y_{s}^{\dag}\right)  ^{\otimes2}\\
0 & I\otimes y_{s}^{\dag}%
\end{pmatrix}
\]
and $\left\langle \cdot,\cdot\right\rangle _{\mathcal{S}_{y^{\dag}%
}^{\mathcal{G}}}$ is the \textquotedblleft inner product\textquotedblright%
\ given by the matrix $\mathcal{S}_{y^{\dag}}^{\mathcal{G}}$. When $s$ is
close to $t$, we have%
\begin{align*}
\int_{s}^{t}\left\langle \left(  \alpha,\alpha^{\dag}\right)  ,d\left(
y^{\psi},\mathbb{X}\right)  \right\rangle _{\mathcal{S}_{y^{\dag}%
}^{\mathcal{G}}}  &  \approx\left(  \alpha_{s},\alpha_{s}^{\dag}\right)
\begin{pmatrix}
I & S_{y_{s}}^{\mathcal{G}}\circ\left(  y_{s}^{\dag}\right)  ^{\otimes2}\\
0 & I\otimes y_{s}^{\dag}%
\end{pmatrix}%
\begin{pmatrix}
y_{s,t}^{\psi}\\
\mathbb{X}_{s,t}%
\end{pmatrix}
\\
&  =\alpha_{s}\left(  \psi\left(  y_{s},y_{t}\right)  +S_{y_{s}}^{\mathcal{G}%
}\left(  y_{s}^{\dag\otimes2}\mathbb{X}_{s,t}\right)  \right)  +\alpha
_{s}^{\dag}\left(  I\otimes y_{s}^{\dag}\right)  \mathbb{X}_{s,t}.
\end{align*}

\end{remark}

\subsection{Almost Additivity Result\label{sub.3.4}}

The following theorem was the key ingredient in the proof of Theorem
\ref{the.3.21} on the existence of rough path integration in the manifold setting.

\begin{theorem}
[Almost Additivity]\label{the.3.26}If $\mathcal{G}:=\left(  \psi,U\right)  $
is a gauge and $\boldsymbol{\alpha}\in CRP_{y}^{U}\left(  M,V\right)  ,$ then
$\tilde{z}_{s,t}\in V$ defined as in Definition \ref{def.3.20} is an almost
additive functional.
\end{theorem}

The proof of Theorem \ref{the.3.26} will be given after Corollary
\ref{cor.3.29} which states that logarithms are \textquotedblleft almost
additive.\textquotedblright\ \ We first need a couple of lemmas. Recall from
Definition \ref{def.2.15} that $\psi_{x}=\psi\left(  x,\cdot\right)  $.

\begin{lemma}
\label{lem.3.27}If $U,\tilde{U}$ are two parallelisms on $M$, then%
\[
S_{y_{t}}^{\tilde{U},U}\circ U\left(  y_{t},y_{s}\right)  ^{\otimes
2}\underset{^{1}}{\approx}U\left(  y_{t},y_{s}\right)  \circ S_{y_{s}}%
^{\tilde{U},U}.
\]

\end{lemma}

\begin{proof}
By the usual patching arguments it suffices to prove this lemma for
$M=\mathbb{R}^{d}.$ In the Euclidean space setting the identity is trivial to
prove since $U\left(  n,m\right)  =_{1}I$ and $S_{n}^{\tilde{U},U}=_{1}%
S_{m}^{\tilde{U},U}.$
\end{proof}

\begin{lemma}
\label{lem.3.28}Let $K$ be a compact, convex set in $\mathbb{R}^{d}.$ If
$\psi$ is a logarithm with domain $\mathcal{D}$ and $K\times K\subseteq
\mathcal{D}$, then there exists a $C_{K}$ such that%
\[
\left\vert \psi_{y}^{\prime}\left(  x\right)  \psi\left(  x,y\right)
+\psi\left(  y,z\right)  -\psi_{y}^{\prime}\left(  x\right)  \psi\left(
x,z\right)  \right\vert \leq C_{K}\max\left\{  \left\vert \psi\left(
x,y\right)  \right\vert ,\left\vert \psi\left(  y,z\right)  \right\vert
,\left\vert \psi\left(  x,z\right)  \right\vert \right\}  ^{3}%
\]
for all $x,y,z$ $\in K$.
\end{lemma}

\begin{proof}
We will use the notation $\left\vert x,y,z\right\vert :=\max\left\{
\left\vert y-x\right\vert ,\left\vert z-y\right\vert ,\left\vert
z-x\right\vert \right\}  $ and write $f\left(  x,y,z\right)  =_{k}g\left(
x,y,z\right)  $ iff $f\left(  x,y,z\right)  =g\left(  x,y,z\right)  +O\left(
\left\vert x,y,z\right\vert ^{k}\right)  .$ Since $\psi$ is zero on the
diagonal and $\psi_{y}^{\prime}\left(  y\right)  =id$ for all $y,$ it follow
from Taylor's theorem (or see Theorem \ref{the.2.24}) that%
\begin{align}
\psi_{y}^{\prime}\left(  x\right)   &  =_{2}id+\psi_{y}^{\prime\prime}\left(
y\right)  \left(  x-y\right)  \text{ and }\nonumber\\
\psi\left(  x,y\right)   &  =_{3}\left(  y-x\right)  +\frac{1}{2}\psi
_{x}^{\prime\prime}\left(  x\right)  \left(  y-x\right)  ^{\otimes
2}\nonumber\\
&  =_{3}\left(  y-x\right)  +\frac{1}{2}\psi_{y}^{\prime\prime}\left(
y\right)  \left(  y-x\right)  ^{\otimes2}. \label{equ.3.17}%
\end{align}
from these approximations we learn,%
\[
\psi\left(  x,y\right)  -\psi\left(  x,z\right)  =_{3}y-z+\frac{1}{2}\psi
_{y}^{\prime\prime}\left(  y\right)  \left[  \left(  y-x\right)  ^{\otimes
2}-\left(  z-x\right)  ^{\otimes2}\right]
\]
and
\begin{align*}
\psi_{y}^{\prime}\left(  x\right)  \psi\left(  x,y\right)   &  -\psi
_{y}^{\prime}\left(  x\right)  \psi\left(  x,z\right) \\
&  =_{3}\left[  id+\psi_{y}^{\prime\prime}\left(  y\right)  \left(
x-y\right)  \otimes\left(  \cdot\right)  \right]  \left(  \psi\left(
x,y\right)  -\psi\left(  x,z\right)  \right) \\
&  =_{3}y-z+\frac{1}{2}\psi_{y}^{\prime\prime}\left(  y\right)  \left[
\left(  y-x\right)  ^{\otimes2}-\left(  z-x\right)  ^{\otimes2}\right]
+\psi_{y}^{\prime\prime}\left(  y\right)  \left[  \left(  x-y\right)
\otimes\left(  y-z\right)  \right]  .
\end{align*}
As simple calculation now shows, with $a=y-x$ and $b=y-z,$ that%
\[
\frac{1}{2}\left[  \left(  y-x\right)  ^{\otimes2}-\left(  z-x\right)
^{\otimes2}\right]  +\left(  x-y\right)  \otimes\left(  y-z\right)  =-\frac
{1}{2}\left[  b^{\otimes2}+b\otimes a-a\otimes b\right]  .
\]
Since $\psi_{y}^{\prime\prime}\left(  y\right)  a\otimes b=\psi_{y}%
^{\prime\prime}\left(  y\right)  b\otimes a$ (mixed partial derivatives
commute), the last two displayed equations give
\begin{align*}
\psi_{y}^{\prime}\left(  x\right)  \psi\left(  x,y\right)  -\psi_{y}^{\prime
}\left(  x\right)  \psi\left(  x,z\right)   &  =_{3}y-z-\frac{1}{2}\psi
_{y}^{\prime\prime}\left(  y\right)  b^{\otimes2}\\
&  =-\left[  \left(  z-y\right)  +\frac{1}{2}\psi_{y}^{\prime\prime}\left(
y\right)  \left(  z-y\right)  ^{\otimes2}\right]  =_{3}-\psi\left(
y,z\right)  .
\end{align*}
The bounds derived above are uniform over a compact set $K$. Because of Eq.
(\ref{equ.3.17}), we may replace $O\left(  \left\vert x,y,z\right\vert
^{3}\right)  \,$with $O\left(  \max\left\{  \left\vert \psi\left(  x,y\right)
\right\vert ,\left\vert \psi\left(  y,z\right)  \right\vert ,\left\vert
\psi\left(  x,z\right)  \right\vert \right\}  ^{3}\right)  $.
\end{proof}

\begin{corollary}
\label{cor.3.29}If $\left(  y_{s},y_{s}^{\dag}\right)  $ is a controlled rough
path and $\psi$ is a logarithm, there exists $C_{\psi},$ $\delta_{\psi}>0$
such that if $0\leq s\leq t\leq u\leq T$ and $u-s\,\leq\delta_{\psi}$, then
\[
\left\vert \psi\left(  y_{t},y_{u}\right)  -\psi\left(  y_{t},\cdot\right)
_{\ast y_{s}}\left[  \psi\left(  y_{s},y_{u}\right)  -\psi\left(  y_{s}%
,y_{t}\right)  \right]  \right\vert _{g}\leq C_{\psi}\omega\left(  s,u\right)
^{3/p}%
\]

\end{corollary}

\begin{proof}
Around every point in $y\left(  \left[  0,T\right]  \right)  ,$ using our
usual techniques, we can find a neighborhood $\mathcal{W}$ such that
$\mathcal{W}\times\mathcal{W}\subseteq\mathcal{D}$ and maps to a convex open
set by a chart. We can then use Remark \ref{rem.2.52} with a slightly modified
version (which includes three variables instead of two) of Lemma
\ref{lem.2.51} to create a global estimate. We can then choose a $\delta$ such
that $u-s\leq\delta$ forces the path to lie within one of these sets
$\mathcal{W}.$ Therefore, it suffices to prove the estimate locally. However,
we can push forward the metric and $\psi$ to a convex set on Euclidean space.
The rest follows from the Lemma \ref{lem.3.28} and the fact that $\left\vert
\psi\left(  y_{s},y_{t}\right)  \right\vert \leq C\omega\left(  s,t\right)
^{1/p}$ for all $\left\vert t-s\right\vert \leq\delta$ for some $C<\infty$ and
$\delta>0.$
\end{proof}

\subsection{Proof of Theorem \ref{the.3.26}}

\begin{proof}
[Proof of Theorem \ref{the.3.26}]Let $0\leq s\leq t\leq u\leq T$ . Throughout
this proof, we will use the notation $\underset{^{i}}{\approx}$ with respect
to the times $s$ and $u.$ To prove the statement, we need to show $\tilde
{z}_{s,t}+\tilde{z}_{t,u}\underset{^{3}}{\approx}\tilde{z}_{s,u}.$ We begin by
working on the three terms for $\tilde{z}_{t,u}$ in the following equation%
\begin{equation}
\tilde{z}_{t,u}=\alpha_{t}\left(  \psi\left(  y_{t},y_{u}\right)  \right)
+\alpha_{t}^{\dag}\left(  I\otimes y_{t}^{\dag}\right)  \mathbb{X}%
_{t,u}+\alpha_{t}\left(  S_{y_{t}}^{\mathcal{G}}\circ y_{t}^{\dag\otimes
2}\mathbb{X}_{t,u}\right)  . \label{equ.3.18}%
\end{equation}
Using Corollary \ref{cor.3.29} followed by Corollary \ref{cor.3.16} we find
\begin{align*}
\alpha_{t}\left(  \psi\left(  y_{t},y_{u}\right)  \right)   &  \underset{^{3}%
}{\approx}\alpha_{t}\psi\left(  y_{t},\cdot\right)  _{\ast y_{s}}\left[
\psi\left(  y_{s},y_{u}\right)  -\psi\left(  y_{s},y_{t}\right)  \right] \\
&  \underset{^{3}}{\approx}\alpha_{t}U\left(  y_{t},y_{s}\right)  \left[
I+S_{y_{s}}^{\mathcal{G}}\left(  \psi\left(  y_{s},y_{t}\right)
\otimes\left(  \cdot\right)  \right)  \right]  \left[  \psi\left(  y_{s}%
,y_{u}\right)  -\psi\left(  y_{s},y_{t}\right)  \right] \\
&  \underset{^{3}}{\approx}\left[  \alpha_{s}+\alpha_{s}^{\dag}x_{s,t}%
\otimes\left(  \cdot\right)  \right]  \left[  I+S_{y_{s}}^{\mathcal{G}}\left(
\psi\left(  y_{s},y_{t}\right)  \otimes\left(  \cdot\right)  \right)  \right]
\left[  \psi\left(  y_{s},y_{u}\right)  -\psi\left(  y_{s},y_{t}\right)
\right] \\
&  \underset{^{3}}{\approx}\alpha_{s}\left[  I+S_{y_{s}}^{\mathcal{G}}\left(
\psi\left(  y_{s},y_{t}\right)  \otimes\left(  \cdot\right)  \right)  \right]
\left[  \psi\left(  y_{s},y_{u}\right)  -\psi\left(  y_{s},y_{t}\right)
\right] \\
&  \quad+\alpha_{s}^{\dag}x_{s,t}\otimes\left[  \psi\left(  y_{s}%
,y_{u}\right)  -\psi\left(  y_{s},y_{t}\right)  \right]  .
\end{align*}
Combining this equation with the estimates
\[
\psi\left(  y_{s},y_{t}\right)  \underset{^{2}}{\approx}y_{s}^{\dag}%
x_{s,t}\text{ and }\psi\left(  y_{s},y_{u}\right)  -\psi\left(  y_{s}%
,y_{t}\right)  \underset{^{2}}{\approx}y_{s}^{\dag}\left[  x_{s,u}%
-x_{s,t}\right]  =y_{s}^{\dag}x_{t,u},
\]
then shows,%
\begin{equation}
\alpha_{t}\left(  \psi\left(  y_{t},y_{u}\right)  \right)  \underset{^{3}%
}{\approx}\alpha_{s}\left[  \psi\left(  y_{s},y_{u}\right)  -\psi\left(
y_{s},y_{t}\right)  \right]  +\alpha_{s}\left(  y_{s}^{\dag}\right)
^{\otimes2}x_{s,t}\otimes x_{t,u}+\alpha_{s}^{\dag}\left(  I\otimes
y_{s}^{\dag}\right)  x_{s,t}\otimes x_{t,u}. \label{equ.3.19}%
\end{equation}
By the definitions of $CRP_{\mathbf{X}}\left(  M\right)  $ and $CRP_{y}%
^{U}\left(  M,V\right)  $ we have%
\begin{align}
\alpha_{t}^{\dag}\left(  I\otimes y_{t}^{\dag}\right)  \mathbb{X}%
_{t,u}\underset{^{3}}{\approx}  &  ~\alpha_{t}^{\dag}\left(  I\otimes U\left(
y_{t},y_{s}\right)  y_{s}^{\dag}\right)  \mathbb{X}_{t,u}\nonumber\\
=  &  ~\alpha_{t}^{\dag}\left(  I\otimes U\left(  y_{t},y_{s}\right)  \right)
\left(  I\otimes y_{s}^{\dag}\right)  \mathbb{X}_{t,u}\underset{^{3}}{\approx
}\alpha_{s}^{\dag}\left(  I\otimes y_{s}^{\dag}\right)  \mathbb{X}_{t,u}.
\label{equ.3.20}%
\end{align}
Lastly by the definitions of $CRP_{\mathbf{X}}\left(  M\right)  $ and
$CRP_{y}^{U}\left(  M,V\right)  $ along with Lemma \ref{lem.3.27} with
$\tilde{U}\left(  m,n\right)  =\left(  \psi_{m}\right)  _{\ast n}$, we have
and%
\begin{align}
\alpha_{t}\left(  S_{y_{t}}^{\mathcal{G}}\circ y_{t}^{\dag\otimes2}%
\mathbb{X}_{t,u}\right)  \underset{^{3}}{\approx}  &  ~\alpha_{t}\left(
S_{y_{t}}^{\mathcal{G}}\circ U\left(  y_{t},y_{s}\right)  ^{\otimes2}\circ
y_{s}^{\dag\otimes2}\mathbb{X}_{t,u}\right) \nonumber\\
\underset{^{3}}{\approx}  &  ~\alpha_{t}\left(  U\left(  y_{t},y_{s}\right)
\circ S_{y_{s}}^{\mathcal{G}}\circ y_{s}^{\dag\otimes2}\mathbb{X}%
_{t,u}\right)  \underset{^{3}}{\approx}\alpha_{s}\left(  S_{y_{s}%
}^{\mathcal{G}}\circ y_{s}^{\dag\otimes2}\mathbb{X}_{t,u}\right)  .
\label{equ.3.21}%
\end{align}
Adding together Eqs. (\ref{equ.3.19}) -- (\ref{equ.3.21}) to%
\[
\tilde{z}_{s,t}=\alpha_{s}\left(  \psi\left(  y_{s},y_{t}\right)  \right)
+\alpha_{s}^{\dag}\left(  I\otimes y_{s}^{\dag}\right)  \mathbb{X}%
_{s,t}+\alpha_{s}\left(  S_{y_{s}}^{\mathcal{G}}\circ y_{s}^{\dag\otimes
2}\mathbb{X}_{s,t}\right)
\]
while making use Chen's identity in Eq. (\ref{equ.2.2}) shows%
\[
\tilde{z}_{s,t}+\tilde{z}_{t,u}\underset{^{3}}{\approx}\alpha_{s}\left(
\psi\left(  y_{s},y_{u}\right)  \right)  +\alpha_{s}^{\dag}\left(  I\otimes
y_{s}^{\dag}\right)  \mathbb{X}_{s,u}+\alpha_{s}\left(  S_{y_{s}}%
^{\mathcal{G}}\circ y_{s}^{\dag\otimes2}\mathbb{X}_{s,u}\right)  =\tilde
{z}_{s,u}.
\]

\end{proof}

\subsection{A Map from $CRP_{y}^{U}\left(  M,V\right)  $ to $CRP_{y}%
^{\tilde{U}}\left(  M,V\right)  $\label{sub.3.6}}

Suppose that $\mathcal{G}=\left(  \psi,U\right)  $ and $\mathcal{\tilde{G}%
}=\left(  \tilde{\psi},\tilde{U}\right)  $ are two gauges on $M.$ Generally,
if $\boldsymbol{\alpha}\mathbf{:=}\left(  \alpha,\alpha^{\dag}\right)  \in
CRP_{y}^{U}\left(  M,V\right)  $, there is no reason to expect it also to be
an element of $CRP_{y}^{\tilde{U}}\left(  M,V\right)  $. However, the main
theorem [Theorem \ref{the.3.32}] of this section shows there is a
\textquotedblleft natural\textquotedblright\ bijection between $CRP_{y}%
^{U}\left(  M,V\right)  $ and $CRP_{y}^{\tilde{U}}\left(  M,V\right)  $ which
preserves the notions of integration. The following proposition is needed in
the proof of Theorem \ref{the.3.32} and moreover motivates the statement of
the theorem.

\begin{proposition}
\label{pro.3.30}If $\mathcal{G}=\left(  \psi,U\right)  $ and $\mathcal{\tilde
{G}}=\left(  \tilde{\psi},\tilde{U}\right)  $ are two gauges on $M$ and
$\mathbf{y}=\left(  y,y^{\dag}\right)  \in CRP_{\mathbf{X}}\left(  M\right)
,$ then%
\begin{equation}
\mathbf{y}_{s,t}^{\mathcal{G}}\underset{3}{\approx}\mathbf{y}_{s,t}%
^{\mathcal{\tilde{G}}}+\left(  S_{y_{s}}^{\tilde{U},U}\left(  \left(
y_{s}^{\dag}\right)  ^{\otimes2}\mathbb{X}_{s,t}\right)  ,0\right)  ,
\label{equ.3.22}%
\end{equation}
where $\mathbf{y}_{s,t}^{\mathcal{G}}$ and $\mathbf{y}_{s,t}^{\mathcal{\tilde
{G}}}$ are as in Definition \ref{def.3.20}.
\end{proposition}

\begin{proof}
From Proposition \ref{pro.3.17},%
\begin{align*}
\psi\left(  y_{s},y_{t}\right)  -\tilde{\psi}\left(  y_{s},y_{t}\right)   &
\underset{3}{\approx}\frac{1}{2}S_{y_{s}}^{\tilde{\psi}_{\ast},\psi_{\ast}%
}\left(  \psi\left(  y_{s},y_{t}\right)  \otimes\psi\left(  y_{s}%
,y_{t}\right)  \right) \\
&  \underset{3}{\approx}\frac{1}{2}S_{y_{s}}^{\tilde{\psi}_{\ast},\psi_{\ast}%
}\left(  \left(  y_{s}^{\dag}\otimes y_{s}^{\dag}\right)  \left[
x_{s,t}\otimes x_{s,t}\right]  \right)  =S_{y_{s}}^{\tilde{\psi}_{\ast}%
,\psi_{\ast}}\left(  \left(  y_{s}^{\dag}\right)  ^{\otimes2}\mathbb{X}%
_{s,t}\right)
\end{align*}
wherein we have used $S_{y_{s}}^{\tilde{\psi}_{\ast},\psi_{\ast}}$ is
symmetric and $\mathbf{X}=\left(  x,\mathbb{X}\right)  $ is a weak-geometric
rough path for the last equality. Making use of this estimate it now follows
that%
\begin{align}
\mathbf{y}_{s,t}^{\mathcal{G}}-\mathbf{y}_{s,t}^{\mathcal{\tilde{G}}}  &
=\left(  \psi\left(  y_{s},y_{t}\right)  -\tilde{\psi}\left(  y_{s}%
,y_{t}\right)  +\left(  S_{y_{s}}^{\mathcal{G}}-S_{y_{s}}^{\mathcal{\tilde{G}%
}}\right)  \left(  \left(  y_{s}^{\dag}\right)  ^{\otimes2}\mathbb{X}%
_{s,t}\right)  ,0\right) \nonumber\\
&  \underset{3}{\approx}\left(  \left(  S_{y_{s}}^{\tilde{\psi}_{\ast}%
,\psi_{\ast}}+S_{y_{s}}^{\mathcal{G}}-S_{y_{s}}^{\mathcal{\tilde{G}}}\right)
\left(  \left(  y_{s}^{\dag}\right)  ^{\otimes2}\mathbb{X}_{s,t}\right)
,0\right)  . \label{equ.3.23}%
\end{align}
On the other hand, by Lemma \ref{lem.3.11},%
\begin{align*}
S^{\tilde{\psi}_{\ast},\psi_{\ast}}  &  =S^{\tilde{\psi}_{\ast},\tilde{U}%
}+S^{\tilde{U},\psi_{\ast}}=S^{\tilde{\psi}_{\ast},\tilde{U}}+S^{\tilde{U}%
,U}+S^{U,\psi_{\ast}}\\
&  =S^{\mathcal{\tilde{G}}}-S^{\mathcal{G}}+S^{\tilde{U},U}%
\end{align*}
which combined with Eq. (\ref{equ.3.23}) gives Eq. (\ref{equ.3.22}).
\end{proof}

\begin{corollary}
\label{cor.3.31}The integral, $\int\left\langle \boldsymbol{\alpha
},d\mathbf{y}^{\mathcal{G}}\right\rangle $ only depends on the choice of
parallelism $U$ and not on the logarithm used to make the gauge $\mathcal{G}%
=\left(  \psi,U\right)  .$
\end{corollary}

\begin{proof}
From Proposition \ref{pro.3.30} with $U=\tilde{U},$ it follows that
\[
\int_{s}^{t}\left\langle \boldsymbol{\alpha},d\mathbf{y}^{\mathcal{G}%
}\right\rangle \underset{3}{\approx}\left\langle \boldsymbol{\alpha}%
_{s},\mathbf{y}_{s,t}^{\mathcal{G}}\right\rangle \underset{3}{\approx
}\left\langle \boldsymbol{\alpha}_{s},\mathbf{y}_{s,t}^{\mathcal{\tilde{G}}%
}\right\rangle \underset{3}{\approx}\int_{s}^{t}\left\langle
\boldsymbol{\alpha},d\mathbf{y}^{\mathcal{\tilde{G}}}\right\rangle
\]
from which it follows that the two additive functionals, $\int\left\langle
\boldsymbol{\alpha},d\mathbf{y}^{\mathcal{G}}\right\rangle $ and
$\int\left\langle \boldsymbol{\alpha},d\mathbf{y}^{\mathcal{\tilde{G}}%
}\right\rangle ,$ must be equal.
\end{proof}

If $\boldsymbol{\alpha}=\left(  \alpha,\alpha^{\dag}\right)  \in CRP_{y}%
^{U}\left(  M,V\right)  $ and $U\neq\tilde{U},$ then%
\begin{equation}
\left\langle \boldsymbol{\alpha}_{s},\mathbf{y}_{s,t}^{\mathcal{G}%
}\right\rangle \underset{3}{\approx}\left\langle \boldsymbol{\alpha}%
_{s},\mathbf{y}_{s,t}^{\mathcal{\tilde{G}}}+\left(  S_{y_{s}}^{\tilde{U}%
,U}\left(  \left(  y_{s}^{\dag}\right)  ^{\otimes2}\mathbb{X}_{s,t}\right)
,0\right)  \right\rangle =\left\langle \boldsymbol{\tilde{\alpha}}%
_{s},\mathbf{y}_{s,t}^{\mathcal{\tilde{G}}}\right\rangle \label{equ.3.24}%
\end{equation}
where $\boldsymbol{\tilde{\alpha}}_{s}$ is defined in Eq. (\ref{equ.3.25})
below. The identity in Eq. (\ref{equ.3.24}) suggests the following theorem.

\begin{theorem}
\label{the.3.32}The map%
\begin{equation}
\boldsymbol{\alpha}_{s}=\left(  \alpha_{s},\alpha_{s}^{\dag}\right)
\longrightarrow\boldsymbol{\tilde{\alpha}}_{s}:=\left(  \tilde{\alpha}%
_{s},\tilde{\alpha}_{s}^{\dag}\right)  :=\left(  \alpha_{s},\alpha_{s}^{\dag
}+\alpha_{s}S_{y_{s}}^{\tilde{U},U}y_{s}^{\dag}\otimes I\right)
\label{equ.3.25}%
\end{equation}
is a bijection from $CRP_{y}^{U}\left(  M,V\right)  $ to $CRP_{y}^{\tilde{U}%
}\left(  M,V\right)  $ such that
\begin{equation}
\int\left\langle \boldsymbol{\alpha},d\mathbf{y}^{\mathcal{G}}\right\rangle
=\int\left\langle \boldsymbol{\tilde{\alpha}},d\mathbf{y}^{\mathcal{\tilde{G}%
}}\right\rangle . \label{equ.3.26}%
\end{equation}

\end{theorem}

\begin{proof}
The only thing that is really left to prove here is the assertion that
$\boldsymbol{\tilde{\alpha}}\in CRP_{y}^{\tilde{U}}\left(  M,V\right)  .$
First we prove that item \ref{Ite.22} of Definition \ref{def.3.1} holds for
$\boldsymbol{\tilde{\alpha}}$.

From Theorem \ref{the.3.15} with $m=y_{s}$ and $n=y_{t},$ we find%
\[
U\left(  y_{s},y_{t}\right)  \tilde{U}\left(  y_{s},y_{t}\right)
^{-1}\underset{^{2}}{\approx}I+S_{y_{s}}^{\tilde{U},U}\left(  \psi\left(
y_{s},y_{t}\right)  \otimes\left(  \cdot\right)  \right)
\]
and then combining this result with Corollary \ref{cor.2.29} shows%
\begin{equation}
\tilde{U}\left(  y_{t},y_{s}\right)  \underset{^{2}}{\approx}U\left(
y_{t},y_{s}\right)  \left[  I+S_{y_{s}}^{\tilde{U},U}\left(  \psi\left(
y_{s},y_{t}\right)  \otimes\left(  \cdot\right)  \right)  \right]  .
\label{equ.3.27}%
\end{equation}
From this equation and the fact that $\boldsymbol{\alpha}\mathbf{\in}%
CRP_{y}^{U}\left(  M,V\right)  ,$ we learn%
\begin{align*}
\alpha_{t}\tilde{U}\left(  y_{t},y_{s}\right)  -\alpha_{s}\underset{^{2}%
}{\approx}  &  ~\alpha_{t}U\left(  y_{t},y_{s}\right)  \left[  I+S_{y_{s}%
}^{\tilde{U},U}\left(  \psi\left(  y_{s},y_{t}\right)  \otimes\left(
\cdot\right)  \right)  \right]  -\alpha_{s}\\
\underset{^{2}}{\approx}  &  ~\left(  \alpha_{s}+\alpha_{s}^{\dag}%
x_{s,t}\right)  \left[  I+S_{y_{s}}^{\tilde{U},U}\left(  \psi\left(
y_{s},y_{t}\right)  \otimes\left(  \cdot\right)  \right)  \right]  -\alpha
_{s}\\
\underset{^{2}}{\approx}  &  ~\alpha_{s}^{\dag}x_{s,t}+\alpha_{s}S_{y_{s}%
}^{\tilde{U},U}\left(  y_{s}^{\dag}x_{s,t}\otimes\left(  \cdot\right)
\right)  =\tilde{\alpha}_{s}^{\dag}\left(  x_{s,t}\otimes\left(  \cdot\right)
\right)
\end{align*}
as desired.

Next we check item \ref{Ite.23} of Definition \ref{def.3.1}. We are given%
\begin{align*}
0\underset{^{1}}{\approx}  &  ~\alpha_{t}^{\dag}\circ\left(  I\otimes U\left(
y_{t},y_{s}\right)  \right)  -\alpha_{s}^{\dag}\\
=  &  ~\tilde{\alpha}_{t}^{\dag}\circ\left(  I\otimes\tilde{U}\left(
y_{t},y_{s}\right)  \right)  -\tilde{\alpha}_{s}^{\dag}\\
&  -\alpha_{t}\circ S_{y_{t}}^{\tilde{U},U}\circ\left(  y_{t}^{\dag}\otimes
U\left(  y_{t},y_{s}\right)  \right)  +\alpha_{s}\circ S_{y_{s}}^{\tilde{U}%
,U}\circ\left(  y_{s}^{\dag}\otimes I\right)
\end{align*}
wherein we have used that $U\left(  y_{s},y_{t}\right)  \underset{^{1}%
}{\approx}\tilde{U}\left(  y_{s},y_{t}\right)  $ (for example, see Eq.
(\ref{equ.3.27})). We therefore must show the last line is approximately $0$.
However, by Lemma \ref{lem.3.27}, we have $S_{y_{t}}^{\tilde{U},U}\circ
U\left(  y_{t},y_{s}\right)  ^{\otimes2}\underset{^{1}}{\approx}U\left(
y_{t},y_{s}\right)  \circ S_{y_{t}}^{\tilde{U},U}.$ Thus%
\begin{align*}
&  \alpha_{t}\circ S_{y_{t}}^{\tilde{U},U}\circ\left(  y_{t}^{\dag}\otimes
U\left(  y_{t},y_{s}\right)  \right)  -\alpha_{s}\circ S_{y_{s}}^{\tilde{U}%
,U}\circ\left(  y_{s}^{\dag}\otimes I\right) \\
&  \quad\underset{^{1}}{\approx}\alpha_{t}\circ S_{y_{t}}^{\tilde{U},U}%
\circ\left(  U\left(  y_{t},y_{s}\right)  y_{s}^{\dag}\otimes U\left(
y_{t},y_{s}\right)  \right)  -\alpha_{s}\circ S_{y_{s}}^{\tilde{U},U}%
\circ\left(  y_{s}^{\dag}\otimes I\right) \\
&  \quad\underset{^{1}}{\approx}\left[  \alpha_{t}\circ U\left(  y_{t}%
,y_{s}\right)  -\alpha_{s}\right]  \left[  S_{y_{s}}^{\tilde{U},U}\circ\left(
y_{s}^{\dag}\otimes I\right)  \right]  \underset{^{1}}{\approx}0.
\end{align*}

\end{proof}

\section{Integrating One-Forms Along a CRP\label{sec.4}}

\begin{lemma}
\label{lem.4.1}Let $V$ be a Banach space and $U$ be a parallelism on $M.$ If
$\alpha\in\Omega^{1}\left(  M,V\right)  $ is a $V$ -- valued smooth one-form
on $M,$ then
\[
\alpha_{n}\circ U\left(  n,m\right)  -\alpha_{m}=_{2}\nabla_{\psi\left(
m,n\right)  }^{U}\alpha
\]
where $\nabla^{U}$ is the covariant derivative defined in Remark
\ref{rem.2.20}.
\end{lemma}

\begin{proof}
By definition, $\nabla_{v_{m}}^{U}\alpha$ is determined by the product rule,
\begin{equation}
v_{m}\left[  \alpha\left(  Y\right)  \right]  =\left(  \nabla_{v_{m}}%
^{U}\alpha\right)  \left(  Y\left(  m\right)  \right)  +\alpha_{m}\left(
\nabla_{v_{m}}^{U}Y\right)  . \label{equ.4.1}%
\end{equation}
However, we may also write
\begin{align*}
v_{m}\left[  \alpha\left(  Y\right)  \right]   &  =\frac{d}{dt}|_{0}%
\alpha\left(  U\left(  m,\sigma_{t}\right)  ^{-1}U\left(  m,\sigma_{t}\right)
Y\left(  \sigma_{t}\right)  \right) \\
&  =\frac{d}{dt}|_{0}\alpha\left(  U\left(  m,\sigma_{t}\right)  ^{-1}Y\left(
m\right)  \right)  +\alpha_{m}\left(  \nabla_{v_{m}}^{U}Y\right)
\end{align*}
where $\sigma_{t}$ is such that $\dot{\sigma}_{0}=v_{m}$. Combining the last
two facts shows that%
\begin{equation}
\nabla_{v_{m}}^{U}\alpha=\frac{d}{dt}|_{0}\left[  \alpha\circ U\left(
m,\sigma_{t}\right)  ^{-1}\right]  . \label{equ.4.2}%
\end{equation}
By Corollary \ref{cor.2.29}, we may alternatively write Eq. (\ref{equ.4.2}) as%
\[
\nabla_{v_{m}}^{U}\alpha=\frac{d}{dt}|_{0}\left[  \alpha\circ U\left(
\sigma_{t},m\right)  \right]  .
\]

To prove the lemma, we note this is a local result and we therefore may assume
$M=\mathbb{R}^{d}.$ Then by Taylor's theorem,%
\begin{align*}
\alpha_{n}\circ U\left(  n,m\right)   &  =\alpha_{m}+D\left[  \alpha_{\left(
\cdot\right)  }\circ U\left(  \cdot,m\right)  \right]  \left(  m\right)
\left(  n-m\right)  +O\left(  \left\vert n-m\right\vert ^{2}\right) \\
&  =\alpha_{m}+\nabla_{\left(  n-m\right)  _{m}}^{U}\alpha+O\left(  \left\vert
n-m\right\vert ^{2}\right) \\
&  =\alpha_{m}+\nabla_{\psi\left(  m,n\right)  }^{U}\alpha+O\left(  \left\vert
\psi\left(  m,n\right)  \right\vert ^{2}\right)  .
\end{align*}

\end{proof}

Suppose that $\alpha\in\Omega^{1}\left(  M,V\right)  $ is a $V$ -- valued
one-form and $U$ is a parallelism on $M.$ We wish to take $\alpha_{s}%
^{U}=\alpha_{y_{s}}:=\alpha|_{T_{y_{s}M}}.$ Making use of Lemma \ref{lem.4.1},
we find%
\begin{equation}
\alpha_{t}^{U}\circ U\left(  y_{t},y_{s}\right)  -\alpha_{s}\underset{^{2}%
}{\approx}\nabla_{\psi\left(  y_{s},y_{t}\right)  }^{U}\alpha\underset{^{2}%
}{\approx}\nabla_{y_{s}^{\dag}x_{s,t}}^{U}\alpha\label{equ.4.3}%
\end{equation}
and this computation suggests the following proposition.

\begin{proposition}
\label{pro.4.2}Suppose that $\alpha\in\Omega^{1}\left(  M,V\right)  $ is a $V$
-- valued one-form and $U$ is a parallelism on $M,$ then
\[
\boldsymbol{\alpha}_{s}^{\left(  y,U\right)  }:=\left(  \alpha_{y_{s}}%
,\alpha_{s}^{\dag\left(  y,U\right)  }\right)  :=\left(  \alpha|_{T_{y_{s}M}%
},\nabla_{y_{s}^{\dag}\left(  \cdot\right)  }^{U}\alpha\right)  \in
CRP_{y}^{U}\left(  M,V\right)  .
\]

\end{proposition}

\begin{proof}
In light of how $\boldsymbol{\alpha}_{s}^{y,U}$ has been defined and of Eq.
(\ref{equ.4.3}), we need only verify Item \ref{Ite.23} in Definition
\ref{def.3.1} is satisfied. To this end, suppose that $w\in W,$ then%
\begin{align}
\alpha_{t}^{\dag\left(  y,U\right)  }\circ\left(  I\otimes U\left(
y_{t},y_{s}\right)  \right)  \left(  w\otimes\left(  \cdot\right)  \right)
&  =\left(  \nabla_{y_{t}^{\dag}w}^{U}\alpha\right)  U\left(  y_{t}%
,y_{s}\right) \nonumber\\
&  \underset{^{1}}{\approx}\left(  \nabla_{U\left(  y_{t},y_{s}\right)
y_{s}^{\dag}w}^{U}\alpha\right)  U\left(  y_{t},y_{s}\right)  \label{equ.4.4}%
\end{align}
wherein we have used Inequality (\ref{equ.2.23}) along with Corollary
\ref{cor.2.29} in the last line. Since for $v_{m}\in T_{m}M$ the function
$F\left(  n\right)  :=\left(  \nabla_{U\left(  n,m\right)  v_{m}}^{U}%
\alpha\right)  U\left(  n,m\right)  \in L\left(  T_{m}M,V\right)  $ is smooth,
it follows by Taylor's theorem that $F\left(  n\right)  =_{1}F\left(
m\right)  $ which translates to%
\[
\left(  \nabla_{U\left(  n,m\right)  v_{m}}^{U}\alpha\right)  U\left(
n,m\right)  =_{1}\nabla_{v_{m}}^{U}\alpha.
\]
Taking $m=y_{s},$ $n=y_{t},$ and $v_{m}=y_{s}^{\dag}w$ in this estimates shows%
\[
\left(  \nabla_{U\left(  y_{t},y_{s}\right)  y_{s}^{\dag}w}^{U}\alpha\right)
U\left(  y_{t},y_{s}\right)  \underset{^{1}}{\approx}\nabla_{y_{s}^{\dag}%
w}^{U}\alpha
\]
which combined with Eq. (\ref{equ.4.4}) completes the proof.
\end{proof}

\begin{theorem}
\label{the.4.3}If $\alpha\in\Omega^{1}\left(  M,V\right)  $ is a $V$ -- valued
one-form, then the integral $\int\left\langle \boldsymbol{\alpha}^{\left(
y,U\right)  },d\mathbf{y}^{\mathcal{G}}\right\rangle $ is independent of any
choice of gauge $\mathcal{G}=\left(  \psi,U\right)  $ on $M.$ In the future we
denote this integral more simply as $\int\left\langle \boldsymbol{\alpha
}\mathbf{,}d\mathbf{y}\right\rangle .$
\end{theorem}

\begin{proof}
Suppose that $U$ and $\tilde{U}$ are two parallelisms. According to Theorem
\ref{the.3.32} it suffices to show
\begin{equation}
\alpha_{s}^{\dag\left(  y,\tilde{U}\right)  }=\alpha_{s}^{\dag\left(
y,U\right)  }+\alpha_{y_{s}}S_{y_{s}}^{\tilde{U},U}\left[  y_{s}^{\dag}\otimes
I\right]  . \label{equ.4.5}%
\end{equation}
We will see that Eq. (\ref{equ.4.5}) is a fairly direct consequence of Example
\ref{exa.3.10} which, when translated to the language of forms (see Eq.
(\ref{equ.4.1})), states%
\begin{equation}
\nabla_{v_{m}}\alpha=\tilde{\nabla}_{v_{m}}\alpha-\alpha\circ S_{m}^{\tilde
{U},U}\left(  v_{m}\otimes\left(  \cdot\right)  \right)  . \label{equ.4.6}%
\end{equation}
So for $w\in W,$ we have%
\begin{align*}
\alpha_{s}^{\dag\left(  y,\tilde{U}\right)  }w  &  =\tilde{\nabla}%
_{y_{s}^{\dag}w}\alpha=\nabla_{y_{s}^{\dag}w}\alpha+\alpha_{y_{s}}%
S_{m}^{\tilde{U},U}\left(  y_{s}^{\dag}w\otimes\left(  \cdot\right)  \right)
\\
&  =\alpha_{s}^{\dag\left(  y,U\right)  }w+\alpha_{y_{s}}S_{m}^{\tilde{U}%
,U}\left(  y_{s}^{\dag}w\otimes\left(  \cdot\right)  \right)
\end{align*}
which proves Eq. (\ref{equ.4.5}).
\end{proof}

Let us now record a number of possible different expressions for computing
$\int_{s}^{t}\alpha\left(  d\mathbf{y}\right)  $ depending on the choice of
gauge we make.

\begin{proposition}
\label{pro.4.4}Let $\mathcal{G=}\left(  \psi,U\right)  $ be a gauge. There
exists a $\delta>0$ such that for $s<t$ and $t-s<\delta,$ the approximation
\[
\left[  \int\alpha\left(  d\mathbf{y}\right)  \right]  _{s,t}^{1}%
\underset{^{3}}{\approx}\alpha_{y_{s}}\left(  \psi\left(  y_{s},y_{t}\right)
\right)  +\left[  \left(  \nabla_{\left(  \cdot\right)  }^{U}\alpha\right)
_{y_{s}}+\alpha_{y_{s}}\circ S_{y_{s}}^{\mathcal{G}}\right]  \circ y_{s}%
^{\dag\otimes2}\mathbb{X}_{s,t}%
\]
holds.
\end{proposition}

In the case that we take $U=U^{\psi},$ we get a slightly simpler formula.

\begin{corollary}
\label{cor.4.5}Let $\psi$ be a logarithm. There exists a $\delta>0$ such that
for $s<t$ and $t-s<\delta,$ the approximation%
\[
\left[  \int\alpha\left(  d\mathbf{y}\right)  \right]  _{s,t}^{1}%
\underset{^{3}}{\approx}\alpha_{y_{s}}\left(  \psi\left(  y_{s},y_{t}\right)
\right)  +d\left(  \alpha_{\left(  \cdot\right)  }\circ\left(  \psi_{\left(
\cdot\right)  }\right)  _{\ast y_{s}}\right)  _{y_{s}}\circ y_{s}^{\dag
\otimes2}\mathbb{X}_{s,t}%
\]
holds.
\end{corollary}

\begin{example}
\label{exa.4.6}Let $\nabla$ be a covariant derivative on $M.$ There exists a
$\delta>0$ such that for $s<t$ and $t-s<\delta,$ the approximation%
\[
\left[  \int\alpha\left(  d\mathbf{y}\right)  \right]  _{s,t}^{1}%
\underset{^{3}}{\approx}\alpha_{y_{s}}\left(  \left(  \exp_{y_{s}}^{\nabla
}\right)  ^{-1}\left(  y_{t}\right)  \right)  +\left[  \left(  \nabla
\alpha\right)  _{y_{s}}+\frac{1}{2}\alpha_{y_{s}}\circ T_{y_{s}}^{\nabla
}\right]  \circ y_{s}^{\dag\otimes2}\mathbb{X}_{s,t}%
\]
holds. Indeed this follows immediately from Proposition \ref{pro.4.4}, Lemma
\ref{lem.3.13}, and the fact that
\begin{align*}
\left(  \nabla\alpha\right)  _{y_{s}}\left(  v_{m},w_{m}\right)  :=  &
~v_{m}\left[  \alpha\left(  \boldsymbol{W}\right)  \right]  -\alpha\left(
\nabla_{v_{m}}\boldsymbol{W}\right) \\
=  &  ~d\left(  \alpha_{\left(  \cdot\right)  }\circ\boldsymbol{W}\left(
\cdot\right)  \right)  _{y_{s}}\left(  v_{m}\right)  -\alpha\left(
\nabla_{v_{m}}\boldsymbol{W}\right)
\end{align*}
where $\boldsymbol{W}$ is any vector field such that $\boldsymbol{W}\left(
m\right)  =w_{m}$. Choosing $\boldsymbol{W}=U^{\nabla}\left(  \cdot,m\right)
w_{m}$, we have
\[
\nabla_{v_{m}}\boldsymbol{W}=\nabla_{v_{m}}U^{\nabla}\left(  \cdot,m\right)
w_{m}=0
\]
by the definition of parallel translation.
\end{example}

\subsection{Integration of a One-Form Using Charts\label{sub.4.1}}

It is easy to see that by independence of gauges, the integral of a one-form
along $\left(  y_{s},y_{s}^{\dag}\right)  $ is an object which we only need to
compute locally. As mentioned in Remark \ref{rem.2.22} we have an example of a
local gauge by using a chart. Plugging this formula into the integral
approximation from Corollary \ref{cor.4.5}, we get the following.

\begin{corollary}
\label{cor.4.7}Let $\phi$ be a chart on $M$. For all $a,b\in\left[
0,T\right]  $ such that $y\left[  a,b\right]  \subset D\left(  \phi\right)  $,
we have the approximation%
\begin{equation}
\left[  \int\alpha\left(  d\mathbf{y}\right)  \right]  _{s,t}^{1}%
\underset{^{3}}{\approx}\alpha_{y_{s}}\left(  \left(  d\phi_{y_{s}}\right)
^{-1}\left[  \phi\left(  y_{t}\right)  -\phi\left(  y_{s}\right)  \right]
\right)  +d\left(  \alpha_{\left(  \cdot\right)  }\circ\left(  d\phi_{\left(
\cdot\right)  }\right)  ^{-1}d\phi_{y_{s}}\right)  _{y_{s}}\circ y_{s}%
^{\dag\otimes2}\mathbb{X}_{s,t} \label{equ.4.7}%
\end{equation}
holds for all $s<t$ $\in\left[  a,b\right]  $.
\end{corollary}

Although this formula looks a bit complicated, it may be reduced to something
that makes more sense. First, note that
\[
\alpha_{m}\circ\left(  d\phi_{m}\right)  ^{-1}=\left[  \left(  \phi
^{-1}\right)  ^{\ast}\alpha\right]  _{\phi\left(  m\right)  }.
\]
Thus we can reduce the right hand side Eq. (\ref{equ.4.7}) to%
\begin{align*}
&  \left[  \left(  \phi^{-1}\right)  ^{\ast}\alpha\right]  _{\phi\left(
y_{s}\right)  }\left(  \phi\left(  y_{t}\right)  -\phi\left(  y_{s}\right)
\right)  +d\left(  \left[  \left(  \phi^{-1}\right)  ^{\ast}\alpha\right]
_{\phi\left(  \cdot\right)  }d\phi_{y_{s}}\right)  _{y_{s}}\circ y_{s}%
^{\dag\otimes2}\mathbb{X}_{s,t}\\
&  =\left[  \left(  \phi^{-1}\right)  ^{\ast}\alpha\right]  _{\phi\left(
y_{s}\right)  }\left(  \phi\left(  y_{t}\right)  -\phi\left(  y_{s}\right)
\right)  +\left[  \left(  \phi^{-1}\right)  ^{\ast}\alpha\right]
_{\phi\left(  y_{s}\right)  }^{\prime}\left[  d\phi_{y_{s}}\circ y_{s}^{\dag
}\right]  ^{\otimes2}\mathbb{X}_{s,t}.
\end{align*}
Now, if we recall Notation \ref{not.2.41}, we see that this is approximately
equal to another rough integral. More precisely%
\[
\left[  \int\alpha\left(  d\mathbf{y}\right)  \right]  _{s,t}^{1}%
\underset{^{3}}{\approx}\left[  \int\left(  \phi^{-1}\right)  ^{\ast}%
\alpha\left(  d\phi_{\ast}\mathbf{y}\right)  \right]  _{s,t}^{1}.
\]
However, additive functionals are unique up to this order, so in fact
\[
\left[  \int\alpha\left(  d\mathbf{y}\right)  \right]  _{s,t}^{1}=\left[
\int\left(  \phi^{-1}\right)  ^{\ast}\alpha\left(  d\phi_{\ast}\mathbf{y}%
\right)  \right]  _{s,t}^{1}%
\]
which is a relation which should hold under any reasonable integral. This is
summarized in the following theorem which gives us an alternative way of
defining this integral.

\begin{theorem}
\label{the.4.8}The integral, $\int\alpha\left(  d\mathbf{y}\right)  ,$ is the
unique $V$ -- valued rough path controlled by $\mathbf{X}$ on $\left[
0,T\right]  $ starting at $0$ determined by

\begin{enumerate}
\item $\left[  \int\alpha\left(  d\mathbf{y}\right)  \right]  _{s,t}%
^{1}=\left[  \int\left(  \left(  \phi^{-1}\right)  ^{\ast}\alpha\right)
\left(  d\phi_{\ast}\mathbf{y}\right)  \right]  _{s,t}^{1}$ for any chart and
$s<t\in\left[  0,T\right]  $ such that $y\left(  \left[  s,t\right]  \right)
\subset D\left(  \phi\right)  $

\item $\left[  \int\alpha\left(  d\mathbf{y}\right)  \right]  _{s}^{\dag
}=\alpha_{y_{s}}\circ y_{s}^{\dag}.$
\end{enumerate}

[See Theorem \ref{the.4.15} below for a more general version of this theorem.]
\end{theorem}

The next theorem with our current toolset can now be proved in two different
ways. We can reduce the result to a special case of Theorem \ref{the.3.24} or,
by using the chart definitions of integration along a one-form, can reduce it
to its validity in the flat case. The first method is quick but may hide the
concept of what is happening. We therefore provide both proofs.

\begin{theorem}
[Associativity Theorem II]\label{the.4.9}Suppose that $y\in CRP\left(
M\right)  ,$ $\alpha\in\Omega^{1}\left(  M,V\right)  ,$ and $K:M\rightarrow
L\left(  V,\tilde{V}\right)  $ is a smooth function so that $K\alpha\in
\Omega^{1}\left(  M,\tilde{V}\right)  .$ If $\mathbf{z}=\int\alpha\left(
d\mathbf{y}\right)  \in CRP\left(  V\right)  ,$ then%
\[
\int\left(  K\alpha\right)  \left(  d\mathbf{y}\right)  =\int\left\langle
K_{\ast}\left(  \mathbf{y}\right)  ,d\mathbf{z}\right\rangle \quad\left(
=:\int\left\langle K_{\ast}\left(  \mathbf{y}\right)  ,d\int\alpha\left(
d\mathbf{y}\right)  \right\rangle \right)  ,
\]
where $K_{\ast}\left(  \mathbf{y}\right)  =\left(  K\left(  y\right)  ,K_{\ast
y}y^{\dag}\right)  \in CRP_{\mathbf{X}}\left(  \operatorname{Hom}\left(
V,V^{\prime}\right)  \right)  .$
\end{theorem}

\begin{proof}
\textbf{Method 1:} Letting $\mathcal{G}=\left(  \psi,U\right)  $ be any gauge,
we define $\mathbf{f}:=\left(  f,f^{\dag}\right)  \in CRP_{\mathbf{X}}\left(
\operatorname{Hom}\left(  V,\tilde{V}\right)  \right)  $ by the formula%
\[
f_{s}:=K\left(  y_{s}\right)  \quad\text{and\quad}f_{s}^{\dag}:=K_{\ast y_{s}%
}y_{s}^{\dag}%
\]
and $\boldsymbol{\alpha}$$^{\left(  y,U\right)  }$ as in Proposition
\ref{pro.4.2} (see Proposition \ref{pro.4.10} below to see why $\mathbf{f}\in
CRP_{\mathbf{X}}\left(  \operatorname{Hom}\left(  V,\tilde{V}\right)  \right)
$). Then by Theorem \ref{the.3.24}, we have%
\begin{equation}
\int\left\langle \mathbf{f}\boldsymbol{\alpha}^{\left(  y,U\right)
},d\mathbf{y}^{\mathcal{G}}\right\rangle =\int\left\langle \mathbf{f}%
,d\mathbf{z}\right\rangle \label{equ.4.8}%
\end{equation}
where $\mathbf{z}=\int\left\langle \boldsymbol{\alpha}^{\left(  y,U\right)
},d\mathbf{y}^{\mathcal{G}}\right\rangle =\int\alpha\left(  d\mathbf{y}%
\right)  $. The right hand side in Equation (\ref{equ.4.8}) is simply
$\int\left\langle K_{\ast}\left(  \mathbf{y}\right)  ,d\mathbf{z}\right\rangle
$ while the $\mathbf{f}$$\boldsymbol{\alpha}$$^{\left(  y,U\right)  }$ term on
the left hand side can be recognized as $\left(  K\alpha\right)  ^{\left(
y,U\right)  }.$ Indeed, by the product rule with $\nabla^{U}$, we have%
\begin{align*}
\left(  K\alpha\right)  _{s}^{\left(  y,U\right)  }  &  =\left(  K\left(
y_{s}\right)  \alpha|_{T_{y_{s}}M},\nabla_{y_{s}^{\dag}\left(  \cdot\right)
}^{U}\left[  K\left(  \cdot\right)  \alpha\right]  \right) \\
&  =\left(  K\alpha|_{T_{y_{s}}M},K_{\ast y_{s}}y_{s}^{\dag}\alpha+K\left(
y_{s}\right)  \nabla_{y_{s}^{\dag}\left(  \cdot\right)  }^{U}\alpha\right) \\
&  =\left(  f_{s}\alpha_{s},f_{s}^{\dag}\alpha+f_{s}\alpha_{s}^{\dag\left(
y,U\right)  }\right) \\
&  =\mathbf{f}\boldsymbol{\alpha}^{\left(  y,U\right)  }.
\end{align*}
Thus%
\[
\int\left(  K\alpha\right)  \left(  d\mathbf{y}\right)  :=\int\left\langle
\left(  K\alpha\right)  _{s}^{\left(  y,U\right)  },d\mathbf{y}^{\mathcal{G}%
}\right\rangle =\int\left\langle \mathbf{f}\boldsymbol{\alpha}^{\left(
y,U\right)  },d\mathbf{y}^{\mathcal{G}}\right\rangle =\int\left\langle
K_{\ast}\left(  \mathbf{y}\right)  ,d\mathbf{z}\right\rangle \mathbf{.}%
\]

\textbf{Method 2: }By a simple patching argument, this is really a local
result and hence using the chart definitions of integration it suffices to
check this result in the case $M$ is an open subset of $\mathbb{R}^{d}.$ First
we check the derivative processes. From the definitions we have%
\[
z_{s}^{\dag}=\alpha_{y_{s}}\circ y_{s}^{\dag}\text{\quad and\quad}\left[
\int\left(  K\alpha\right)  \left(  d\mathbf{y}\right)  \right]  _{s}^{\dag
}=\left(  K\alpha\right)  _{y_{s}}\circ y_{s}^{\dag}=K\left(  y_{s}\right)
\alpha_{y_{s}}\circ y_{s}^{\dag}=K\left(  y_{s}\right)  z_{s}^{\dag}.
\]
Thus
\[
\left[  \int\left(  K\alpha\right)  \left(  d\mathbf{y}\right)  \right]
_{s}^{\dag}=K\left(  y_{s}\right)  z_{s}^{\dag}.
\]
On the other hand
\[
\left[  \int\left\langle K_{\ast}\left(  \mathbf{y}\right)  ,d\mathbf{z}%
\right\rangle \right]  _{s}^{\dag}=\left[  K\left(  y\right)  \right]
_{s}z_{s}^{\dag}=K\left(  y_{s}\right)  z_{s}^{\dag}%
\]

Similarly for the paths%
\[
z_{s,t}\underset{^{3}}{\approx}\alpha\left(  y_{s,t}\right)  +\alpha_{y_{s}%
}^{\prime}y_{s}^{\dag\otimes2}\mathbb{X}_{s,t}.
\]
and so%
\begin{align*}
\left[  \int\left(  K\alpha\right)  \left(  d\mathbf{y}\right)  \right]
_{s,t}^{1}  &  \underset{^{3}}{\approx}\left(  K\alpha\right)  _{y_{s}}%
y_{s,t}+\left(  K\alpha\right)  _{y_{s}}^{\prime}y_{s}^{\dag\otimes
2}\mathbb{X}_{s,t}\\
&  =K\left(  y_{s}\right)  \alpha_{y_{s}}y_{s,t}+K\left(  y_{s}\right)
\alpha_{y_{s}}^{\prime}y_{s}^{\dag\otimes2}\mathbb{X}_{s,t}+\left[  K_{y_{s}%
}^{\prime}\left(  y_{s}^{\dag}\left(  \cdot\right)  \otimes\alpha y_{s}^{\dag
}\left(  \cdot\right)  \right)  \right]  \mathbb{X}_{s,t}\\
&  \underset{^{3}}{\approx}K\left(  y_{s}\right)  z_{s,t}+K_{y_{s}}^{\prime
}\left(  y_{s}^{\dag}\otimes z_{s}^{\dag}\right)  \mathbb{X}_{s,t}.
\end{align*}
On the other hand
\begin{align*}
\left[  \int\left\langle K_{\ast}\left(  \mathbf{y}\right)  ,d\mathbf{z}%
\right\rangle \right]  _{s,t}^{1}  &  \underset{^{3}}{\approx}K\left(
y_{s}\right)  z_{s,t}+\left[  K_{\ast}\left(  \mathbf{y}\right)  \right]
_{s}^{\dag}z_{s}^{\dag}\mathbb{X}_{s,t}\\
&  =K\left(  y_{s}\right)  z_{s,t}+K_{y_{s}}^{\prime}\left(  y_{s}^{\dag
}\otimes z_{s}^{\dag}\right)  \mathbb{X}_{s,t}.
\end{align*}
Comparing these expressions completes the proof.
\end{proof}

\subsection{Push-forwards of Controlled Rough Paths\label{sub.4.2}}

Let $M=M^{d}$ and $\tilde{M}=\tilde{M}^{\tilde{d}^{\prime}}$ be manifolds. Let
$f:M\rightarrow\tilde{M}$ be smooth and suppose $\mathbf{y}_{s}=\left(
y_{s},y_{s}^{\dag}\right)  $ $\in CRP_{\mathbf{X}}\left(  M\right)  $. In
Definition \ref{def.4.11} below, we are going to give a definition of the
push-forward of $\mathbf{y}$ by $f$ which generalizes Example \ref{exa.2.56}.

\begin{proposition}
\label{pro.4.10}The pair $\left(  f\left(  y_{s}\right)  ,f_{\ast}\circ
y_{s}^{\dag}\right)  $ is an element of $CRP_{\mathbf{X}}\left(  \tilde
{M}\right)  $.
\end{proposition}

\begin{proof}
Suppose $\tilde{\phi}$ is a chart on $\tilde{M}$ such that $f\circ y\left(
\left[  a,b\right]  \right)  \subseteq D\left(  \tilde{\phi}\right)  $. We
must show that
\begin{equation}
\left\vert \tilde{\phi}\circ f\left(  y_{t}\right)  -\tilde{\phi}\circ
f\left(  y_{s}\right)  -d\tilde{\phi}\circ f_{\ast}y_{s}^{\dag}x_{s,t}%
\right\vert \leq C_{\tilde{\phi},a,b}\omega\left(  s,t\right)  ^{2/p}
\label{equ.4.9}%
\end{equation}
and%
\begin{equation}
\left\vert d\tilde{\phi}\circ f_{\ast}y_{t}^{\dag}-d\tilde{\phi}\circ f_{\ast
}y_{s}^{\dag}\right\vert \leq C_{\tilde{\phi},a,b}\omega\left(  s,t\right)
^{1/p} \label{equ.4.10}%
\end{equation}
hold for some $C_{\tilde{\phi},a,b}$ for all $s\leq t$ in $\left[  a,b\right]
$. We can again use our proof strategy outlined in Remark \ref{rem.2.52} to
treat this problem in nice neighborhoods. We leave it to the reader to follow
the pattern of earlier proofs to see that we can assume without loss of
generality that there is a chart $\phi$ on $M$ such that $y\left(  \left[
a,b\right]  \right)  \subseteq D\left(  \phi\right)  $ and $R\left(
\phi\right)  $ is convex. Which these simplifications, we note that $\left(
z_{s},z_{s}^{\dag}\right)  :=\left(  \phi\left(  y_{s}\right)  ,d\phi\circ
y_{s}^{\dag}\right)  $ is a controlled rough path on $R\left(  \phi\right)  $
and the function $F:=\tilde{\phi}\circ f\circ\phi^{-1}:R\left(  \phi\right)
\rightarrow R\left(  \tilde{\phi}\right)  $ is a map between Euclidean spaces.
Therefore Inequalities (\ref{equ.4.9}) and (\ref{equ.4.10}) reduce to the fact
that the pair $\left(  F\left(  z_{s}\right)  ,F^{\prime}\left(  z_{s}\right)
\circ z_{s}^{\dag}\right)  $ is a controlled rough path in $\mathbb{R}%
^{\tilde{d}}$ (which is trivial by applying Taylor's theorem after we check
that we get the correct terms); indeed, by a simple computation, we have
\begin{align*}
F^{\prime}\left(  z_{s}\right)  \circ z_{s}^{\dag}  &  =d\tilde{\phi}\circ
f_{\ast}\circ\left(  d\phi^{-1}\right)  _{z_{s}}\circ d\phi_{y_{s}}\circ
y_{s}^{\dag}\\
&  =d\tilde{\phi}\circ f_{\ast}\circ\left(  d\phi_{y_{s}}\right)  ^{-1}\circ
d\phi_{y_{s}}\circ y_{s}^{\dag}\\
&  =d\tilde{\phi}\circ f_{\ast}y_{s}^{\dag}%
\end{align*}
and clearly $F\left(  z_{s}\right)  =\tilde{\phi}\circ f\left(  y_{s}\right)
$.
\end{proof}

\begin{definition}
\label{def.4.11}The \textbf{push-forward} of $\mathbf{y}$ denoted by $f_{\ast
}\mathbf{y}$ or $f_{\ast}\left(  y,y^{\dag}\right)  $ is the rough path
controlled by $\mathbf{X}$ with path $f\left(  y_{s}\right)  $ and derivative
process $f_{\ast}\circ y_{s}^{\dag}$. If $\tilde{M}=\mathbb{R}^{\tilde{d}}$,
we will abuse notation an write $f_{\ast}\mathbf{y}_{s}$ to mean $\left(
f\left(  y_{s}\right)  ,df\circ y_{s}^{\dag}\right)  $ (i.e. we forget the
base point on the derivative process).
\end{definition}

\begin{remark}
\label{rem.4.12}The push-forward operation on elements in $CRP_{\mathbf{X}%
}\left(  M\right)  $ is clearly covariant, i.e. if $f:M\rightarrow N$ and
$g:N\rightarrow P$ are two smooth maps of manifolds, $M,N,$ and $P,$ then
$\left(  g\circ f\right)  _{\ast}\left(  \mathbf{y}\right)  =g_{\ast}\left(
f_{\ast}\left(  \mathbf{y}\right)  \right)  .$
\end{remark}

This definition is consistent with how we defined the integral of a one-form
along a controlled rough path in the sense that we have a fundamental theorem
of calculus. Let $V$ be a Banach space.

\begin{theorem}
\label{the.4.13} Let $\mathbf{y}_{s}=\left(  y_{s},y_{s}^{\dag}\right)  $ $\in
CRP_{\mathbf{X}}\left(  M\right)  $ and $f$ be a smooth function from $M$ to
$V$. Then%
\[
f\left(  y_{s}\right)  -f\left(  y_{0}\right)  =\left[  \int df\left[
d\mathbf{y}\right]  \right]  _{0,s}^{1}%
\]
where $df$ is interpreted as a one-form. Since we have $df\circ y_{s}^{\dag
}=\left[  \int df\left[  d\mathbf{y}\right]  \right]  _{s}^{\dag}$ we have the
equality%
\[
f_{\ast}\left(  y,y^{\dag}\right)  -\left(  f\left(  y_{0}\right)  ,0\right)
=\int df\left(  d\mathbf{y}\right)  .
\]

\end{theorem}

\begin{proof}
Although there are ways to do this proof without much machinery, we find it
more instructive to work on a Riemannian manifold with the Levi-Civita
covariant derivative. Since we have proved that the integral is independent of
choice of metric, it does not matter which one we pick. With this in mind, we
have the approximation
\[
\left[  \int df\left[  d\mathbf{y}\right]  \right]  _{s,t}^{1}\underset{^{3}%
}{\approx}df_{y_{s}}\left(  \exp_{y_{s}}^{-1}\left(  y_{t}\right)  \right)
+\left(  \nabla df\right)  _{y_{s}}\left[  y_{s}^{\dag\otimes2}\mathbb{X}%
_{s,t}\right]
\]
and as $\nabla df$ is symmetric, it follows that%
\begin{align*}
\left[  \int df\left[  d\mathbf{y}\right]  \right]  _{s,t}^{1}  &
\underset{^{3}}{\approx}df_{y_{s}}\left(  \exp_{y_{s}}^{-1}\left(
y_{t}\right)  \right)  +\frac{1}{2}\left(  \nabla df\right)  _{y_{s}}\left[
y_{s}^{\dag\otimes2}\left(  x_{s,t}\otimes x_{s,t}\right)  \right] \\
&  \underset{^{3}}{\approx}df_{y_{s}}\left(  \exp_{y_{s}}^{-1}\left(
y_{t}\right)  \right)  +\frac{1}{2}\left(  \nabla df\right)  _{y_{s}}\left[
\exp_{y_{s}}^{-1}\left(  y_{t}\right)  ^{\otimes2}\right] \\
&  \underset{^{3}}{\approx}f\left(  y_{t}\right)  -f\left(  y_{s}\right)  .
\end{align*}
The last approximation above follows from Taylor's Theorem on manifolds
(Theorem \ref{the.6.1} in the Appendix). Note here that $f\left(
y_{t}\right)  -f\left(  y_{s}\right)  $ is additive so that
\[
\left[  \int df\left[  d\mathbf{y}\right]  \right]  _{s,t}^{1}=f\left(
y_{t}\right)  -f\left(  y_{s}\right)  .
\]

\end{proof}

\begin{remark}
\label{rem.4.14}If $M\subseteq$ is an embedded submanifold of $W=\mathbb{R}%
^{k}$, $\left(  y_{s},y_{s}^{\dag}\right)  $ $\in CRP_{\mathbf{X}}\left(
M\right)  $, $I:M\rightarrow W$ denotes the identity (or embedding) map, and
$\left(  z_{s},z_{s}^{\dag}\right)  :=I_{\ast}\left(  y_{s},y_{s}^{\dag
}\right)  ,$ then we have
\[
z_{s}=y_{s}\text{\quad and\quad}z_{s}^{\dag}=\pi_{2}\circ y_{s}^{\dag}%
\]
where $\pi_{2}$ is the projection of the tangent vector component (i.e. it
forgets the base point). We can associate to it a unique rough path $\left(
y,\mathbb{Y}\right)  $ in $W$ such that
\[
\left(  z_{s}^{\dag}\otimes z_{s}^{\dag}\right)  \mathbb{X}_{s,t}%
\underset{^{3}}{\approx}\mathbb{Y}_{s,t}.
\]
In this case, this is a rough path in the embedded sense (See \cite{CDL13})
since
\[
\left[  I\left(  y_{s}\right)  \otimes Q\left(  y_{s}\right)  \right]  \left[
\mathbb{Y}\right]  _{s,t}\underset{^{3}}{\approx}\left[  I\left(
y_{s}\right)  \otimes Q\left(  y_{s}\right)  \right]  \left[  z_{s}^{\dag
}\otimes z_{s}^{\dag}\right]  \mathbb{X}_{s,t}=0
\]
as $Q\left(  y_{s}\right)  \circ z_{s}^{\dag}=0$ where $Q=I-P$ and $P\left(
x\right)  $ is orthogonal projection onto the tangent space at $x$.
\end{remark}

Lastly, we have a relation between push-forwards of paths and pull-backs of one-forms.

\begin{theorem}
[Push me-Pull me]\label{the.4.15}Let $f:M\rightarrow\tilde{M}$, let
$\mathbf{y}_{s}=\left(  y_{s},y_{s}^{\dag}\right)  \in CRP_{\mathbf{X}}\left(
M\right)  $ and let $\tilde{\alpha}\in\Omega^{1}\left(  \tilde{M},V\right)  $.
Then%
\begin{equation}
\left[  \int f^{\ast}\alpha\left(  d\mathbf{y}\right)  \right]  ^{1}=\left[
\int\alpha\left(  d\left(  f_{\ast}\mathbf{y}\right)  \right)  \right]  ^{1}.
\label{equ.4.11}%
\end{equation}
Moreover%
\[
\int f^{\ast}\alpha\left(  d\mathbf{y}\right)  =\int\alpha\left(  d\left(
f_{\ast}\mathbf{y}\right)  \right)
\]

\end{theorem}

\begin{proof}
This is a statement we only have to prove locally. Indeed for each
$s\in\left[  0,T\right]  $, there are charts $\phi^{s}$ and $\tilde{\phi}^{s}$
on $M$ and $\tilde{M}$ respectively such that $y_{s}\in D\left(  \phi
^{s}\right)  $ and $f\left(  y_{s}\right)  \in D\left(  \tilde{\phi}%
^{s}\right)  $ which are open. We take $\mathcal{U}_{s}:=$ $f^{-1}\left(
D\left(  \tilde{\phi}^{s}\right)  \right)  \cap D\left(  \phi^{s}\right)  $
and shrink it if necessary so that $\mathcal{V}_{s}=\phi\left(  \mathcal{U}%
_{s}\right)  $ is convex. Thus if we can prove that Eq. (\ref{equ.4.11}) holds
whenever $y\left(  \left[  a,b\right]  \right)  $ $\subseteq\mathcal{U}$ such
that $\phi\left(  \mathcal{U}\right)  $ is convex and such that $f\left(
y\left(  \left[  a,b\right]  \right)  \right)  \subseteq D\left(  \tilde{\phi
}\right)  $, we will be done. We do this now:

By Theorem \ref{the.4.8}, the fact that pull-backs are contravariant, and that
push-forwards are covariant, we have%
\begin{align*}
\left[  \int f^{\ast}\alpha\left(  d\mathbf{y}\right)  \right]  _{s,t}^{1}  &
=\left[  \int\left(  \phi^{-1}\right)  ^{\ast}f^{\ast}\alpha\left(
d\phi_{\ast}\mathbf{y}\right)  \right]  _{s,t}^{1}\\
&  =\left[  \int\left(  f\circ\phi^{-1}\right)  ^{\ast}\alpha\left(
d\phi_{\ast}\mathbf{y}\right)  \right]  _{s,t}^{1}\\
&  =\left[  \int\left(  \tilde{\phi}^{-1}\circ\tilde{\phi}\circ f\circ
\phi^{-1}\right)  ^{\ast}\alpha\left(  d\phi_{\ast}\mathbf{y}\right)  \right]
_{s,t}^{1}\\
&  =\left[  \int\left(  \tilde{\phi}\circ f\circ\phi^{-1}\right)  ^{\ast
}\left(  \left(  \tilde{\phi}^{-1}\right)  ^{\ast}\alpha\right)  \left(
d\phi_{\ast}\mathbf{y}\right)  \right]  _{s,t}^{1}\\
&  =\left[  \int\left(  \tilde{\phi}^{-1}\right)  ^{\ast}\alpha\left(
d\left(  \left(  \tilde{\phi}\circ f\circ\phi^{-1}\right)  _{\ast}\phi_{\ast
}\mathbf{y}\right)  \right)  \right]  _{s,t}^{1}%
\end{align*}
where the last step is just Eq. (\ref{equ.4.11}) on Euclidean space. This is a
simple computation (for example, see the appendix of \cite{CDL13}). Thus, we
have%
\begin{align*}
\left[  \int f^{\ast}\alpha\left(  d\mathbf{y}\right)  \right]  _{s,t}^{1}  &
=\left[  \int\left(  \tilde{\phi}^{-1}\right)  ^{\ast}\alpha\left(  d\left(
\left(  \tilde{\phi}\circ f\circ\phi^{-1}\right)  _{\ast}\phi_{\ast}%
\mathbf{y}\right)  \right)  \right]  _{s,t}^{1}\\
&  =\left[  \int\left(  \tilde{\phi}^{-1}\right)  ^{\ast}\alpha\left(
d\left(  \tilde{\phi}_{\ast}\left(  f_{\ast}\mathbf{y}\right)  \right)
\right)  \right]  _{s,t}^{1}\\
&  =\left[  \int\alpha\left(  d\left(  f_{\ast}\mathbf{y}\right)  \right)
\right]  _{s,t}^{1}.
\end{align*}
The fact that
\[
\left[  \int f^{\ast}\alpha\left(  d\mathbf{y}\right)  \right]  ^{\dag
}=\left[  \int\alpha\left(  d\left(  f_{\ast}\mathbf{y}\right)  \right)
\right]  ^{\dag}%
\]
is trivial.
\end{proof}

\section{Rough Differential Equations}

Before discussing rough differential equations on a manifold, we will give an
equivalent condition for a controlled rough path $\mathbf{z}\in
CRP_{\mathbf{X}}\left(  \mathbb{R}^{d}\right)  $ to satisfy the RDE
approximation on a compact interval in the flat case using logarithms.

For the next proposition, let $\psi$ be a logarithm on $\mathbb{R}^{d}$ such
that $\psi\left(  x,y\right)  =\left(  x,\bar{\psi}\left(  x,y\right)
\right)  .$

\begin{proposition}
\label{pro.5.1}Let $z:\left[  a,b\right]  \rightarrow\mathbb{R}^{d}$ be a path
and let $\mathcal{W\subseteq\mathbb{R}}^{d}$ be an open convex set such that
$z\left(  \left[  a,b\right]  \right)  \subseteq\mathcal{W}$ and
$\mathcal{W\times W}\subseteq D\left(  \psi\right)  .$ Then
\begin{equation}
z_{s,t}\underset{^{3}}{\approx}F_{x_{s,t}}\left(  z_{s}\right)  +\left(
\partial_{F_{w}\left(  z_{s}\right)  }F_{\tilde{w}}\right)  \left(
z_{s}\right)  |_{w\otimes\tilde{w}=\mathbb{X}_{s,t}} \label{equ.5.1}%
\end{equation}
if and only if%
\begin{equation}
\bar{\psi}\left(  z_{s},z_{t}\right)  \underset{^{3}}{\approx}F_{x_{s,t}%
}\left(  z_{s}\right)  +\left(  \partial_{F_{w}\left(  z_{s}\right)  }\left[
\bar{\psi}_{z_{s}}^{\prime}\left(  \cdot\right)  F_{\tilde{w}}\left(
\cdot\right)  \right]  \right)  \left(  z_{s}\right)  |_{w\otimes\tilde
{w}=\mathbb{X}_{s,t}} \label{equ.5.2}%
\end{equation}

\end{proposition}

\begin{proof}
If $z_{\cdot}$ satisfies Eq. (\ref{equ.5.1}), then from Eq. (\ref{equ.2.17})
of Theorem \ref{the.2.24} with $y=z_{t}$ and $x=z_{s}$ we find,%
\begin{align}
\bar{\psi}\left(  z_{s,}z_{t}\right)   &  =z_{s,t}+\frac{1}{2}\bar{\psi}%
_{x}^{\prime\prime}\left(  x\right)  \left(  z_{s,t}\right)  ^{\otimes
2}+C\left(  z_{s,}z_{t}\right)  \left(  z_{s,t}\right)  ^{\otimes
3}\label{equ.5.3}\\
&  \underset{^{3}}{\approx}F_{x_{s,t}}\left(  z_{s}\right)  +\left(
\partial_{F_{w}\left(  z_{s}\right)  }F_{\tilde{w}}\right)  \left(
z_{s}\right)  |_{w\otimes\tilde{w}=\mathbb{X}_{s,t}}+\frac{1}{2}\bar{\psi
}_{z_{s}}^{\prime\prime}\left(  z_{s}\right)  \left[  F_{x_{s,t}}\left(
z_{s}\right)  \right]  ^{\otimes2}, \label{equ.5.4}%
\end{align}
wherein $C$ is a smooth function and we have made use of the fact that
$z_{s,t}\underset{^{1}}{\approx}0$. By the product rule and the fact that
$\psi$ is a logarithm it follows that%
\begin{align}
\left(  \partial_{F_{w}\left(  z_{s}\right)  }\left[  \bar{\psi}_{z_{s}%
}^{\prime}\left(  \cdot\right)  F_{\tilde{w}}\left(  \cdot\right)  \right]
\right)  \left(  z_{s}\right)   &  =\bar{\psi}_{z_{s}}^{\prime\prime}\left(
z_{s}\right)  F_{w}\left(  z_{s}\right)  \otimes F_{\tilde{w}}\left(
z_{s}\right)  +\bar{\psi}_{z_{s}}^{\prime}\left(  z_{s}\right)  \left(
\partial_{F_{w}\left(  z_{s}\right)  }F_{\tilde{w}}\right)  \left(
z_{s}\right) \nonumber\\
&  =\bar{\psi}_{z_{s}}^{\prime\prime}\left(  z_{s}\right)  F_{w}\left(
z_{s}\right)  \otimes F_{\tilde{w}}\left(  z_{s}\right)  +\left(
\partial_{F_{w}\left(  z_{s}\right)  }F_{\tilde{w}}\right)  \left(
z_{s}\right)  . \label{equ.5.5}%
\end{align}
Since $\mathbf{X}$ is a weak-geometric rough path and $\bar{\psi}_{z_{s}%
}^{\prime\prime}\left(  z_{s}\right)  $ is symmetric, we also have,%
\[
\bar{\psi}_{z_{s}}^{\prime\prime}\left(  z_{s}\right)  F_{w}\left(
z_{s}\right)  \otimes F_{\tilde{w}}\left(  z_{s}\right)  |_{|_{w\otimes
\tilde{w}=\mathbb{X}_{s,t}}}=\frac{1}{2}\bar{\psi}_{z_{s}}^{\prime\prime
}\left(  z_{s}\right)  \left[  F_{x_{s,t}}\left(  z_{s}\right)  \right]
^{\otimes2},
\]
which combined with Eq. (\ref{equ.5.5}) shows,%
\begin{align}
&  \left(  \partial_{F_{w}\left(  z_{s}\right)  }\left[  \bar{\psi}_{z_{s}%
}^{\prime}\left(  \cdot\right)  F_{\tilde{w}}\left(  \cdot\right)  \right]
\right)  \left(  z_{s}\right)  |_{w\otimes\tilde{w}=\mathbb{X}_{s,t}%
}\label{equ.5.6}\\
&  \quad=\left(  \partial_{F_{w}\left(  z_{s}\right)  }F_{\tilde{w}}\right)
\left(  z_{s}\right)  |_{w\otimes\tilde{w}=\mathbb{X}_{s,t}}+\frac{1}{2}%
\bar{\psi}_{z_{s}}^{\prime\prime}\left(  z_{s}\right)  \left[  F_{x_{s,t}%
}\left(  z_{s}\right)  \right]  ^{\otimes2}. \label{equ.5.7}%
\end{align}
Equation (\ref{equ.5.2}) now follows directly from Eqs. (\ref{equ.5.4}) and
(\ref{equ.5.6}).

Conversely, now assume that Eq. (\ref{equ.5.2}) holds. From Eq. (\ref{equ.5.2}%
) and the fact that $\mathbf{X}$ is a rough path there exists $C_{1}<\infty$
such that $\left\vert \bar{\psi}\left(  z_{s},z_{t}\right)  \right\vert \leq
C_{1}\omega\left(  s,t\right)  ^{1/p}.$ Combining this observation with Eq.
(\ref{equ.5.3}) easily implies $z_{s,t}\underset{^{1}}{\approx}0.$ Indeed, by
uniform continuity, there exists a $\delta>0$ such that if $\left\vert
t-s\right\vert \leq\delta$, we have
\begin{align*}
\left\vert z_{s,t}\right\vert  &  \leq\left\vert \bar{\psi}\left(  z_{s}%
,z_{t}\right)  \right\vert +\left\vert \frac{1}{2}\psi_{z_{s}}^{\prime\prime
}\left(  z_{s}\right)  \left(  z_{s,t}\right)  ^{\otimes2}+C\left(
z_{s},z_{t}\right)  \left(  z_{s,t}\right)  ^{\otimes3}\right\vert \\
&  \leq C_{1}\omega\left(  s,t\right)  ^{1/p}+\frac{1}{2}\left\vert
z_{s,t}\right\vert .
\end{align*}
By using an argument similar to the proof of Theorem \ref{the.2.47} we can
bootstrap these local inequalities to prove the existence of a $C_{2}<\infty$
such that $\left\vert z_{s,t}\right\vert \leq C_{2}\omega\left(  s,t\right)
^{1/p}$ for $a\leq s\leq t\leq b$.

From Eqs. (\ref{equ.5.3}) and (\ref{equ.5.2}),%
\begin{align*}
z_{s,t}  &  =\bar{\psi}\left(  z_{s},z_{t}\right)  -\frac{1}{2}\bar{\psi
}_{z_{s}}^{\prime\prime}\left(  z_{s}\right)  \left(  \psi\left(  z_{s}%
,z_{t}\right)  \right)  ^{\otimes2}+C\left(  z_{s},z_{t}\right)  \left(
z_{s,t}\right)  ^{\otimes3}\\
&  \underset{^{3}}{\approx}F_{x_{s,t}}\left(  z_{s}\right)  +\left(
\partial_{F_{w}\left(  z_{s}\right)  }\left[  \bar{\psi}_{z_{s}}^{\prime
}\left(  \cdot\right)  F_{\tilde{w}}\left(  \cdot\right)  \right]  \right)
\left(  z_{s}\right)  |_{w\otimes\tilde{w}=\mathbb{X}_{s,t}}-\frac{1}{2}%
\bar{\psi}_{z_{s}}^{\prime\prime}\left(  z_{s}\right)  \left(  F_{x_{s,t}%
}\left(  z_{s}\right)  \right)  ^{\otimes2}\\
&  =F_{x_{s,t}}\left(  z_{s}\right)  +\left(  \partial_{F_{w}\left(
z_{s}\right)  }F_{\tilde{w}}\right)  \left(  z_{s}\right)  ,
\end{align*}
wherein we have used Eq. (\ref{equ.5.6}) for the last equality.
\end{proof}

\subsection{RDEs on a Manifold\label{sub.5.1}}

We now move to the manifold case. Let $F:M\rightarrow L\left(  W,TM\right)  $
be smooth such that $F\left(  m\right)  \in L\left(  W,T_{m}M\right)  .$
Alternatively we can think of $F:W\rightarrow\Gamma\left(  TM\right)  $ where
the map $w\rightarrow F_{w}\left(  \cdot\right)  $ is linear. We wish to give
meaning to the differential equation
\begin{equation}
d\mathbf{y}_{t}=F_{d\mathbf{X}_{t}}\left(  y_{t}\right)  \label{equ.5.8}%
\end{equation}
with initial condition $y_{0}=\bar{y}_{0}.$ To do this, first recall that any
vector field can be transferred to Euclidean space by using charts. If
$\mathcal{U}\subseteq D\left(  \phi\right)  $ where $\phi$ is a chart and
$\mathcal{V}:=\phi\left(  \mathcal{U}\right)  $ then
\[
F^{\phi}:=d\phi\circ\left(  F\circ\phi^{-1}\right)
\]
is a vector field on $\mathcal{V}$ (which does not carry the base point). If
$\mathbf{y}_{t}$ is to \textquotedblleft solve\textquotedblright%
\ (\ref{equ.5.8}) then $\mathbf{z}_{t}:=\phi_{\ast}\mathbf{y}_{t}$ should
solve the differential equation
\begin{equation}
d\mathbf{z}_{t}=F_{d\mathbf{X}_{t}}^{\phi}\left(  z_{t}\right)  .
\label{equ.5.9}%
\end{equation}

In the Euclidean case, Equation (\ref{equ.5.9}) is satisfied if%
\begin{align}
z_{t} &  \underset{^{3}}{\approx}z_{s}+F_{x_{s,t}}^{\phi}\left(  z_{s}\right)
+\left(  \partial_{F_{w}^{\phi}\left(  z_{s}\right)  }F_{\tilde{w}}^{\phi
}\right)  \left(  z_{s}\right)  |_{w\otimes\tilde{w}=\mathbb{X}_{s,t}%
}\label{equ.5.10}\\
z_{s}^{\dag} &  =F_{\left(  \cdot\right)  }^{\phi}\left(  z_{s}\right)
\nonumber
\end{align}
By writing out Equation (\ref{equ.5.10}) we have%
\begin{align}
\phi\left(  y_{t}\right)   &  \underset{^{3}}{\approx}\phi\left(
y_{s}\right)  +d\phi\circ F_{x_{s,t}}\left(  y_{s}\right)  +\left(
\partial_{d\phi\circ F_{w}\left(  y_{s}\right)  }d\phi\circ\left(
F_{\tilde{w}}\circ\phi^{-1}\right)  \right)  \left(  \phi\left(  y_{s}\right)
\right)  |_{w\otimes\tilde{w}=\mathbb{X}_{s,t}}\nonumber\\
&  =\phi\left(  y_{s}\right)  +d\phi\circ F_{x_{s,t}}\left(  y_{s}\right)
+F_{w}\left(  y_{s}\right)  \left[  d\phi\circ F_{\tilde{w}}\right]
|_{w\otimes\tilde{w}=\mathbb{X}_{s,t}}.\label{equ.5.11}%
\end{align}
We note that $F$ is linear with its range in the algebra of differential
operators, we can extend it uniquely to $\mathcal{F}$ which acts on the tensor
algebra $T\left(  \mathbb{R}^{n}\right)  $. In that case, we may write
(\ref{equ.5.11}) more concisely as%
\begin{equation}
\phi\left(  y_{t}\right)  \underset{^{3}}{\approx}\phi\left(  y_{s}\right)
+\left(  \mathcal{F}_{\mathbf{X}_{s,t}}\phi\right)  \left(  y_{s}\right)
.\label{equ.5.12}%
\end{equation}
This approximation will be satisfied for our solution to a rough differential
equation on a manifold. However, we will opt to define our solution in a
coordinate-free but equivalent way:

\begin{definition}
\label{def.5.2}$\mathbf{y=}\left(  y,y^{\dag}\right)  $ on $I_{0}=\left[
0,T\right]  $ or $[0,T)$ solves (\ref{equ.5.8}) if $y_{s}^{\dag}=F_{\left(
\cdot\right)  }\left(  y_{s}\right)  $ and for every $f\in C^{\infty}\left(
M\right)  $ and $\left[  a,b\right]  \subseteq I_{0}$, the approximation
\[
f\left(  y_{t}\right)  -f\left(  y_{s}\right)  \underset{^{3}}{\approx}\left(
\mathcal{F}_{\mathbf{X}_{s,t}}f\right)  \left(  y_{s}\right)
\]
holds for $a\leq s\leq t\leq b.$

If in addition $y_{0}=\bar{y}_{0}$, we say $\mathbf{y}$ solves (\ref{equ.5.8})
with initial condition $y_{0}=\bar{y}_{0}.$
\end{definition}

While this is an intuitive definition, there are many workable
characterizations of solving a rough differential equation. Before presenting
a few more, we note that if $\alpha\in\Omega^{1}\left(  M,V\right)  $ and
$F:M\rightarrow L\left(  W,TM\right)  $ is smooth, then the composition
$\alpha\circ F_{\left(  \cdot\right)  }$ is a smooth map from $M$ to $V$.
Given $\mathbf{y}\in CRP_{\mathbf{X}}\left(  M\right)  $, we can then define
the push-forward $\left[  \alpha\circ F_{\left(  \cdot\right)  }\right]
_{\ast}\mathbf{y}\in CRP_{\mathbf{X}}\left(  L\left(  W,V\right)  \right)  $.
Recall from Theorem \ref{the.2.7} that we can define the integral increment%
\[
\int_{s}^{t}\left\langle \left(  \left[  \alpha\circ F_{\left(  \cdot\right)
}\right]  _{\ast}\mathbf{y}\right)  _{\tau},d\mathbf{X}_{\tau}\right\rangle .
\]
With this idea in mind, we now give other characterizations of solving Eq.
(\ref{equ.5.8}).

\begin{theorem}
\label{the.5.3}Let $y$ be a path in $M$ on $I_{0}$ with $y_{s}^{\dag}%
=F_{\cdot}\left(  y_{s}\right)  .$ Let $\mathbf{y}=\left(  y,y^{\dag}\right)
\in CRP_{\mathbf{X}}\left(  M\right)  $. The following are equivalent.

\begin{enumerate}
\item For every chart $\phi$ with $a,b\in I_{0}$ such that $y\left(  \left[
a,b\right]  \right)  \subseteq D\left(  \phi\right)  $ the approximation
\begin{equation}
\phi\left(  y_{t}\right)  \underset{^{3}}{\approx}\phi\left(  y_{s}\right)
+d\phi\circ F_{x_{s,t}}\left(  y_{s}\right)  +F_{w}\left(  y_{s}\right)
\left[  d\phi\circ F_{\tilde{w}}\right]  |_{w\otimes\tilde{w}=\mathbb{X}%
_{s,t}} \label{equ.5.13}%
\end{equation}
holds $a\leq s\leq t\leq b$; that is%
\[
\phi\left(  y_{t}\right)  -\phi\left(  y_{s}\right)  =\int_{s}^{t}\left\langle
\left(  \left[  d\phi\circ F_{\left(  \cdot\right)  }\right]  _{\ast
}\mathbf{y}\right)  _{\tau},d\mathbf{X}_{\tau}\right\rangle
\]
for $a\leq s\leq t\leq b$.

\item If $V$ is a Banach space, $\alpha\in\Omega^{1}\left(  M,V\right)  $, and
$\left[  a,b\right]  $ is such that $\left[  a,b\right]  \subseteq I_{0}$ then%
\[
\int_{s}^{t}\alpha\left(  d\mathbf{y}\right)  \underset{^{3}}{\approx}%
\alpha\left(  F_{x_{s,t}}\left(  y_{s}\right)  \right)  +F_{w}\left(
y_{s}\right)  \left[  \alpha\circ F_{\tilde{w}}\right]  |_{w\otimes\tilde
{w}=\mathbb{X}_{s,t}}%
\]
for $a\leq s\leq t\leq b$; that is
\[
\int_{s}^{t}\alpha\left(  d\mathbf{y}\right)  =\int_{s}^{t}\left\langle
\left(  \left[  \alpha\circ F_{\left(  \cdot\right)  }\right]  _{\ast
}\mathbf{y}\right)  _{\tau},d\mathbf{X}_{\tau}\right\rangle
\]
for $a\leq s\leq t\leq b.$

\item $\mathbf{y}$ solves (\ref{equ.5.8}); that is%
\[
f\left(  y_{t}\right)  -f\left(  y_{s}\right)  =\int_{s}^{t}\left\langle
\left(  \left[  df\circ F_{\left(  \cdot\right)  }\right]  _{\ast}%
\mathbf{y}\right)  _{\tau},d\mathbf{X}_{\tau}\right\rangle
\]
for every $f\in C^{\infty}\left(  M\right)  $
\end{enumerate}
\end{theorem}

\begin{proof}
We will only prove the approximations in each case, that is the first
statement of each item. The second statements are immediate from the definitions.

$\left(  1\implies2\right)  $ We assume that $\mathbf{y}$ satisfies the
approximation in Eq. (\ref{equ.5.13}) for any chart. Let $\left[  a,b\right]
\subseteq I_{0}$ be given. For every $m\in y\left(  \left[  a,b\right]
\right)  $, we have there exists a chart $\phi^{m}$ with open domain
$\mathcal{V}_{m}:=D\left(  \phi^{m}\right)  $ containing $m$ whose range
$R\left(  \phi^{m}\right)  $ is convex. We may now use our patching strategy
outlined in Remark \ref{rem.2.52} with the cover $\left\{  V_{m}\right\}
_{m\in y\left(  \left[  a,b\right]  \right)  }$ applied to the function%
\[
\left(  s,t\right)  \longrightarrow\int_{s}^{t}\alpha\left(  d\mathbf{y}%
\right)  -\alpha\left(  F_{x_{s,t}}\left(  y_{s}\right)  \right)
-F_{w}\left(  y_{s}\right)  \left[  \alpha\circ F_{\tilde{w}}\right]
|_{w\otimes\tilde{w}=\mathbb{X}_{s,t}}%
\]
to reduce to the case where $y\left(  \left[  a,b\right]  \right)  $ is
contained in the domain of a single chart.

With this reduction, we can further reduce to the flat case by defining
$\mathbf{z}_{t}:=\left(  \phi\left(  y_{t}\right)  ,F_{\cdot}\left(
y_{s}\right)  \right)  $ and $F^{\phi}:=d\phi\left(  F\circ\phi^{-1}\right)  $
and showing
\begin{align*}
&  \int_{s}^{t}\alpha\left(  d\mathbf{y}\right)  -\alpha\left(  F_{x_{s,t}%
}\left(  y_{s}\right)  \right)  -F_{w}\left(  y_{s}\right)  \left[
\alpha\circ F_{\tilde{w}}\right]  |_{w\otimes\tilde{w}=\mathbb{X}_{s,t}}\\
&  \quad=\int_{s}^{t}\left(  \left(  \phi^{-1}\right)  ^{\ast}\alpha\right)
\left(  d\mathbf{z}\right)  -\left(  \left(  \phi^{-1}\right)  ^{\ast}%
\alpha\right)  _{z_{s}}\left(  F_{x_{s,t}}^{\phi}\left(  z_{s}\right)
\right)  -\left(  \partial_{F_{w}^{\phi}\left(  z_{s}\right)  }\left[  \left(
\phi^{-1}\right)  ^{\ast}\alpha\circ F_{\tilde{w}}^{\phi}\right]  \right)
\left(  z_{s}\right)  |_{w\otimes\tilde{w}=\mathbb{X}_{s,t}}.
\end{align*}
The above equality is true due to the following three identities:%
\begin{align}
\int_{s}^{t}\alpha\left(  d\mathbf{y}\right)   &  =\int_{s}^{t}\left(  \left(
\phi^{-1}\right)  ^{\ast}\alpha\right)  \left(  d\mathbf{z}\right)
,\label{equ.5.14}\\
\alpha\left(  F_{x_{s,t}}\left(  y_{s}\right)  \right)   &  =\left(  \left(
\phi^{-1}\right)  ^{\ast}\alpha\right)  _{z_{s}}\left(  F_{x_{s,t}}^{\phi
}\left(  z_{s}\right)  \right)  ,\text{ and}\label{equ.5.15}\\
F_{w}\left(  y_{s}\right)  \left[  \alpha\circ F_{\tilde{w}}\right]   &
=\left(  \partial_{F_{w}^{\phi}\left(  z_{s}\right)  }\left[  \left(
\phi^{-1}\right)  ^{\ast}\alpha\circ F_{\tilde{w}}^{\phi}\right]  \right)
\left(  z_{s}\right)  . \label{equ.5.16}%
\end{align}
\qquad\qquad Equation (\ref{equ.5.14}) is true by Theorem \ref{the.4.15}. The
differential geometric identities in Eqs. (\ref{equ.5.15}) and (\ref{equ.5.16})
are simply a matter of unwinding the definitions.

$\left(  2\implies3\right)  $ By letting $\alpha=df$ and using Theorem
\ref{the.4.13}, we have
\begin{align*}
&  f\left(  y_{t}\right)  -f\left(  y_{s}\right)  =\int_{s}^{t}df\left(
d\mathbf{y}\right) \\
&  \quad\underset{^{3}}{\approx}df\left(  F_{x_{s,t}}\left(  y_{s}\right)
\right)  +F_{w}\left(  y_{s}\right)  \left[  df\circ F_{\tilde{w}}\right]
|_{w\otimes\tilde{w}=\mathbb{X}_{s,t}}\\
&  \quad=\left(  \mathcal{F}_{\mathbf{X}_{s,t}}f\right)  \left(  y_{s}\right)
\end{align*}

$\left(  3\implies1\right)  $ We leave it to the reader to work through the
details of this step which follow exactly as in the proof of Theorem
\ref{the.2.57} by letting $f^{i}$ be the coordinates of $\phi.$
\end{proof}

By Theorem \ref{the.6.9} in the Appendix, we see that a solution to a rough
differential equation in flat space does actually satisfy Eq. (\ref{equ.5.8}).
Moreover, we immediately get local existence of solutions:

\begin{theorem}
\label{the.5.4}Let $F:W\rightarrow\Gamma\left(  TM\right)  $ be linear and let
$\bar{y}_{0}$ be a point in $M$. There exists a local in time solution to the
differential Eq. (\ref{equ.5.8}) with initial condition $y_{0}=\bar{y}_{0}.$
\end{theorem}

\begin{proof}
Let $\phi$ be any chart such that $\bar{y}_{0}\in D\left(  \phi\right)  $.
Then there exists a solution on some time interval $\left[  0,\tau\right]  $
in $R\left(  \phi\right)  $ to the differential equation%
\[
d\mathbf{z}_{t}=F_{d\mathbf{X}_{t}}^{\phi}\left(  z_{t}\right)
\]
with initial condition $z_{0}=\phi\left(  \bar{y}_{0}\right).$ If
$\tilde{\phi}$ is any other chart such that $\left[  a,b\right]
\subseteq\left[  0,\tau\right]  $ and $y\left(  \left[  a,b\right]  \right)
\subseteq D\left(  \tilde{\phi}\right)  $, then the transition map
$\tilde{\phi}\circ\phi^{-1}$ has a domain containing $z\left(  \left[
a,b\right]  \right)  $. It is easy to check that
\[
F^{\tilde{\phi}}=\left(  F^{\phi}\right)  ^{\tilde{\phi}\circ\phi^{-1}}%
\]
and by Corollary \ref{cor.6.12}, after unraveling the notation, we have
\[
\tilde{\phi}\left(  y_{t}\right)  \underset{^{3}}{\approx}\tilde{\phi}\left(
y_{s}\right)  +d\tilde{\phi}\circ F_{x_{s,t}}\left(  y_{s}\right)
+F_{w}\left(  y_{s}\right)  \left[  d\tilde{\phi}\circ F_{\tilde{w}}\right]
|_{w\otimes\tilde{w}=\mathbb{X}_{s,t}}.
\]
Thus satisfying the rough differential equation approximation in one chart is
sufficient prove that it hold in all charts.
\end{proof}

Solutions to rough differential equations will be unique on the intersection
of their time domain up to some possible explosion time. This is stated more
precisely in the following theorem.

\begin{theorem}
\label{the.5.5}Let $T>0$. There is unique solution $\mathbf{y}_{t}\in
CRP_{\mathbf{X}}\left(  M\right)  $ to $d\mathbf{y}_{t}=F_{d\mathbf{X}_{t}%
}\left(  y_{t}\right)  $ with initial condition $y_{0}=\bar{y}_{0}$ existing
either on all of $\left[  0,T\right]  $ or on $[0,\tau)$ for some $\tau<T$
such that the closure of $\left\{  y_{t}:0\leq t<\tau\right\}  $ is not compact.
\end{theorem}

\begin{proof}
This proof follows the strategy of the proof of Theorem 4.2 in \cite{CDL13}.
First we will show that we can always concatenate a solution $\mathbf{y}$
provided it has not exploded yet:

Suppose there exists a $\mathbf{y}$ solving $d\mathbf{y}_{t}=F_{d\mathbf{X}%
_{t}}\left(  y_{t}\right)  $ with initial condition $y_{0}=\bar{y}_{0}$ on
$[0,\tau)$. If there exists a compact $K\subseteq M$ such that $\left\{
y_{t}:0\leq t<\tau\right\}  \subseteq K$, then there is a sequence of
increasing times $t_{n}\in\lbrack0,\tau)$ such that $t_{n}\rightarrow\tau$ and
$y_{\infty}:=\lim_{n\rightarrow\infty}y\left(  t_{n}\right)  $ exists and is
in $K$. We can now choose a chart $\phi$ such that the closure of $D\left(
\phi\right)  $ is compact and such that $y_{\infty}\in D\left(  \phi\right)
.$ Let $\mathbf{z}_{t}$ and $a$ be such that $\mathbf{z}_{t}:=\phi_{\ast
}\mathbf{y}$ on some time interval $[a,\tau)$ such that $y\left(
[a,\tau)\right)  \subseteq D\left(  \phi\right)  $. By appealing to Lemma
\ref{lem.6.8} in the Appendix, there exists an $\epsilon>0$ and a $U\subseteq
D\left(  \phi\right)  $ containing $y_{\infty}$ such that for any $s\in\left[
\tau-\epsilon,\tau\right]  $ and $\bar{z}\in U$, there exists $\mathbf{\tilde
{z}\in}CRP_{\mathbf{X}}\left(  \mathbb{R}^{d}\right)  $ defined on $\left[
s,\tau+\epsilon\right]  $ which solves%
\[
d\mathbf{\tilde{z}}_{t}=F_{x_{s,t}}^{\phi}\left(  \tilde{z}_{t}\right)
\quad\text{with\quad}\tilde{z}_{s}=\bar{z}.
\]
Letting $n$ be sufficiently large, we have that $t_{n}\in\left[  \tau
-\epsilon,\tau\right]  $ and we let $\mathbf{\tilde{z}}$ be the solution to
$d\mathbf{\tilde{z}}_{t}=F_{x_{s,t}}^{\phi}\left(  \tilde{z}_{t}\right)  $
with initial condition $\tilde{z}_{s}=z\left(  t_{n}\right)  .$ Then we can
concatenate $\mathbf{z}$ and $\mathbf{\tilde{z}}$ in the sense of Lemma
\ref{lem.2.6}. By pulling these back to the manifold by $\phi^{-1}$, we now
have a solution $\mathbf{\tilde{y}}$ on $M$ which is defined on $\left[
0,\tau+\epsilon\right]  .$

With the preceding fact shown, we may now prove the theorem. We define%
\[
\tau:=\sup\left\{  T_{0}\in\left(  0,T\right)  :\exists\mathbf{y}\text{
solving }d\mathbf{y}_{t}=F_{d\mathbf{X}_{t}}\left(  y_{t}\right)  \text{ with
}y_{0}=\bar{y}_{0}\right\}  .
\]
We can then for any $t<\tau$ define $\mathbf{y}_{t}:=\mathbf{\hat{y}}_{t}$
where $\mathbf{\hat{y}}_{t}$ is any solution to $d\mathbf{y}_{t}%
=F_{d\mathbf{X}_{t}}\left(  y_{t}\right)  $ with initial condition $y_{0}%
=\bar{y}_{0}.$ By the uniqueness of solutions to rough differential equations
on flat space and the fact that we can cover any portion of the path with the
domain of a chart, we know that $\mathbf{y}_{t}$ is well defined, and in fact
satisfies $d\mathbf{y}_{t}=F_{d\mathbf{X}_{t}}\left(  y_{t}\right)  $ on all
of $[0,\tau)$. If the closure of $\left\{  y_{t}:0\leq t<\tau\right\}  $ is
compact, then from what we showed above, we can produce a solution
$\mathbf{\tilde{y}}$ which is defined on $\left[  0,\tau+\epsilon\right]  $
for some $\epsilon>0$. In this case, $\tau$ must be $T$ and $\mathbf{\tilde
{y}|}_{\left[  0,T\right]  }$ is a solution defined on all of $\left[
0,T\right]  $.
\end{proof}

\begin{definition}
\label{def.5.6}Let $f:M\rightarrow N$ be a smooth map between manifolds. Let
$F:W\rightarrow\Gamma\left(  TM\right)  $ and $\tilde{F}:W\rightarrow
\Gamma\left(  TN\right)  $ be linear. We say $F$ and $\tilde{F}$ are $f$ $-$
\textbf{related dynamical systems} if%
\[
f_{\ast}F_{w}=\tilde{F}_{w}\circ f\text{ for all }w\in W.
\]

\end{definition}

As in the flat case and shown in the Appendix in Theorem \ref{the.6.11}, we
have a relation between dynamical systems.\ The proof is no different in the
manifold case, and so we omit it.

\begin{theorem}
\label{the.5.7}Suppose $f:M\rightarrow N$ is a smooth map between manifolds
and let $F:W\rightarrow\Gamma\left(  TM\right)  $ and $\tilde{F}%
:W\rightarrow\Gamma\left(  TN\right)  $ be $f-$related dynamical systems. If
$\mathbf{y}$ solves the initial value problem Eq. (\ref{equ.5.8}), then
$\mathbf{\tilde{y}}_{t}:=\left(  \tilde{y}_{t},\tilde{y}_{s}^{\dag}\right)
:=f_{\ast}\mathbf{\tilde{y}}_{t}$ solves%
\[
d\mathbf{\tilde{y}}_{t}=\tilde{F}_{d\mathbf{X}_{t}}\left(  \tilde{y}%
_{t}\right)  \text{\quad with\quad}\tilde{y}_{0}=f\left(  \bar{y}_{0}\right)
.
\]

\end{theorem}

\subsubsection{RDEs from the Gauge Perspective\label{sub.5.1.1}}

Following the theme of Theorem \ref{the.2.44}, we also have a way
to view a solution to a differential equation using the gauge perspective. Let
$\psi$ be a logarithm on $M$ with diagonal domain $\mathcal{D}$.

\begin{theorem}
\label{the.5.8}Let $y$ be a path in $M$ on $I_{0}$ with $y_{s}^{\dag}%
=F_{\cdot}\left(  y_{s}\right)  .$ Let $\mathbf{y}=\left(  y,y^{\dag}\right)
$. Then $\mathbf{y}$ solves Equation (\ref{equ.5.8}) if and only if for every
$a,b$ such that $\left[  a,b\right]  \subseteq I_{0}$, there exists a
$\delta>0$ such that
\begin{equation}
\psi\left(  y_{s},y_{t}\right)  \underset{^{3}}{\approx}F_{x_{s,t}}\left(
y_{s}\right)  +F_{w}\left(  y_{s}\right)  \left[  \left(  \psi_{y_{s}}\right)
_{\ast}F_{\tilde{w}}\right]  |_{w\otimes\tilde{w}=\mathbb{X}_{s,t}}.
\label{equ.5.17}%
\end{equation}
provided $a\leq s\leq t\leq b$ \ and $t-s<\delta$.
\end{theorem}

\begin{proof}
This proof will be similar to the proof of Theorem \ref{the.2.44}.

First we show the condition of Theorem \ref{the.5.8} implies that $\mathbf{y}$
solves Equation (\ref{equ.5.8}). Let $\phi$ be a chart and let $\left[
a,b\right]  $ be such that $y\left(  \left[  a,b\right]  \right)  \subseteq
D\left(  \phi\right)  $. By defining
\begin{align*}
z_{s} &  :=\phi\left(  y_{s}\right)  \\
\psi^{\phi}\left(  x,y\right)   &  :=\phi_{\ast}\psi\left(  \phi^{-1}\left(
x\right)  ,\phi^{-1}\left(  y\right)  \right)  \\
F_{w}^{\phi}\left(  x\right)   &  :=d\phi\left(  F_{w}\left(  \phi^{-1}\left(
x\right)  \right)  \right)
\end{align*}
and denoting $\psi^{\phi}\left(  x,y\right)  =\left(  x,\bar{\psi}^{\phi
}\left(  x,y\right)  \right)  $, Eq. (\ref{equ.5.17}), once pushed forward by
$\phi$, can be written as%
\[
\bar{\psi}^{\phi}\left(  z_{s},z_{t}\right)  \underset{^{3}}{\approx
}F_{x_{s,t}}^{\phi}\left(  z_{s}\right)  +\left(  \partial_{F_{w}^{\phi
}\left(  z_{s}\right)  }\left[  \bar{\psi}_{z_{s}}^{\phi\prime}\left(
\cdot\right)  F_{\tilde{w}}^{\phi}\left(  \cdot\right)  \right]  \right)
\left(  z_{s}\right)  |_{w\otimes\tilde{w}=\mathbb{X}_{s,t}}%
\]
provided $a\leq s\leq t\leq b$ \ and $t-s<\delta.$ We then must prove that $z$
solves Eq. (\ref{equ.5.10}) for all $a\leq s\leq t\leq b.$ However, by
appealing to Lemma \ref{lem.2.48} and Lemma \ref{lem.6.13} of the Appendix, we
only need to prove Eq. (\ref{equ.5.10}) holds for every $u$ in $\left[
a,b\right]  $ for $s\leq t$ in $\left(  u-\delta_{u},u+\delta_{u}\right)
\cap\left[  a,b\right]  $ for some $\delta_{u}$. We do this now:

For any $u\in\left[  a,b\right]  $, let $\mathcal{W}_{u}$ be an open convex
set of $z_{u}$ such that $\mathcal{W}_{u}\times\mathcal{W}_{u}\subseteq
D\left(  \psi^{\phi}\right)  $. We then choose $\delta_{u}>0$ such that
$z\left(  \left[  u-\delta_{u}\text{,}u+\delta_{u}\right]  \cap\left[
a,b\right]  \right)  \subseteq\mathcal{W}_{u}$ and $2\delta_{u}\leq\delta$. We
are now in the setting of Proposition \ref{pro.5.1} and have therefore shown
$\mathbf{y}$ solves Eq. (\ref{equ.5.8}).

For the reverse implication, let $\left[  a,b\right]  \subseteq I_{0}$ be
given. Choose $\delta>0$ such that $\left\vert t-s\right\vert \leq\delta$ for
$a\leq s\leq t\leq b$ implies that $\left\vert \psi\left(  y_{s},y_{t}\right)
\right\vert _{g}$ is bounded. Around every point $m$ of $y\left(  \left[
a,b\right]  \right)  $, there exists an open $\mathcal{O}_{m}$ containing $m$
such that $\mathcal{O}_{m}\times\mathcal{O}_{m}\subseteq\mathcal{D}$.
Additionally for each $m$ there exists a chart $\phi^{m}$ such that $m\in
D\left(  \phi^{m}\right)  $, $D\left(  \phi^{m}\right)  \subseteq
\mathcal{O}_{m}$, and $\mathcal{W}_{m}:=R\left(  \phi^{m}\right)  $ is convex.
We may now use Remark \ref{rem.2.52} with the cover $\left\{  \mathcal{V}%
_{m}\right\}  _{m\in y\left(  \left[  a,b\right]  \right)  }$ and $D=\left\{
\left(  s,t\right)  :a\leq s\leq t\leq b\text{ and }\left\vert t-s\right\vert
\leq\delta\right\}  $ with the function
\[
\left(  s,t\right)  \longrightarrow\psi\left(  y_{s},y_{t}\right)
-F_{x_{s,t}}\left(  y_{s}\right)  -F_{w}\left(  y_{s}\right)  \left[  \left(
\psi_{y_{s}}\right)  _{\ast}\circ F_{\tilde{w}}\right]  |_{w\otimes\tilde
{w}=\mathbb{X}_{s,t}}.
\]
Doing this, we have reduced to considering the case of our path being
contained in the domain of a single chart $\phi$ such that $D\left(
\phi\right)  \times D\left(  \phi\right)  \subseteq\mathcal{D}$ and $R\left(
\phi\right)  $ is convex. By using the same definitions above for $z_{s}$,
$F^{\phi}$, and $\psi^{\phi}$, we reduce proving
\[
\psi\left(  y_{s},y_{t}\right)  \underset{^{3}}{\approx}F_{x_{s,t}}\left(
y_{s}\right)  +F_{w}\left(  y_{s}\right)  \left[  \left(  \psi_{y_{s}}\right)
_{\ast}\circ F_{\tilde{w}}\right]  |_{w\otimes\tilde{w}=\mathbb{X}_{s,t}}%
\]
to the flat case%
\[
\bar{\psi}^{\phi}\left(  z_{s},z_{t}\right)  \underset{^{3}}{\approx
}F_{x_{s,t}}^{\phi}\left(  z_{s}\right)  +\left(  \partial_{F_{w}^{\phi
}\left(  z_{s}\right)  }\left[  \bar{\psi}_{z_{s}}^{\phi\prime}\left(
\cdot\right)  F_{\tilde{w}}^{\phi}\left(  \cdot\right)  \right]  \right)
\left(  z_{s}\right)  |_{w\otimes\tilde{w}=\mathbb{X}_{s,t}}.
\]
This is now in the setting of Proposition \ref{pro.5.1} and hence we are finished.
\end{proof}

Akin to the integral formulas, there is also a characterization of solving a
differential equation which involves a gauge $\left(  \psi,U\right)  $.

\begin{theorem}
\label{the.5.9}$\mathbf{y}=\left(  y,y^{\dag}\right)  $ on $I_{0}$ solves
(\ref{equ.5.8}) if and only if $y_{s}^{\dag}=F_{\left(  \cdot\right)  }\left(
y_{s}\right)  $ and for all $\left[  a,b\right]  \subseteq I_{0}$, there
exists a $\delta>0$ \ such that $\left\vert t-s\right\vert \leq\delta$, and
$a\leq s\leq t\leq b$ implies%
\[
\psi\left(  y_{s},y_{t}\right)  \underset{^{3}}{\approx}F_{x_{s,t}}\left(
y_{s}\right)  +\left(  -S_{y_{s}}^{\psi_{\ast},U}\left[  F_{w}\left(
y_{s}\right)  \otimes F_{\tilde{w}}\left(  y_{s}\right)  \right]
+F_{w}\left(  y_{s}\right)  \left[  U\left(  y_{s},\cdot\right)  F_{\tilde{w}%
}\right]  \right)  |_{w\otimes\tilde{w}=\mathbb{X}_{s,t}}.
\]

\end{theorem}

\begin{proof}
This follows immediately from the product rule:%
\begin{align*}
F_{w}\left(  y_{s}\right)  \left[  \left(  \psi_{y_{s}}\right)  _{\ast
}F_{\tilde{w}}\right]   &  =F_{w}\left(  y_{s}\right)  \left[  \left(
\psi_{y_{s}}\right)  _{\ast\left(  \cdot\right)  }U\left(  y_{s},\cdot\right)
^{-1}U\left(  y_{s},\cdot\right)  F_{\tilde{w}}\right]  \\
&  =-S_{y_{s}}^{\psi_{\ast},U}\left[  F_{w}\left(  y_{s}\right)  \otimes
F_{\tilde{w}}\left(  y_{s}\right)  \right]  +F_{w}\left(  y_{s}\right)
\left[  U\left(  y_{s},\cdot\right)  F_{\tilde{w}}\right]
\end{align*}

\end{proof}

\begin{example}
If $\nabla$ is a covariant derivative, then $\mathbf{y}$ on $I_{0}$ solves
(\ref{equ.5.8}) if and only if $y_{s}^{\dag}=F\left(  y_{s}\right)  $ and
\[
\exp_{y_{s}}^{-1}\left(  y_{t}\right)  \underset{^{3}}{\approx}F_{x_{s,t}%
}\left(  y_{s}\right)  +\left(  \nabla_{F_{w}\left(  y_{s}\right)  }%
F_{\tilde{w}}\right)  -\frac{1}{2}T^{\nabla}\left[  F_{w}\left(  y_{s}\right)
\otimes F_{\tilde{w}}\left(  y_{s}\right)  \right]  |_{w\otimes\tilde
{w}=\mathbb{X}_{s,t}}%
\]
for $s$ and $t$ close.
\end{example}

\section{Appendix\label{sec.6}}

\subsection{Taylor Expansion on a Riemannian manifold\label{sub.6.1}}

Let $\left(  M,g\right)  $ be a Riemannian manifold, $\nabla$ be the
Levi-Civita covariant derivative, $\exp\left(  tv\right)  $ be the geodesic
flow, and $//_{t}\left(  \sigma\right)  $ denote parallel translation relative
to $\nabla.$ Recall that Taylor's formula with integral remainder states for
any smooth function $g$ on $\left[  0,1\right]  ,$ that%
\begin{equation}
G\left(  1\right)  =\sum_{k=0}^{n}\frac{1}{n!}G^{\left(  k\right)  }\left(
0\right)  +\frac{1}{n!}\int_{0}^{1}G^{\left(  n+1\right)  }\left(  t\right)
\left(  1-t\right)  ^{n}dt. \label{equ.6.1}%
\end{equation}
We now apply this result to $G\left(  t\right)  :=f\left(  \exp_{m}\left(
tv\right)  \right)  $ where $f\in C^{\infty}\left(  M\right)  ,$ $v\in T_{m}M$
and $m\in M.$ To this end let $\sigma\left(  t\right)  :=\exp\left(
tv\right)  $ so that $\nabla\dot{\sigma}\left(  t\right)  /dt=0.$ It then
follows that%
\begin{align}
\dot{G}\left(  t\right)   &  =df\left(  \dot{\sigma}\left(  t\right)  \right)
=df_{\sigma\left(  t\right)  }\left(  \dot{\sigma}\left(  t\right)  \right)
,\nonumber\\
\ddot{G}\left(  t\right)   &  =\frac{d}{dt}df_{\sigma\left(  t\right)
}\left(  \dot{\sigma}\left(  t\right)  \right)  =\left(  \nabla_{\dot{\sigma
}\left(  t\right)  }df\right)  \left(  \dot{\sigma}\left(  t\right)  \right)
+df_{\sigma\left(  t\right)  }\left(  \frac{\nabla}{dt}\dot{\sigma}\left(
t\right)  \right) \nonumber\\
&  =\left(  \nabla_{\dot{\sigma}\left(  t\right)  }df\right)  \left(
\dot{\sigma}\left(  t\right)  \right)  =\left(  \nabla df\right)  \left(
\dot{\sigma}\left(  t\right)  \otimes\dot{\sigma}\left(  t\right)  \right)
\nonumber\\
&  \vdots\nonumber\\
G^{\left(  k\right)  }\left(  t\right)   &  =\left(  \nabla^{k-1}df\right)
\left(  \dot{\sigma}\left(  t\right)  ^{\otimes k}\right)  =\left(
\nabla^{k-1}df\right)  \left(  \overset{k\text{ times}}{\overbrace{\dot
{\sigma}\left(  t\right)  \otimes\dots\otimes\dot{\sigma}\left(  t\right)  }%
}\right)  . \label{equ.6.2}%
\end{align}
Therefore we may conclude that
\begin{align}
f\left(  \exp_{m}\left(  v\right)  \right)   &  =G\left(  1\right)
=\sum_{k=0}^{n}\frac{1}{n!}G^{\left(  k\right)  }\left(  0\right) \nonumber\\
&  =f\left(  x\right)  +\sum_{k=1}^{n}\frac{1}{k!}\left(  \nabla
^{k-1}df\right)  \left(  v^{\otimes k}\right)  +\frac{1}{n!}\int_{0}%
^{1}\left(  \nabla^{n}df\right)  \left(  \dot{\sigma}\left(  t\right)
^{\otimes\left(  n+1\right)  }\right)  \left(  1-t\right)  ^{n}dt.
\label{equ.6.3}%
\end{align}
Letting $n=\exp_{m}\left(  v\right)  $ in this formula then gives the
following version of Taylor's theorem on a manifold.

\begin{theorem}
\label{the.6.1}Let $f\in C^{\infty}\left(  M\right)  $ and $m,n\in M$ with
$d_{g}\left(  m,n\right)  $ sufficiently small so that there exists a unique
$v\in T_{m}M$ such that $\left\vert v\right\vert _{g_{m}}\leq d\left(
m,n\right)  $ and $n=\exp_{m}\left(  v\right)  .$ Then we have%
\begin{align}
f\left(  n\right)   &  =f\left(  m\right)  +\sum_{k=1}^{n}\frac{1}{k!}\left(
\nabla^{k-1}df\right)  \left(  v^{\otimes k}\right)  +\frac{1}{n!}\int_{0}%
^{1}\left(  \nabla^{n}df\right)  \left(  \dot{\sigma}\left(  t\right)
^{\otimes\left(  n+1\right)  }\right)  \left(  1-t\right)  ^{n}dt\nonumber\\
&  =f\left(  m\right)  +\sum_{k=1}^{n}\frac{1}{k!}\left(  \nabla
^{k-1}df\right)  \left(  \left[  \exp_{m}^{-1}\left(  n\right)  \right]
^{\otimes k}\right)  +\frac{1}{n!}\int_{0}^{1}\left(  \nabla^{n}df\right)
\left(  \dot{\sigma}\left(  t\right)  ^{\otimes\left(  n+1\right)  }\right)
\left(  1-t\right)  ^{n}dt \label{equ.6.4}%
\end{align}
where $\sigma\left(  t\right)  =\exp_{m}\left(  tv\right)  .$ In particular
since $\left\vert \dot{\sigma}\left(  t\right)  \right\vert _{g}=\left\vert
v\right\vert _{g}=d_{g}\left(  m,n\right)  $ it follows that
\begin{equation}
f\left(  n\right)  =f\left(  m\right)  +\sum_{k=1}^{n}\frac{1}{k!}\left(
\nabla^{k-1}df\right)  \left(  \left[  \exp_{m}^{-1}\left(  n\right)  \right]
^{\otimes k}\right)  +O\left(  d\left(  m,n\right)  ^{n+1}\right)  .
\label{equ.6.5}%
\end{equation}

\end{theorem}

\begin{lemma}
\label{lem.6.2}Let $M$ be an embedded submanifold of $W=\mathbb{R}^{k}$ and
$P\left(  m\right)  :W\rightarrow T_{m}M$ be orthogonal projection onto the
tangent space. If $m,n\in M$ are close, then;

\begin{enumerate}
\item \label{Ite.29}$P\left(  m\right)  \left[  \exp_{m}^{-1}\left(  n\right)
-\left(  n-m\right)  \right]  =O\left(  \left\vert n-m\right\vert ^{3}\right)
.$

Moreover, $\exp_{m}^{-1}\left(  n\right)  -\left(  n-m\right)  =O\left(
\left\vert n-m\right\vert ^{2}\right)  $

\item $U^{\nabla}\left(  n,m\right)  =P\left(  m\right)  +dP\left(  \exp
_{m}^{-1}\left(  n\right)  \right)  +O\left(  \left\vert n-m\right\vert
^{2}\right)  =P\left(  n\right)  +O\left(  \left\vert n-m\right\vert
^{2}\right)  $

\item $P\left(  n\right)  -P\left(  m\right)  =dP\left(  \exp_{m}^{-1}\left(
n\right)  \right)  +O\left(  \left\vert n-m\right\vert ^{2}\right)  .$

Here $U^{\nabla}\left(  n,m\right)  $ refers to the parallelism defined in
Example \ref{exa.2.19}.
\end{enumerate}
\end{lemma}

\begin{proof}
We will denote $v:=\exp_{m}^{-1}\left(  n\right)  $ $\in T_{m}M$ and
$\sigma\left(  t\right)  =\exp_{m}\left(  tv\right)  $.

For 1, we have by Taylor expansion on manifolds (Theorem \ref{the.6.1}) that
\[
G\left(  n\right)  =G\left(  m\right)  +dG\left(  v\right)  +\frac{1}%
{2}\left(  \nabla dG\right)  \left(  v\otimes v\right)  +\frac{1}{2}\int%
_{0}^{1}\left(  \nabla^{2}dG\right)  \left(  \dot{\sigma}\left(  t\right)
^{\otimes3}\right)  \left(  1-t\right)  ^{2}dt
\]
where $G\in C^{\infty}(M,W).$ Letting $G(m)=m$ as a function into $W$, we have%
\[
n=m+\exp_{m}^{-1}\left(  n\right)  +\frac{1}{2}\left(  \nabla P\right)
\left(  v\otimes v\right)  +O\left(  \left\vert v\right\vert _{g}^{3}\right)
.
\]
Rearranging, we have%
\begin{equation}
\exp_{m}^{-1}\left(  n\right)  -\left(  n-m\right)  =-\frac{1}{2}\left(
\nabla P\right)  \left(  v\otimes v\right)  +O\left(  \left\vert v\right\vert
_{g}^{3}\right)  \label{equ.6.6}%
\end{equation}
so that%
\[
P\left(  m\right)  \left[  \exp_{m}^{-1}\left(  n\right)  -\left(  n-m\right)
\right]  =-\frac{1}{2}P(m)\left(  \nabla P\right)  \left(  v\otimes v\right)
+O\left(  \left\vert v\right\vert _{g}^{3}\right)  .
\]
Note that $\left(  \nabla P\right)  \left(  v\otimes v\right)  =dP\left(
v\right)  v=dP\left(  v\right)  P\left(  m\right)  v.$ Using the identities
$dPQ-PdQ=0$ and $dP=-dQ$, where $Q=I-P$, we get that $PdPP=0.$ Thus we have%
\[
P\left(  m\right)  \left[  \exp_{m}^{-1}\left(  n\right)  -\left(  n-m\right)
\right]  =O\left(  \left\vert v\right\vert ^{3}\right)  .
\]
Lastly, in a small neighborhood around $m$, $\left\vert v\right\vert
_{g}=\left\vert m-n\right\vert +o\left(  \left\vert m-n\right\vert \right)  $
so that
\[
P\left(  m\right)  \left[  \exp_{m}^{-1}\left(  n\right)  -\left(  n-m\right)
\right]  =O\left(  \left\vert n-m\right\vert ^{3}\right)
\]
The fact that $\exp_{m}^{-1}\left(  n\right)  -\left(  n-m\right)  =O\left(
\left\vert n-m\right\vert ^{2}\right)  $ is immediate from Eq. (\ref{equ.6.6}).

For 3, we use Taylor's theorem again this time with $G:M\longrightarrow
L\left(  W,W\right)  $ defined by $G\left(  n\right)  :=P\left(  n\right)  $
to see that%
\[
P\left(  n\right)  -P\left(  m\right)  =dP\left(  \exp_{m}^{-1}\left(
n\right)  \right)  +O\left(  \left\vert v\right\vert ^{2}\right)  .
\]
As before, this is equivalent to $P\left(  n\right)  -P\left(  m\right)
=dP\left(  \exp_{m}^{-1}\left(  n\right)  \right)  +O\left(  \left\vert
m-n\right\vert ^{2}\right)  .$

Lastly for 2, Taylor applied to $G_{m}:M\longrightarrow L\left(
T_{m}M,\mathbb{R}^{N}\right)  $ defined by $G_{m}\left(  n\right)  =U^{\nabla
}\left(  n,m\right)  $ gives
\[
U^{\nabla}\left(  n,m\right)  -P\left(  m\right)  =dG_{m}\left(  \exp_{m}%
^{-1}\left(  n\right)  \right)  +O\left(  \left\vert m-n\right\vert
^{2}\right)  .
\]
But
\begin{align*}
dG_{m}\left(  \exp_{m}^{-1}\left(  n\right)  \right)   &  =\frac{d}{dt}%
|_{0}U\left(  \sigma\left(  t\right)  ,m\right)  \\
&  =-dQ\left(  \dot{\sigma}\left(  t\right)  \right)  |_{0}\\
&  =-dQ\left(  \exp_{m}^{-1}\left(  n\right)  \right)  \\
&  =dP\left(  \exp_{m}^{-1}\left(  n\right)  \right)  .
\end{align*}
Thus we have
\[
U^{\nabla}\left(  n,m\right)  =P\left(  m\right)  +dP\left(  \exp_{m}%
^{-1}\left(  n\right)  \right)  +O\left(  \left\vert m-n\right\vert
^{2}\right)
\]
which is the first equality of 2. The second equality follows trivially from
this and 3.
\end{proof}

\subsection{Equivalence of Riemannian Metrics on Compact Sets\label{sub.6.2}}

\begin{proposition}
\label{pro.6.3}Let $\pi:E\rightarrow N$ be a real rank $d<\infty$ vector
bundle over a finite dimensional manifold $N.$ Further suppose that $E$ is
equipped with smoothly varying fiber inner product $g$ and let $S_{g}%
:=\left\{  \xi\in E:g\left(  \xi,\xi\right)  =1\right\}  $ be a sub-bundle of
$E.$ Then for any compact $K\subseteq N$, $\pi^{-1}\left(  K\right)  \cap
S_{g}$ is a compact sets.
\end{proposition}

\begin{proof}
We wish to show that every sequence $\left\{  \xi_{l}\right\}  _{l=1}^{\infty
}\subset\pi^{-1}\left(  K\right)  \cap S_{g}$ has a convergent subsequence.
Since $\left\{  \pi\left(  \xi_{l}\right)  \right\}  _{l=1}^{\infty}$ is a
sequence in $K,$ by passing to a subsequence if necessary we may assume that
$m:=\lim_{l\rightarrow\infty}\pi\left(  \xi_{l}\right)  $ exists in $K.$ By
passing to a further subsequence if necessary we may assume that $\left\{
\xi_{l}\right\}  _{l=1}^{\infty}\in\pi^{-1}\left(  K_{0}\right)  \cap S_{g}$
where $K_{0}$ is a compact neighborhood of $m$ which is contained in an open
neighborhood $U$ over which $E$ is trivializable and hence we may now assume
that $\pi^{-1}\left(  U\right)  =U\times\mathbb{R}^{d}$ and that $\xi
_{l}=\left(  n_{l},v_{l}\right)  $ where $\lim_{l\rightarrow\infty}n_{l}=m\in
K_{0}.$

Let $S^{d-1}$ denote the standard Euclidean unit sphere inside of
$\mathbb{R}^{d}.$ The function, $F:U\times S^{d-1}\rightarrow\left(
0,\infty\right)  $ defined by $F\left(  n,v\right)  =g\left(  \left(
n,v\right)  ,\left(  n,v\right)  \right)  $ is smooth and hence has a minimum
$c>0$ and a maximum, $C<\infty$ on the compact set, $K\times S^{d-1}.$
Therefore by a simple scaling argument we conclude that
\begin{equation}
c\left\vert v\right\vert ^{2}\leq g\left(  \left(  n,v\right)  ,\left(
n,v\right)  \right)  \leq C\left\vert v\right\vert ^{2}\text{ }\forall~n\in
K\text{ and }v\in\mathbb{R}^{d}. \label{equ.6.7}%
\end{equation}
From the lower bound in Inequality (\ref{equ.6.7}) and the assumption that
$1=g\left(  \xi_{l},\xi_{l}\right)  $ it follows that $\left\vert
v_{l}\right\vert _{\mathbb{R}^{d}}\leq1/\sqrt{c}$ for all $l$ and therefore
has a convergent sub-sequence $\left\{  v_{l_{k}}\right\}  _{k=1}^{\infty}.$
This completes the proof as $\left\{  \xi_{l_{k}}=\left(  n_{l_{k}},v_{l_{k}%
}\right)  \right\}  _{k=1}^{\infty}$ is convergent as well.
\end{proof}

\begin{corollary}
\label{cor.6.4}If $g,\tilde{g}$ are two Riemannian metrics on $TM$,
$K\subseteq M$ is compact, then there exists $0<c_{K},C_{K}<\infty$ such that%
\begin{equation}
c_{K}\left\vert v\right\vert _{\tilde{g}_{m}}\leq\left\vert v\right\vert
_{g_{m}}\leq C_{K}\left\vert v\right\vert _{\tilde{g}_{m}}~\forall~v\in
\pi^{-1}\left(  K\right)  . \label{equ.6.8}%
\end{equation}
In other words, all Riemannian metrics are equivalent when restricted to
compact subsets, $K\subset M.$
\end{corollary}

\begin{proof}
The function, $F:TM\rightarrow\lbrack0,\infty),$ defined by $F\left(
v\right)  :=g\left(  v,v\right)  $ is smooth and positive when restricted to
$S_{\tilde{g}}\cap\pi^{-1}\left(  K\right)  $ which is compact by Proposition
\ref{pro.6.3}. Therefore there exists $0<c_{K}<C_{K}<\infty$ such that
$c_{K}^{2}\leq g\left(  v,v\right)  \leq C_{K}^{2}$ for all $v\in S_{\tilde
{g}}\cap\pi^{-1}\left(  K\right)  $ from which Inequality (\ref{equ.6.8})
follows by a simple scaling argument.
\end{proof}

\subsection{Covariant Derivatives on Euclidean Space\label{sub.6.3}}

On $\mathbb{R}^{d}$ every covariant derivative takes the form $\nabla_{\left(
x,v\right)  }=\partial_{v}+A_{x}\left\langle v\right\rangle $ where
$A:\mathbb{R}^{d}\rightarrow L\left(  \mathbb{R}^{d},L\left(  \mathbb{R}%
^{d},\mathbb{R}^{d}\right)  \right)  .$ If $\sigma_{x}^{v}\left(  t\right)
=\exp_{x}\left(  tv\right)  $ where $\exp=\exp^{\nabla},$ we have by
definition
\begin{align*}
\partial_{\dot{\sigma}_{x}^{v}\left(  t\right)  }\dot{\sigma}_{x}^{v} &
=-A_{\sigma_{x}^{v}\left(  t\right)  }\left\langle \dot{\sigma}_{x}^{v}\left(
t\right)  \right\rangle \dot{\sigma}_{x}^{v}\left(  t\right)  \\
\dot{\sigma}_{x}^{v}\left(  0\right)   &  =v\\
\sigma_{x}^{v}\left(  0\right)   &  =x
\end{align*}
In particular if $f_{x}=\exp_{x}\left(  \cdot\right)  $ plugging in at $t=0$
we get%
\[
f_{x}^{\prime\prime}\left(  0\right)  \left[  v\otimes v\right]
=-A_{x}\left\langle v\right\rangle v.
\]
Now if we denote $G_{x}:=\exp_{x}^{-1}\left(  \cdot\right)  $ and by
differentiating $f_{x}\circ G_{x}$ twice, we get that%
\[
G_{x}^{\prime\prime}\left(  x\right)  \left[  v\otimes v\right]
=A_{x}\left\langle v\right\rangle v.
\]
Indeed we have%
\begin{align*}
0 &  =\left(  f_{x}\circ G_{x}\right)  ^{\prime\prime}\left(  x\right)  \\
&  =\left[  f_{x}^{\prime}\left(  G_{x}\left(  x\right)  \right)
G_{x}^{\prime}\left(  x\right)  \right]  ^{\prime}\\
&  =f_{x}^{\prime\prime}\left(  G_{x}\left(  x\right)  \right)  \left[
G_{x}^{\prime}\left(  x\right)  \otimes G_{x}^{\prime}\left(  x\right)
\right]  +f_{x}^{\prime}\left(  G_{x}\left(  x\right)  \right)  G_{x}%
^{\prime\prime}\left(  x\right)  .
\end{align*}
Since $G_{x}\left(  x\right)  =0$, $G_{x}^{\prime}\left(  x\right)  =I$, and
$f_{x}^{\prime}\left(  0\right)  =I$ we have
\[
f_{x}^{\prime\prime}\left(  0\right)  =-G_{x}^{\prime\prime}\left(  x\right)
.
\]
Parallel translation $U^{\nabla}\left(  \sigma_{x}^{v}\left(  t\right)
,x\right)  $ solves%
\begin{align*}
\frac{d}{dt}U^{\nabla}\left(  \sigma_{x}^{v}\left(  t\right)  ,x\right)   &
=-A_{\sigma_{x}^{v}\left(  t\right)  }\left\langle \dot{\sigma}_{x}^{v}\left(
t\right)  \right\rangle U^{\nabla}\left(  \sigma_{x}^{v}\left(  t\right)
,x\right)  \\
U^{\nabla}\left(  x,x\right)   &  =I
\end{align*}
Again, using $t=0$ we have that if $\tilde{G}_{x}=U^{\nabla}\left(
\cdot,x\right)  $ then
\[
\tilde{G}_{x}^{\prime}\left(  x\right)  v=-A_{x}\left\langle v\right\rangle .
\]

To summarize, we have
\begin{equation}
\left(  \exp_{x}^{-1}\right)  ^{\prime\prime}\left(  x\right)  \left[
v\otimes v\right]  =A_{x}\left\langle v\right\rangle v\label{equ.6.9}%
\end{equation}
and%
\[
\left(  U^{\nabla}\left(  \cdot,x\right)  \right)  ^{\prime}\left(  x\right)
v=-A_{x}\left\langle v\right\rangle .
\]
Since $\left(  \exp_{x}^{-1}\right)  ^{\prime\prime}\left(  x\right)  $ is
symmetric, we have that%
\begin{align}
\left(  \exp_{x}^{-1}\right)  ^{\prime\prime}\left(  x\right)  \left[
v\otimes w\right]   &  =\frac{1}{2}\left(  \exp_{x}^{-1}\right)
^{\prime\prime}\left(  x\right)  \left(  v\otimes w+w\otimes v\right)
+\frac{1}{2}\left(  \exp_{x}^{-1}\right)  ^{\prime\prime}\left(  x\right)
\left(  v\otimes w-w\otimes v\right)  \nonumber\\
&  =\frac{1}{2}\left(  \exp_{x}^{-1}\right)  ^{\prime\prime}\left(  x\right)
\left(  v\otimes w+w\otimes v\right)  \nonumber\\
&  =\frac{1}{2}A_{x}\left(  v\otimes w+w\otimes v\right)  \nonumber\\
&  =\frac{1}{2}\left(  A_{x}\left\langle v\right\rangle w+A_{x}\left\langle
w\right\rangle v\right)  \label{equ.6.10}%
\end{align}
Another way of saying this is that $\left(  \exp_{x}^{-1}\right)
^{\prime\prime}\left(  x\right)  $ equals the symmetric part of $A_{x}$. By
using this fact and Taylor's theorem, we get the following result.

\begin{lemma}
\label{lem.6.5}If $\nabla_{\left(  x,v\right)  }=\partial_{v}+A_{x}%
\left\langle v\right\rangle $ is a covariant derivative on $\mathbb{R}^{d}$,
then
\[
\left(  \exp_{x}^{\nabla}\right)  ^{-1}\left(  y\right)  -\left(  y-x\right)
-\frac{1}{2}A_{x}\left\langle y-x\right\rangle \left\langle y-x\right\rangle
=O\left(  \left\vert y-x\right\vert ^{3}\right)
\]%
\begin{equation}
U^{\nabla}\left(  y,x\right)  -I+A_{x}\left\langle y-x\right\rangle =O\left(
\left\vert y-x\right\vert ^{2}\right)  \label{equ.6.11}%
\end{equation}

where $\left\vert x-y\right\vert $ is small enough for these terms to make sense.
\end{lemma}

\begin{corollary}
\label{cor.6.6}If $\nabla_{\left(  x,v\right)  }=\partial_{v}+A_{x}%
\left\langle v\right\rangle $ is a covariant derivative on $\mathbb{R}^{d}$,
then
\[
U^{\nabla}\left(  y,x\right)  -I-A_{y}\left\langle x-y\right\rangle =O\left(
\left\vert y-x\right\vert ^{2}\right)
\]
where $\left\vert x-y\right\vert $ is small enough for these terms to make
sense. In particular, we have%
\[
\left(  U^{\nabla}\left(  x,\cdot\right)  \right)  ^{\prime}\left(  x\right)
v=A_{x}\left\langle v\right\rangle
\]

\end{corollary}

\begin{proof}
This is immediate after expanding $A_{\left(  \cdot\right)  }$ about $x$ in
the direction $y-x$ in Eq. (\ref{equ.6.11}) with Taylor's theorem.
\end{proof}

\subsection{Second order gauge inequality does not imply second order chart
inequality.}

\begin{example}
\label{exa.6.7}Let $x_{s}$ and $y_{s}$ be the $C\left(  \left[  0,2\right]
,\mathbb{R}\right)  $ paths defined by
\end{example}

\[
y_{s}=x_{s}=\left\{
\begin{array}
[c]{ccc}%
0 & \text{if} & 0\leq s\leq1\\
s^{1/p}-1 & \text{if} & 1\leq s\leq2
\end{array}
\right.
\]
and the control $\omega\left(  s,t\right)  $ be defined by%
\[
\omega\left(  s,t\right)  =\left\{
\begin{array}
[c]{ccc}%
0 & \text{if} & t\leq1\\
t-\left(  s\vee1\right)  & \text{if} & t\geq1
\end{array}
\right.  .
\]
Then it is easy to check that
\[
\left\vert x_{s,t}\right\vert \leq\omega\left(  s,t\right)  ^{1/p}%
\]

Let
\[
y_{s}^{\dag}=\left\{
\begin{array}
[c]{ccc}%
2-2s & \text{if} & 0\leq s\leq\frac{1}{2}\\
1 & \text{else} & \frac{1}{2}\leq s\leq2
\end{array}
\right.  .
\]
Then if $t-s\leq1/2$, $y_{s,t}-y_{s}^{\dag}x_{s,t}=0$ so that $\left(
y,y^{\dag}\right)  $ satisfies Inequality (\ref{equ.2.22}) with $\delta=1/2$
and $\psi\left(  x,y\right)  =y-x$. On the other hand if $s=0$ and
$t=1+\epsilon$, then%
\[
y_{s,t}-y_{s}^{\dag}x_{s,t}=\epsilon^{1/p}-2\epsilon^{1/p}=-\epsilon^{1/p}.
\]
Thus%
\[
\frac{\left\vert y_{0,1+\epsilon}-y_{0}^{\dag}x_{0,1+\epsilon}\right\vert
}{\omega\left(  0,1+\epsilon\right)  ^{2/p}}=\frac{1}{\epsilon^{1/p}}%
\]
so that $\left(  y,y^{\dag}\right)  $ does not satisfy Inequality
(\ref{equ.2.24}) with the identity chart.

\subsection{Rough Differential Equation Results in Euclidean
Space\label{sub.6.5}}

The following lemma (which is Corollary 2.17 in \cite{CDL13} and was proved
using Theorem 10.14 of \cite{FV}) proves useful in the manifold case.

\begin{lemma}
\label{lem.6.8}Let $U\subseteq\mathbb{R}^{d}$ be an open set and $U_{1}$ be a
precompact open set whose closure is contained in $U.$ There exists $\ $a
$\delta>0$ such that for all $\left(  \bar{z}_{0},t_{0}\right)  \in
U_{1}\times\left[  0,T\right]  $, the rough differential equation%
\[
d\mathbf{z}_{t}=F_{d\mathbf{X}_{t}}\left(  z_{t}\right)  \quad\text{with\quad
}z_{t_{0}}=\bar{z}_{0}%
\]
has a unique solution $\mathbf{z\in}CRP_{\mathbf{X}}\left(  \mathbb{R}%
^{d}\right)  $ which is defined on $\left[  t_{0,}t_{0}+\delta\wedge T\right]
$ with $z_{t}\in U$ for all $t\in\left[  t_{0,}t_{0}+\delta\wedge T\right]  .$
\end{lemma}

We now state an equivalent condition for the path $\mathbf{z}$ to solve Eq.
(\ref{equ.2.11}).

\begin{theorem}
\label{the.6.9}Let $U\subseteq\mathbb{R}^{d}$ be open such and $\mathbf{z}%
=\left(  z,z^{\dag}\right)  \in CRP_{\mathbf{X}}\left(  \mathbb{R}^{d}\right)
$ defined on $I_{0}$ such that $z\left(  I_{0}\right)  \subseteq U$. Then
$\mathbf{z}$ solves Eq. (\ref{equ.2.11}) if and only if $z_{s}^{\dag}%
=F_{\cdot}\left(  z_{s}\right)  $ and for every $\left[  a,b\right]  \subseteq
I_{0},$ Banach space $V$, and $\alpha\in\Omega^{1}\left(  U,V\right)  $, the
approximation
\[
\int_{s}^{t}\alpha\left(  d\mathbf{z}\right)  \underset{^{3}}{\approx}%
\alpha_{z_{s}}\left(  F_{x_{s,t}}\left(  z_{s}\right)  \right)  +\left(
\partial_{F_{w}\left(  z_{s}\right)  }\left[  \alpha\circ F_{\tilde{w}%
}\right]  \right)  \left(  z_{s}\right)  |_{w\otimes\tilde{w}=\mathbb{X}%
_{s,t}}%
\]
holds.
\end{theorem}

\begin{proof}
This is proved in \cite{CDL13} [Theorem 4.5 by letting $M=U$] but included
here for completeness. To prove the \textquotedblleft if\textquotedblright%
\ direction, it suffices to let $\alpha=d\left(  I_{U}\right)  $ and notice
that%
\[
\int_{s}^{t}d\left(  I_{U}\right)  \left(  d\mathbf{z}\right)  =z_{t}-z_{s}%
\]
by Theorem \ref{the.4.13} and that $d\left(  I_{U}\right)  _{u}\left(
\tilde{u}\right)  =\tilde{u}$ so that
\[
d\left(  I_{U}\right)  _{z_{s}}\left(  F_{x_{s,t}}\left(  z_{s}\right)
\right)  =F_{x_{s,t}}\left(  z_{s}\right)
\]
and%
\[
\left(  \partial_{F_{w}\left(  z_{s}\right)  }\left[  d\left(  I_{U}\right)
\circ F_{\tilde{w}}\right]  \right)  \left(  z_{s}\right)  =\left(
\partial_{F_{w}\left(  z_{s}\right)  }F_{\tilde{w}}\right)  \left(
z_{s}\right)  .
\]

To prove the \textquotedblleft only if\textquotedblright\ direction, by
definition we have%
\[
z_{s,t}\underset{^{3}}{\approx}F_{x_{s,t}}\left(  z_{s}\right)  +\left(
\partial_{F_{w}\left(  z_{s}\right)  }F_{\tilde{w}}\right)  \left(
z_{s}\right)  |_{w\otimes\tilde{w}=\mathbb{X}_{s,t}}%
\]
and
\[
\int_{s}^{t}\alpha\left(  d\mathbf{z}\right)  \underset{^{3}}{\approx}%
\alpha_{z_{s}}\left(  z_{s,t}\right)  +\alpha_{z_{s}}^{\prime}\left(
F_{\cdot}\left(  z_{s}\right)  \otimes F_{\cdot}\left(  z_{s}\right)
\mathbb{X}_{s,t}\right)  .
\]
Combining these approximations, we have%
\begin{align*}
\int_{s}^{t}\alpha\left(  d\mathbf{z}\right)   &  \underset{^{3}}{\approx
}\alpha_{z_{s}}\left(  z_{s,t}\right)  +\alpha_{z_{s}}^{\prime}\left(
F_{\cdot}\left(  z_{s}\right)  \otimes F_{\cdot}\left(  z_{s}\right)
\mathbb{X}_{s,t}\right) \\
&  \underset{^{3}}{\approx}\alpha_{z_{s}}\left(  F_{x_{s,t}}\left(
z_{s}\right)  +\left(  \partial_{F_{w}\left(  z_{s}\right)  }F_{\tilde{w}%
}\right)  \left(  z_{s}\right)  \right)  +\alpha_{z_{s}}^{\prime}\left(
F_{w}\left(  z_{s}\right)  \otimes F_{\tilde{w}}\left(  z_{s}\right)  \right)
|_{w\otimes\tilde{w}=\mathbb{X}_{s,t}}\\
&  =\alpha_{z_{s}}\left(  F_{x_{s,t}}\left(  z_{s}\right)  \right)  +\left(
\partial_{F_{w}\left(  z_{s}\right)  }\left[  \alpha\circ F_{\tilde{w}%
}\right]  \right)  \left(  z_{s}\right)  |_{w\otimes\tilde{w}=\mathbb{X}%
_{s,t}}%
\end{align*}
where the last equality follows from the calculation%
\begin{align*}
\left(  \partial_{F_{w}\left(  z_{s}\right)  }\left[  \alpha\circ F_{\tilde
{w}}\right]  \right)  \left(  z_{s}\right)   &  =\left(  \partial
_{F_{w}\left(  z_{s}\right)  }\left[  \alpha_{z_{s}}\circ F_{\tilde{w}}\left(
\cdot\right)  \right]  \right)  \left(  z_{s}\right)  +\left(  \partial
_{F_{w}\left(  z_{s}\right)  }\alpha_{\left(  \cdot\right)  }\circ
F_{\tilde{w}}\left(  z_{s}\right)  \right)  \left(  z_{s}\right) \\
&  =\alpha_{z_{s}}\left(  \left(  \partial_{F_{w}\left(  z_{s}\right)
}F_{\tilde{w}}\right)  \left(  z_{s}\right)  \right)  +\alpha_{z_{s}}^{\prime
}\left(  F_{w}\left(  z_{s}\right)  \otimes F_{\tilde{w}}\left(  z_{s}\right)
\right)
\end{align*}

\end{proof}

Theorem \ref{the.6.11} below is useful in showing that a solution to an RDE in
the flat case satisfies our manifold Definition \ref{def.5.2}. Let $U$ and
$\tilde{U}$ be open sets for the remainder of this subsection.

\begin{definition}
\label{def.6.10}Let $f:U\subseteq\mathbb{R}^{d}\rightarrow\tilde{U}%
\subseteq\mathbb{R}^{\tilde{d}}$ be a smooth map. Let $F:U\rightarrow L\left(
W,\mathbb{R}^{d}\right)  $ and $\tilde{F}:\tilde{U}\rightarrow L\left(
W,\mathbb{R}^{\tilde{d}}\right)  $ be smooth. We say $F$ and $\tilde{F}$ are
$f$ $-$ \textbf{related dynamical systems} if%
\[
f^{\prime}\left(  x\right)  F_{w}\left(  x\right)  =\tilde{F}_{w}\circ
f\left(  x\right)  \text{ for all }w\in W.
\]

\end{definition}

\begin{theorem}
\label{the.6.11}Suppose $f:U\subseteq\mathbb{R}^{d}\rightarrow\tilde
{U}\subseteq\mathbb{R}^{\tilde{d}}$ is a smooth map and let $F:U\rightarrow
L\left(  W,\mathbb{R}^{d}\right)  $ and $\tilde{F}:\tilde{U}\rightarrow
L\left(  W,\mathbb{R}^{\tilde{d}}\right)  $ be $f-$related dynamical systems.
If $\mathbf{z}$ solves
\[
d\mathbf{z}_{t}=F_{d\mathbf{X}_{t}}\left(  z_{t}\right)
\]
with initial condition $z_{0}=\bar{z}_{0}$, then $\mathbf{\tilde{z}}%
_{t}:=\left(  \tilde{z}_{t},\tilde{z}_{s}^{\dag}\right)  :=f_{\ast}%
\mathbf{z}_{t}$ solves%
\[
d\mathbf{\tilde{z}}_{t}=\tilde{F}_{d\mathbf{X}_{t}}\left(  \tilde{z}%
_{t}\right)
\]
with initial condition $\tilde{z}_{0}=f\left(  \bar{z}_{0}\right)  .$
\end{theorem}

\begin{proof}
We have by letting $\alpha:=df$ in Theorem \ref{the.6.9}
\begin{align*}
\tilde{z}_{s,t}  &  =f\left(  z_{t}\right)  -f\left(  z_{s}\right) \\
&  \underset{^{3}}{\approx}f^{\prime}\left(  z_{s}\right)  F_{x_{s,t}}\left(
z_{s}\right)  +\partial_{F_{w}\left(  z_{s}\right)  }\left[  f^{\prime}\left(
\cdot\right)  F_{\tilde{w}}\left(  \cdot\right)  \right]  \left(
z_{s}\right)  |_{w\otimes\tilde{w}=\mathbb{X}_{s,t}}\\
&  \underset{^{3}}{\approx}\tilde{F}_{x_{s,t}}\left(  \tilde{z}_{s}\right)
+\left(  \partial_{F_{w}\left(  z_{s}\right)  }\tilde{F}_{\tilde{w}}\circ
f\right)  \left(  z_{s}\right)  |_{w\otimes\tilde{w}=\mathbb{X}_{s,t}}\\
&  \underset{^{3}}{\approx}\tilde{F}_{x_{s,t}}\left(  \tilde{z}_{s}\right)
+\tilde{F}_{\tilde{w}}^{\prime}\left(  f\left(  z_{s}\right)  \right)
f^{\prime}\left(  z_{s}\right)  F_{w}\left(  z_{s}\right)  |_{w\otimes
\tilde{w}=\mathbb{X}_{s,t}}\\
&  \underset{^{3}}{\approx}\tilde{F}_{x_{s,t}}\left(  \tilde{z}_{s}\right)
+\tilde{F}_{\tilde{w}}^{\prime}\left(  f\left(  z_{s}\right)  \right)
\tilde{F}_{w}\circ f\left(  z_{s}\right)  |_{w\otimes\tilde{w}=\mathbb{X}%
_{s,t}}\\
&  \underset{^{3}}{\approx}\tilde{F}_{x_{s,t}}\left(  \tilde{z}_{s}\right)
+\left(  \partial_{\tilde{F}_{w}\left(  \tilde{z}_{s}\right)  }\tilde
{F}_{\tilde{w}}\right)  \left(  \tilde{z}_{s}\right)  |_{w\otimes\tilde
{w}=\mathbb{X}_{s,t}}.
\end{align*}
Additionally
\[
\tilde{z}_{t}^{\dag}=f^{\prime}\left(  z_{t}\right)  z_{t}^{\dag}=f^{\prime
}\left(  z_{t}\right)  F_{\left(  \cdot\right)  }\left(  z_{t}\right)
=\tilde{F}_{\left(  \cdot\right)  }\left(  \tilde{z}_{t}\right)  .
\]

\end{proof}

\begin{corollary}
\label{cor.6.12}Let $\phi:U\subseteq\mathbb{R}^{d}\rightarrow\tilde
{U}\subseteq\mathbb{R}^{d}$ be a diffeomorphism with $\phi\left(  z\left(
I_{0}\right)  \right)  \subseteq U$. Then $\mathbf{z}$ on $I_{0}$ solves
\[
d\mathbf{z}_{t}=F_{d\mathbf{X}_{t}}\left(  z_{t}\right)
\]
with initial condition $z_{0}=\bar{z}_{0}$ if and only if $\mathbf{\tilde{z}%
}:=\phi_{\ast}\mathbf{z}$ on $I_{0}$ solves
\[
d\mathbf{\tilde{z}}_{t}=F_{d\mathbf{X}_{t}}^{\phi}\left(  \tilde{z}%
_{t}\right)
\]
with initial condition $\tilde{z}_{0}=\phi\left(  \bar{z}_{0}\right)  $ where
$F^{\phi}:=d\phi\circ\left(  F\circ\phi^{-1}\right)  .$
\end{corollary}

\begin{proof}
This follows from Theorem \ref{the.6.11} by seeing that $F$ is $\phi-$related
to $F^{\phi}.$
\end{proof}

This last lemma helps patch solutions in the manifold case.

\begin{lemma}
\label{lem.6.13}Let $z\in C\left(  \left[  0,T\right]  ,V\right)  $ and let
$0=t_{0}<t_{1}<\ldots<t_{l}=T$ be a partition of $\left[  0,T\right]  .$ If
\begin{equation}
z_{s,t}\underset{^{3}}{\approx}F_{x_{s,t}}\left(  z_{s}\right)  +\left(
\partial_{F_{w}\left(  z_{s}\right)  }F_{\tilde{w}}\right)  \left(
z_{s}\right)  |_{w\otimes\tilde{w}=\mathbb{X}_{s,t}}\label{equ.6.12}%
\end{equation}
holds for all $t_{i}\leq s\leq t\leq t_{i+1}$ and $0\leq i<l$ then Eq.
(\ref{equ.6.12}) holds for $0\leq s\leq t\leq T$.

In particular, if $\mathbf{z}_{t}$ solves $d\mathbf{z}_{t}=F_{d\mathbf{X}_{t}%
}\left(  z_{t}\right)  $ with $z_{0}=\bar{z}_{0}$ on $\left[  0,\tau\right]  $
and $\mathbf{\tilde{z}}_{t}$ solves $d\mathbf{\tilde{z}}_{t}=F_{d\mathbf{X}%
_{t}}\left(  \tilde{z}_{t}\right)  $ with $\tilde{z}_{\tau}=z_{\tau}$ on
$\left[  \tau,T\right]  $, then the concatenation of $\mathbf{z}_{t}$ and
$\mathbf{\tilde{z}}_{t}$ in the sense of Lemma \ref{lem.2.6} solves
$d\mathbf{z}_{t}=F_{d\mathbf{X}_{t}}\left(  z_{t}\right)  $ with $z_{0}%
=\bar{z}_{0}$ on $\left[  0,T\right]  $.
\end{lemma}

\begin{proof}
This proof is identical from \cite{CDL13} [Lemma A.2], adapted here with
different notation. We will only prove it in the case of two subintervals.
First note that
\[
F_{w}\left(  y\right)  =F_{w}\left(  x\right)  +F_{w}^{\prime}\left(
x\right)  \left(  y-x\right)  +O\left(  \left\vert w\right\vert \left\vert
y-x\right\vert ^{2}\right)
\]
and%
\[
\left(  \partial_{F_{w}\left(  y\right)  }F_{\tilde{w}}\right)  \left(
y\right)  =\left(  \partial_{F_{w}\left(  x\right)  }F_{\tilde{w}}\right)
\left(  x\right)  +O\left(  \left\vert w\right\vert \left\vert \tilde
{w}\right\vert \left\vert y-x\right\vert \right)
\]
by Taylor's theorem and the fact that $w\rightarrow F_{w}$ is linear. Using
these facts, we have%
\begin{align*}
z_{s,t}  &  =z_{s,\tau}+z_{\tau,t}\\
&  \underset{^{3}}{\approx}F_{x_{s,\tau}}\left(  z_{s}\right)  +\left(
\partial_{F_{w}\left(  z_{s}\right)  }F_{\tilde{w}}\right)  \left(
z_{s}\right)  |_{w\otimes\tilde{w}=\mathbb{X}_{s,\tau}}+F_{x_{\tau,t}}\left(
z_{\tau}\right)  +\left(  \partial_{F_{w}\left(  z_{\tau}\right)  }%
F_{\tilde{w}}\right)  \left(  z_{\tau}\right)  |_{w\otimes\tilde{w}%
=\mathbb{X}_{\tau,t}}\\
&  \underset{^{3}}{\approx}F_{x_{s,t}}\left(  z_{s}\right)  +F_{x_{\tau,t}%
}^{\prime}\left(  z_{s}\right)  \left(  z_{s,\tau}\right)  +\left(
\partial_{F_{w}\left(  z_{s}\right)  }F_{\tilde{w}}\right)  \left(
z_{s}\right)  |_{w\otimes\tilde{w}=\mathbb{X}_{s,\tau}}+\left(  \partial
_{F_{w}\left(  z_{s}\right)  }F_{\tilde{w}}\right)  \left(  z_{s}\right)
|_{w\otimes\tilde{w}=\mathbb{X}_{\tau,t}}\\
&  \underset{^{3}}{\approx}F_{x_{s,t}}\left(  z_{s}\right)  +F_{x_{\tau,t}%
}^{\prime}\left(  z_{s}\right)  \left(  F_{x_{s,\tau}}\left(  z_{s}\right)
\right)  +\left(  \partial_{F_{w}\left(  z_{s}\right)  }F_{\tilde{w}}\right)
\left(  z_{s}\right)  |_{w\otimes\tilde{w}=\mathbb{X}_{s,\tau}+\mathbb{X}%
_{\tau,t}}\\
&  =F_{x_{s,t}}\left(  z_{s}\right)  +\left(  \partial_{F_{w}\left(
z_{s}\right)  }F_{\tilde{w}}\right)  \left(  z_{s}\right)  |_{w\otimes
\tilde{w}=\mathbb{X}_{s,\tau}+\mathbb{X}_{\tau,t}+x_{s,\tau}\otimes x_{\tau
,t}}\\
&  =F_{x_{s,t}}\left(  z_{s}\right)  +\left(  \partial_{F_{w}\left(
z_{s}\right)  }F_{\tilde{w}}\right)  \left(  z_{s}\right)  |_{w\otimes
\tilde{w}=\mathbb{X}_{s,t}}.
\end{align*}

\end{proof}

\bibliographystyle{amsplain}

\bibliography{CRPI}
\newpage
\section*{Contact Information}
\noindent
Bruce K. Driver \newline
Department of Mathematics, 0112 \newline
University of California, San Diego \newline
9500 Gilman Drive \newline
La Jolla, California 92093-0112 USA \newline \newline
\emph{E-mail address}: \texttt{bdriver@ucsd.edu}
\newline \newline \newline 
Jeremy S. Semko \newline
Department of Mathematics, 0112 \newline
University of California, San Diego \newline
9500 Gilman Drive \newline
La Jolla, California 92093-0112 USA \newline \newline
\emph{E-mail address}: \texttt{jsemko@ucsd.edu}


\section{Smooth Horizontal Lifts for Principal Bundles\label{sec.7}}

This section is provided as a smooth warm-up for things to come in the rough
category. Before we get down to business let us recall some basic facts about
connections on principal bundles.

\subsection{Connections\label{sub.7.1}}

For the sake of motivation, let $E\rightarrow M$ be a vector bundle with fiber
$V$ and let $G=\operatorname*{Aut}\left(  V\right)  .$ Further let $P$ be the
associated principal bundle to $E,$ i.e. $P\rightarrow M$ is a fiber bundle
with fibers, $P_{m}:=\operatorname*{Aut}\left(  V,E_{m}\right)  $ for each
$m\in M.$ Notice that these fibers are diffeomorphic to $G,$ $G$ acts on the
right of $P$ by composition, so that $u_{m}g=u_{m}\circ g$ for all $u_{m}\in
P_{m}$ and $g\in G.$

If $E$ is equipped with a covariant derivative, $\nabla,$ we may construct a
$\mathfrak{g:}=\operatorname*{Lie}\left(  G\right)  =\operatorname*{End}%
\left(  V\right)  $ -- valued one-form $\omega=\omega^{\nabla}$ on $P$ by
\[
\omega^{\nabla}\left(  \dot{u}\left(  0\right)  \right)  :=u\left(  0\right)
^{-1}\frac{\nabla}{dt}|_{t=0}u\left(  t\right)
\]
for all smooth paths in $P.$ This one-form has the following properties;

\begin{enumerate}
\item If $u\left(  t\right)  =u_{0}e^{tA}$ for some $u_{o}\in P$ and
$A\in\mathfrak{g}$ then
\[
\omega^{\nabla}\left(  \dot{u}\left(  0\right)  \right)  :=u_{o}^{-1}%
\frac{\nabla}{dt}|_{t=0}u_{0}e^{tA}=u_{o}^{-1}u_{0}A=A.
\]
In the future we denote $\dot{u}\left(  0\right)  $ in this example by
$\tilde{A}\left(  u_{0}\right)  $ or simply by $u_{0}A.$

\item If $g\in G,$ then
\begin{align*}
\left(  R_{g}^{\ast}\omega^{\nabla}\right)  \left(  \dot{u}\left(  0\right)
\right)   &  =\omega^{\nabla}\left(  \left(  R_{g}\right)  _{\ast}\dot
{u}\left(  0\right)  \right)  =\omega^{\nabla}\left(  \frac{d}{dt}|_{0}\left[
u\left(  t\right)  g\right]  \right) \\
&  =\left[  u\left(  0\right)  g\right]  ^{-1}\frac{\nabla}{dt}|_{t=0}\left[
u\left(  t\right)  g\right]  =g^{-1}\left[  u\left(  0\right)  ^{-1}%
\frac{\nabla}{dt}|_{t=0}u\left(  t\right)  \right]  g\\
&  =Ad_{g^{-1}}\left[  \omega^{\nabla}\left(  \dot{u}\left(  0\right)
\right)  \right]  .
\end{align*}

\end{enumerate}

This shows that every covariant derivatives gives rise to connection one-form
$\left(  \omega\right)  $ as in Definition \ref{def.7.1} on $P\left(
E\right)  .$

\begin{definition}
\label{def.7.1}Let $G$ be a Lie group, $P\rightarrow M$ be a principal bundle
with structure group, $G,$ and $\mathfrak{g}:=\operatorname*{Lie}\left(
G\right)  :=T_{e}G.$ [In the future we write $G\rightarrow P\overset{\pi
}{\rightarrow}M$ to denote that $P$ is a principal bundle over $M$ with
structure group $G$ and projection map $\pi.]$ A $\mathfrak{g}$ -- valued
one-form, $\omega,$ on $P$ is a connection one-form provided;

\begin{enumerate}
\item $\omega\left(  \tilde{A}\left(  \cdot\right)  \right)  =A$ for all
$A\in\mathfrak{g}$ where $\tilde{A}\left(  u_{0}\right)  :=\frac{d}{dt}%
|_{0}u_{o}e^{tA}$ -- which is a typical \textquotedblleft
vertical\textquotedblright\ vector in $TP.$

\item $R_{g}^{\ast}\omega=Ad_{g^{-1}}\omega$ for all $g\in G,$ i.e.
\[
\left(  R_{g}^{\ast}\omega\right)  \left(  \xi_{u}\right)  =Ad_{g^{-1}}%
\omega\left(  \xi_{u}\right)  ~\forall~g\in G\text{ and }\xi_{u}\in TP.
\]

\end{enumerate}
\end{definition}

\begin{example}
[Trivial Bundle Case]\label{exa.7.2}Suppose that $P=M\times G$ is a trivial
principal bundle and $\omega$ is a connection form on $P.$ In this case we may
associate to $\omega$ a one-form on $M$ with values in $\mathfrak{g}$ by
setting
\[
\Gamma\left(  v_{m}\right)  :=\omega\left(  \left(  v_{m},0_{e}\right)
\right)  \in\mathfrak{g~\forall~}v_{m}\in T_{m}M.
\]
Furthermore we may reconstruct $\omega$ from $\Gamma$ as follows. Let
$v_{m}\in T_{m}M,$ $A\in\mathfrak{g},$ and $g\in G$ so that $\tilde{A}\left(
g\right)  $ is the generic element of $T_{g}G,$ then we must have
\begin{align*}
\omega\left(  v_{m},\tilde{A}\left(  g\right)  \right)   &  =\omega\left(
v_{m},0_{g}\right)  +\omega\left(  0_{m},\tilde{A}\left(  g\right)  \right) \\
&  =\omega\left(  \left(  R_{g}\right)  _{\ast}\left(  v_{m},0_{e}\right)
\right)  +A\\
&  =\left(  R_{g}^{\ast}\omega\right)  \left(  \left(  v_{m},0_{e}\right)
\right)  +A\\
&  =Ad_{g^{-1}}\left[  \omega\left(  \left(  v_{m},0_{e}\right)  \right)
\right]  +A\\
&  =A+Ad_{g^{-1}}\Gamma\left(  v_{m}\right)  .
\end{align*}
In this way we see that connections on $M\times G$ are in one to one
correspondence with $\mathfrak{g}$ -- valued one-forms on $M.$

Before finishing this example let us compute $\omega\left(  \dot{u}\left(
t\right)  \right)  $ where $u\left(  t\right)  =\left(  y\left(  t\right)
,g\left(  t\right)  \right)  $ is any smooth curve in $P.$ The key point is to
observe that $\dot{g}\left(  t\right)  =\tilde{A}\left(  g\left(  t\right)
\right)  $ where $A\mathbf{:=}L_{g\left(  t\right)  ^{-1}\ast}\dot{g}\left(
t\right)  $ and therefore,%
\begin{equation}
\omega\left(  \dot{u}\left(  t\right)  \right)  =\omega\left(  \left(  \dot
{y}\left(  t\right)  ,\dot{g}\left(  t\right)  \right)  \right)  =L_{g\left(
t\right)  ^{-1}\ast}\dot{g}\left(  t\right)  +Ad_{g\left(  t\right)  ^{-1}%
}\Gamma\left(  \dot{y}\left(  t\right)  \right)  . \label{equ.7.1}%
\end{equation}
Alternatively stated, if $\left(  v_{m},\xi_{g}\right)  \in T_{\left(
m,g\right)  }\left(  M\times G\right)  \cong T_{m}M\times T_{g}G,$ then
\begin{equation}
\omega\left(  \left(  v_{m},\xi_{g}\right)  \right)  =\theta\left(  \xi
_{g}\right)  +Ad_{g^{-1}}\Gamma\left(  v_{m}\right)  , \label{equ.7.2}%
\end{equation}
where
\begin{equation}
\theta\left(  \xi_{g}\right)  :=L_{g^{-1}\ast}\xi_{g}\in\mathfrak{g}
\label{equ.7.3}%
\end{equation}
is the left \textbf{Maurer--Cartan}\textit{ }form on $G.$
\end{example}

\subsection{Horizontal Lifts\label{sub.7.2}}

\begin{definition}
[Smooth Horiontal Lifts]\label{def.7.3}Let $\left(  G\rightarrow
P\overset{\pi}{\rightarrow}M,\omega\right)  $ be a principal bundle with
connection, $\omega,$ and $y\left(  t\right)  $ be a smooth curve in $M.$ We
say that $t\rightarrow u\left(  t\right)  \in P$ is a horizontal lift of $y$
provided; i) it is a lift, i.e. $\pi\circ u=y$ (or equivalently $u\left(
t\right)  \in P_{y\left(  t\right)  }$ for all $t)$ and ii) it is horizontal,
i.e. $\omega\left(  \dot{u}\left(  t\right)  \right)  =0$ for all $t.$
\end{definition}

\begin{example}
[Trivial Bundle Case II]\label{exa.7.4}Let us continue the notation in Example
\ref{exa.7.2} and suppose that $y\left(  t\right)  \in M$ is a smooth curve.
Any lift of $y$ is of the form $u\left(  t\right)  =\left(  y\left(  t\right)
,g\left(  t\right)  \right)  $ for some smooth curve, $t\rightarrow g\left(
t\right)  \in G.$ From Eq. (\ref{equ.7.1}) it follows that $u$ is horizontal
iff
\[
0=L_{g\left(  t\right)  ^{-1}\ast}\dot{g}\left(  t\right)  +Ad_{g\left(
t\right)  ^{-1}}\Gamma\left(  \dot{y}\left(  t\right)  \right)
\]
or equivalently by applying $Ad_{g\left(  t\right)  \ast}$ to both sides of
this equation iff
\[
0=R_{g\left(  t\right)  ^{-1}\ast}\dot{g}\left(  t\right)  +\Gamma\left(
\dot{y}\left(  t\right)  \right)  .
\]
In the matrix group case this is equivalent to solving,%
\[
\dot{g}\left(  t\right)  +\Gamma\left(  \dot{y}\left(  t\right)  \right)
g\left(  t\right)  =0.
\]
These differential equations have global unique solutions once we specify
$g\left(  0\right)  =g_{0}$ for some $g_{0}\in G.$ Hence for trivial bundles
we have shown that to each $u_{0}\in P_{y\left(  0\right)  }$ there exists a
unique horizontal lift, $u\left(  \cdot\right)  ,$ of $y$ such that $u\left(
0\right)  =u_{0}.$
\end{example}

Before ending this section, let us consider what happens to all of these
structures under pull backs.

\begin{example}
\label{exa.7.5}Suppose $G\rightarrow\tilde{P}\overset{\tilde{\pi}%
}{\rightarrow}\tilde{M}$ is another principal bundle with the same structure
group, $f:M\rightarrow\tilde{M}$ is a smooth map, and $F:P\rightarrow\tilde
{P}$ is a bundle map above $f,$ i.e. $\tilde{\pi}\circ F=f\circ\pi$ and
$F\left(  ug\right)  =F\left(  u\right)  g$ for all $u\in P$ and $g\in G.$
This last statement may be written as $F\circ R_{g}=R_{g}\circ F.$

A connection $\tilde{\omega}$ on $\tilde{P}$ pulls back to a $\mathfrak{g}$ --
valued one-form, $\omega:=F^{\ast}\tilde{\omega}$ on $P.$ This one-form is
again a connection on $P$ since;
\begin{align*}
\omega\left(  u\cdot A\right)   &  =\left(  F^{\ast}\tilde{\omega}\right)
\left(  \frac{d}{dt}|_{0}ue^{tA}\right)  =\tilde{\omega}\left(  F_{\ast}%
\frac{d}{dt}|_{0}ue^{tA}\right) \\
&  =\tilde{\omega}\left(  \frac{d}{dt}|_{0}F\left(  ue^{tA}\right)  \right)
=\tilde{\omega}\left(  \frac{d}{dt}|_{0}F\left(  u\right)  e^{tA}\right) \\
&  =\tilde{\omega}\left(  F\left(  u\right)  \cdot A\right)  =A
\end{align*}
and%
\begin{align*}
R_{g}^{\ast}\omega &  =R_{g}^{\ast}F^{\ast}\tilde{\omega}=\left(  F\circ
R_{g}\right)  ^{\ast}\tilde{\omega}=\left(  R_{g}\circ F\right)  ^{\ast}%
\tilde{\omega}\\
&  =F^{\ast}R_{g}^{\ast}\tilde{\omega}=F^{\ast}\left(  Ad_{g^{-1}}%
\tilde{\omega}\right)  =Ad_{g^{-1}}F^{\ast}\tilde{\omega}=Ad_{g^{-1}}\omega.
\end{align*}

Moreover, if $y\left(  t\right)  $ is a smooth curve in $M$ and $u\left(
t\right)  $ is a horizontal lift of $y,$ then $F\circ u$ is a horizontal lift
of $f\circ y.$ To see this is the case we have $\tilde{\pi}\circ F\circ
u=f\circ\pi\circ u=f\circ y$ so that $F\circ u$ is a lift of $f\circ y.$
Moreover,%
\[
\tilde{\omega}\left(  \frac{d}{dt}F\left(  u\left(  t\right)  \right)
\right)  =\left(  F^{\ast}\tilde{\omega}\right)  \left(  \dot{u}\left(
t\right)  \right)  =\omega\left(  \dot{u}\left(  t\right)  \right)  =0.
\]

\end{example}

As a consequence of these examples we may easily prove the following theorem.

\begin{theorem}
[Existence of Horizontal Lifts]\label{the.7.6}Let $G\rightarrow P\overset{\pi
}{\rightarrow}M$ be a principal bundle with connection $\omega,$ $y$ be a
smooth curve in $M,$ and $u_{0}\in P_{y\left(  0\right)  }.$ Then there exists
a unique horizontal lift $u\left(  t\right)  $ above $y$ such that $u\left(
0\right)  =u_{0}.$
\end{theorem}

\begin{proof}
Let us first prove local existence and uniqueness. We choose an open
neighborhood, $U\subseteq M$ of $y\left(  0\right)  $ such that $\pi
^{-1}\left(  U\right)  $ may be trivialized, i.e. there exists a bundle
isomorphism $F:U\times G\rightarrow\pi^{-1}\left(  U\right)  $ such that
$\pi\circ F=I_{U}\circ\tilde{\pi}$ where $\tilde{\pi}$ is projection onto the
first factor of $U\times G.$ We then let $\tilde{\omega}:=F^{\ast}\omega$ and
$g_{0}\in G$ be determined by $\left(  y\left(  0\right)  ,g_{0}\right)
=F^{-1}\left(  u_{0}\right)  .$

From Example \ref{exa.7.5} we know that $\tilde{\omega}$ is a connection on
$U\times G.$ Now choose $\tau>0$ so that $y\left(  \left[  0,\tau\right]
\right)  \subseteq U.$ We may then use Example \ref{exa.7.4} to conclude there
exists a unique $g\left(  t\right)  \in G$ for $0\leq t\leq\tau$ such that
$\tilde{u}\left(  t\right)  :=\left(  y\left(  t\right)  ,g\left(  t\right)
\right)  $ is the $\tilde{\omega}$ -- horizontal lift of $y$ starting at
$\tilde{u}\left(  0\right)  =\left(  y\left(  0\right)  ,g_{0}\right)  .$ It
then follows from Example \ref{exa.7.5} that $u\left(  t\right)  =F\left(
\tilde{u}\left(  t\right)  \right)  $ is the unique horizontal lift of
$y\left(  t\right)  $ for $0\leq t\leq\tau$ starting at $u_{0}.$ By a finite
covering argument we may continue this horizontal lift to the time interval
for which $y$ is defined. The uniqueness is also easily proved at the same time.
\end{proof}

There is one last horizontal lifting proposition we should record here.

\begin{notation}
\label{not.7.7}Let $G\rightarrow P\overset{\pi}{\rightarrow}M$ be a principal
bundle with connection $\omega.$ The \textbf{vertical subspace }at $u\in P$ is
defined by%
\[
\mathcal{V}_{u}:=\left\{  \xi\in T_{u}P:\pi_{\ast}\xi=0\right\}  \subseteq
T_{u}P
\]
and the \textbf{horizontal subspace} at $u\in P$ is defined by%
\[
\mathcal{H}_{u}^{\omega}:=\left\{  \xi\in T_{u}P:\omega\left(  \xi\right)
=0\right\}  \subseteq T_{u}P.
\]

\end{notation}

\begin{proposition}
[Tangent Horizontal Lifting]\label{pro.7.8}If $G\rightarrow P\overset{\pi
}{\rightarrow}M$ be a principal bundle with connection $\omega,$ then;

\begin{enumerate}
\item $\mathcal{V}_{u}=\left\{  u\cdot A:A\in\mathfrak{g}\right\}  ,$

\item $T_{u}P=\mathcal{V}_{u}\oplus\mathcal{H}_{u}^{\omega}$ for all $u\in P,$ and

\item $\pi_{\ast}:\mathcal{H}_{u}^{\omega}\rightarrow T_{\pi\left(  u\right)
}M$ is a linear isomorphism.
\end{enumerate}
\end{proposition}

\begin{proof}
The results of this proposition are purely local and hence we may assume that
$P$ is the trivial bundle $M\times G.$ In this model with $u=\left(
m,g\right)  ,$
\begin{align*}
\mathcal{V}_{u}  &  =\left\{  \left(  0_{m},\xi_{g}\right)  :\xi_{g}\in
T_{g}G\right\}  =\left\{  \left(  0_{m},\tilde{A}\left(  g\right)  \right)
:A\in\mathfrak{g}\right\} \\
&  =\left\{  \left(  0_{m},L_{g\ast}A\right)  :A\in\mathfrak{g}\right\}
=\left\{  u\cdot A:A\in\mathfrak{g}\right\}
\end{align*}
which prove item 1. Letting $\Gamma\left(  v_{m}\right)  :=\omega\left(
\left(  v_{m},0_{e}\right)  \right)  $ as in Example \ref{exa.7.2}, we have
\[
\omega\left(  v_{m},\tilde{A}\left(  g\right)  \right)  =A+Ad_{g^{-1}}%
\Gamma\left(  v_{m}\right)
\]
and so $\left(  v_{m},\tilde{A}\left(  g\right)  \right)  $ is horizontal iff
$A=-Ad_{g^{-1}}\Gamma\left(  v_{m}\right)  .$ Therefore it follows that
\begin{align*}
\mathcal{H}_{u}^{\omega}  &  =\left\{  \left(  v_{m},-L_{g\ast}\cdot
Ad_{g^{-1}}\Gamma\left(  v_{m}\right)  \right)  :v_{m}\in T_{m}M\right\} \\
&  =\left\{  \left(  v_{m},-R_{g\ast}\Gamma\left(  v_{m}\right)  \right)
:v_{m}\in T_{m}M\right\}  .
\end{align*}
From these descriptions of $\mathcal{V}_{u}$ and $\mathcal{H}_{u}^{\omega}$ it
is easily seen that $\mathcal{V}_{u}\cap\mathcal{H}_{u}^{\omega}=\left\{
0\right\}  $ and $T_{u}P=\mathcal{V}_{u}+\mathcal{H}_{u}^{\omega}$ and hence
item 2. is proved. Item 3. is also now trivial to check since it is clear that
$\pi_{\ast}\left(  v_{m},-R_{g\ast}\Gamma\left(  v_{m}\right)  \right)
=v_{m}$ defines an isomorphism from $\mathcal{H}_{u}^{\omega}$ onto $T_{m}M.$
\end{proof}

\begin{notation}
\label{not.7.9}If $G\rightarrow P\overset{\pi}{\rightarrow}M$ be a principal
bundle with connection $\omega,$ let $\mathcal{B}_{u}^{\omega}:T_{\pi\left(
u\right)  }M\rightarrow\mathcal{H}_{u}^{\omega}$ be the inverse of $\pi_{\ast
u}|_{\mathcal{H}_{u}^{\omega}}.$ Thus if $v\in T_{\pi\left(  u\right)  }M,$
then $\xi=\mathcal{B}_{u}^{\omega}v$ iff $\pi_{\ast}\xi=v$ and $\omega\left(
\xi\right)  =0.$ We refer to $\mathcal{B}_{u}^{\omega}v$ as the
\textbf{horizontal lift }of $v$ to $T_{u}P.$
\end{notation}

\section{Some Auxiliary Results\label{sec.8}}

This section gathers some further results which are needed below to discuss
rough horizontal lifts.

\begin{lemma}
\label{lem.8.1}\marginpar{Should we just state this.\ Haven't thought about it
completely but seems like you could just define a chart on the product as the
product chart from which it should be obvious}Let $M$ and $N$ are two
manifolds and $\pi^{M}$ and $\pi^{N}$ be the projection maps from $M\times N$
to $M$ and $N$ respectively. The map,%
\begin{equation}
\left(  \pi_{\ast}^{M},\pi_{\ast}^{N}\right)  :CRP_{\mathbf{X}}\left(  M\times
N\right)  \rightarrow CRP_{\mathbf{X}}\left(  M\right)  \times CRP_{\mathbf{X}%
}\left(  N\right)  , \label{equ.8.1}%
\end{equation}
is a bijection. \marginpar{I think this Lemma should be proved with general
gauges. I think the notion of a product gauge should be formalized as well.
The covariant derivative associated to a product parallelism is the
\textquotedblleft sum\textquotedblright\ of covariant derivatives.}
\end{lemma}

\begin{proof}
Let $\mathbf{u}=\left(  u,u^{\dag}\right)  \in CRP_{\mathbf{X}}\left(  M\times
N\right)  ,$ $\mathbf{y}=\left(  y,y^{\dag}\right)  :=\pi_{\ast}^{M}\left(
\mathbf{u}\right)  \in CRP_{\mathbf{X}}\left(  M\right)  $ and $\mathbf{z=}%
\left(  z,z^{\dag}\right)  :=\pi_{\ast}^{N}\left(  \mathbf{u}\right)
\mathbf{\in}CRP_{\mathbf{X}}\left(  N\right)  .$ Then $u_{t}=\left(
y_{t},z_{t}\right)  $ is uniquely determined by $y$ and $z.$ Similarly, since
\[
\left(  \pi_{\ast}^{M},\pi_{\ast}^{N}\right)  :T_{u_{t}}\left(  M\times
N\right)  \longrightarrow T_{y_{t}}M\times T_{z_{t}}N
\]
is an isomorphism it follows that $u_{t}^{\dag}:W\rightarrow T_{u_{t}}\left(
M\times N\right)  $ is uniquely determined by $\left(  \pi_{\ast}^{M}%
,\pi_{\ast}^{N}\right)  u_{t}^{\dag}=\left(  y_{t}^{\dag},z_{t}^{\dag}\right)
.$ So the only real content of the Lemma that the map in Eq. (\ref{equ.8.1})
is surjective.

So now suppose that $\mathbf{y}=\left(  y,y^{\dag}\right)  \mathbf{\in
}CRP_{\mathbf{X}}\left(  M\right)  $ and $\mathbf{z=}\left(  z,z^{\dag
}\right)  \mathbf{\in}CRP_{\mathbf{X}}\left(  N\right)  $ are given and define
$\mathbf{u}=\left(  u,u^{\dag}\right)  $ so that $u=\left(  y,z\right)  $ and
$\left(  \pi_{\ast}^{M},\pi_{\ast}^{N}\right)  u^{\dag}=\left(  y_{t}^{\dag
},z_{t}^{\dag}\right)  .$ To finish the proof we need only verify that
$\mathbf{u\in}CRP_{\mathbf{X}}\left(  M\times N\right)  .$ To this end suppose
that $\psi^{M}$ and $\psi^{N}$ are logarithms on $M$ and $N$ respectively. We
then let
\[
\psi\left(  \left(  m,n\right)  ,\left(  \tilde{m},\tilde{n}\right)  \right)
:=\left(  \psi^{M}\left(  m,\tilde{m}\right)  ,\psi^{N}\left(  n,\tilde
{n}\right)  \right)  \in T_{\left(  m,n\right)  }\left[  M\times N\right]
\]
where we are now using $\left(  \pi_{\ast}^{M},\pi_{\ast}^{N}\right)  $ to
identify $T\left[  M\times N\right]  $ with $TM\times TN.$ The map $\psi$ is
now a logarithm on $M\times N.$ The parallelism, $U^{\psi},$ associated to
this logarithm satisfies,
\[
U^{\psi}\left(  \left(  m,n\right)  ,\left(  \tilde{m},\tilde{n}\right)
\right)  =U^{\psi^{M}}\left(  m,\tilde{m}\right)  \times U^{\psi^{N}}\left(
n,\tilde{n}\right)  .
\]
Therefore we have
\[
\psi\left(  \left(  y_{s},z_{s}\right)  ,\left(  y_{t},z_{t}\right)  \right)
=\left(  \psi^{M}\left(  y_{s},y_{t}\right)  ,\psi^{N}\left(  z_{s}%
,z_{t}\right)  \right)  \underset{^{2}}{\approx}\left(  y_{s}^{\dag}%
x_{s,t},z_{s}^{\dag}x_{s,t}\right)  =u_{s}^{\dag}x_{s,t}%
\]
and
\begin{align*}
U^{\psi}\left(  \left(  y_{s},z_{s}\right)  ,\left(  y_{t},z_{t}\right)
\right)  u_{t}^{\dag}  &  =\left(  U^{\psi^{M}}\left(  y_{s},y_{t}\right)
\times U^{\psi^{N}}\left(  z_{s},z_{t}\right)  \right)  \left(  y_{t}^{\dag
},z_{t}^{\dag}\right) \\
&  =\left(  U^{\psi^{M}}\left(  y_{s},y_{t}\right)  y_{t}^{\dag},U^{\psi^{N}%
}\left(  z_{s},z_{t}\right)  z_{t}^{\dag}\right) \\
&  \underset{^{1}}{\approx}\left(  y_{s}^{\dag},z_{s}^{\dag}\right)
=u_{s}^{\dag}%
\end{align*}
from which it follows that $\mathbf{u\in}CRP_{\mathbf{X}}\left(  M\times
N\right)  .$
\end{proof}

\begin{theorem}
\label{the.8.2}Let $G$ be a Lie group and set $F_{A}\left(  g\right)
:=-A\cdot g=-R_{g\ast}A$ for all $A\in\mathfrak{g.}$ Let $\mathbf{z\in
}CRP_{\mathbf{X}}\left(  \mathfrak{g}\right)  $ be defined on $\left[
0,T\right]  $. There exists a unique global solution $\mathbf{g}\in
CRP_{\mathbf{X}}\left(  G\right)  $ solving
\begin{equation}
d\mathbf{g}_{t}=F_{d\mathbf{z}_{t}}\left(  \mathbf{g}_{t}\right)
\label{equ.8.2}%
\end{equation}
with initial condition
\[
g_{0}=e\in G.
\]
Here we are abusing notation and mean $d\mathbf{Z}_{t}$ in Eq. (\ref{equ.8.2})
where $\mathbf{Z}$ is the rough path associated to the controlled rough path
$\mathbf{z\in}CRP_{\mathbf{X}}\left(  \mathfrak{g}\right)  $.
\end{theorem}

\begin{proof}
This proof is adapted from Theorem 4.20 in \cite{CDL13}. By Theorem
\ref{the.5.5}, we know we only have to rule out the case that $\mathbf{g}$
exists maximally on $[0,\tau)$ for some $\tau\leq T$ where $\left\{
g_{t}:0\leq t<\tau\right\}  $ does not have compact closure.

By Lemma \ref{lem.6.8}, there exists an $\epsilon>0$ such that for any
$t_{0}\in\left[  0,T\right]  $ there is a solution $\mathbf{h}$ defined on
$\left[  t_{0},\left(  t_{0}+\epsilon\right)  \wedge T\right]  $ solving
$d\mathbf{h}=F_{d\mathbf{z}_{t}}\left(  \mathbf{h}_{t}\right)  $ with initial
condition $h_{t_{0}}=e$. If right multiplication map is given by $R_{g_{2}%
}\left(  g_{1}\right)  =g_{1}g_{2}$, then it is easy to see that $F$ is
$R_{g}-$related to itself so that $\mathbf{k}=\left(  R_{\bar{g}}\right)
_{\ast}\mathbf{h}$ on $\left[  t_{0},\left(  t_{0}+\epsilon\right)  \wedge
T\right]  $ solves
\[
d\mathbf{k}_{t}=F_{d\mathbf{z}_{t}}\left(  \mathbf{k}_{t}\right)
\]
with initial condition
\[
k_{t_{0}}=\bar{g}.
\]
Choosing $t_{0}\in\left(  0\vee\left(  \tau-\epsilon/2\right)  ,\tau\right)
$, we can concatenate $\mathbf{g}$ and $\mathbf{k}$ where we start
$\mathbf{k}$ at $t_{0}$ with initial condition $k_{t_{0}}=g\left(
t_{0}\right)  $. This gives us a solution defined on a strictly larger
interval than $[0,\tau)$ and thus shows that the second case mentioned above
can not occur.
\end{proof}

\begin{theorem}
\label{the.8.3}Let $G$ be a Lie group and $\theta:=\theta_{r}\in\Omega
^{1}\left(  G,\mathfrak{g}\right)  $ be the right -- Maurer-Cartan form on
$G,$ i.e.
\[
\theta\left(  \xi_{g}\right)  =\xi_{g}\cdot g^{-1}:=R_{g^{-1}\ast}\xi_{g}.
\]
Further set $F_{A}\left(  g\right)  :=-A\cdot g=-R_{g\ast}A$ for all
$A\in\mathfrak{g}.$ If $\mathbf{z:=}\left(  z,z^{\dag}\right)  \in
CRP_{\mathbf{X}}\left(  \mathfrak{g}\right)  $ is a $\mathfrak{g}$ -- valued
rough path, then $\mathbf{g}=\left(  g,g^{\dag}\right)  \in CRP_{\mathbf{X}%
}\left(  G\right)  $ satisfies
\begin{equation}
d\mathbf{g}=F_{d\mathbf{Z}}\left(  g\right)  \label{equ.8.3}%
\end{equation}
iff%
\begin{equation}
\int\theta\left(  d\mathbf{g}\right)  =-\mathbf{z}. \label{equ.8.4}%
\end{equation}

\end{theorem}

\begin{proof}
First we note that
\[
\theta\left(  F_{b}\left(  g\right)  \right)  =-b\text{\quad for all\quad}%
b\in\mathfrak{g},g\in G
\]
and%
\[
F_{\theta\left(  \xi_{g}\right)  }\left(  g\right)  =-\xi_{g}\quad\text{for
all}\quad\xi_{g}\in T_{g}G
\]

Assume Eq. (\ref{equ.8.3}) holds. Using Theorem \ref{the.5.3} we learn that
\[
\left[  \int\theta\left(  d\mathbf{g}\right)  \right]  _{s,t}^{1}%
\underset{^{3}}{\approx}\theta\left(  F_{z_{s,t}}\left(  g_{s}\right)
\right)  +F_{a}\left(  g_{s}\right)  \theta\left(  F_{b}\right)  |_{a\otimes
b=z_{s}^{\dag}\otimes z_{s}^{\dag}\mathbb{X}_{s,t}}=-z_{s,t}%
\]
wherein we have used $\theta\left(  F_{b}\left(  g\right)  \right)  =-b$ for
all $b\in\mathfrak{g}$, $g\in G$ and therefore $\left[  F_{a}\left(
g_{s}\right)  \right]  \theta\left(  F_{b}\right)  =0.$ Moreover we have
\[
\left[  \int\theta\left(  d\mathbf{g}\right)  \right]  _{s}^{^{\dag}}%
=\theta\left(  F_{z_{s}^{\dag}\left(  \cdot\right)  }\left(  g_{s}\right)
\right)  =-z_{s}^{\dag}%
\]
and hence we have shown that Eq. (\ref{equ.8.4}) holds.

Conversely, let us suppose that Eq. (\ref{equ.8.4}) holds and $\alpha\in
\Omega^{1}\left(  G,V\right)  $ is a smooth one-form on $G.$ We will show that%
\begin{equation}
\int_{s}^{t}\alpha\left(  d\mathbf{g}\right)  \underset{^{3}}{\approx}%
\alpha\left(  F_{z_{s,t}}\left(  g_{s}\right)  \right)  +F_{A}\left(
g_{s}\right)  \alpha\left(  F_{B}\right)  |_{A\otimes B=z_{s}^{\dag}\otimes
z_{s}^{\dag}\mathbb{X}_{s,t}} \label{equ.8.5}%
\end{equation}
and%
\begin{equation}
g_{s}^{\dag}=F_{z_{s}^{\dag}\left(  \cdot\right)  }\left(  g_{s}\right)
\label{equ.8.6}%
\end{equation}
To prove Eq. (\ref{equ.8.6}), we note that
\[
g_{s}^{\dag}=-F_{\theta\left(  g_{s}^{\dag}\left(  \cdot\right)  \right)
}\left(  g_{s}\right)  =F_{z_{s}^{\dag}\left(  \cdot\right)  }\left(
g_{s}\right)
\]
where the second equality follows from the fact that%
\[
\theta\left(  g_{s}^{\dag}\left(  \cdot\right)  \right)  =-z_{s}^{\dag}%
\]
To prove Eq. (\ref{equ.8.5}), we will first write $\alpha$ as the composition
of two functions. Given $\xi_{g}\in T_{g}G$ we have $\xi_{g}=\theta\left(
\xi_{g}\right)  \cdot g=R_{g\ast}\theta\left(  \xi_{g}\right)  $ and therefore%
\[
\alpha\left(  \xi_{g}\right)  =\alpha\left(  R_{g\ast}\theta\left(  \xi
_{g}\right)  \right)  =\left(  R_{g}^{\ast}\alpha\right)  \theta\left(
\xi_{g}\right)  .
\]
This shows $\alpha=K\theta$ where $K:G\rightarrow\operatorname*{End}\left(
\mathfrak{g},V\right)  $ is the function defined by $K\left(  g\right)
=R_{g}^{\ast}\alpha=\alpha_{g}\circ R_{g\ast}.$ Applying Theorem \ref{the.4.9}
with $\mathbf{y}$ replaced by $\mathbf{g}$ shows,%
\[
\int\alpha\left(  d\mathbf{g}\right)  =\int\left(  K\theta\right)  \left(
d\mathbf{g}\right)  =-\int K_{\ast}\left(  \mathbf{g}\right)  d\mathbf{z.}%
\]
So, according to Theorem \ref{the.5.3}, it only remains to show%
\begin{equation}
-\left[  \int K_{\ast}\left(  \mathbf{g}\right)  d\mathbf{z}\right]
_{s,t}^{1}\underset{^{3}}{\approx}\alpha\left(  F_{z_{s,t}}\left(
g_{s}\right)  \right)  +F_{A}\left(  g_{s}\right)  \alpha\left(  F_{B}\right)
|_{A\otimes B=z_{s}^{\dag}\otimes z_{s}^{\dag}\mathbb{X}_{s,t}}.
\label{equ.8.7}%
\end{equation}

In order to work out the left side of Eq. (\ref{equ.8.7}) we need to expand
out $K_{\ast g_{s}}g_{s}^{\dag}$ which we now do. If $\xi_{g}=\dot{g}\left(
0\right)  \in T_{g}G$ and $A\in\mathfrak{g,}$ then
\begin{align*}
\left(  K_{\ast}\xi_{g}\right)  A  &  =\frac{d}{dt}|_{0}K\left(  g\left(
t\right)  \right)  A=\frac{d}{dt}|_{0}\left(  R_{g\left(  t\right)  }^{\ast
}\alpha\right)  A\\
&  =\frac{d}{dt}|_{0}\alpha\left(  R_{g\left(  t\right)  \ast}A\right)
=-\frac{d}{dt}|_{0}\alpha\left(  F_{A}\left(  g\left(  t\right)  \right)
\right) \\
&  =-\xi_{g}\alpha\left(  F_{A}\right)  .
\end{align*}
Therefore for $A,B\in\mathfrak{g},$%
\begin{equation}
\left(  K_{\ast g_{s}}g_{s}^{\dag}A\right)  B=-\left(  g_{s}^{\dag}A\right)
\alpha\left(  F_{B}\right)  . \label{equ.8.8}%
\end{equation}
As mentioned above, we have%
\[
g_{s}^{\dag}A=F_{z_{s}^{\dag}A}\left(  g_{s}\right)
\]
which combined with Eq. (\ref{equ.8.8}) shows
\[
\left(  K_{\ast g_{s}}g_{s}^{\dag}A\right)  B=-F_{z_{s}^{\dag}A}\left(
g_{s}\right)  \alpha\left(  F_{B}\right)  .
\]

Using these result it now follows that%
\begin{align*}
&  -\left[  \int K_{\ast}\left(  \mathbf{g}\right)  d\mathbf{z}\right]
_{s,t}^{1}\underset{^{3}}{\approx}-K\left(  g_{s}\right)  z_{s,t}-\left(
K_{\ast g_{s}}g_{s}^{\dag}\right)  \left(  z_{s}^{\dag}\otimes z_{s}^{\dag
}\right)  \mathbb{X}_{s,t}\\
&  \quad=-\alpha_{g_{s}}R_{g_{s}\ast}z_{s,t}+F_{z_{s}^{\dag}A}\left(
g_{s}\right)  \alpha\left(  F_{z_{s}^{\dag}B}\right)  |_{A\otimes
B=\mathbb{X}_{s,t}}\\
&  \quad=\alpha\left(  F_{z_{s,t}}\left(  g_{s}\right)  \right)  +F_{A}\left(
g_{s}\right)  \alpha\left(  F_{B}\right)  |_{A\otimes B=z_{s}^{\dag}\otimes
z_{s}^{\dag}\mathbb{X}_{s,t}}%
\end{align*}
which is the desired result.
\end{proof}

\section{Controlled Rough Path Horizontal Lifts\label{sec.9}}

Here we show that rough horizontals always exist and are unique. Given all of
the preparation, this section is now fairly straightforward and clean.

\begin{definition}
[Rough Horiontal Lifts]\label{def.9.1}Let $\left(  G\rightarrow P\overset{\pi
}{\rightarrow}M,\omega\right)  $ be a principal bundle with connection,
$\omega,$ and $\mathbf{y}=\left(  y,y^{\dag}\right)  \in CRP_{\mathbf{X}%
}\left(  M\right)  .$ A controlled rough path $\mathbf{u}=\left(  u,u^{\dag
}\right)  \in CRP_{\mathbf{X}}\left(  P\right)  $ is said to be a\textbf{
horizontal lift }of $\mathbf{y}$ if
\begin{equation}
\pi_{\ast}\left(  \mathbf{u}\right)  =\mathbf{y}\text{\quad and\quad}%
\int\omega\left(  d\mathbf{u}\right)  \equiv\mathbf{0} \label{equ.9.1}%
\end{equation}
where $\mathbf{0}$ is the controlled rough path whose path and derivative
process are identically $0$. We will write $\mathbf{u}\in CRP_{\mathbf{y}}%
^{H}\left(  P\right)  $ for such lifts.
\end{definition}

\begin{remark}
\label{rem.9.2}One might ask that Definition \ref{def.9.1} should include the
requirement that%
\[
u^{\dag}=\mathcal{B}_{u}^{\omega}y^{\dag}%
\]
or more explicitly $u_{t}^{\dag}w=\mathcal{B}_{u_{t}}^{\omega}y_{t}^{\dag}w$
for all $t$ and $w\in W.$ However, this is redundant as $\pi_{\ast}\left(
\mathbf{u}\right)  =\mathbf{y}$ implies that%
\begin{equation}
\pi_{\ast}u_{s}^{\dag}=y_{s}^{\dag} \label{equ.9.2}%
\end{equation}
while $\int\omega\left(  d\mathbf{u}\right)  \equiv0$ implies that
\begin{equation}
\omega\left(  u_{s}^{\dag}\right)  =0. \label{equ.9.3}%
\end{equation}
However, Eq. (\ref{equ.9.2}) and Eq. (\ref{equ.9.3}) together imply that
$u^{\dag}=\mathcal{B}_{u}^{\omega}y^{\dag}.$
\end{remark}

\begin{proposition}
\label{pro.9.3}Suppose $G\rightarrow\tilde{P}\overset{\tilde{\pi}%
}{\rightarrow}\tilde{M}$ is another principal bundle with connection
$\tilde{\omega}.$ Further suppose that $f:M\rightarrow\tilde{M}$ is a smooth
map, $F:P\rightarrow\tilde{P}$ is a bundle map above $f,$ and $\omega=F^{\ast
}\tilde{\omega}.$ If $\mathbf{u}=\left(  u,u^{\dag}\right)  \in
CRP_{\mathbf{X}}\left(  P\right)  $ a horizontal lift of $\mathbf{y\in}\in
CRP_{\mathbf{X}}\left(  M\right)  $, then $\mathbf{v}:=F_{\ast}\left(
\mathbf{u}\right)  \in CRP_{\mathbf{X}}\left(  \tilde{P}\right)  $ is a a
horizontal lift of $F_{\ast}\left(  \mathbf{y}\right)  \mathbf{\in}\in
CRP_{\mathbf{X}}\left(  \tilde{M}\right)  .$
\end{proposition}

\begin{proof}
By Remark \ref{rem.4.12} and the fact that $\mathbf{u}$ is a lift of
$\mathbf{y}$ we have
\[
\tilde{\pi}_{\ast}F_{\ast}\left(  \mathbf{u}\right)  =\left(  \tilde{\pi}\circ
F\right)  _{\ast}\left(  \mathbf{u}\right)  =\left(  f\circ\pi\right)  _{\ast
}\left(  \mathbf{u}\right)  =f_{\ast}\pi_{\ast}\left(  \mathbf{u}\right)
=f_{\ast}\left(  \mathbf{y}\right)
\]
and hence $F_{\ast}\left(  \mathbf{u}\right)  $ is a lift of $f_{\ast}\left(
\mathbf{y}\right)  .$ Secondly we observe from the push-me pull-me Theorem
\ref{the.4.15} that%
\[
\int\tilde{\omega}\left(  d\mathbf{v}\right)  =\int\tilde{\omega}\left(
dF_{\ast}\mathbf{u}\right)  =\int\left(  F^{\ast}\tilde{\omega}\right)
\left(  d\mathbf{u}\right)  =\int\omega\left(  d\mathbf{u}\right)
\equiv\mathbf{0}.
\]

\end{proof}

\begin{theorem}
[Existence of Horizontal Lifts]\label{the.9.4}Let $G\rightarrow P\overset{\pi
}{\rightarrow}M$ be a principal bundle with connection $\omega,$
$\mathbf{y}=\left(  y,y^{\dag}\right)  \in CRP_{\mathbf{X}}\left(  M\right)
,$ and $\bar{u}_{0}\in P_{y_{0}}.$ Then there exists a unique horizontal lift
$\mathbf{u}=\left(  u,u^{\dag}\right)  \in CRP_{\mathbf{X}}\left(  P\right)  $
above $\mathbf{y}$ such that $u_{0}=\bar{u}_{0}.$
\end{theorem}

\begin{proof}
The proof follows the lines of the smooth case. Because of Proposition
\ref{pro.9.3} and simple patching arguments we may reduce to considering the
case that $P=M\times G$ is a trivial bundle where $M$ is now an open subset of
$\mathbb{R}^{d}.$ In light of Lemma \ref{lem.8.1}, the desired horizontal lift
may be expressed in the form, $\mathbf{u}_{s}=\left(  \mathbf{y}%
_{s},\mathbf{g}_{s}\right)  $ for some $\mathbf{g}=\left(  g,g^{\dag}\right)
\mathbf{\in}CRP_{\mathbf{X}}\left(  G\right)  $ which is to be determined. We
now find the equations that $\mathbf{g}$ has to solve in order for\textbf{
}$\mathbf{u}$ to be horizontal.

Let $\pi_{G}:P=M\times G\rightarrow G$ and $\pi_{M}:P\rightarrow M$ be the
natural projection maps. Making use of Theorem \ref{the.4.9}, we may deduce
that
\begin{equation}
\int\omega\left(  d\mathbf{u}\right)  \equiv\mathbf{0\iff}\int\left(
Ad_{\pi_{G}}\omega\right)  \left(  d\mathbf{u}\right)  \equiv0.
\label{equ.9.4}%
\end{equation}
where by $Ad_{\pi_{G}}\omega$, we mean the one-form given by%
\[
Ad_{\pi_{G}}\omega\left(  v_{m},\xi_{g}\right)  =g\left[  \omega\left(
v_{m},\xi\,g\right)  \right]  g^{-1}.
\]
Indeed if $\int\omega\left(  d\mathbf{u}\right)  \equiv\mathbf{0}$, then
\[
\int Ad_{\pi_{G}}\omega\left(  d\mathbf{u}\right)  =\int\left(  Ad_{\pi_{G}%
}\right)  _{\ast}\left(  \mathbf{u}\right)  \left(  \int\omega\left(
d\mathbf{u}\right)  \right)  \equiv0
\]
and if $\int\left(  Ad_{\pi_{G}}\omega\right)  \left(  d\mathbf{u}\right)
\equiv0$
\begin{align*}
\int\omega\left(  d\mathbf{u}\right)   &  =\int\left(  Ad_{\pi_{G}^{-1}%
}Ad_{\pi_{G}}\omega\right)  \left(  d\mathbf{u}\right) \\
&  =\int Ad_{\pi_{G}^{-1}}\left(  \mathbf{u}\right)  \left(  \int Ad_{\pi_{G}%
}\omega\left(  d\mathbf{u}\right)  \right)  \equiv0
\end{align*}
On the other hand
\begin{align*}
\left(  Ad_{\pi_{G}}\omega\right)  \left(  \left(  v_{m},\xi_{g}\right)
\right)   &  =Ad_{g}\left[  \theta_{l}\left(  \xi_{g}\right)  +Ad_{g^{-1}%
}\Gamma\left(  v_{m}\right)  \right] \\
&  =\theta_{r}\left(  \xi_{g}\right)  +\Gamma\left(  v_{m}\right)
\end{align*}
from which we deduce
\[
Ad_{\pi_{G}}\omega=\pi_{G}^{\ast}\theta_{r}+\pi_{M}^{\ast}\Gamma.
\]
An application of the push-me pull-me Theorem \ref{the.4.15} then shows
\begin{align*}
\int\left(  Ad_{\pi_{G}}\omega\right)  \left(  d\mathbf{u}\right)   &
=\int\left(  \pi_{G}^{\ast}\theta_{r}\right)  \left(  d\mathbf{u}\right)
+\int\left(  \pi_{M}^{\ast}\Gamma\right)  \left(  d\mathbf{u}\right) \\
&  =\int\theta_{r}\left(  d\mathbf{g}\right)  +\int\Gamma\left(
d\mathbf{y}\right)  .
\end{align*}
Combining these statements shows $\int\omega\left(  d\mathbf{u}\right)
\equiv\mathbf{0}$ is equivalent to $\mathbf{u}_{s}=\left(  \mathbf{y}%
_{s},\mathbf{g}_{s}\right)  $ satisfying
\[
\int\theta_{r}\left(  d\mathbf{g}\right)  =-\int\Gamma\left(  d\mathbf{y}%
\right)
\]
which according to Theorem \ref{the.8.3} is equivalent to $\mathbf{g}$
satisfying
\begin{equation}
d\mathbf{g}=F_{d\mathbf{z}}\left(  g\right)  \label{equ.9.5}%
\end{equation}
where $F_{A}\left(  g\right)  =-R_{g\ast}A$ for all $A\in\mathfrak{g}$ and
$\mathbf{z}:=\int\Gamma\left(  d\mathbf{y}\right)  .$ It is now know by
Theorem \ref{the.8.2} that (given initial conditions for $g_{0})$ that Eq.
(\ref{equ.9.5}) has global unique solutions.
\end{proof}

We can now specialize the above to define parallel translation.

\begin{definition}
\label{def.9.5}Let $GL\left(  M\right)  $ be the frame bundle above $M$
with$,$ let $\mathbf{y}=\left(  y,y^{\dag}\right)  \in CRP_{\mathbf{X}}\left(
M\right)  $, and let $\nabla$ be a covariant derivative on $M$. Let $\bar
{u}_{0}\in GL\left(  M\right)  _{y_{0}}$. \textbf{Parallel translation along}
$\mathbf{y}$ \textbf{starting at} $\bar{u}_{0}$ is the unique element
$\mathbf{u}\in CRP_{\mathbf{y}}^{H,\nabla}\left(  GL\left(  M\right)  \right)
$ such that

\begin{enumerate}
\item $\pi_{\ast}\mathbf{u}=\mathbf{y}$

\item $\int\omega^{\nabla}\left(  d\mathbf{u}\right)  \equiv\mathbf{0}$
\end{enumerate}

where $\omega^{\nabla}$ is the connection form associated to $\nabla$.
\end{definition}

\section{Rough Rolling and Unrolling}

We will first introduce some terminology which will be useful for this section.

\begin{definition}
\label{def.10.1}A manifold $M$ is \textbf{parallelizable }\ if there exists a
linear map, $Y:\mathbb{R}^{d}\rightarrow\Gamma\left(  TM\right)  $ such that%
\[
Y_{\left(  \cdot\right)  }\left(  m\right)  :a\rightarrow Y_{a}\left(
m\right)
\]
is an isomorphism for all $m\in M.$

For every $Y$ that parallelizes $M$, there exists an $\mathbb{R}^{d}-$ valued
one-form given by%
\[
\theta^{Y}\left(  v_{m}\right)  :=\left[  Y_{\left(  \cdot\right)  }\left(
m\right)  \right]  ^{-1}v_{m}.
\]

\end{definition}

\begin{example}
\label{exa.10.3}Let $M$ be a manifold and $GL\left(  M\right)  $ be the
associated frame bundle. Let $\nabla$ be a covariant derivative on $M$. Then
$GL\left(  M\right)  $ is parallelizable with $T_{g}GL\left(  M\right)
\cong\mathbb{R}^{d}\times gl\left(  d\right)  $ where $gl\left(  d\right)  $
are the $d\times d$ matrices. One choice of $Y^{GL\left(  M\right)  }$ is this
case is defined by%
\[
Y^{GL\left(  M\right)  }\left(  g\right)  \left(  a,A\right)  :=B_{a}^{\nabla
}\left(  g\right)  +\tilde{A}\left(  g\right)
\]
where $B_{a}^{\nabla}$ is the horizontal vector field defined by
\[
B_{a}^{\nabla}\left(  g\right)  =\dot{u}\left(  0\right)  \text{\quad
where\quad}u\left(  t\right)  :=//_{t}\left(  \exp^{\nabla}\left(  \left(
\cdot\right)  ga\right)  \right)  g
\]
and where $\tilde{A}$ is the vertical vector field given by
\[
\tilde{A}\left(  g\right)  :=\frac{d}{dt}|_{0}ge^{tA}=g\cdot A\text{.}%
\]
Moreover, we have the associated $\mathbb{R}^{d}\times gl\left(  d\right)
-$valued one-form $\theta^{Y^{GL\left(  M\right)  }}:=\left(  \hat{\theta
},\omega^{\nabla}\right)  $ which is constructed such that
\[
\left(  \hat{\theta},\omega^{\nabla}\right)  \left(  B_{a}^{\nabla}\left(
g\right)  +\tilde{A}\left(  g\right)  \right)  =\left(  a,A\right)
\]
for all $\left(  a,A\right)  \in\mathbb{R}^{d}\times gl\left(  d\right)  $ and
$g\in GL\left(  M\right)  $ where $\hat{\theta}$ is the canonical one-form
given by%
\[
\hat{\theta}\left(  \dot{u}_{0}\right)  =u_{0}^{-1}\left(  \pi_{\ast}\dot
{u}_{0}\right)  .
\]

\end{example}

\subsection{Rolling and Unrolling of paths}

Here we have our main theorem about rolling and unrolling of paths.

\begin{theorem}
\label{the.10.4}Let $M$ be a parallelizable manifold (by $Y$) and let
$\theta^{Y}$ the associated one-form. Fix some point $o\in M$. Then every
$\mathbf{y}\in CRP_{\mathbf{X}}\left(  M\right)  $ with $y_{0}=o$ on the
interval $\left[  0,T\right]  $ determines a path $\mathbf{z}\in
CRP_{\mathbf{X}}\left(  \mathbb{R}^{d}\right)  $ with $z_{0}=0$ on the
interval $\left[  0,T\right]  $ by the map%
\[
\mathbf{y}\longrightarrow\mathbf{z}:=\int\theta^{Y}\left(  d\mathbf{y}\right)
.
\]
such that
\[
d\mathbf{y}_{t}=Y_{d\mathbf{z}_{t}}\left(  \mathbf{y}_{t}\right)
\quad\text{and\quad}y_{0}=o.
\]

Alternatively, suppose that $\mathbf{z}\in CRP_{\mathbf{X}}\left(
\mathbb{R}^{d}\right)  $ with $z_{0}=0$ on the interval $\left[  0,T\right]  $
and let $\mathbf{y}$ be the solution to%
\[
d\mathbf{y}_{t}=Y_{d\mathbf{z}_{t}}\left(  \mathbf{y}_{t}\right)
\]
with initial condition $y_{0}=o$ with possible explosion time $\tau$. Then
over $[0,\tau)$ we have%
\[
\int\theta^{Y}\left(  d\mathbf{y}\right)  =\mathbf{z}.
\]

\end{theorem}

\begin{proof}
This proof follows nearly word for word the proof of Theorem \ref{the.8.3} if
we replace $F$ with $-Y$. Though this earlier proof was specialized to the Lie
group case (which could not explode in finite time), the only fact we used
about $F$ and $\theta$ was how they interacted with each other.
\end{proof}

We can now use Theorem \ref{the.10.4} with Example \ref{exa.10.3} to
specialize to a correspondence of paths of $GL\left(  M\right)  $ and those of
$\mathbb{R}^{d}\times gl\left(  d\right)  $.

\begin{theorem}
\label{the.10.5}Fix some point $o\in M$ and $g_{o}$ a frame at $o$. Every
$\mathbf{u}\in CRP_{\mathbf{X}}\left(  GL\left(  M\right)  \right)  $ with
$u_{0}=g_{o}$ on the interval $\left[  0,T\right]  $ gives rise to a
$\mathbf{z}\in CRP_{\mathbf{X}}\left(  \mathbb{R}^{d}\times gl\left(
d\right)  \right)  $ with $z_{0}=\left(  0,0\right)  $ on the interval
$\left[  0,T\right]  $ via the map%
\[
\mathbf{u}\longrightarrow\mathbf{z}:=\int\theta^{Y^{GL\left(  M\right)  }%
}\left(  d\mathbf{u}\right)  .
\]
such that%
\[
d\mathbf{u}_{t}=Y_{d\mathbf{z}_{t}}^{GL\left(  M\right)  }\left(
\mathbf{u}_{t}\right)  \quad\text{and\quad}u_{0}=g_{o}%
\]

Alternatively, every $\mathbf{z}\in CRP_{\mathbf{X}}\left(  \mathbb{R}%
^{d}\times gl\left(  d\right)  \right)  $ with $z_{0}=\left(  0,0\right)  $ on
the interval $\left[  0,T\right]  $ gives rise to $\mathbf{u}\in
CRP_{\mathbf{X}}\left(  GL\left(  M\right)  \right)  $ with $u_{0}=g_{o}$ on
with possible explosion time $\tau$ via the differential equation%
\[
d\mathbf{u}_{t}=Y_{d\mathbf{z}_{t}}^{GL\left(  M\right)  }\left(
\mathbf{u}_{t}\right)  \quad\text{and\quad}u_{0}=g_{o}.
\]
In this case, over $[0,\tau)$, we have
\[
\mathbf{z}:=\int\theta^{Y^{GL\left(  M\right)  }}\left(  d\mathbf{u}\right)
.
\]

\end{theorem}

We can now use Theorem \ref{the.10.5} to give an alternative characterization
of parallel translation.

\begin{theorem}
\label{the.par}Let $\mathbf{u}\in CRP_{\mathbf{X}}\left(  GL\left(  M\right)
\right)  $ such that $\pi_{\ast}\mathbf{u}=\mathbf{y}$. Then $\mathbf{u}$ is
an element of $CRP_{\mathbf{y}}^{H,\nabla}\left(  GL\left(  M\right)  \right)
$ if and only if there exists an $\mathbb{R}^{d}$-valued controlled rough path
$\mathbf{a}=\left(  a,a^{\dag}\right)  $ such that\footnote{We are abusing
notation slightly as we still need $u_{s}^{\dag}=B_{a_{s}^{\dag}\left(
\cdot\right)  }^{\nabla}\left(  u_{s}\right)  $ so that $\mathbf{u}$ is indeed
a rough path controlled by $\mathbf{X}$.}
\[
d\mathbf{u}_{t}=B_{d\mathbf{a}_{t}}^{\nabla}\left(  \mathbf{u}_{t}\right)
\]
where $B^{\nabla}$ are the horizontal vector fields introduced in Example
\ref{exa.10.3}.
\end{theorem}

\begin{proof}
If $d\mathbf{u}_{t}=B_{d\mathbf{a}_{t}}^{\nabla}\left(  \mathbf{u}_{t}\right)
$, then by Theorem \ref{the.5.3} we have
\begin{align*}
\int_{s}^{t}\omega^{\nabla}\left(  d\mathbf{u}\right)   &  \underset{^{3}%
}{\approx}\omega^{\nabla}\left(  B_{a_{s,t}}^{\nabla}\left(  u_{s}\right)
\right)  +B_{a_{s}^{\dag}w}\left(  u_{s}\right)  \left[  \omega^{\nabla}\circ
B_{a_{s}^{\dag}\tilde{w}}\right]  |_{w\otimes\tilde{w}=\mathbb{X}_{s,t}}\\
&  =0
\end{align*}
as $\omega^{\nabla}\circ B_{a}^{\nabla}=0$ for all $a\in\mathbb{R}^{d}$.
Additionally $\left[  \int\omega^{\nabla}\left(  d\mathbf{u}\right)  \right]
_{s}^{\dag}=\omega^{\nabla}\left(  B_{a_{s}^{\dag}\left(  \cdot\right)
}\left(  u_{s}\right)  \right)  =0.$

Conversely, if $\mathbf{u}\in CRP_{\mathbf{y}}^{H,\nabla}\left(  GL\left(
M\right)  \right)  $, then by Theorem \ref{the.10.5}, we have the existence of
$\mathbf{z}$ given by
\[
\mathbf{z}:=\int\theta^{Y^{GL\left(  M\right)  }}\left(  d\mathbf{u}\right)
\]
where $\theta^{Y^{GL\left(  M\right)  }}=\left(  \hat{\theta},\omega^{\nabla
}\right)  $. Additionally, we have
\[
d\mathbf{u}_{t}=Y_{d\mathbf{z}_{t}}^{GL\left(  M\right)  }\left(
\mathbf{u}_{t}\right)
\]
Let $\mathbf{a}=\pi_{1\ast}\mathbf{z}$ where $\pi_{1}:\mathbb{R}^{d}\times
gl\left(  d\right)  $ is projection onto the first component. To finish the
proof, we must verify that $Y_{z_{s,t}}^{GL\left(  M\right)  }=B_{\pi
_{1}\left(  z_{t}\right)  -\pi_{1}\left(  z_{s}\right)  }^{\nabla}$ and
$Y_{z_{s}^{\dag}w}^{GL\left(  M\right)  }Y_{z_{s}^{\dag}\tilde{w}}^{GL\left(
M\right)  }=B_{\pi_{1\ast}z_{s}^{\dag}w}^{\nabla}B_{\pi_{1}\ast z_{s}^{\dag
}\tilde{w}}^{\nabla}$. However, this is trivial due to the fact that
\begin{align*}
Y_{\left(  a,0\right)  }^{GL\left(  M\right)  }\left(  g\right)   &
=B_{a}^{\nabla}\left(  g\right)  +\tilde{0}\left(  g\right) \\
&  =B_{a}^{\nabla}\left(  g\right)
\end{align*}
and that $\mathbf{z}:=\int\theta^{Y^{GL\left(  M\right)  }}\left(
d\mathbf{u}\right)  =\int\left(  \hat{\theta},\omega^{\nabla}\right)  \left(
d\mathbf{u}\right)  =\left(  \int\hat{\theta}\left(  d\mathbf{u}\right)
,\mathbf{0}\right)  $.
\end{proof}

We can now put parallel translation and Theorem \ref{the.10.5} together to get
a correspondence of paths between $CRP_{\mathbf{X}}\left(  M\right)  $ and
$CRP_{\mathbf{X}}\left(  \mathbb{R}^{d}\right)  .$

\begin{corollary}
\label{cor.10.6}Let $\nabla$ be a covariant derivative on $M$, $o\in M$, and
$g_{o}\in GL\left(  M\right)  _{o}$. Let $\mathbf{h}\left(  \mathbf{y}%
,g_{y_{0}}\right)  $ be the map that takes a path $\mathbf{y}$ and initial
frame $g_{y_{0}}$ to its parallel translation. There is a one-to-one map from
$CRP_{\mathbf{X}}\left(  M\right)  $ starting at $o$ defined on $\left[
0,T\right]  $ and $CRP_{\mathbf{X}}\left(  \mathbb{R}^{d}\right)  $ starting
at $0$ defined on $\left[  0,T\right]  $ given by%
\[%
\begin{array}
[c]{ccccc}%
CRP_{\mathbf{X}}\left(  M\right)  & \longrightarrow & CRP_{\mathbf{y}%
}^{H,\nabla}\left(  GL\left(  M\right)  \right)  & \longrightarrow &
CRP_{\mathbf{X}}\left(  \mathbb{R}^{d}\right) \\
\mathbf{y} & \longrightarrow & \mathbf{h}\left(  \mathbf{y},g_{o}\right)  &
\longrightarrow & \int\hat{\theta}\left(  d\mathbf{h}\left(  \mathbf{y}%
,g_{o}\right)  \right)
\end{array}
\]
where $\hat{\theta}$ is the canonical one-form.
\end{corollary}

\subsection{Rolling and Unrolling of rough one-forms}

\marginpar{This section is still pretty messy}Let $\nabla$ be a covariant derivative.

\begin{lemma}
\label{lem.10.7}Let $\mathbf{y}=\left(  y,y^{\dag}\right)  \in CRP_{\mathbf{X}%
}\left(  M\right)  $ and $U^{\nabla}$ be the parallelism associated to
$\nabla$ and $\mathbf{u}=\left(  u,u^{\dag}\right)  $ be parallel translation
along $\mathbf{y}$\textbf{ }with any starting point. We have the approximation%
\begin{equation}
u_{t}-U_{y_{t},y_{s}}^{\nabla}u_{s}\underset{^{2}}{\approx}0. \label{equ.10.1}%
\end{equation}

\end{lemma}

\begin{proof}
These are all local statements, thus we may assume that we are working in
$\mathbb{R}^{d}\times GL\left(  d\right)  $. In that case, we have
$\nabla_{\left(  m,v_{m}\right)  }=\partial_{\left(  m,v_{m}\right)  }%
+\Gamma_{m}\left\langle v_{m}\right\rangle $ and may write Eq. (\ref{equ.10.1}%
) as $\left(  y_{t},g_{t}\right)  -$ $U_{y_{t},y_{s}}^{\nabla}\left(
y_{s},g_{s}\right)  $ where $U_{y_{t},y_{s}}^{\nabla}\left(  y_{s},g\right)
=\left(  y_{t},\bar{U}_{y_{t},y_{s}}^{\nabla}g\right)  $ for some $\bar
{U}:\left(  \mathbb{R}^{d}\right)  ^{2}\rightarrow\operatorname*{Aut}\left(
GL\left(  d\right)  \right)  $. Thus%
\begin{align*}
u_{t}-U_{y_{t},y_{s}}^{\nabla}u_{s}  &  =\left(  y_{t},g_{t}\right)
-U_{y_{t},y_{s}}^{\nabla}\left(  y_{s},g_{s}\right) \\
&  =\left(  0,g_{t}-\bar{U}_{y_{t},y_{s}}^{\nabla}g_{s}\right)  .
\end{align*}
Therefore we just need to show
\[
g_{t}-\bar{U}_{y_{t},y_{s}}^{\nabla}g_{s}\underset{^{2}}{\approx}0.
\]
By Lemma \ref{lem.6.5}, we have $\bar{U}_{y_{t},y_{s}}^{\nabla}\underset{^{2}%
}{\approx}I-\Gamma_{y_{s}}\left\langle y_{t}-y_{s}\right\rangle $ and
therefore
\[
g_{t}-\bar{U}_{y_{t},y_{s}}^{\nabla}g_{s}\underset{^{2}}{\approx}g_{t}-\left(
I-\Gamma_{y_{s}}\left\langle y_{t}-y_{s}\right\rangle \right)  g_{s}%
\]
On the other hand, by the proof of Theorem \ref{the.9.4}, we have that
\[
\theta_{r}\left(  \left[  g_{t}-g_{s}\right]  _{g_{s}}\right)  \underset{^{2}%
}{\approx}-\Gamma_{y_{s}}\left\langle y_{t}-y_{s}\right\rangle
\]
where $\theta_{r}\left(  \left[  g_{t}-g_{s}\right]  _{g_{s}}\right)
=g_{t}g_{s}^{-1}-I$ (we leave it to the reader to verify the $\Gamma$ in this
case indeed corresponds to $\Gamma$ that determines the covariant derivative).
Thus $g_{t}\underset{^{2}}{\approx}g_{s}-\left(  \Gamma_{y_{s}}\left\langle
y_{t}-y_{s}\right\rangle \right)  g_{s}$ and
\begin{align*}
g_{t}-\bar{U}_{y_{t},y_{s}}^{\nabla}g_{s}  &  \underset{^{2}}{\approx}%
g_{t}-\left(  I-\Gamma_{y_{s}}\left\langle y_{t}-y_{s}\right\rangle \right)
g_{s}\\
&  \underset{^{2}}{\approx}g_{s}-\left(  \Gamma_{y_{s}}\left\langle
y_{t}-y_{s}\right\rangle \right)  g_{s}-g_{s}+\Gamma_{y_{s}}\left\langle
y_{t}-y_{s}\right\rangle g_{s}\\
&  =0.
\end{align*}

\end{proof}

Given the Lemma, we have the following.

\begin{proposition}
\label{pro.10.8}Let $\mathbf{u}=\left(  u_{t},u_{t}^{\dag}\right)  \in
CRP_{\mathbf{X}}\left(  GL\left(  M\right)  \right)  $ be parallel translation
started at $u_{0}$ along $\mathbf{y}:=\left(  y,y^{\dag}\right)  $ with
respect to $\nabla$.

\begin{enumerate}
\item Let $\boldsymbol{\tilde{\alpha}}:=\left(  \tilde{\alpha},\tilde{\alpha
}^{\dag}\right)  \in CRP_{\mathbf{X}}\left(  L\left(  \mathbb{R}^{d},V\right)
\right)  $. Then $\boldsymbol{\alpha}^{\nabla}:=\left(  \alpha^{\nabla
},\left(  \alpha^{\dag}\right)  ^{\nabla}\right)  $ defined by
\begin{align*}
\alpha_{s}^{\nabla}  &  :=\tilde{\alpha}_{s}\circ u_{s}^{-1}\\
\left(  \alpha_{s}^{\dag}\right)  ^{\nabla}  &  :=\tilde{\alpha}_{s}^{\dag
}\circ\left(  I\otimes u_{s}^{-1}\right)
\end{align*}
is an element of $CRP_{y}^{U^{\nabla}}\left(  M,V\right)  .$

\item Let $\boldsymbol{\alpha}\in CRP_{y}^{U^{\nabla}}\left(  M,V\right)  .$
Then $\boldsymbol{\tilde{\alpha}}^{\nabla}:=\left(  \tilde{\alpha}^{\nabla
},\left(  \tilde{\alpha}^{\dag}\right)  ^{\nabla}\right)  $ defined by%
\begin{align*}
\tilde{\alpha}_{s}^{\nabla}  &  :=\alpha_{s}\circ u_{s}\\
\left(  \tilde{\alpha}_{s}^{\dag}\right)  ^{\nabla}  &  :=\alpha_{s}^{\dag
}\left(  I\otimes u_{s}\right)
\end{align*}
is an element of $CRP_{\mathbf{X}}\left(  L\left(  \mathbb{R}^{d},V\right)
\right)  .$
\end{enumerate}
\end{proposition}

\begin{proof}
In the first case (suppressing the $\nabla)$, we have%
\begin{align*}
&  \alpha_{t}\circ U_{y_{t},y_{s}}-\alpha_{s}-\alpha_{s}^{\dag}\left(
x_{s,t}\otimes\left(  \cdot\right)  \right) \\
&  \quad=\tilde{\alpha}_{t}\circ u_{t}^{-1}\circ U_{y_{t},y_{s}}-\tilde
{\alpha}_{s}\circ u_{s}^{-1}-\alpha_{s}^{\dag}\left(  x_{s,t}\otimes
u_{s}^{-1}\left(  \cdot\right)  \right) \\
&  \quad=\left[  \tilde{\alpha}_{t}-\tilde{\alpha}_{s}-\alpha_{s}^{\dag
}\left(  x_{s,t}\otimes\left(  \cdot\right)  \right)  \right]  u_{s}%
^{-1}+\tilde{\alpha}_{t}\left(  u_{t}^{-1}\circ U_{y_{t},y_{s}}-u_{s}%
^{-1}\right) \\
&  \quad\underset{^{2}}{\approx}\tilde{\alpha}_{t}\left(  u_{t}^{-1}\circ
U_{y_{t},y_{s}}-u_{s}^{-1}\right) \\
&  \quad\underset{^{2}}{\approx}0
\end{align*}
where the last step is from the Lemma. Additionally, we have%
\begin{align*}
\alpha_{t}^{\dag}\circ\left(  I\otimes U_{y_{t},y_{s}}\right)  -\alpha
_{s}^{\dag}  &  =\tilde{\alpha}_{t}^{\dag}\circ\left(  I\otimes u_{t}%
^{-1}U_{y_{t},y_{s}}\right)  -\tilde{\alpha}_{s}^{\dag}\circ\left(  I\otimes
u_{s}^{-1}\right) \\
&  \underset{^{1}}{\approx}\left(  \tilde{\alpha}_{t}^{\dag}-\tilde{\alpha
}_{s}^{\dag}\right)  \circ\left(  I\otimes u_{s}^{-1}\right)  +\tilde{\alpha
}_{t}^{\dag}\left(  I\otimes\left(  u_{t}^{-1}U_{y_{t},y_{s}}-u_{s}%
^{-1}\right)  \right) \\
&  \underset{^{1}}{\approx}0
\end{align*}

The second step is similar and also reduces to the validity of the lemma.
\end{proof}

Given a covariant derivative, we have a way to roll both a path and a rough
one-form onto flat space. The following theorem says that we get the same
answer if we do the integral on the manifold as we get if we roll both the
path and the rough one-form onto flat space and compute the integral there.

\begin{theorem}
\label{the.10.9}Let $\mathbf{y}=\left(  y,y^{\dag}\right)  \in CRP_{\mathbf{X}%
}\left(  M\right)  $, $\nabla$ a covariant derivative, and $u_{y_{0}}\in
GL\left(  M\right)  $. Further let $\boldsymbol{\alpha}\in CRP_{y}^{U^{\nabla
}}\left(  M,V\right)  $, $\boldsymbol{\tilde{\alpha}}^{\nabla}:=\left(
\tilde{\alpha}^{\nabla},\left(  \tilde{\alpha}^{\dag}\right)  ^{\nabla
}\right)  \in CRP_{\mathbf{X}}\left(  L\left(  \mathbb{R}^{d},V\right)
\right)  $, $\mathbf{\tilde{y}}:=\int\hat{\theta}\left(  d\mathbf{h}\left(
\mathbf{y},u_{o}\right)  \right)  \in CRP_{\mathbf{X}}\left(  \mathbb{R}%
^{d}\right)  $, $\psi$ be a logarithm and $\mathcal{G}=\left(  \psi,U^{\nabla
}\right)  $ Then%
\[
\int\left\langle \boldsymbol{\alpha},d\mathbf{y}^{\mathcal{G}}\right\rangle
=\int\left\langle \boldsymbol{\tilde{\alpha}}^{\nabla},d\mathbf{\tilde{y}%
}\right\rangle .
\]

\end{theorem}

\begin{proof}
Here we may trivialize appropriately to assume that $GL\left(  M\right)
=\mathbb{R}^{d}\times GL\left(  d\right)  $ and that $\nabla$ is in the form
$\nabla=\partial+\Gamma$. In this setting, we may write $h\left(
\mathbf{y},u_{0}\right)  =\mathbf{u}=\left(  \mathbf{y},\mathbf{g}\right)  $
for parallel translation in $\mathbb{R}^{d}\times GL\left(  d\right)  $. We
have (suppressing the $\nabla$)%
\begin{align}
\int_{s}^{t}\left\langle \boldsymbol{\tilde{\alpha}},d\mathbf{\tilde{y}%
}\right\rangle  &  \underset{^{3}}{\approx}\tilde{\alpha}_{s}\left(  \tilde
{y}_{s,t}\right)  +\tilde{\alpha}_{s}^{\dag}\left(  I\otimes\tilde{y}%
_{s}^{\dag}\right)  \mathbb{X}_{s,t}\\
&  =\alpha_{s}\left(  g_{s}\tilde{y}_{s,t}\right)  +\alpha_{s}^{\dag}\left(
I\otimes g_{s}\tilde{y}_{s}^{\dag}\right)  \mathbb{X}_{s,t}. \label{stop1}%
\end{align}
To continue the proof, we will have to expand $\mathbf{\tilde{y}}$ which we
will do after we introduce a few relevant geometrical objects.

Suppose $\psi$ is any logarithm on $M=\mathbb{R}^{d}$ and $\psi^{GL\left(
d\right)  }$ is any logarithm on $GL\left(  d\right)  $. It is easy to see
that $\left(  \psi,\psi^{GL\left(  d\right)  }\right)  $ is a logarithm on
$GL\left(  M\right)  =\mathbb{R}^{d}\times GL\left(  d\right)  $. In a similar
manner, let $\nabla^{GL\left(  d\right)  }$ be any covariant derivative on
$GL\left(  d\right)  $. Then $\left(  \nabla,\nabla^{GL\left(  d\right)
}\right)  $ is a covariant derivative on $\mathbb{R}^{d}\times GL\left(
d\right)  $ and $\left(  U^{\nabla},U^{\nabla^{GL\left(  d\right)  }}\right)
$ is a parallelism on $\mathbb{R}^{d}\times GL\left(  d\right)  .$ Putting it
all together we have $\left(  \mathcal{G},\mathcal{G}^{GL\left(  d\right)
}\right)  $ is a gauge on $\mathbb{R}^{d}\times GL\left(  d\right)  $ where
$\mathcal{G=}\left(  \psi,U^{\nabla}\right)  $ and $\mathcal{G}^{GL\left(
d\right)  }=$ $\left(  \psi^{GL\left(  d\right)  },U^{\nabla^{GL\left(
d\right)  }}\right)  $.

By the definition of $\mathbf{\tilde{y}}$, we have%
\begin{align*}
\tilde{y}_{s,t}  &  \underset{^{3}}{\approx}\hat{\theta}\left(  \left(
\psi,\psi^{GL\left(  d\right)  }\right)  \left(  \left(  y_{s},g_{s}\right)
,\left(  y_{t},g_{t}\right)  \right)  +\mathcal{S}^{\left(  \mathcal{G}%
,\mathcal{G}^{GL\left(  d\right)  }\right)  }\left(  y_{s}^{\dag},g_{s}^{\dag
}\right)  ^{\otimes2}\mathbb{X}_{s,t}\right) \\
&  \quad+\left[  \left(  \nabla,\nabla^{GL\left(  d\right)  }\right)
\hat{\theta}\right]  \left(  y_{s}^{\dag},g_{s}^{\dag}\right)  ^{\otimes
2}\mathbb{X}_{s,t}\\
&  =A+B
\end{align*}
where
\begin{align*}
A  &  :=\hat{\theta}\left(  \left(  \psi,\psi^{GL\left(  d\right)  }\right)
\left(  \left(  y_{s},g_{s}\right)  ,\left(  y_{t},g_{t}\right)  \right)
+\mathcal{S}^{\left(  \mathcal{G},\mathcal{G}^{GL\left(  d\right)  }\right)
}\left(  y_{s}^{\dag},g_{s}^{\dag}\right)  ^{\otimes2}\mathbb{X}_{s,t}\right)
\\
B  &  :=\left[  \left(  \nabla,\nabla^{GL\left(  d\right)  }\right)
\hat{\theta}\right]  \left(  y_{s}^{\dag},g_{s}^{\dag}\right)  ^{\otimes
2}\mathbb{X}_{s,t}.
\end{align*}
By the definition of $\hat{\theta},$ we have
\begin{align*}
A  &  =\hat{\theta}\left(  \left(  \psi,\psi^{GL\left(  d\right)  }\right)
\left(  \left(  y_{s},g_{s}\right)  ,\left(  y_{t},g_{t}\right)  \right)
+\mathcal{S}^{\left(  \mathcal{G},\mathcal{G}^{GL\left(  d\right)  }\right)
}\left(  y_{s}^{\dag},g_{s}^{\dag}\right)  ^{\otimes2}\mathbb{X}_{s,t}\right)
\\
&  =\hat{\theta}\left(  \left(  \psi\left(  y_{s},y_{t}\right)  ,\psi
^{GL\left(  d\right)  }\left(  g_{s},g_{t}\right)  \right)  +\left(
\mathcal{S}^{\mathcal{G}}y_{s}^{\dag\otimes2},\mathcal{S}^{\mathcal{G}%
^{GL\left(  d\right)  }}g_{s}^{\dag\otimes2}\right)  \mathbb{X}_{s,t}\right)
\\
&  =g_{s}^{-1}\left[  \psi\left(  y_{s},y_{t}\right)  +\mathcal{S}%
^{\mathcal{G}}y_{s}^{\dag\otimes2}\mathbb{X}_{s,t}\right]  .
\end{align*}
Additionally, we have%
\begin{align*}
\tilde{y}_{s}^{\dag}  &  =\hat{\theta}\left(  y_{s}^{\dag},g_{s}^{\dag}\right)
\\
&  =g_{s}^{-1}y_{s}^{\dag}.
\end{align*}
Plugging these terms into Eq. (\ref{stop1}), we have%
\begin{align*}
&  \int_{s}^{t}\left\langle \boldsymbol{\tilde{\alpha}},d\mathbf{\tilde{y}%
}\right\rangle \\
&  \quad\underset{^{3}}{\approx}\alpha_{s}\circ g_{s}\left(  A+B\right)
+\alpha_{s}^{\dag}\left(  I\otimes g_{s}g_{s}^{-1}y_{s}^{\dag}\right)
\mathbb{X}_{s,t}\\
&  \quad=\alpha_{s}\left(  \psi\left(  y_{s},y_{t}\right)  +\mathcal{S}%
^{\mathcal{G}}y_{s}^{\dag\otimes2}\mathbb{X}_{s,t}\right)  +\alpha_{s}^{\dag
}\left(  I\otimes y_{s}^{\dag}\right)  \mathbb{X}_{s,t}+\alpha_{s}\circ
g_{s}\left(  B\right) \\
&  \quad\underset{^{3}}{\approx}\int_{s}^{t}\left\langle \boldsymbol{\alpha
},d\mathbf{y}^{\mathcal{G}}\right\rangle +\alpha_{s}\circ g_{s}\left(
B\right)  .
\end{align*}
Thus, to prove that $\int_{s}^{t}\left\langle \boldsymbol{\tilde{\alpha}%
},d\mathbf{\tilde{y}}\right\rangle $ $\underset{^{3}}{\approx}\int_{s}%
^{t}\left\langle \boldsymbol{\alpha},d\mathbf{y}^{\mathcal{G}}\right\rangle $,
it is sufficient to prove that
\[
\left[  \left(  \nabla,\nabla^{GL\left(  d\right)  }\right)  \hat{\theta
}\right]  \left(  y_{s}^{\dag},g_{s}^{\dag}\right)  ^{\otimes2}\equiv0.
\]
\qquad

First we examine $g^{\dag}$. By Eq. (\ref{equ.9.5}) in the proof of Theorem
\ref{the.9.4}, we have that%
\begin{align*}
g_{s}^{\dag}  &  =F_{\Gamma_{y_{s}}\left\langle y_{s}^{\dag}\left(
\cdot\right)  \right\rangle }\left(  g_{s}\right) \\
&  =-\Gamma_{y_{s}}\left\langle y_{s}^{\dag}\left(  \cdot\right)
\right\rangle g_{s}.
\end{align*}
Letting $\left(  \mathbf{Y},\mathbf{G}\right)  $ be the vector field given by
\[
\mathbf{Y}:=U^{\nabla}\left(  \cdot,y_{s}\right)  y_{s}^{\dag}\tilde{w}%
\quad\text{and\quad}\mathbf{G}:=U^{\nabla^{GL\left(  d\right)  }}\left(
\cdot,g_{s}\right)  g_{s}^{\dag}\tilde{w}%
\]
for any vector $\tilde{w}\in W$, by Eq. (\ref{equ.6.11}), we have
$v_{m}\left[  \mathbf{Y}\right]  =-\Gamma_{y_{s}}\left\langle v_{m}%
\right\rangle \left\langle y_{s}\tilde{w}\right\rangle $. With these formulas,
we have
\begin{align*}
\left[  \left(  \nabla,\nabla^{GL\left(  d\right)  }\right)  \hat{\theta
}\right]  \left(  y_{s}^{\dag},g_{s}^{\dag}\right)  ^{\otimes2}\left(
w\otimes\tilde{w}\right)   &  =\left[  y_{s}^{\dag}w,g_{s}^{\dag}w\right]
\left[  \hat{\theta}\left(  \mathbf{Y},\mathbf{G}\right)  \right] \\
&  =\left[  y_{s}^{\dag}w,g_{s}^{\dag}w\right]  \left[  \left(  y,g\right)
\longrightarrow g^{-1}\mathbf{Y}\left(  y\right)  \right] \\
&  =-g_{s}^{-1}\left(  g_{s}^{\dag}w\right)  g_{s}^{-1}y_{s}^{\dag}\tilde
{w}-g_{s}^{-1}\Gamma_{y_{s}}\left\langle y_{s}^{\dag}w\right\rangle
\left\langle y_{s}^{\dag}\tilde{w}\right\rangle \\
&  =g_{s}^{-1}\Gamma_{y_{s}}\left\langle y_{s}^{\dag}w\right\rangle
\left\langle g_{s}g_{s}^{-1}y_{s}^{\dag}\tilde{w}\right\rangle -g_{s}%
^{-1}\Gamma_{y_{s}}\left\langle y_{s}^{\dag}w\right\rangle \left\langle
y_{s}^{\dag}\tilde{w}\right\rangle \\
&  =0.
\end{align*}

The fact that $\left[  \int\left\langle \boldsymbol{\alpha},d\mathbf{y}%
^{\mathcal{G}}\right\rangle \right]  _{s}^{\dag}=\left[  \int\left\langle
\boldsymbol{\tilde{\alpha}}^{\nabla},d\mathbf{\tilde{y}}\right\rangle \right]
_{s}^{\dag}$ is clear.
\end{proof}

\begin{proof}
[Alternative Proof?]\marginpar{Since we are trivializing down to
$\mathbb{R}^{d}\times GL\left(  d\right)  $, I\ suppose we might as well be as
lazy as possible?}Use instead the gauge $\psi\left(  y_{s},y_{t}\right)
=\left[  y_{t}-y_{s}\right]  _{y_{s}}$ and the usual derivative for $GL\left(
d\right)  $. Then%
\begin{align*}
\tilde{y}_{s,t}  &  \underset{^{3}}{\approx}\hat{\theta}\left(  \left[
y_{s,t}\right]  _{y_{s}},\psi^{GL\left(  d\right)  }\left(  g_{s}%
,g_{t}\right)  \right)  +\theta_{\left(  y_{s},g_{s}\right)  }^{\prime}\left(
\left(  y_{s}^{\dag},g_{s}^{\dag}\right)  ^{\otimes2}\mathbb{X}_{s,t}\right)
\\
&  =g_{s}^{-1}\left[  y_{s,t}\right]  _{y_{s}}+\frac{d}{dt}|_{0}%
\theta_{\left(  y_{s}+ty_{s}^{\dag}w,g_{s}+tg_{s}^{\dag}w\right)  }\left(
y_{s}^{\dag}\tilde{w},g_{s}^{\dag}w\right)  |_{w\otimes\tilde{w}%
=\mathbb{X}_{s,t}}\\
&  =g_{s}^{-1}\left[  y_{s,t}\right]  _{y_{s}}+\frac{d}{dt}|_{0}\theta\left(
\left[  y_{s}^{\dag}\tilde{w}\right]  _{y_{s}+ty_{s}^{\dag}w},\left[
g_{s}^{\dag}w\right]  _{g_{s}+tg_{s}^{\dag}w}\right)  |_{w\otimes\tilde
{w}=\mathbb{X}_{s,t}}\\
&  =g_{s}^{-1}\left[  y_{s,t}\right]  _{y_{s}}+\frac{d}{dt}|_{0}\left(
g_{s}+tg_{s}^{\dag}w\right)  ^{-1}y_{s}^{\dag}\tilde{w}|_{w\otimes\tilde
{w}=\mathbb{X}_{s,t}}\\
&  =g_{s}^{-1}\left[  y_{s,t}\right]  _{y_{s}}-g_{s}^{-1}\left(  g_{s}^{\dag
}w\right)  g_{s}^{-1}y_{s}\tilde{w}|_{w\otimes\tilde{w}=\mathbb{X}_{s,t}}%
\end{align*}
so that
\begin{align*}
\int_{s}^{t}\left\langle \boldsymbol{\tilde{\alpha}},d\mathbf{\tilde{y}%
}\right\rangle  &  \underset{^{3}}{\approx}\alpha_{s}\left(  g_{s}\tilde
{y}_{s,t}\right)  +\alpha_{s}^{\dag}\left(  I\otimes g_{s}\tilde{y}_{s}^{\dag
}\right)  \mathbb{X}_{s,t}\\
&  =\alpha_{s}\left(  y_{s,t}\right)  -\alpha_{s}\left(  \left(  g_{s}^{\dag
}w\right)  g_{s}^{-1}y_{s}\tilde{w}|_{w\otimes\tilde{w}=\mathbb{X}_{s,t}%
}\right)  +\alpha_{s}^{\dag}\left(  I\otimes g_{s}\tilde{y}_{s}^{\dag}\right)
\mathbb{X}_{s,t}%
\end{align*}
As before
\begin{align*}
g_{s}^{\dag}  &  =F_{\Gamma_{y_{s}}\left\langle y_{s}^{\dag}\left(
\cdot\right)  \right\rangle }\left(  g_{s}\right) \\
&  =-\Gamma_{y_{s}}\left\langle y_{s}^{\dag}\left(  \cdot\right)
\right\rangle g_{s}.
\end{align*}
so this is%
\begin{align*}
\int_{s}^{t}\left\langle \boldsymbol{\tilde{\alpha}},d\mathbf{\tilde{y}%
}\right\rangle  &  \underset{^{3}}{\approx}\alpha_{s}\left(  y_{s,t}\right)
+\alpha_{s}\left(  \left(  g_{s}^{\dag}w\right)  g_{s}^{-1}y_{s}\tilde
{w}|_{w\otimes\tilde{w}=\mathbb{X}_{s,t}}\right)  +\alpha_{s}^{\dag}\left(
I\otimes g_{s}\tilde{y}_{s}^{\dag}\right)  \mathbb{X}_{s,t}\\
&  =\alpha_{s}\left(  y_{s,t}\right)  +\alpha_{s}\left(  \mathcal{S}_{y_{s}%
}^{\mathcal{G}}\left\langle y_{s}^{\dag}w\right\rangle \left\langle
y_{s}\tilde{w}\right\rangle \right)  |_{w\otimes\tilde{w}=\mathbb{X}_{s,t}%
}+\alpha_{s}^{\dag}\left(  I\otimes g_{s}\tilde{y}_{s}^{\dag}\right)
\mathbb{X}_{s,t}\\
&  =\alpha_{s}\left(  y_{s,t}\right)  +\alpha_{s}\left(  \mathcal{S}%
^{\mathcal{G}}y_{s}^{\dag\otimes2}\right)  \mathbb{X}_{s,t}+\alpha_{s}^{\dag
}\left(  I\otimes y_{s}^{\dag}\right)  \mathbb{X}_{s,t}\\
&  \underset{^{3}}{\approx}\int\left\langle \boldsymbol{\alpha},d\mathbf{y}%
^{\mathcal{G}}\right\rangle .
\end{align*}
The fact that $\mathcal{S}^{\mathcal{G}}=\Gamma$ when $\psi\left(  x,y\right)
=\left[  y-x\right]  _{x}$ and $U=U^{\nabla}$ follows straight from its definition.
\end{proof}

\section{Bone Yard A. [My first brute force attempts follow.]\label{sec.11}}

The results of this sections could surely be pushed through to give a proof of
existence and uniqueness of horizontal lifts. However the computations were
becoming to ugly and so I changed tact which resulted in the proof given above
which I\ thinks is more appealing and fairly natural. Nevertheless,
Proposition \ref{p.bd.10} below could be used as an example of a gauge with
its compatibility tensor computed out and should probably still appear earlier
in the paper.

\begin{proposition}
\label{pro.11.1}A controlled rough path $\mathbf{y}=\left(  y,y^{\dag}\right)
\in CRP\left(  M\right)  $ solves
\[
d\mathbf{y}=F_{d\mathbf{x}}\left(  y\right)
\]
iff for any logarithm, $\psi,$ on $M$ we have%
\[
\psi\left(  y\left(  s\right)  ,y\left(  t\right)  \right)  =_{3}F_{x_{s,t}%
}\left(  y\left(  s\right)  \right)  +\left(  \nabla_{F_{a}}^{\psi}%
F_{b}\right)  \left(  y\left(  s\right)  \right)  |_{a\otimes b=\mathbb{X}%
_{s,t}}.
\]

\end{proposition}

\begin{proof}
This is a restatement of Theorem \ref{the.5.8} where it has already been
proved. Nevertheless, let us give a heuristic proof in order to motivate the result.

We start with the identity,%
\[
\frac{d}{dt}\psi\left(  y\left(  s\right)  ,y\left(  t\right)  \right)
=\psi\left(  y\left(  s\right)  ,\cdot\right)  _{\ast y\left(  t\right)  }%
\dot{y}\left(  t\right)  =U^{\psi}\left(  y\left(  s\right)  ,y\left(
t\right)  \right)  F_{\dot{x}\left(  t\right)  }\left(  y\left(  t\right)
\right)  .
\]
Using the fundamental theorem of calculus a couple of times and making
reasonable estimates shows%
\begin{align*}
\psi\left(  y\left(  s\right)  ,y\left(  t\right)  \right)   &  =\int_{s}%
^{t}U^{\psi}\left(  y\left(  s\right)  ,y\left(  \tau\right)  \right)
F_{\dot{x}\left(  \tau\right)  }\left(  y\left(  \tau\right)  \right)  d\tau\\
&  =\int_{s}^{t}d\tau\left[  F_{\dot{x}\left(  \tau\right)  }\left(  y\left(
s\right)  \right)  +\int_{s}^{\tau}d\sigma\frac{d}{d\sigma}U^{\psi}\left(
y\left(  s\right)  ,y\left(  \sigma\right)  \right)  F_{\dot{x}\left(
\tau\right)  }\left(  y\left(  \sigma\right)  \right)  \right] \\
&  \underset{^{3}}{\approx}\int_{s}^{t}d\tau\left[  F_{\dot{x}\left(
\tau\right)  }\left(  y\left(  s\right)  \right)  +\int_{s}^{\tau}d\sigma
U^{\psi}\left(  y\left(  s\right)  ,y\left(  \sigma\right)  \right)  \left(
\nabla_{\dot{y}\left(  \sigma\right)  }F_{\dot{x}\left(  \tau\right)
}\right)  \left(  y\left(  \sigma\right)  \right)  \right] \\
&  \underset{^{3}}{\approx}\int_{s}^{t}d\tau\left[  F_{\dot{x}\left(
\tau\right)  }\left(  y\left(  s\right)  \right)  +\int_{s}^{\tau}%
d\sigma\left(  \nabla_{\dot{y}\left(  \sigma\right)  }F_{\dot{x}\left(
\tau\right)  }\right)  \left(  y\left(  s\right)  \right)  \right] \\
&  =\int_{s}^{t}d\tau\left[  F_{\dot{x}\left(  \tau\right)  }\left(  y\left(
s\right)  \right)  +\int_{s}^{\tau}d\sigma\left(  \nabla_{F_{\dot{x}\left(
\sigma\right)  }\left(  y\left(  \sigma\right)  \right)  }F_{\dot{x}\left(
\tau\right)  }\right)  \left(  y\left(  s\right)  \right)  \right] \\
&  =F_{x_{s,t}}\left(  y\left(  s\right)  \right)  +\left(  \nabla_{F_{a}%
}^{\psi}F_{b}\right)  \left(  y\left(  s\right)  \right)  |_{a\otimes
b=\mathbb{X}_{s,t}}.
\end{align*}

\end{proof}

\begin{theorem}
[Existence of Horizontal Lifts]\label{the.11.2}Let $G\rightarrow
P\overset{\pi}{\rightarrow}M$ be a principal bundle with connection $\omega,$
$\mathbf{y}=\left(  y,y^{\dag}\right)  \in CRP\left(  M\right)  ,$ and
$u_{0}\in P_{y\left(  0\right)  }.$ Then there exists a unique horizontal lift
$\mathbf{u}=\left(  u,u^{\dag}\right)  \in CRP\left(  P\right)  $ above
$\mathbf{y}$ such that $u\left(  0\right)  =u_{0}.$
\end{theorem}

\begin{proof}
The proof follows the lines of the smooth case. Because of Proposition
\ref{pro.9.3} and simple patching arguments we may reduce to considering the
case that $P=M\times G$ is a trivial bundle where $M$ is now an open subset of
$\mathbb{R}^{d}.$ In light of Lemma \ref{lem.8.1}, the desired horizontal lift
may be expressed in the form, $\mathbf{u}_{s}=\left(  \mathbf{y}%
_{s},\mathbf{g}_{s}\right)  $ for some $\mathbf{g}=\left(  g,g^{\dag}\right)
\mathbf{\in}CRP\left(  G\right)  $ to be determined. We now find the equations
that $\mathbf{g}$ must satisfy. To this end we must write out what it means
for $\int\omega\left(  d\mathbf{u}\right)  \equiv0.$ To do this we take the
product of the standard gauge on $M\subset\mathbb{R}^{d}$ and the gauge
$\mathcal{G}^{\nabla}$ on $G$ introduced in Proposition \ref{p.bd.10} above.
Recall that%
\[
\omega\left(  \left(  v_{m},\xi_{g}\right)  \right)  =\theta\left(  \xi
_{g}\right)  +Ad_{g^{-1}}\Gamma\left(  v_{m}\right)
\]
and thus%
\begin{align*}
0  &  =\omega\circ u_{s}^{\dag}w=\omega\left(  \left(  y_{s}^{\dag}%
w,g_{s}^{^{\dag}}w\right)  \right) \\
&  =\theta\left(  g_{s}^{^{\dag}}w\right)  +Ad_{g_{s}^{-1}}\Gamma\left(
\left[  y_{s}^{\dag}w\right]  _{y_{s}}\right) \\
&  =g_{s}^{-1}g_{s}^{^{\dag}}w+Ad_{g_{s}^{-1}}\Gamma\left(  \left[
y_{s}^{\dag}w\right]  _{y_{s}}\right) \\
&  =\theta\left(  g_{s}^{\dag}w\right)  +Ad_{g_{s}^{-1}}\Gamma\left(  \left[
y_{s}^{\dag}w\right]  _{y_{s}}\right)  .
\end{align*}
Therefore,%
\begin{align*}
T^{\nabla}\left[  g_{s}^{\dag\otimes2}\mathbb{X}_{s,t}\right]   &
=-L_{g_{s}\ast}\left[  Ad_{g_{s}^{-1}}\Gamma y_{s}^{\dag}a,Ad_{g_{s}^{-1}%
}\Gamma y_{s}^{\dag}b\right]  |_{a\otimes b=\mathbb{X}_{s,t}}\\
&  =-L_{g_{s}\ast}Ad_{g_{s}^{-1}}\left[  \Gamma y_{s}^{\dag}a,\Gamma
y_{s}^{\dag}b\right]  |_{a\otimes b=\mathbb{X}_{s,t}}\\
&  =-R_{g_{s}\ast}\left[  \Gamma y_{s}^{\dag}a,\Gamma y_{s}^{\dag}b\right]
|_{a\otimes b=\mathbb{X}_{s,t}}.
\end{align*}
and so
\begin{align*}
\omega\left(  T^{\nabla}\left[  g_{s}^{\dag\otimes2}\mathbb{X}_{s,t}\right]
\right)   &  =-\omega\left(  R_{g_{s}\ast}\left[  \Gamma y_{s}^{\dag}a,\Gamma
y_{s}^{\dag}b\right]  |_{a\otimes b=\mathbb{X}_{s,t}}\right) \\
&  =-Ad_{g_{s}^{-1}}\omega\left(  \left[  \Gamma y_{s}^{\dag}a,\Gamma
y_{s}^{\dag}b\right]  |_{a\otimes b=\mathbb{X}_{s,t}}\right) \\
&  =-Ad_{g_{s}^{-1}}\left[  \Gamma y_{s}^{\dag}a,\Gamma y_{s}^{\dag}b\right]
|_{a\otimes b=\mathbb{X}_{s,t}}.
\end{align*}
Using this information we find,%
\begin{align*}
0  &  =\int_{s}^{t}\omega\left(  \mathbf{du}\right)  =_{2}\omega\left(
\left(  y_{s,t}\right)  _{y_{s}},\psi\left(  g_{s},g_{t}\right)  \right)
+\omega\left(  S\left[  \left(  y_{s}^{\dag},g_{s}^{\dag}\right)  ^{\otimes
2}\mathbb{X}_{s,t}\right]  \right)  +\left(  \nabla\omega\right)  \left(
y_{s}^{\dag},g_{s}^{\dag}\right)  ^{\otimes2}\mathbb{X}_{s,t}\\
&  =Ad_{g_{s}^{-1}}\Gamma\left(  \left(  y_{s,t}\right)  _{y_{s}}\right)
+\theta\left(  \psi\left(  g_{s},g_{t}\right)  \right)  +\frac{1}{2}%
Ad_{g_{s}^{-1}}\left[  \Gamma y_{s}^{\dag}a,\Gamma y_{s}^{\dag}b\right]
|_{a\otimes b=\mathbb{X}_{s,t}}+\left(  \nabla\omega\right)  \left(
y_{s}^{\dag},g_{s}^{\dag}\right)  ^{\otimes2}\mathbb{X}_{s,t}\\
&  =Ad_{g_{s}^{-1}}\Gamma\left(  \left(  y_{s,t}\right)  _{y_{s}}\right)
+\log\left(  g_{s}^{-1}g_{t}\right)  +\frac{1}{2}Ad_{g_{s}^{-1}}\left[  \Gamma
y_{s}^{\dag}a,\Gamma y_{s}^{\dag}b\right]  |_{a\otimes b=\mathbb{X}_{s,t}%
}+\left(  \nabla\omega\right)  \left(  y_{s}^{\dag},g_{s}^{\dag}\right)
^{\otimes2}\mathbb{X}_{s,t}.
\end{align*}
Moreover we have
\begin{align*}
\left(  \nabla_{\left(  w_{m},gB\right)  }\omega\right)  \left(  v_{m},\xi
_{m}\right)   &  =\frac{d}{dt}|_{0}\omega_{\left(  m+tw,ge^{tB}\right)  }%
\circ\left(  v_{m+tw},L_{e^{tB}\ast}\xi_{m}\right) \\
&  =\frac{d}{dt}|_{0}\left[  \theta\left(  L_{e^{tB}\ast}\xi_{m}\right)
+Ad_{\left(  ge^{tB}\right)  ^{-1}}\Gamma\left(  v_{m+tw}\right)  \right] \\
&  =\frac{d}{dt}|_{0}\left[  \theta\left(  \xi_{m}\right)  +Ad_{\left(
ge^{tB}\right)  ^{-1}}\Gamma\left(  v_{m+tw}\right)  \right] \\
&  =-ad_{B}Ad_{g^{-1}}\Gamma\left(  v_{m}\right)  +Ad_{g^{-1}}\partial
_{w}\Gamma\left(  v_{m}\right)
\end{align*}
and therefore,%
\begin{align*}
\left(  \nabla\omega\right)  \left(  y_{s}^{\dag},g_{s}^{\dag}\right)
^{\otimes2}\mathbb{X}_{s,t}  &  =\left(  \nabla_{\left(  y_{s}^{\dag}%
a,g_{s}^{\dag}a\right)  }\omega\right)  \left(  \left(  y_{s}^{\dag}%
b,g_{s}^{\dag}b\right)  \right)  |_{a\otimes b=\mathbb{X}_{s,t}}\\
&  =\left[  -ad_{\theta\left(  g_{s}^{\dag}a\right)  }Ad_{g_{s}^{-1}}%
\Gamma\left(  y_{s}^{\dag}b\right)  +Ad_{g_{s}^{-1}}\left(  \partial
_{y_{s}^{\dag}a}\Gamma\right)  \left(  y_{s}^{\dag}b\right)  \right]
_{_{a\otimes b=\mathbb{X}_{s,t}}}.
\end{align*}
Putting this all together we find,%
\begin{align*}
0  &  =Ad_{g_{s}^{-1}}\Gamma\left(  \left(  y_{s,t}\right)  _{y_{s}}\right)
+\log\left(  g_{s}^{-1}g_{t}\right)  +\frac{1}{2}Ad_{g_{s}^{-1}}\left[  \Gamma
y_{s}^{\dag}a,\Gamma y_{s}^{\dag}b\right]  |_{a\otimes b=\mathbb{X}_{s,t}%
}+\left(  \nabla\omega\right)  \left(  y_{s}^{\dag},g_{s}^{\dag}\right)
^{\otimes2}\mathbb{X}_{s,t}\\
&  =Ad_{g_{s}^{-1}}\Gamma\left(  \left(  y_{s,t}\right)  _{y_{s}}\right)
+\log\left(  g_{s}^{-1}g_{t}\right)  +\frac{1}{2}Ad_{g_{s}^{-1}}\left[  \Gamma
y_{s}^{\dag}a,\Gamma y_{s}^{\dag}b\right]  |_{a\otimes b=\mathbb{X}_{s,t}}\\
&  +\left[  -ad_{\theta\left(  g_{s}^{\dag}a\right)  }Ad_{g_{s}^{-1}}%
\Gamma\left(  y_{s}^{\dag}b\right)  +Ad_{g_{s}^{-1}}\left(  \partial
_{y_{s}^{\dag}a}\Gamma\right)  \left(  y_{s}^{\dag}b\right)  \right]
_{_{a\otimes b=\mathbb{X}_{s,t}}}.
\end{align*}%
\[
\omega\left(  \left(  v_{m},\xi_{g}\right)  \right)  =\theta\left(  \xi
_{g}\right)  +Ad_{g^{-1}}\Gamma\left(  v_{m}\right)
\]
where $\nabla_{\left(  v,\xi\right)  }=\partial_{v}+\nabla_{\xi}^{\psi}$ is
the covariant derivative associated to the product gauge. Recall from Eq.
(\ref{equ.7.2}) that%
\[
\omega\left(  \left(  v_{m},\xi_{g}\right)  \right)  =\theta\left(  \xi
_{g}\right)  +Ad_{g^{-1}}\Gamma\left(  v_{m}\right)
\]
so with our choices%
\[
\theta\left(  \psi\left(  g_{s},g_{t}\right)  \right)  =\log\left(  g_{s}%
^{-1}g_{t}\right)
\]
and hence
\[
\omega\left(  \left(  y_{s,t}\right)  _{y_{s}},\psi\left(  g_{s},g_{t}\right)
\right)  =\log\left(  g_{s}^{-1}g_{t}\right)  +Ad_{g_{s}^{-1}}\Gamma\left(
\left(  y_{s,t}\right)  _{y_{s}}\right)  .
\]
The $\nabla\omega$ term is more painful to work out.
\[
\nabla_{\left(  w_{m},\eta_{m}\right)  }\omega=\left(  \partial_{w_{m}}%
+\nabla_{\eta_{m}}^{\psi}\right)  \theta+\nabla_{\xi}^{\psi}\theta
\]
Thus%
\[
\omega\left(  \left(  y_{s,t}\right)  _{y_{s}},\psi\left(  g_{s},g_{t}\right)
\right)  =\theta\left(  \psi\left(  g_{s},g_{t}\right)  \right)
+Ad_{g_{s}^{-1}}\Gamma\left(  \left(  y_{s,t}\right)  _{y_{s}}\right)
\]
while .......
\end{proof}

\section{Bone Yard B. Old Vector Bundle Rough Parallel
Translation\label{sec.12}}

[Old into for this section follows: Section \ref{sec.12} (now \ref{sec.6})
develops the notion of rough parallel translation along manifold valued rough
paths. We take a general perspective, defining (and proving the existence of)
rough parallel translation along any trivial vector bundle and obtain results
for parallel translation along a rough path in $M$ as a special case. A key
ingredient to our construction is a non-explosion result for the RDEs driven
along right-invariant vector fields on a Lie group. The very useful chain rule
involving rough parallel translation and covariant differentiation is
explained in Theorem \ref{the.12.11}.]

\bigskip

In this chapter we will develop the notion of parallel translation along an
$M$ -- valued rough path. We begin with parallel translation in the trivial
bundle $M\times G$ where $G$ is a Lie group. We then use these results with
$G=O\left(  N\right)  $ in order to construct parallel translation on $TM.$

\subsection{Parallel Translation on Trivial Principle Bundles\label{sub.12.1}}

\begin{notation}
\label{not.12.1}Let $\mathfrak{g}$ denote the Lie algebra of $G$ and for
$A\in\mathfrak{g},$ let $\hat{A}$ be the right invariant vector field on $G$
agreeing with $A$ at $g=e,$ i.e.
\begin{equation}
\hat{A}\left(  g\right)  :=\frac{d}{dt}|_{0}e^{tA}g=R_{g\ast}A\text{ for all
}g\in G. \label{equ.12.1}%
\end{equation}
Furthermore, if $F\in C^{\infty}\left(  M\times G\right)  $ and $\left(
x,g\right)  \in M\times G$ we let $F^{\prime}\left(  x,g\right)
\in\mathfrak{g}^{\ast}$ be defined by
\begin{equation}
F^{\prime}\left(  x,g\right)  A:=\left(  \hat{A}F\right)  \left(  x,g\right)
=\frac{d}{dt}|_{0}F\left(  x,e^{tA}g\right)  . \label{equ.12.2}%
\end{equation}

\end{notation}

\begin{notation}
\label{not.12.2}Let $Y:\mathbb{R}^{n}\rightarrow\Gamma\left(  TM\right)  $ be
a linear map from $\mathbb{R}^{n}$ into the vector fields, $\Gamma\left(
TM\right)  ,$ on $M.$ (In more detail, for each $a\in\mathbb{R}^{n}$ we have
$Y_{a}:M\rightarrow E$ such that $Q\left(  m\right)  Y\left(  m\right)  a=0$
for all $m\in M$ and $\nu\in\mathbb{R}^{n}.$ Moreover to each connection
one-form $\Gamma\in\Omega^{1}\left(  M,\mathfrak{g}\right)  $ ($\mathfrak{g}%
=\operatorname*{Lie}\left(  G\right)  $ as above) and $a\in\mathbb{R}^{n}$ we
let $Y_{a}^{\Gamma}=Y_{a}-\widehat{\Gamma\left(  Y_{a}\right)  }\in
\Gamma\left(  T\left(  M\times G\right)  \right)  $ which is explicitly defied
by
\begin{equation}
Y_{a}^{\Gamma}\left(  m,g\right)  :=\left(  Y_{a}\left(  m\right)
,Y_{\Gamma\left(  Y_{a}\left(  m\right)  \right)  }^{G}\left(  g\right)
\right)  =\left(  Y_{a}\left(  m\right)  ,-R_{g\ast}\Gamma\left(  Y_{a}\left(
m\right)  \right)  \right)  . \label{equ.12.3}%
\end{equation}

\end{notation}

\begin{lemma}
\label{lem.12.3}For all $a,b\in\mathbb{R}^{n}$ we have%
\begin{equation}
\left[  Y_{a},\widehat{\Gamma\left(  Y_{b}\right)  }\right]  =\widehat{Y_{a}%
\Gamma\left(  Y_{b}\right)  } \label{equ.12.4}%
\end{equation}
and%
\begin{equation}
Y_{a}^{\Gamma}Y_{b}^{\Gamma}=Y_{a}Y_{b}-Y_{b}\widehat{\Gamma\left(
Y_{a}\right)  }-Y_{a}\widehat{\Gamma\left(  Y_{b}\right)  }+\widehat{Y_{b}%
\Gamma\left(  Y_{a}\right)  }+\widehat{\Gamma\left(  Y_{a}\right)
}\widehat{\Gamma\left(  Y_{b}\right)  }. \label{equ.12.5}%
\end{equation}

\end{lemma}

\begin{proof}
If $F\in C^{\infty}\left(  M\times G\right)  $ and $\left(  x,g\right)  \in
M\times G,$ then
\begin{align*}
Y_{a}\left(  \widehat{\Gamma\left(  Y_{b}\right)  }F\right)  \left(
x,g\right)   &  =Y_{a}\left(  x\rightarrow F^{\prime}\left(  x,g\right)
\Gamma\left(  Y_{b}\left(  x\right)  \right)  \right) \\
&  =\left(  Y_{a}F^{\prime}\right)  \left(  x,g\right)  \Gamma\left(
Y_{b}\left(  x\right)  \right)  +F^{\prime}\left(  x,g\right)  Y_{a}\left(
x\right)  \Gamma\left(  Y_{b}\right) \\
&  =\left(  Y_{a}F\right)  ^{\prime}\left(  x,g\right)  \Gamma\left(
Y_{b}\left(  x\right)  \right)  +F^{\prime}\left(  x,g\right)  Y_{a}\left(
x\right)  \Gamma\left(  Y_{b}\right) \\
&  =\left(  \widehat{\Gamma\left(  Y_{b}\right)  }Y_{a}F\right)  \left(
x,g\right)  +\left(  \widehat{Y_{a}\left(  x\right)  \Gamma\left(
Y_{b}\right)  }F\right)  \left(  x,g\right)
\end{align*}
from which Eq. (\ref{equ.12.4}) follows. The computation,
\begin{align*}
Y_{a}^{\Gamma}Y_{b}^{\Gamma}  &  =\left(  Y_{a}-\widehat{\Gamma\left(
Y_{a}\right)  }\right)  \left(  Y_{b}-\widehat{\Gamma\left(  Y_{b}\right)
}\right) \\
&  =Y_{a}Y_{b}-\widehat{\Gamma\left(  Y_{a}\right)  }Y_{b}-Y_{a}%
\widehat{\Gamma\left(  Y_{b}\right)  }+\widehat{\Gamma\left(  Y_{a}\right)
}\widehat{\Gamma\left(  Y_{b}\right)  }\\
&  =Y_{a}Y_{b}-\left(  Y_{b}\widehat{\Gamma\left(  Y_{a}\right)  }+\left[
\widehat{\Gamma\left(  Y_{a}\right)  },Y_{b}\right]  \right)  -Y_{a}%
\widehat{\Gamma\left(  Y_{b}\right)  }+\widehat{\Gamma\left(  Y_{a}\right)
}\widehat{\Gamma\left(  Y_{b}\right)  }\\
&  =Y_{a}Y_{b}-Y_{b}\widehat{\Gamma\left(  Y_{a}\right)  }-Y_{a}%
\widehat{\Gamma\left(  Y_{b}\right)  }+\left[  Y_{b},\widehat{\Gamma\left(
Y_{a}\right)  }\right]  +\widehat{\Gamma\left(  Y_{a}\right)  }\widehat{\Gamma
\left(  Y_{b}\right)  }\\
&  =Y_{a}Y_{b}-Y_{b}\widehat{\Gamma\left(  Y_{a}\right)  }-Y_{a}%
\widehat{\Gamma\left(  Y_{b}\right)  }+\widehat{Y_{b}\Gamma\left(
Y_{a}\right)  }+\widehat{\Gamma\left(  Y_{a}\right)  }\widehat{\Gamma\left(
Y_{b}\right)  },
\end{align*}
then proves Eq. (\ref{equ.12.5}).
\end{proof}

\begin{theorem}
\label{the.12.4}Let $\mathbf{Z\in}WG_{p}\left(  \mathbb{R}^{n}\right)  $ and
$F\in C^{\infty}\left(  M\times G\right)  .$ If $\mathbf{U=}\left(  u:=\left(
x,g\right)  ,\mathbb{U}\right)  \in WG_{p}\left(  M\times G\right)  $ solves
the RDE,%
\begin{equation}
d\mathbf{U}_{t}=Y_{d\mathbf{Z}_{t}}^{\Gamma}\left(  x_{t},g_{t}\right)  \text{
with }\left(  x\left(  0\right)  ,g\left(  0\right)  \right)  =\left(
x_{0},e\right)  \in M\times G, \label{equ.12.6}%
\end{equation}
then
\begin{equation}
\left[  F\left(  u\right)  \right]  _{s,t}\simeq\left(  \left(  \mathcal{Y}%
_{\mathbf{Z}_{s,t}}+\mathcal{Y}_{\mathbf{Z}_{s,t}}^{G}\right)  F\right)
\left(  u_{s}\right)  -\left(  \widehat{Y_{a}\Gamma\left(  Y_{b}\right)
}\right)  F\left(  u_{s}\right)  |_{a\otimes b=\mathbb{Z}_{s,t}}%
-\widehat{\Gamma\left(  Y_{z_{s,t}}\right)  }Y_{z_{s,t}}F. \label{equ.12.7}%
\end{equation}

\end{theorem}

\begin{proof}
By Theorem \ref{the.4.5},%
\[
\left[  F\left(  u\right)  \right]  _{s,t}\simeq\left(  Y_{z_{s,t}}^{\Gamma
}F\right)  \left(  u_{s}\right)  +\left(  Y_{\left(  \cdot\right)  }^{\Gamma
}Y_{\left(  \cdot\right)  }^{\Gamma}F\right)  \left(  u_{s}\right)
\mathbb{Z}_{s,t}.
\]
Since%
\[
\left(  Y_{z_{s,t}}^{\Gamma}F\right)  \left(  u_{s}\right)  =\left(
Y_{z_{s,t}}F\right)  \left(  u_{s}\right)  -\left(  \widehat{\Gamma\left(
Y_{z_{s,t}}\right)  }F\right)  \left(  u_{s}\right)
\]
and
\[
\left(  Y_{\left(  \cdot\right)  }^{\Gamma}Y_{\left(  \cdot\right)  }^{\Gamma
}F\right)  \left(  u_{s}\right)  \mathbb{Z}_{s,t}=\left(  Y_{a}Y_{b}%
-Y_{b}\widehat{\Gamma\left(  Y_{a}\right)  }-Y_{a}\widehat{\Gamma\left(
Y_{b}\right)  }+\widehat{Y_{b}\Gamma\left(  Y_{a}\right)  }+\widehat{\Gamma
\left(  Y_{a}\right)  }\widehat{\Gamma\left(  Y_{b}\right)  }\right)  F\left(
u_{s}\right)  |_{a\otimes b=\mathbb{Z}_{s,t}},
\]
it follows that%
\begin{align}
F\left(  u\right)  _{s,t}  &  \simeq\left(  \left(  \mathcal{Y}_{\mathbf{Z}%
_{s,t}}+\mathcal{Y}_{\mathbf{Z}_{s,t}}^{G}\right)  F\right)  \left(
u_{s}\right)  -\left(  Y_{b}\widehat{\Gamma\left(  Y_{a}\right)  }%
+Y_{a}\widehat{\Gamma\left(  Y_{b}\right)  }-\widehat{Y_{b}\Gamma\left(
Y_{a}\right)  }\right)  F\left(  u_{s}\right)  |_{a\otimes b=\mathbb{Z}_{s,t}%
}\nonumber\\
&  =\left(  \left(  \mathcal{Y}_{\mathbf{Z}_{s,t}}+\mathcal{Y}_{\mathbf{Z}%
_{s,t}}^{G}\right)  F\right)  \left(  u_{s}\right)  -\left(  Y_{z_{s,t}%
}\widehat{\Gamma\left(  Y_{z_{s,t}}\right)  }F\right)  \left(  u_{s}\right)
+\left(  \widehat{Y_{b}\Gamma\left(  Y_{a}\right)  }\right)  F\left(
u_{s}\right)  |_{a\otimes b=\mathbb{Z}_{s,t}}. \label{equ.12.8}%
\end{align}

We now finish the proof by showing Eq. (\ref{equ.12.8}) is in fact the same as
Eq. (\ref{equ.12.7}). To this end we have%
\begin{align*}
\left(  \widehat{Y_{b}\Gamma\left(  Y_{a}\right)  }\right)  F\left(
u_{s}\right)  |_{a\otimes b=\mathbb{Z}_{s,t}}  &  =\left(  \widehat{Y_{b}%
\Gamma\left(  Y_{a}\right)  }\right)  F\left(  u_{s}\right)  |_{a\otimes
b=\mathbb{Z}_{s,t}^{a}}+\left(  \widehat{Y_{b}\Gamma\left(  Y_{a}\right)
}\right)  F\left(  u_{s}\right)  |_{a\otimes b=\frac{1}{2}z_{s,t}\otimes
z_{s,t}}\\
&  =-\left(  \widehat{Y_{a}\Gamma\left(  Y_{b}\right)  }\right)  F\left(
u_{s}\right)  |_{a\otimes b=\mathbb{Z}_{s,t}^{a}}+\left(  \widehat{Y_{a}%
\Gamma\left(  Y_{b}\right)  }\right)  F\left(  u_{s}\right)  |_{a\otimes
b=\frac{1}{2}z_{s,t}\otimes z_{s,t}}%
\end{align*}
and%
\begin{align*}
Y_{z_{s,t}}\widehat{\Gamma\left(  Y_{z_{s,t}}\right)  }F  &
=\widehat{Y_{z_{s,t}}\Gamma\left(  Y_{z_{s,t}}\right)  }F+\widehat{\Gamma
\left(  Y_{z_{s,t}}\right)  }Y_{z_{s,t}}F\\
&  =\left(  \widehat{Y_{a}\Gamma\left(  Y_{b}\right)  }\right)  F\left(
u_{s}\right)  |_{a\otimes b=z_{s,t}\otimes z_{s,t}}+\widehat{\Gamma\left(
Y_{z_{s,t}}\right)  }Y_{z_{s,t}}F.
\end{align*}
From these last two displayed equations we conclude that%
\begin{align*}
\left(  \widehat{Y_{b}\Gamma\left(  Y_{a}\right)  }\right)   &  F\left(
u_{s}\right)  |_{a\otimes b=\mathbb{Z}_{s,t}}-\left(  Y_{z_{s,t}%
}\widehat{\Gamma\left(  Y_{z_{s,t}}\right)  }F\right)  \left(  u_{s}\right) \\
&  =-\left(  \widehat{Y_{a}\Gamma\left(  Y_{b}\right)  }\right)  F\left(
u_{s}\right)  |_{a\otimes b=\mathbb{Z}_{s,t}^{a}}-\left(  \widehat{Y_{a}%
\Gamma\left(  Y_{b}\right)  }\right)  F\left(  u_{s}\right)  |_{a\otimes
b=\frac{1}{2}z_{s,t}\otimes z_{s,t}}-\widehat{\Gamma\left(  Y_{z_{s,t}%
}\right)  }Y_{z_{s,t}}F\\
&  =-\left(  \widehat{Y_{a}\Gamma\left(  Y_{b}\right)  }\right)  F\left(
u_{s}\right)  |_{a\otimes b=\mathbb{Z}_{s,t}}-\widehat{\Gamma\left(
Y_{z_{s,t}}\right)  }Y_{z_{s,t}}F
\end{align*}
which combined with Eq. (\ref{equ.12.8}) gives Eq. (\ref{equ.12.7}).
\end{proof}

Let $\pi_{M}$ and $\pi_{G}$ be the projection maps from $M\times G$ to $M$ and
$G$ respectively. We are now going to characterize $\mathbf{U}$ solving Eq.
(\ref{equ.12.6}) in terms of its components
\begin{equation}
\mathbf{G}:=\left(  \pi_{G}\right)  _{\ast}\left(  \mathbf{U}\right)  \text{
and }\mathbf{X}:=\left(  \pi_{M}\right)  _{\ast}\left(  \mathbf{U}\right)  .
\label{equ.12.9}%
\end{equation}

\begin{theorem}
\label{the.12.5}Let $\mathbf{U}$ solves Eq. (\ref{equ.12.6}), $\mathbf{G}%
=\left(  \pi_{G}\right)  _{\ast}\left(  \mathbf{U}\right)  ,$ $\mathbf{X}%
:=\left(  \pi_{M}\right)  _{\ast}\left(  \mathbf{U}\right)  ,$ and
\begin{equation}
\mathbf{A}:=\int\Gamma\left(  d\mathbf{X}\right)  \in W_{p}\left(
\mathfrak{g}\right)  . \label{equ.12.10}%
\end{equation}
Then $\mathbf{X}$ and $\mathbf{G}$ solve%
\begin{align}
d\mathbf{X}_{t}  &  =Y_{d\mathbf{Z}_{t}}\left(  x_{t}\right)  \text{ with
}x\left(  0\right)  =x_{0}\in M\label{equ.12.11}\\
d\mathbf{G}_{t}  &  =Y_{\mathbf{dA}}^{G}\left(  g_{t}\right)  \text{ with
}g\left(  0\right)  =e\in G. \label{equ.12.12}%
\end{align}

\end{theorem}

\begin{proof}
Since $\left(  \pi_{M}\right)  _{\ast}Y_{a}^{\Gamma}=Y_{a}\circ\pi_{M},$ the
assertion in Eq. (\ref{equ.12.11}) is a direct consequence of Theorem
\ref{the.4.11}. For the second equation we first observe that
\[
\left(  \pi_{G}\right)  _{\ast}Y_{a}^{\Gamma}\left(  m,g\right)
=Y_{\Gamma\left(  Y_{a}\left(  m\right)  \right)  }^{G}\left(  g\right)
\]
and so if $f\in C^{\infty}\left(  G\right)  $ then for $a,b\in\mathbb{R}^{n},$%
\[
Y_{b}^{\Gamma}\left(  f\circ\pi_{G}\right)  =\left(  Y_{\Gamma\left(
Y_{b}\left(  m\right)  \right)  }^{G}f\right)  \left(  g\right)
\]
and
\begin{align*}
Y_{a}^{\Gamma}Y_{b}^{\Gamma}\left(  f\circ\pi_{G}\right)   &  =Y_{a}^{\Gamma
}\left[  \left(  m,g\right)  \rightarrow\left(  Y_{\Gamma\left(  Y_{b}\left(
m\right)  \right)  }^{G}f\right)  \left(  g\right)  \right] \\
&  =\left(  Y_{Y_{a}\Gamma\left(  Y_{b}\left(  m\right)  \right)  }%
^{G}f\right)  \left(  g\right)  +\left(  Y_{\Gamma\left(  Y_{a}\left(
m\right)  \right)  }^{G}Y_{\Gamma\left(  Y_{b}\left(  m\right)  \right)  }%
^{G}f\right)  \left(  g\right)  .
\end{align*}
Thus applying Theorem \ref{the.4.5} to the function $f\circ\pi_{G}\in
C^{\infty}\left(  M\times G\right)  $ we discover that%
\begin{align}
f\left(  g_{t}\right)  -f\left(  g_{s}\right)   &  =f\circ\pi_{G}\left(
u_{t}\right)  -f\circ\pi_{G}\left(  u_{s}\right) \nonumber\\
&  \simeq\left(  \mathcal{Y^{\Gamma}}_{d\mathbf{Z}_{t}}\left(  f\circ\pi
_{G}\right)  \right)  \left(  x_{s}\right) \nonumber\\
&  =\left(  Y_{\Gamma\left(  Y_{z_{s,t}}\left(  x_{s}\right)  \right)  }%
^{G}f\right)  \left(  g_{s}\right)  +\left[  \left(  Y_{Y_{\left(
\cdot\right)  }\Gamma\left(  Y_{\left(  \cdot\right)  }\left(  x_{s}\right)
\right)  }^{G}f\right)  \left(  g_{s}\right)  +\left(  Y_{\Gamma\left(
Y_{\left(  \cdot\right)  }\left(  x_{s}\right)  \right)  }^{G}Y_{\Gamma\left(
Y_{\left(  \cdot\right)  }\left(  x_{s}\right)  \right)  }^{G}f\right)
\left(  g_{s}\right)  \right]  \mathbb{Z}_{s,t}. \label{equ.12.13}%
\end{align}
On the other hand from Theorem \ref{the.4.5}, Eq. (\ref{equ.12.12}) is
equivalent to%
\begin{equation}
f\left(  g_{t}\right)  -f\left(  g_{s}\right)  \simeq\left(  Y_{a_{s,t}}%
^{G}f\right)  \left(  g_{s}\right)  +\left(  \mathcal{Y}_{\mathbb{A}_{s,t}%
}^{G}f\right)  \left(  g_{s}\right)  , \label{equ.12.14}%
\end{equation}
where $\mathcal{Y}_{a\otimes b}^{G}:=Y_{a}Y_{b}.$ Since, by item 2. of Theorem
\ref{the.4.5},%
\begin{equation}
a_{s,t}=\left[  \int_{s}^{t}\Gamma\left(  d\mathbf{X}\right)  \right]
^{1}\simeq\Gamma\left(  Y_{z_{s,t}}\left(  x_{s}\right)  \right)  +\left[
Y_{\left(  \cdot\right)  }\left(  x_{s}\right)  \Gamma\left(  Y_{\left(
\cdot\right)  }\right)  \right]  \mathbb{Z}_{s,t}, \label{equ.12.15}%
\end{equation}
and%
\begin{equation}
\mathbb{A}_{s,t}=\left[  \int_{s}^{t}\Gamma\left(  d\mathbf{X}\right)
\right]  ^{2}\simeq\left[  \Gamma_{x_{s}}Y\left(  x_{s}\right)  \otimes
\Gamma_{x_{s}}Y\left(  x_{s}\right)  \right]  \mathbb{Z}_{s,t},
\label{equ.12.16}%
\end{equation}
we see that Eqs. (\ref{equ.12.13}) and (\ref{equ.12.14}) are indeed the same.
Moreover,
\begin{align*}
\mathbb{G}_{s,t}  &  \simeq\left(  \pi_{G}\right)  _{\ast}\otimes\left(
\pi_{G}\right)  _{\ast}\mathbb{U}_{s,t}\\
&  \simeq\left(  \pi_{G}\right)  _{\ast}\otimes\left(  \pi_{G}\right)  _{\ast
}Y_{\left(  \cdot\right)  }^{\Gamma}\left(  u_{t}\right)  \otimes Y_{\left(
\cdot\right)  }^{\Gamma}\left(  u_{t}\right)  \mathbb{Z}_{s,t}\\
&  =Y_{\Gamma\left(  Y_{\left(  \cdot\right)  }\left(  x_{s}\right)  \right)
}^{G}\left(  g_{s}\right)  \otimes Y_{\Gamma\left(  Y_{\left(  \cdot\right)
}\left(  x_{s}\right)  \right)  }^{G}\left(  g_{s}\right)  \mathbb{Z}_{s,t}\\
&  \simeq Y_{\left(  \cdot\right)  }^{G}\left(  g_{s}\right)  \otimes
Y_{\left(  \cdot\right)  }^{G}\left(  g_{s}\right)  \left[  \Gamma_{x_{s}%
}Y\left(  x_{s}\right)  \otimes\Gamma_{x_{s}}Y\left(  x_{s}\right)  \right]
\mathbb{Z}_{s,t}\\
&  \simeq Y_{\left(  \cdot\right)  }^{G}\left(  g_{s}\right)  \otimes
Y_{\left(  \cdot\right)  }^{G}\left(  g_{s}\right)  \mathbb{A}_{s,t}%
\end{align*}
which shows that the level two part of $\left(  \pi_{G}\right)  _{\ast}\left(
\mathbf{U}\right)  $ satisfies the equation dictated by Eq. (\ref{equ.12.12}).
\end{proof}

\begin{proposition}
\label{pro.12.6}If $\mathbf{X}$ solves Eq. (\ref{equ.12.11}) and $\mathbf{G}$
solves Eq. (\ref{equ.12.12}) (which exists on $\left[  0,T\right]  $ by
Theorem \ref{the.4.20}), then $u_{t}:=\left(  x_{t},g_{t}\right)  $ solves the
level one equations dictated by the RDE\ Eq. (\ref{equ.12.6}).
\end{proposition}

\begin{proof}
If $F\in C^{\infty}\left(  M\times G\right)  $ and $0\leq s<t\leq T,$ then
using Theorem \ref{the.4.5} for the RDE in Eq. (\ref{equ.12.11}) and then
using Theorem \ref{the.4.5} for the RDE in Eq. (\ref{equ.12.12}) we conclude%
\begin{align*}
F\left(  x_{t},g_{t}\right)  \simeq &  F\left(  x_{t},g_{s}\right)  -\left(
\left(  a_{s,t}+\mathbb{A}_{s,t}\right)  ^{\symbol{94}}F\right)  \left(
x_{t},g_{s}\right) \\
\simeq &  F\left(  x_{s},g_{s}\right)  +\left(  \mathcal{Y}_{\mathbf{Z}_{s,t}%
}F\right)  \left(  x_{s},g_{s}\right)  -\left(  \left(  a_{s,t}+\mathbb{A}%
_{s,t}\right)  ^{\symbol{94}}F\right)  \left(  x_{s},g_{s}\right) \\
&  -\left(  \mathcal{Y}_{\mathbf{Z}_{s,t}}\left(  a_{s,t}+\mathbb{A}%
_{s,t}\right)  ^{\symbol{94}}F\right)  \left(  x_{s},g_{s}\right) \\
\simeq &  F\left(  x_{s},g_{s}\right)  +\left(  \mathcal{Y}_{\mathbf{Z}_{s,t}%
}F\right)  \left(  x_{s},g_{s}\right)  -\left(  \left(  a_{s,t}+\mathbb{A}%
_{s,t}\right)  ^{\symbol{94}}F\right)  \left(  x_{s},g_{s}\right)  -\left(
Y_{z_{s,t}}\hat{a}_{s,t}F\right)  \left(  x_{s},g_{s}\right) \\
\simeq &  F\left(  x_{s},g_{s}\right)  +\left(  \mathcal{Y}_{\mathbf{Z}_{s,t}%
}F\right)  \left(  x_{s},g_{s}\right)  -\left(  \left(  a_{s,t}+\mathbb{A}%
_{s,t}\right)  ^{\symbol{94}}F\right)  \left(  x_{s},g_{s}\right)  -\left(
\hat{a}_{s,t}Y_{z_{s,t}}F\right)  \left(  x_{s},g_{s}\right)  .
\end{align*}
Replacing $a_{s,t}$ and $\mathbb{A}_{s,t}$ above by the approximate identities
in Eqs. (\ref{equ.12.15}) and (\ref{equ.12.16}) respectively shows%
\begin{align*}
\left[  F\left(  u\right)  \right]  _{s,t}  &  \simeq\left(  \mathcal{Y}%
_{\mathbf{Z}_{s,t}}F\right)  \left(  u_{s}\right) \\
&  -\left(  \left(  \Gamma\left(  Y_{z_{s,t}}\left(  x_{s}\right)  \right)
+\left[  Y_{\left(  \cdot\right)  }\left(  x_{s}\right)  \Gamma\left(
Y_{\left(  \cdot\right)  }\right)  \right]  \mathbb{Z}_{s,t}+\left[
\Gamma_{x_{s}}Y\left(  x_{s}\right)  \otimes\Gamma_{x_{s}}Y\left(
x_{s}\right)  \right]  \mathbb{Z}_{s,t}\right)  ^{\symbol{94}}F\right)
\left(  u_{s}\right) \\
&  -\left(  \widehat{\Gamma\left(  Y_{z_{s,t}}\left(  x_{s}\right)  \right)
}Y_{z_{s,t}}F\right)  \left(  u_{s}\right) \\
&  \simeq\left(  \mathcal{Y}_{\mathbf{Z}_{s,t}}F\right)  \left(  u_{s}\right)
-\left(  \left(  \Gamma\left(  Y_{z_{s,t}}\left(  x_{s}\right)  \right)
+\left[  \Gamma_{x_{s}}Y\left(  x_{s}\right)  \otimes\Gamma_{x_{s}}Y\left(
x_{s}\right)  \right]  \mathbb{Z}_{s,t}\right)  ^{\symbol{94}}F\right)
\left(  u_{s}\right) \\
&  -\left(  \left[  Y_{\left(  \cdot\right)  }\left(  x_{s}\right)
\Gamma\left(  Y_{\left(  \cdot\right)  }\right)  \right]  \mathbb{Z}%
_{s,t}\right)  ^{\symbol{94}}F\left(  u_{s}\right)  -\left(  \widehat{\Gamma
\left(  Y_{z_{s,t}}\left(  x_{s}\right)  \right)  }Y_{z_{s,t}}F\right)
\left(  u_{s}\right) \\
&  \simeq\left(  \left(  \mathcal{Y}_{\mathbf{Z}_{s,t}}+\mathcal{Y}%
_{\mathbf{Z}_{s,t}}^{G}\right)  F\right)  \left(  u_{s}\right)  -\left(
\left[  Y_{\left(  \cdot\right)  }\left(  x_{s}\right)  \Gamma\left(
Y_{\left(  \cdot\right)  }\right)  \right]  \mathbb{Z}_{s,t}\right)
^{\symbol{94}}F\left(  u_{s}\right)  -\left(  \widehat{\Gamma\left(
Y_{z_{s,t}}\left(  x_{s}\right)  \right)  }Y_{z_{s,t}}F\right)  \left(
u_{s}\right)
\end{align*}
which indeed matches Eq. (\ref{equ.12.7}).

\end{proof}

\begin{corollary}
\label{cor.12.7}If the solution $\mathbf{X}=\left(  x,\mathbb{X}\right)  $ to
Eq. (\ref{equ.12.12}) exists on $\left[  0,T\right]  ,$ then the solution
$\mathbf{U}\in WG_{p}\left(  M\times G\right)  $ to the RDE in Eq.
(\ref{equ.12.6}) exists on $\left[  0,T\right]  .$ (In other words, the
explosion time of $\mathbf{U}$ is the same as the explosion time for the level
one component $(x)$ of $\mathbf{X}.$)
\end{corollary}

\begin{proof}
Let $\mathbf{G}$ solves Eq. (\ref{equ.12.12}) which exists on $\left[
0,T\right]  $ by Theorem \ref{the.4.20}). Then by Proposition \ref{pro.12.6}
$u_{t}:=\left(  x_{t},g_{t}\right)  $ solves the level one equations dictated
by the RDE\ in Eq. (\ref{equ.12.6}).The result now follows by an application
of Lemma \ref{lem.2.18}.
\end{proof}

\begin{definition}
[Parallel Translation I.]\label{def.12.8}If $\Gamma\in\Omega^{1}\left(
M,\mathfrak{g}\right)  $ and $\mathbf{X\in}WG_{p}\left(  M\right)  ,$ we refer
to $\mathbf{U=}\left(  u:=\left(  x,g\right)  ,\mathbb{U}\right)  \in
WG_{p}\left(  M\times G\right)  $ solving Eq. (\ref{equ.12.6}) with $k=N,$
$\mathbf{Z}:=\mathbf{X},$ and $Y=V$ as in Example \ref{exa.3.7} as parallel
translation in the trivial principal bundle, $M\times G,$ relative to the
connection $\Gamma.$
\end{definition}

\subsection{Rough Parallel Translation on Trivial Vector
Bundles\label{sub.12.2}}

Let $\mathcal{E}$ be a finite dimensional vector space and $M\times
\mathcal{E}$ be the trivial vector bundle with fiber being $\mathcal{E}.$
(Eventually we will take $\mathcal{E}=\mathbb{R}^{N}$ -- the ambient space in
which $M$ is embedded.) As usual we identify sections of $M\times\mathcal{E}$
with functions, $T:M\rightarrow\mathcal{E}.$ We are now going to specialize
the results in the previous subsection to the case where $G=GL\left(
\mathcal{E}\right)  ,$ the matrix Lie group of invertible linear
transformations on $\mathcal{E}.$ In his case $\mathfrak{g}%
=\operatorname*{Lie}GL\left(  \mathcal{E}\right)  =\operatorname*{End}\left(
\mathcal{E}\right)  .$ Let us fix $\mathbf{X}\in WG_{p}\left(  M,\omega
\right)  $ and a connection one-form $\Gamma\in\Omega^{1}\left(
M,\operatorname*{End}\left(  \mathcal{E}\right)  \right)  .$

\begin{notation}
\label{not.12.9}The \textbf{covariant derivative} associated to $\Gamma$ of a
smooth function, $T:M\rightarrow\mathcal{E},$ is defined by%
\begin{equation}
\nabla_{v_{m}}T:=\left(  dT\right)  \left(  v_{m}\right)  +\Gamma\left(
v_{m}\right)  T\left(  m\right)  \text{ for all }v_{m}\in T_{m}M.
\label{equ.12.17}%
\end{equation}

\end{notation}

\begin{remark}
[Rough Parallel Translation]\label{rem.12.10}Recall that \textbf{rough
parallel translation }$\left(  \mathbf{G}:=\left(  g,\mathbb{G}\right)
\right)  $ along $\mathbf{X}$ relative to $\Gamma$ is the solution to the RDE
in Eq. (\ref{equ.12.12}) which in this setting is equivalent to $\mathbf{G}%
:=\left(  g,\mathbb{G}\right)  $ solving%
\begin{equation}
dg+\left(  d\mathbf{A}\right)  g=0\text{ with }g_{0}=I \label{equ.12.18}%
\end{equation}
where%
\begin{equation}
\mathbf{A}=\left(  a,\mathbb{A}\right)  :=\int\Gamma\left(  \mathbf{dX}%
\right)  \in WG_{p}\left(  \operatorname*{End}\left(  \mathcal{E}\right)
,\omega\right)  . \label{equ.12.19}%
\end{equation}

\end{remark}

\begin{theorem}
\label{the.12.11}Let $\mathbf{G}:=\left(  g,\mathbb{G}\right)  $ be rough
parallel translation along $\mathbf{X}=\left(  x,\mathbb{X}\right)  $ as in
Definition \ref{rem.12.10}. If $T:M\rightarrow\mathcal{E}$ is a smooth
function, then
\begin{equation}
\left[  g^{-1}T\left(  x\right)  \right]  _{s,t}\simeq g_{s}^{-1}\left[
\left[  \nabla_{V_{z_{s,t}}}T\right]  \left(  x_{s}\right)  +\left[
\nabla_{V_{\left(  \cdot\right)  }}\nabla_{V_{\left(  \cdot\right)  }%
}T\right]  _{\mathbb{Z}_{s,t}}\left(  x_{s}\right)  \right]  ,
\label{equ.12.20}%
\end{equation}
where $\nabla T:=dT+\Gamma T,$ and for $w,z\in\mathbb{R}^{d}$ and $x\in M,$
\[
\left[  \nabla_{V_{\left(  \cdot\right)  }}\nabla_{V_{\left(  \cdot\right)  }%
}T\right]  _{w\otimes z}\left(  x\right)  =\left[  \nabla_{V_{w}}\nabla
_{V_{z}}T\right]  \left(  x\right)  .
\]
In particular, if $T$ is $\nabla$ -- parallel, i.e. $\nabla T\equiv0,$ then
$g_{t}^{-1}T\left(  x_{t}\right)  =T\left(  x_{0}\right)  $ for all
$t\in\left[  0,T\right]  .$
\end{theorem}

\begin{proof}
If we let $F\left(  x,g\right)  :=g^{-1}T\left(  x\right)  $ we find
\[
\left(  V_{z}^{\Gamma}F\right)  \left(  x,g\right)  =g^{-1}dT\left(
P_{x}z\right)  +g^{-1}\Gamma\left(  P_{x}z\right)  T\left(  x\right)
=g^{-1}\left(  \nabla_{V_{z}}T\right)  \left(  x\right)  .
\]
Moreover if $z,w\in\mathbb{R}^{N}$ we have
\[
\left(  \mathcal{V^{\Gamma}}_{w\otimes z}F\right)  \left(  x,g\right)
=\left(  V_{w}^{\Gamma}V_{z}^{\Gamma}F\right)  \left(  x,g\right)
=g^{-1}\left(  \nabla_{V_{w}}\nabla_{V_{z}}T\right)  \left(  x\right)  .
\]
From Theorem \ref{the.4.5} and the last two equation we learn%
\[
\left[  F\left(  x,g\right)  \right]  _{s,t}\simeq\left(  \mathcal{V^{\Gamma}%
}_{z_{s,t}+\mathbb{Z}_{s,t}}F\right)  \left(  x_{s},g_{s}\right)  =g_{s}%
^{-1}\left[  \left[  \nabla_{V_{z_{s,t}}}T\right]  \left(  x_{s}\right)
+\left[  \nabla_{V_{\left(  \cdot\right)  }}\nabla_{V_{\left(  \cdot\right)
}}T\right]  _{\mathbb{Z}_{s,t}}\left(  x_{s}\right)  \right]  .
\]

\end{proof}

\begin{notation}
\label{not.12.12}Using the product rule, this covariant derivative in Notation
\ref{not.12.9} extends to a covariant derivative on any associated bundle to
$M\times\mathcal{E}$ and in particular on $M\times\operatorname*{End}\left(
\mathcal{E}\right)  .$ In detail, if $U:M\rightarrow\operatorname*{End}\left(
\mathcal{E}\right)  $ is a smooth function (thought of as a section of the
trivial vector bundle, $M\times\operatorname*{End}\left(  \mathcal{E}\right)
\mathcal{)},$ and $v_{m}\in T_{m}M$ then%
\[
\nabla_{v_{m}}U:=dU\left(  v_{m}\right)  +\left[  \Gamma\left(  v_{m}\right)
,U\left(  m\right)  \right]  =dU\left(  v_{m}\right)  +ad_{\Gamma\left(
v_{m}\right)  }U\left(  m\right)  \in\operatorname*{End}\left(  \mathcal{E}%
\right)  .
\]

\end{notation}

\begin{lemma}
\label{lem.12.13}Continuing the notation above, the product rule,%
\begin{equation}
\nabla_{v}\left[  UT\right]  =\left[  \nabla vU\right]  T\left(  x\right)
+U\left(  x\right)  \nabla_{v}T\text{ for all }v\in T_{x}M \label{equ.12.21}%
\end{equation}
for all smooth functions, $T:M\rightarrow\mathcal{E}$ and $U:M\rightarrow
\operatorname*{End}\left(  \mathcal{E}\right)  .$
\end{lemma}

\begin{proof}
The product rule is a simple consequence of the product rule for $d$ and the
definition of $\nabla;$%
\begin{align*}
\nabla\left[  UT\right]   &  =d\left[  UT\right]  +\Gamma\left[  UT\right]
=\left(  dU\right)  T+UdT+\Gamma UT\\
&  =\left(  dU\right)  T+UdT+\left[  \Gamma,U\right]  T+U\Gamma T\\
&  =\left(  dU\right)  T+\left[  \Gamma,U\right]  T+U\left(  dT+\Gamma
T\right) \\
&  =\left[  \nabla U\right]  T+U\left[  \nabla T\right]  .
\end{align*}

\end{proof}

\begin{lemma}
\label{lem.12.14}Parallel translation relative to $ad_{\Gamma\left(
\cdot\right)  }$ is given by $Ad_{g_{t}}$ where $Ad_{g}U:=gUg^{-1}$ and
$g_{t}$ is the the solution to Eq. (\ref{equ.12.18}).
\end{lemma}

\begin{proof}
From Theorem \ref{the.G.1} and the fact that $\rho\left(  g\right)  :=Ad_{g}$
is a homomorphism from $GL\left(  \mathcal{E}\right)  $ to $GL\left(
\operatorname*{End}\left(  \mathcal{E}\right)  \right)  $ such that
\begin{align*}
d\rho\left(  a\right)  b  &  :=\frac{d}{dt}|_{0}\rho\left(  e^{ta}\right)
b=\frac{d}{dt}|_{0}\left[  e^{ta}be^{-ta}\right] \\
&  =ab-ba=\left[  a,b\right]  =ad_{a}b
\end{align*}
for all $a,b\in\operatorname*{End}\left(  \mathcal{E}\right)  .$
\end{proof}

Applying this result with $\mathcal{E}$ replaced by $\operatorname*{End}%
\left(  \mathcal{E}\right)  $ and $\Gamma$ by $ad_{\Gamma}$ leads to the
following corollary.

\begin{corollary}
\label{cor.12.15}Continue the notation of Theorem \ref{the.12.11} but now
suppose that $U:M\rightarrow\operatorname*{End}\left(  \mathcal{E}\right)  $
is a smooth function which we view as a section of the trivial bundle,
$M\times\operatorname*{End}\left(  \mathcal{E}\right)  .$ Then
\begin{equation}
\left[  Ad_{g^{-1}}U\left(  x\right)  \right]  _{s,t}\simeq Ad_{g_{s}^{-1}%
}\left[  \left[  \nabla_{V_{z_{s,t}}}U\right]  \left(  x_{s}\right)  +\left[
\nabla_{V_{\left(  \cdot\right)  }}\nabla_{V_{\left(  \cdot\right)  }%
}U\right]  _{\mathbb{Z}_{s,t}}\left(  x_{s}\right)  \right]  ,/
\label{equ.12.22}%
\end{equation}
where $\nabla U:=dU+\left[  \Gamma,U\right]  ,$ and for $w,z\in\mathbb{R}^{d}$
and $x\in M,$
\[
\left[  \nabla_{V_{\left(  \cdot\right)  }}\nabla_{V_{\left(  \cdot\right)  }%
}U\right]  _{w\otimes z}\left(  x\right)  =\left[  \nabla_{V_{w}}\nabla
_{V_{z}}U\right]  \left(  x\right)  .
\]
In particular, if $U$ is $\nabla$ -- parallel, i.e. $\nabla U\equiv0,$ then
$g_{t}^{-1}U\left(  x_{t}\right)  =U\left(  x_{0}\right)  $ for all
$t\in\left[  0,T\right]  .$
\end{corollary}

\begin{corollary}
\label{cor.12.16}As in Lemma \ref{lem.G.4}, let $A$ be any
$\operatorname*{End}\left(  \mathbb{R}^{N}\right)  $ -- valued one-form on $M$
such that $\left[  A\left(  v_{x}\right)  ,P_{x}\right]  =0$ for all $x\in M$
and $v_{x}\in T_{x}M$ and define $\Gamma:=dPQ+dQP+A$ which is an
$\operatorname*{End}\left(  \mathbb{R}^{N}\right)  $ -- valued one-form on
$M.$ Then for every $\mathbf{X}\in WG_{p}\left(  M,\omega\right)  $ we have%
\[
\left[  Ad_{g}^{-1}P_{x}\right]  _{s,t}\simeq0\implies P_{x_{t}}g_{t}%
=g_{t}P_{x_{0}}\text{ for }0\leq t\leq T,
\]
where $\mathbf{G}=\left(  g_{s,t},\mathbb{G}_{s,t}\right)  $ is parallel
translation along $\mathbf{X}.$
\end{corollary}

\section{Boneyard Semko}

First we will need a lemma that is in the same spirit as Lemma \ref{lem.1.32}.
Its proof is the same besides the few obvious modifications.

\begin{lemma}
\label{lem.13.1}Suppose $\left(  TM,g\right)  $ and $\left(  TN,h\right)  $
are tangent bundles with inner products over $M$ and $N$ respectively. Let
$\zeta$ be a smooth function from $M$ to $N$ and consider the vector bundle
$\mathcal{\tilde{E}}:=\left(  \pi_{M}^{\ast}TM\otimes\pi_{M}^{\ast}%
TM\otimes\pi_{M}^{\ast}TM\right)  ^{\ast}\otimes\left(  \pi_{M}\circ
\zeta\right)  ^{\ast}TN$ on $TM$. Let $F:TM\rightarrow\mathcal{\tilde{E}}$ be
a smooth section. For every compact $K\subseteq TM$, there exists a $C_{K}$
such that
\begin{equation}
\left\vert F\left(  v\right)  \left(  v^{1}\otimes v^{2}\otimes v^{3}\right)
\right\vert _{h_{\zeta\left(  m\right)  }}\leq C_{K}\left\vert v^{1}%
\right\vert _{g_{m}}\left\vert v^{2}\right\vert _{g_{m}}\left\vert
v^{3}\right\vert _{g_{m}}%
\end{equation}
for $v\in K$.
\end{lemma}

Before we prove this, we will mention and prove a simple geometrical lemma.

\begin{lemma}
[Product rule]\label{lem.13.2}Let $m\in M$ and suppose $X,Y$ are vector fields
on $M.$ If $\phi$ is a chart such that $m\in D\left(  \phi\right)  $ then%
\[
XYf\left(  m\right)  =\left(  f\circ\phi^{-1}\right)  ^{\prime\prime}\left(
\phi\left(  m\right)  \right)  \left[  d\phi_{m}X\left(  m\right)  \otimes
d\phi_{m}Y\left(  m\right)  \right]  +\left(  f\circ\phi^{-1}\right)
^{\prime}\left(  \phi\left(  m\right)  \right)  \left[  \partial_{X\left(
m\right)  }d\phi\circ Y\right]
\]

\end{lemma}

\begin{proof}
We have
\begin{align*}
XYf\left(  m\right)   &  =X\left(  Y\left(  f\circ\phi^{-1}\circ\phi\right)
\right)  \left(  m\right) \\
&  =X\left(  \left(  f\circ\phi^{-1}\right)  ^{\prime}\left(  \phi\left(
\cdot\right)  \right)  d\phi\circ Y\left(  \cdot\right)  \right)  \left(
m\right) \\
&  =\left(  f\circ\phi^{-1}\right)  ^{\prime\prime}\left(  \phi\left(
m\right)  \right)  \left[  d\phi_{m}X\left(  m\right)  \otimes d\phi
_{m}Y\left(  m\right)  \right]  +\left(  f\circ\phi^{-1}\right)  ^{\prime
}\left(  \phi\left(  m\right)  \right)  \left[  \partial_{X\left(  m\right)
}d\phi\circ Y\right]
\end{align*}

\end{proof}

\begin{proof}
This will be similar to the proof of Theorem \ref{the.1.22}.

$\left(  \implies\right)  $First we show that if Definition \ref{def.5.2}
holds, then Equation (\ref{equ.4.8}) holds. This will mimic the proof of step
two of Theorem \ref{the.1.22}.

Let $a,b$ be such that $\left[  a,b\right]  \subseteq I_{0}$. Then around
every point $m\in y\left(  \left[  a,b\right]  \right)  $, we can find a chart
$\phi^{m}$ and an open $\mathcal{W}_{m}$ such that $\mathcal{W}_{m}\subseteq
D\left(  \phi^{m}\right)  $, $m\in\mathcal{W}_{m}$ and $\mathcal{V}_{m}%
:=\phi\left(  \mathcal{W}_{m}\right)  $ is convex. By using Remark
\ref{rem.1.33} on these sets, we may assume WLOG that $y\left(  \left[
a,b\right]  \right)  \subseteq\mathcal{W}$ where $\phi\left(  \mathcal{W}%
\right)  =\mathcal{V}$ is convex. We can then apply Taylor's Theorem to the
function%
\[
x\longrightarrow f\circ\phi^{-1}\left(  x\right)
\]
to see that%
\begin{align*}
f\circ\phi^{-1}\left(  \tilde{x}\right)   &  =f\circ\phi^{-1}\left(  x\right)
+\left(  f\circ\phi^{-1}\right)  ^{\prime}\left(  x\right)  \left[  \tilde
{x}-x\right]  +\frac{1}{2}\left(  f\circ\phi^{-1}\right)  ^{\prime\prime
}\left(  x\right)  \left[  \tilde{x}-x\right]  ^{\otimes2}\\
&  \quad+\frac{1}{2}\int_{0}^{1}\left(  f\circ\phi^{-1}\right)  ^{\prime
\prime\prime}\left(  t\left(  \tilde{x}-x\right)  +x\right)  \left[  \tilde
{x}-x\right]  ^{\otimes3}\left(  1-t\right)  ^{2}dt.
\end{align*}
We have%
\[
\left\vert \frac{1}{2}\int_{0}^{1}\left(  f\circ\phi^{-1}\right)
^{\prime\prime\prime}\left(  t\left(  \tilde{x}-x\right)  +x\right)  \left[
\tilde{x}-x\right]  ^{\otimes3}\left(  1-t\right)  ^{2}dt\right\vert \leq
C_{1}\left\vert \tilde{x}-x\right\vert ^{3}%
\]
provided $x,\tilde{x}$ come from a compact set (which the convex hull of
$\phi\left(  y\left(  \left[  a,b\right]  \right)  \right)  $ is). If we plug
in $\tilde{x}=\phi\left(  y_{t}\right)  $ and $x=\phi\left(  y_{s}\right)  $,
we see that%
\begin{align*}
f\left(  y_{t}\right)  -f\left(  y_{s}\right)   &  \underset{^{3}}{\approx
}\left(  f\circ\phi^{-1}\right)  ^{\prime}\left(  \phi\left(  y_{s}\right)
\right)  \left[  \phi\left(  y_{t}\right)  -\phi\left(  y_{s}\right)  \right]
+\frac{1}{2}\left(  f\circ\phi^{-1}\right)  ^{\prime\prime}\left(  \phi\left(
y_{s}\right)  \right)  \left[  \phi\left(  y_{t}\right)  -\phi\left(
y_{s}\right)  \right]  ^{\otimes2}\\
&  =\left(  f\circ\phi^{-1}\right)  ^{\prime}\left(  \phi\left(  y_{s}\right)
\right)  \left[  d\phi\circ F_{x_{s,t}}\left(  y_{s}\right)  +\left(
\partial_{F_{w}\left(  y_{s}\right)  }d\phi\circ F_{\tilde{w}}\right)
|_{w\otimes\tilde{w}=\mathbb{X}_{s,t}}\right] \\
&  \quad+\frac{1}{2}\left(  f\circ\phi^{-1}\right)  ^{\prime\prime}\left(
\phi\left(  y_{s}\right)  \right)  \left[  d\phi\circ F_{x_{s,t}}\left(
y_{s}\right)  \right]  ^{\otimes2}%
\end{align*}
by Definition \ref{def.5.2} and the fact that the other terms are
approximately zero. We have by the chain rule that%
\begin{align*}
\left(  f\circ\phi^{-1}\right)  ^{\prime}\left(  \phi\left(  y_{s}\right)
\right)  \left[  d\phi\circ F_{x_{s,t}}\left(  y_{s}\right)  \right]   &
=f_{\ast y_{s}}F_{x_{s,t}}\left(  y_{s}\right) \\
&  =\left(  F_{x_{s,t}}f\right)  \left(  y_{s}\right)  .
\end{align*}
Therefore, to finish this step, we just need to show that%
\[
\left(  f\circ\phi^{-1}\right)  ^{\prime}\left(  \phi\left(  y_{s}\right)
\right)  \left(  \partial_{F_{w}\left(  y_{s}\right)  }d\phi\circ F_{\tilde
{w}}\right)  |_{w\otimes\tilde{w}=\mathbb{X}_{s,t}}+\frac{1}{2}\left(
f\circ\phi^{-1}\right)  ^{\prime\prime}\left(  \phi\left(  y_{s}\right)
\right)  \left[  d\phi\circ F_{x_{s,t}}\left(  y_{s}\right)  \right]
^{\otimes2}=\left(  \mathcal{F}_{\mathbb{X}_{s,t}}f\right)  \left(
y_{s}\right)  .
\]
First, note that $\left(  f\circ\phi^{-1}\right)  ^{\prime\prime}\left(
\phi\left(  y_{s}\right)  \right)  $ is symmetric so that%
\begin{equation}
\frac{1}{2}\left(  f\circ\phi^{-1}\right)  ^{\prime\prime}\left(  \phi\left(
y_{s}\right)  \right)  \left[  d\phi\circ F_{x_{s,t}}\left(  y_{s}\right)
\right]  ^{\otimes2}=\left(  f\circ\phi^{-1}\right)  ^{\prime\prime}\left(
\phi\left(  y_{s}\right)  \right)  \left[  d\phi\circ F_{w}\left(
y_{s}\right)  \right]  \otimes\left[  d\phi\circ F_{\tilde{w}}\left(
y_{s}\right)  \right]  |_{w\otimes\tilde{w}=\mathbb{X}_{s,t}}.
\label{equ.13.2}%
\end{equation}
Thus when \textquotedblleft$w\otimes\tilde{w}=\mathbb{X}_{s,t}$%
\textquotedblright, we have by Lemma \ref{lem.13.2} with $X=F_{w}$,
$Y=F_{\tilde{w}}$and Equation (\ref{equ.13.2}) that
\begin{align*}
&  \left(  f\circ\phi^{-1}\right)  ^{\prime}\left(  \phi\left(  y_{s}\right)
\right)  \left(  \partial_{F_{w}\left(  y_{s}\right)  }d\phi\circ F_{\tilde
{w}}\right)  +\frac{1}{2}\left(  f\circ\phi^{-1}\right)  ^{\prime\prime
}\left(  \phi\left(  y_{s}\right)  \right)  \left[  d\phi\circ F_{x_{s,t}%
}\left(  y_{s}\right)  \right]  ^{\otimes2}\\
&  \quad=\left(  f\circ\phi^{-1}\right)  ^{\prime}\left(  \phi\left(
y_{s}\right)  \right)  \left(  \partial_{F_{w}\left(  y_{s}\right)  }%
d\phi\circ F_{\tilde{w}}\right)  +\left(  f\circ\phi^{-1}\right)
^{\prime\prime}\left(  \phi\left(  y_{s}\right)  \right)  \left[  d\phi\circ
F_{w}\left(  y_{s}\right)  \right]  \otimes\left[  d\phi\circ F_{\tilde{w}%
}\left(  y_{s}\right)  \right]  .\\
&  \quad=\left(  F_{w}F_{\tilde{w}}f\right)  \left(  y_{s}\right)
\end{align*}
Which is what we needed to prove.

$\left(  \Longleftarrow\right)  $For the reverse implication, we assume that
the approximation in Equation (\ref{equ.4.8}) holds for each $f\in C^{\infty
}\left(  M\right)  $. Let $\phi$ be a chart and let $a,b$ be such that
$\left[  a,b\right]  \subseteq I_{0}$ and $y\left(  \left[  a,b\right]
\right)  \subseteq D\left(  \phi\right)  $. Let $\mathcal{U\subseteq}D\left(
\phi\right)  $ be an open set in $M$ that contains $y\left(  \left[
a,b\right]  \right)  $ which has compact closure inside a domain of a chart,
we can find an open set which contains the interval which has compact closure.
Then using a smoothing function, we can manufacture global functions $f^{i}$
which agree with $\phi^{i}$ on $\mathcal{U}$. It is easy then to see that
Equation (\ref{equ.5.12}) holds.
\end{proof}

\subsection{Rough Parallel Translation}

Let $M^{d}$ be a Riemannian manifold and $O\left(  M\right)  $ be the
orthogonal frame bundle of $M$ where the fiber over $m$ is given by the
isometric linear transformations from $\mathbb{R}^{d}$ to $T_{m}M.$ Let $\pi$
denote the projection map and$.$let $\left(  y_{s},y_{s}^{\dag}\right)  $ be a
rough path controlled by $\mathbf{X}=\left(  x,\mathbb{X}\right)  .$

\begin{definition}
We call $\left(  u_{s},\acute{u}_{s}\right)  $ parallel translation along
$\left(  y_{s},y_{s}^{\dag}\right)  $ starting at $u_{0}$ if

\begin{enumerate}
\item $\pi\left(  u_{t}\right)  =y_{t}$ for all $t\in\left[  0,T\right]  $

\item $\int\omega\left(  d\left(  u_{t},\acute{u}_{t}\right)  \right)
=\left(  0,\acute{0}\right)  $ if $\omega$ is the $so\left(  d\right)  $
valued one-form defined by $\omega\left(  \dot{u}_{0}\right)  =u_{0}^{-1}%
\frac{\nabla u}{dt}\left(  0\right)  $
\end{enumerate}
\end{definition}

\begin{theorem}
Given $\left(  y,y^{\dag}\right)  \in CR_{\mathbf{X}}\left(  M\right)  $ and
$\pi\left(  u_{0}\right)  =y_{0}$, there exists a parallel translation locally.
\end{theorem}

Before proving this, we will set up some notation and prove some lemmas.

Let $\mathcal{O}$ be an open set containing $u_{0}$ such that $\left(
\pi,\phi\right)  :\mathcal{O}\longrightarrow V\subseteq M\times O_{d}$ is a
diffeomorphic trivialization, i.e. if $h\in O_{d}$ then
\[
\phi\left(  uh\right)  =\phi\left(  u\right)  h
\]
Next we define%
\[
\alpha:=\alpha^{\phi}:=\left[  \left(  \pi,\phi\right)  ^{-1}\right]  ^{\ast
}\omega
\]

\begin{lemma}
$\alpha$ satisfies the relation%
\[
\alpha\left(  v_{m},x_{g}\right)  =g^{-1}\alpha\left(  v_{m},0_{e}\right)
g+g^{-1}x_{g}%
\]

\end{lemma}

\begin{proof}
Let $\sigma_{t}$ be such that $\dot{\sigma}_{0}=v_{m}$ and let $\rho
_{t}=ge^{tg^{-1}x_{g}}.$ Then
\begin{align*}
\alpha\left(  v_{m},x_{g}\right)   &  =\left[  \left(  \pi,\phi\right)
^{-1}\right]  ^{\ast}\omega\left(  v_{m},x_{g}\right) \\
&  =\omega\left(  \left[  \left(  \pi,\phi\right)  ^{-1}\right]  _{\ast
}\left(  v_{m},x_{g}\right)  \right) \\
&  =\omega\left(  \frac{d}{dt}|_{t=0}\left(  \pi,\phi\right)  ^{-1}\left(
\sigma_{t},ge^{tg^{-1}x_{g}}\right)  \right) \\
&  =\omega\left(  \frac{d}{dt}|_{t=0}\left(  \pi,\phi\right)  ^{-1}\left(
\sigma_{t},e\right)  ge^{tg^{-1}x_{g}}\right) \\
&  =\omega\left(  \frac{d}{dt}|_{t=0}\left(  \pi,\phi\right)  ^{-1}\left(
\sigma_{t},e\right)  g+\frac{d}{dt}|_{t=0}\left(  \pi,\phi\right)
^{-1}\left(  m,e\right)  ge^{tg^{-1}x_{g}}\right) \\
&  =\omega\left(  \frac{d}{dt}|_{t=0}\left(  \pi,\phi\right)  ^{-1}\left(
\sigma_{t},e\right)  g\right)  +\omega\left(  \frac{d}{dt}|_{t=0}\left(
\pi,\phi\right)  ^{-1}\left(  m,g\right)  e^{tg^{-1}x_{g}}\right) \\
&  =g^{-1}\omega\left(  \frac{d}{dt}|_{t=0}\left(  \pi,\phi\right)
^{-1}\left(  \sigma_{t},e\right)  \right)  g+g^{-1}x_{g}%
\end{align*}

\end{proof}

Now let $\nabla:=\left(  \nabla^{M},\nabla^{O_{d}}\right)  $ where
$\nabla_{O_{d}}$ is the Levi-Civita covariant derivative determined by the
inner product $\left\langle A,B\right\rangle =Tr\left(  A^{t}B\right)  .$

We'll state the next lemma without proof

\begin{lemma}
If $X,Y$ are both either right or left invariant vector fields and $\nabla$ is
determined by a bi-invariant Riemannian metric, then%
\[
\nabla_{X}Y=\frac{1}{2}\left[  X,Y\right]
\]

\end{lemma}

\begin{corollary}
\label{cor.13.7}If $A,B\in so\left(  d\right)  $ then if $\tilde{A}\left(
g\right)  =gA$, $\hat{A}\left(  g\right)  =Ag$ then
\[
\nabla_{\tilde{A}}^{O_{d}}\tilde{B}=\frac{1}{2}\left[  \tilde{A},\tilde
{B}\right]  =\frac{1}{2}\widetilde{\left[  A,B\right]  }%
\]

and%
\[
\nabla_{\hat{A}}^{O_{d}}\hat{B}=\frac{1}{2}\left[  \hat{A},\hat{B}\right]
=\frac{1}{2}\widehat{\left[  B,A\right]  }%
\]

\end{corollary}

\begin{lemma}
$\nabla\alpha$ satisfies the relation%
\begin{align*}
\nabla\alpha\left(  \left(  v_{m},x_{g}\right)  \otimes\left(  w_{m}%
,y_{g}\right)  \right)   &  =g^{-1}\nabla\alpha\left(  \left(  v_{m}%
,0_{e}\right)  \otimes\left(  w_{m},0_{e}\right)  \right)  g-g^{-1}x_{g}%
g^{-1}\alpha\left(  w_{m},0_{e}\right)  g\\
&  +g^{-1}\alpha\left(  w_{m},0_{e}\right)  x_{g}-\frac{1}{2}g^{-1}\left[
x_{g}g^{-1}y_{g}-y_{g}g^{-1}x_{g}\right]
\end{align*}

\end{lemma}

\begin{proof}
We have%
\[
\nabla\alpha\left(  \left(  v_{m},x_{g}\right)  \otimes\left(  w_{m}%
,y_{g}\right)  \right)  =\left[  v_{m},x_{g}\right]  \alpha\left(  W,Y\right)
-\alpha\left(  \nabla_{\left(  v_{m},x_{g}\right)  }\left(  W,Y\right)
\right)
\]
where $W$ and $Y$ are any vector fields with $W\left(  m\right)  =w_{m}$ and
$Y\left(  g\right)  =y_{g}$. We'll choose $Y\left(  h\right)  =hg^{-1}%
y_{g}=\widetilde{g^{-1}y_{g}}\left(  h\right)  $ and $\rho_{t}=ge^{tg^{-1}%
x_{g}}.$ Let $\sigma_{t}$ be such that $\sigma_{0}=m,\dot{\sigma}_{0}=v_{m}$.

Also, note that $\rho_{t}^{-1}Y\left(  \rho_{t}\right)  =\rho_{t}^{-1}\rho
_{t}g^{-1}y_{g}=g^{-1}y_{g}$. Examining the first term, we have%
\begin{align*}
\left[  v_{m},x_{g}\right]  \alpha\left(  W,Y\right)   &  =\frac{d}{dt}%
|_{t=0}\alpha\left(  W\left(  \sigma_{t}\right)  ,Y\left(  \rho_{t}\right)
\right) \\
&  =\frac{d}{dt}|_{t=0}\alpha\left(  W\left(  \sigma_{t}\right)  ,Y\left(
\rho_{0}\right)  \right)  +\frac{d}{dt}|_{t=0}\alpha\left(  W\left(
\sigma_{0}\right)  ,Y\left(  \rho_{t}\right)  \right) \\
&  =g^{-1}\frac{d}{dt}|_{t=0}\alpha\left(  W\left(  \sigma_{t}\right)
,0_{e}\right)  g+\frac{d}{dt}|_{t=0}\rho_{t}^{-1}\alpha\left(  W\left(
\sigma_{0}\right)  ,0_{e}\right)  \rho_{t}+\frac{d}{dt}|_{t=0}\rho_{t}%
^{-1}Y\left(  \rho_{t}\right) \\
&  =g^{-1}\frac{d}{dt}|_{t=0}\alpha\left(  W\left(  \sigma_{t}\right)
,0_{e}\right)  g+\frac{d}{dt}|_{t=0}\rho_{t}^{-1}\alpha\left(  W\left(
\sigma_{0}\right)  ,0_{e}\right)  \rho_{t}\\
&  =g^{-1}\left[  v_{m},0_{e}\right]  \alpha\left(  W,0\right)  g-g^{-1}%
x_{g}g^{-1}\alpha\left(  w_{m},0_{e}\right)  g+g^{-1}\alpha\left(  w_{m}%
,0_{e}\right)  x_{g}\\
&  =g^{-1}\nabla\alpha\left(  \left(  v_{m},0_{e}\right)  \otimes\left(
w_{m},0_{e}\right)  \right)  g+g^{-1}\alpha\left(  \nabla_{v_{m}}%
W,0_{e}\right)  g-g^{-1}x_{g}g^{-1}\alpha\left(  w_{m},0_{e}\right)
g+g^{-1}\alpha\left(  w_{m},0_{e}\right)  x_{g}%
\end{align*}
Examining the second term, we have%
\[
\alpha\left(  \nabla_{\left(  v_{m},x_{g}\right)  }\left(  W,Y\right)
\right)  =\alpha\left(  \nabla_{v_{m}}W,\nabla_{x_{g}}Y\right)
\]
Lets first compute $\nabla_{x_{g}}Y$. We note that $X\left(  h\right)
=hg^{-1}x_{g}=\widetilde{g^{-1}x_{g}}\left(  h\right)  $ has the property that
$X\left(  g\right)  =x_{g}.$ Thus%
\begin{align*}
\left(  \nabla_{X}Y\right)  \left(  g\right)   &  =\frac{1}{2}\left[
XY-YX\right]  \left(  g\right) \\
&  =\frac{1}{2}\widetilde{\left[  g^{-1}x_{g}g^{-1}y_{g}-g^{-1}y_{g}%
g^{-1}x_{g}\right]  }\left(  g\right) \\
&  =\frac{1}{2}\left[  x_{g}g^{-1}y_{g}-y_{g}g^{-1}x_{g}\right]
\end{align*}
Thus we have%
\begin{align*}
\alpha\left(  \nabla_{\left(  v_{m},x_{g}\right)  }\left(  W,Y\right)
\right)   &  =\alpha\left(  \nabla_{v_{m}}W,\nabla_{x_{g}}Y\right) \\
&  =g^{-1}\alpha\left(  \nabla_{v_{m}}W,0_{e}\right)  g+\frac{1}{2}%
g^{-1}\left[  x_{g}g^{-1}y_{g}-y_{g}g^{-1}x_{g}\right]
\end{align*}
Putting the two together, we have%
\begin{align*}
\nabla\alpha\left(  \left(  v_{m},x_{g}\right)  \otimes\left(  w_{m}%
,y_{g}\right)  \right)   &  =\left[  v_{m},x_{g}\right]  \alpha\left(
W,Y\right)  -\alpha\left(  \nabla_{\left(  v_{m},x_{g}\right)  }\left(
W,Y\right)  \right) \\
&  =g^{-1}\nabla\alpha\left(  \left(  v_{m},0_{e}\right)  \otimes\left(
w_{m},0_{e}\right)  \right)  g+g^{-1}\alpha\left(  \nabla_{v_{m}}%
W,0_{e}\right)  g\\
&  -g^{-1}x_{g}g^{-1}\alpha\left(  w_{m},0_{e}\right)  g+g^{-1}\alpha\left(
w_{m},0_{e}\right)  x_{g}\\
&  -\left(  g^{-1}\alpha\left(  \nabla_{v_{m}}W,0_{e}\right)  g+\frac{1}%
{2}g^{-1}\left[  x_{g}g^{-1}y_{g}-y_{g}g^{-1}x_{g}\right]  \right) \\
&  =g^{-1}\nabla\alpha\left(  \left(  v_{m},0_{e}\right)  \otimes\left(
w_{m},0_{e}\right)  \right)  g-g^{-1}x_{g}g^{-1}\alpha\left(  w_{m}%
,0_{e}\right)  g\\
&  +g^{-1}\alpha\left(  w_{m},0_{e}\right)  x_{g}-\frac{1}{2}g^{-1}\left[
x_{g}g^{-1}y_{g}-y_{g}g^{-1}x_{g}\right]
\end{align*}

\end{proof}

Now we are ready to prove the theorem.

\begin{proof}
[Proof of Theorem]Let
\[
\left(  a_{s},\mathbb{A}_{s,t}\right)  :=\int\alpha\left(  d\left(  \left(
y,0_{e}\right)  ,\left(  y^{\dag},\acute{0}\right)  \right)  \right)  .
\]
This is a rough path living in $so\left(  d\right)  \oplus so\left(  d\right)
^{\otimes2}$. Note that $\mathbb{A}_{s,t}\approx\left(  \alpha\circ
y_{s}^{\dag}\right)  ^{\otimes2}\mathbb{X}_{s,t}.$ We can then solve the rough
differential equation%
\[
dg=V_{d\mathbf{A}}\left(  g\right)
\]
with $g_{0}=\phi\left(  u_{0}\right)  $ and $V_{a}\left(  g\right)
=-ag=\widehat{-a}\left(  g\right)  .$ If we then define%
\[
\acute{g}_{s}:=-\alpha\left(  y_{s}^{\dag}\left(  \cdot\right)  \right)  g_{s}%
\]
it is easy to check that the pair $\left(  g_{s},\acute{g}_{s}\right)  $ is a
rough path on $O_{d}$ controlled by $\mathbf{X}.$We will now show that
\[
\left(  u_{s},\acute{u}_{s}\right)  :=\left[  \left(  \pi,\phi\right)
^{-1}\right]  _{\ast}\left(  \left(  y_{s,}g_{s}\right)  ,\left(  y_{s,}%
^{\dag}\acute{g}_{s}\right)  \right)
\]
is parallel translation. To show this, it suffices to show that
\[
\left[  \int\alpha\left(  d\left(  \left(  y,g\right)  ,\left(  y^{\dag
},\acute{g}\right)  \right)  \right)  \right]  _{s}=\left(  0,\acute
{0}\right)
\]
By definition of $g_{s,}$we have
\begin{align*}
\exp_{g_{s}}^{-1}g_{t}  &  \approx-a_{s,t}g_{s}+\left(  \nabla_{\widehat{-a}%
}\widehat{-b}\right)  \left(  g_{s}\right)  |_{a\otimes b=\mathbb{A}_{s,t}}\\
&  \approx-a_{s,t}g_{s}+\left(  \nabla_{\widehat{-a}}\widehat{-b}\right)
\left(  g_{s}\right)  |_{a\otimes b=\left(  \alpha\circ y_{s}^{\dag}\right)
^{\otimes2}\mathbb{X}_{s,t}}%
\end{align*}
From corollary \ref{cor.13.7} we have
\[
\left(  \nabla_{\widehat{-a}}\widehat{-b}\right)  \left(  g_{s}\right)
=\frac{1}{2}\widehat{\left(  ab-ba\right)  }\left(  g_{s}\right)  =-\frac
{1}{2}\left(  ab-ba\right)  g_{s}%
\]
Now%
\[
\left[  \int\alpha\left(  d\left(  \left(  y,g\right)  ,\left(  y^{\dag
},\acute{g}\right)  \right)  \right)  \right]  _{s,t}^{1}\approx\alpha\left(
\exp_{y_{s}}^{-1}y_{t},\exp_{g_{s}}^{-1}g_{t}\right)  +\nabla\alpha\left(
\left(  y_{s}^{\dag},\acute{g}_{s}\right)  ^{\otimes2}\mathbb{X}_{s,t}\right)
\]
We have%
\begin{align*}
\alpha\left(  \exp_{y_{s}}^{-1}y_{t},\exp_{g_{s}}^{-1}g_{t}\right)   &
=g_{s}^{-1}\alpha\left(  \exp_{y_{s}}^{-1}y_{t},0_{e}\right)  g_{s}+g_{s}%
^{-1}\exp_{g_{s}}^{-1}g_{t}\\
&  \approx g_{s}^{-1}\alpha\left(  \exp_{y_{s}}^{-1}y_{t},0_{e}\right)
g_{s}+g_{s}^{-1}\left[  -a_{s,t}g_{s}-\frac{1}{2}\left(  ab-ba\right)
g_{s}|_{a\otimes b=\left(  \alpha\circ y_{s}^{\dag}\right)  ^{\otimes
2}\mathbb{X}_{s,t}}\right] \\
&  \approx g_{s}^{-1}\alpha\left(  \exp_{y_{s}}^{-1}y_{t},0_{e}\right)
g_{s}\\
&  +g_{s}^{-1}\left[  -\left[  \alpha\left(  \exp_{y_{s}}^{-1}y_{t}%
,0_{e}\right)  +\nabla\alpha\left(  \left(  y_{s},0_{e}\right)  ^{\otimes
2}\mathbb{X}_{s,t}\right)  \right]  g_{s}-\frac{1}{2}\left(  ab-ba\right)
g_{s}|_{a\otimes b=\left(  \alpha\circ y_{s}^{\dag}\right)  ^{\otimes
2}\mathbb{X}_{s,t}}\right] \\
&  =-g_{s}^{-1}\nabla\alpha\left(  \left(  y_{s},0_{e}\right)  ^{\otimes
2}\mathbb{X}_{s,t}\right)  g_{s}-\frac{1}{2}g_{s}^{-1}\left(  ab-ba\right)
g_{s}|_{a\otimes b=\left(  \alpha\circ y_{s}^{\dag}\right)  ^{\otimes
2}\mathbb{X}_{s,t}}%
\end{align*}
Therefore%
\begin{align*}
&  \left[  \int\alpha\left(  d\left(  \left(  y,g\right)  ,\left(  y^{\dag
},\acute{g}\right)  \right)  \right)  \right]  _{s,t}^{1}\\
&  \approx\alpha\left(  \exp_{y_{s}}^{-1}y_{t},\exp_{g_{s}}^{-1}g_{t}\right)
+\nabla\alpha\left(  \left(  y_{s}^{\dag},\acute{g}_{s}\right)  ^{\otimes
2}\mathbb{X}_{s,t}\right) \\
&  \approx-g_{s}^{-1}\nabla\alpha\left(  \left(  y_{s}^{\dag},0_{e}\right)
^{\otimes2}\mathbb{X}_{s,t}\right)  g_{s}-\frac{1}{2}g_{s}^{-1}\left(
ab-ba\right)  g_{s}|_{a\otimes b=\left(  \alpha\circ y_{s}^{\dag}\right)
^{\otimes2}\mathbb{X}_{s,t}}+\nabla\alpha\left(  \left(  y_{s}^{\dag}%
,\acute{g}_{s}\right)  ^{\otimes2}\mathbb{X}_{s,t}\right) \\
&  \approx-g_{s}^{-1}\nabla\alpha\left(  \left(  y_{s}^{\dag}a,0_{e}\right)
\mathbb{\otimes}\left(  y_{s}^{\dag}b,0_{e}\right)  \right)  g_{s}|_{a\otimes
b=\mathbb{X}_{s,t}}-\frac{1}{2}g_{s}^{-1}\left(  ab-ba\right)  g_{s}%
|_{a\otimes b=\left(  \alpha\circ y_{s}^{\dag}\right)  ^{\otimes2}%
\mathbb{X}_{s,t}}\\
&  +\nabla\alpha\left(  \left(  y_{s}^{\dag}a,\acute{g}_{s}a\right)
\otimes\left(  y_{s}^{\dag}b,\acute{g}_{s}b\right)  \right)  |_{a\otimes
b=\mathbb{X}_{s,t}}\\
&  \approx-g_{s}^{-1}\nabla\alpha\left(  \left(  y_{s}^{\dag}a,0_{e}\right)
\mathbb{\otimes}\left(  y_{s}^{\dag}b,0_{e}\right)  \right)  g_{s}-\frac{1}%
{2}g_{s}^{-1}\left[  \alpha\left(  y_{s}^{\dag}a,0_{e}\right)  \alpha\left(
y_{s}^{\dag}b,0_{e}\right)  -\alpha\left(  y_{s}^{\dag}b,0_{e}\right)
\alpha\left(  y_{s}^{\dag}a,0_{e}\right)  \right]  g_{s}\\
&  +\nabla\alpha\left(  \left(  y_{s}^{\dag}a,\acute{g}_{s}a\right)
\otimes\left(  y_{s}^{\dag}b,\acute{g}_{s}b\right)  \right)  |_{a\otimes
b=\mathbb{X}_{s,t}}\\
&  =Ad_{g_{s}^{-1}}\left[  -\nabla\alpha\left(  \left(  y_{s}^{\dag}%
a,0_{e}\right)  \mathbb{\otimes}\left(  y_{s}^{\dag}b,0_{e}\right)  \right)
-\frac{1}{2}\left[  \alpha\left(  y_{s}^{\dag}a,0_{e}\right)  \alpha\left(
y_{s}^{\dag}b,0_{e}\right)  -\alpha\left(  y_{s}^{\dag}b,0_{e}\right)
\alpha\left(  y_{s}^{\dag}a,0_{e}\right)  \right]  \right] \\
&  +\nabla\alpha\left(  \left(  y_{s}^{\dag}a,\acute{g}_{s}a\right)
\otimes\left(  y_{s}^{\dag}b,\acute{g}_{s}b\right)  \right)  |_{a\otimes
b=\mathbb{X}_{s,t}}\\
&  =Ad_{g_{s}^{-1}}\left[
\begin{array}
[c]{c}%
-\nabla\alpha\left(  \left(  y_{s}^{\dag}a,0_{e}\right)  \mathbb{\otimes
}\left(  y_{s}^{\dag}b,0_{e}\right)  \right)  -\frac{1}{2}\left[
\alpha\left(  y_{s}^{\dag}a,0_{e}\right)  \alpha\left(  y_{s}^{\dag}%
b,0_{e}\right)  -\alpha\left(  y_{s}^{\dag}b,0_{e}\right)  \alpha\left(
y_{s}^{\dag}a,0_{e}\right)  \right] \\
+\nabla\alpha\left(  \left(  y_{s}^{\dag}a,0_{e}\right)  \mathbb{\otimes
}\left(  y_{s}^{\dag}b,0_{e}\right)  \right)  -\acute{g}_{s}ag_{s}^{-1}%
\alpha\left(  y_{s}^{\dag}b,0_{e}\right) \\
+\alpha\left(  y_{s}^{\dag}b,0_{e}\right)  \acute{g}_{s}ag_{s}^{-1}-\frac
{1}{2}\left[  \acute{g}_{s}ag_{s}^{-1}\acute{g}_{s}b-\acute{g}_{s}bg_{s}%
^{-1}g_{s}a\right]  g_{s}^{-1}%
\end{array}
\right]  |_{a\otimes b=\mathbb{X}_{s,t}}\\
&  =Ad_{g^{-1}}\left[
\begin{array}
[c]{c}%
-\frac{1}{2}\left[  \alpha\left(  y_{s}^{\dag}a,0_{e}\right)  \alpha\left(
y_{s}^{\dag}b,0_{e}\right)  -\alpha\left(  y_{s}^{\dag}b,0_{e}\right)
\alpha\left(  y_{s}^{\dag}a,0_{e}\right)  \right]  -\acute{g}_{s}ag_{s}%
^{-1}\alpha\left(  y_{s}^{\dag}b,0_{e}\right) \\
+\alpha\left(  y_{s}^{\dag}b,0_{e}\right)  \acute{g}_{s}ag_{s}^{-1}-\frac
{1}{2}\left[  \acute{g}_{s}ag_{s}^{-1}\acute{g}_{s}b-\acute{g}_{s}bg_{s}%
^{-1}g_{s}a\right]  g_{s}^{-1}%
\end{array}
\right]  |_{a\otimes b=\mathbb{X}_{s,t}}\\
&  =Ad_{g^{-1}}\left[
\begin{array}
[c]{c}%
-\frac{1}{2}\left[  \alpha\left(  y_{s}^{\dag}a,0_{e}\right)  \alpha\left(
y_{s}^{\dag}b,0_{e}\right)  -\alpha\left(  y_{s}^{\dag}b,0_{e}\right)
\alpha\left(  y_{s}^{\dag}a,0_{e}\right)  \right]  +\alpha\left(  y_{s}^{\dag
}a,0_{e}\right)  \alpha\left(  y_{s}^{\dag}b,0_{e}\right) \\
-\alpha\left(  y_{s}^{\dag}b,0_{e}\right)  \alpha\left(  y_{s}^{\dag}%
a,0_{e}\right)  -\frac{1}{2}\left[  \alpha\left(  y_{s}^{\dag}a,0_{e}\right)
\alpha\left(  y_{s}^{\dag}b,0_{e}\right)  -\alpha\left(  y_{s}^{\dag}%
b,0_{e}\right)  \alpha\left(  y_{s}^{\dag}a,0_{e}\right)  \right]
\end{array}
\right]  |_{a\otimes b=\mathbb{X}_{s,t}}\\
&  =Ad_{g^{-1}}\left[  0\right]  |_{a\otimes b=\mathbb{X}_{s,t}}\\
&  =0
\end{align*}

Next we have to check that the derivative process of $\int\alpha\left(
d\left(  \left(  y,g\right)  ,\left(  y^{\dag},\acute{g}\right)  \right)
\right)  $ is 0. By definition, it is
\[
\alpha\circ\left(  y_{s}^{\dag},\acute{g}_{s}\right)
\]
Acting on a vector $v$ in $\mathbb{R}^{n}$ it is given by%
\begin{align*}
&  \alpha\circ\left(  y_{s}^{\dag},\acute{g}_{s}\right)  \left(  v\right) \\
&  =\alpha\left(  y_{s}^{\dag}v,\acute{g}_{s}v\right) \\
&  =g_{s}^{-1}\alpha\left(  y_{s}^{\dag}v,0_{e}\right)  g_{s}+g_{s}^{-1}%
\acute{g}_{s}v\\
&  =g_{s}^{-1}\alpha\left(  y_{s}^{\dag}v,0_{e}\right)  g_{s}-g_{s}^{-1}%
\alpha\left(  y_{s}^{\dag}v,0_{e}\right)  g_{s}\\
&  =0
\end{align*}

\end{proof}

We have%
\begin{align*}
0  &  =\left[  \int\alpha\left(  d\left(  \left(  y,g\right)  ,\left(
y^{\dag},\acute{g}\right)  \right)  \right)  \right]  _{s,t}^{1}\\
&  \approx\alpha\left(  \exp_{y_{s}}^{-1}y_{t},\exp_{g_{s}}^{-1}g_{t}\right)
+\nabla\alpha\left(  \left(  y_{s}^{\dag},\acute{g}_{s}\right)  ^{\otimes
2}\mathbb{X}_{s,t}\right) \\
&  \approx g_{s}^{-1}\alpha\left(  \exp_{y_{s}}^{-1}y_{t},0_{e}\right)
g_{s}+g_{s}^{-1}\exp_{g_{s}}^{-1}g_{t}\\
&  +g_{s}^{-1}\nabla\alpha\left(  \left(  y_{s}^{\dag}a,0_{e}\right)
\mathbb{\otimes}\left(  y_{s}^{\dag}b,0_{e}\right)  \right)  g_{s}-g_{s}%
^{-1}\acute{g}_{s}ag_{s}^{-1}\alpha\left(  y_{s}^{\dag}b,0_{e}\right)  g_{s}\\
&  +g_{s}^{-1}\alpha\left(  y_{s}^{\dag}b,0_{e}\right)  \acute{g}_{s}%
a-\frac{1}{2}\left[  g_{s}^{-1}\acute{g}_{s}ag_{s}^{-1}\acute{g}_{s}%
b-g_{s}^{-1}\acute{g}_{s}bg_{s}^{-1}\acute{g}_{s}a\right]
\end{align*}
Which means%
\begin{align*}
\exp_{g_{s}}^{-1}g_{t}  &  \approx-\alpha\left(  \exp_{y_{s}}^{-1}y_{t}%
,0_{e}\right)  g_{s}-\nabla\alpha\left(  \left(  y_{s}^{\dag}a,0_{e}\right)
\mathbb{\otimes}\left(  y_{s}^{\dag}b,0_{e}\right)  \right)  g_{s}+\acute
{g}_{s}ag_{s}^{-1}\alpha\left(  y_{s}^{\dag}b,0_{e}\right)  g_{s}\\
&  -\alpha\left(  y_{s}^{\dag}b,0_{e}\right)  \acute{g}_{s}a+\frac{1}%
{2}\left[  \acute{g}_{s}ag_{s}^{-1}\acute{g}_{s}b-\acute{g}_{s}bg_{s}%
^{-1}g_{s}a\right] \\
&  \approx a_{s,t}g_{s}+\acute{g}_{s}ag_{s}^{-1}\alpha\left(  y_{s}^{\dag
}b,0_{e}\right)  g_{s}-\alpha\left(  y_{s}^{\dag}b,0_{e}\right)  \acute{g}%
_{s}a+\frac{1}{2}\left[  \acute{g}_{s}ag_{s}^{-1}\acute{g}_{s}b-\acute{g}%
_{s}bg_{s}^{-1}\acute{g}_{s}a\right]
\end{align*}

\input{end}

\bibliographystyle{plain}
\bibliography{CRPI}

\end{document}